\documentclass[reqno,twoside,12pt]{article}
 \usepackage[english]{babel}
 \usepackage{times}
 \usepackage[
    paper=a4paper,
    portrait=true,
    textwidth=425pt,
    textheight=650pt,
    tmargin=3cm,
    marginratio=1:1
            ]{geometry}
 \usepackage[all]{xy}       
 \usepackage[centertags]{amsmath}
 \usepackage{amsfonts}
 \usepackage{amsthm}
 \usepackage{amssymb}
 \usepackage{enumerate}
 \usepackage{mathrsfs}
 \usepackage{titling}

 \theoremstyle{plain}
 \newtheorem{Thm}{Theorem}[section]
 \newtheorem{Cor}[Thm]{Corollary}
 \newtheorem{Lemma}[Thm]{Lemma}
 \newtheorem{Prop}[Thm]{Proposition}

 \newtheorem{introtheorem}{Theorem}
 
 \newtheorem{IntroThm}[introtheorem]{Theorem}

 \theoremstyle{definition}
 \newtheorem{Rem}[Thm]{Remark}
 \newtheorem{Ex}[Thm]{Example}

 \numberwithin{Thm}{section}
 \numberwithin{equation}{section}

\def\qquad{\quad\quad}

\def\msy#1{{\mathbb #1}}
\def\C{{\msy C}}

\def\N{{\msy N}}
\def\P{{\msy P}}
\def\Q{{\msy Q}}
\def\R{{\msy R}}
\def\Z{{\msy Z}}

\def\fa{{\mathfrak a}}
\def\fb{{\mathfrak b}}

\def\fg{{\mathfrak g}}
\def\fh{{\mathfrak h}}
\def\fj{{\mathfrak j}}
\def\fk{{\mathfrak k}}
\def\fl{{\mathfrak l}}
\def\fm{{\mathfrak m}}
\def\fn{{\mathfrak n}}
\def\fp{{\mathfrak p}}
\def\fq{{\mathfrak q}}

\def\fs{{\mathfrak s}}
\def\ft{{\mathfrak t}}

\def\fA{{\mathfrak A}}
\def\fB{{\mathfrak B}}
\def\fC{{\mathfrak C}}
\def\fN{{\mathfrak N}}
\def\fU{{\mathfrak U}}

\def\cA{{\mathcal A}}
\def\cB{{\mathcal B}}
\def\cC{{\mathcal C}}
\def\cD{{\mathcal D}}
\def\cE{{\mathcal E}}
\def\cF{{\mathcal F}}

\def\cH{{\mathcal H}}
\def\cI{{\mathcal I}}

\def\cL{{\mathcal L}}
\def\cM{{\mathcal M}}
\def\cN{{\mathcal N}}
\def\cO{{\mathcal O}}
\def\cP{{\mathcal P}}
\def\cQ{{\mathcal Q}}
\def\cR{{\mathcal R}}
\def\cS{{\mathcal S}}
\def\cT{{\mathcal T}}
\def\cU{{\mathcal U}}
\def\cV{{\mathcal V}}
\def\cW{{\mathcal W}}
\def\cX{{\mathcal X}}

\def\cZ{{\mathcal Z}}

\def\Ft{\mathscr{F}}
\def\WP{\cW\cP}
\def\Const{\mathrm{CT}}
\def\const{\Gamma}

\def\to{\rightarrow}
\def\Re{\mathrm{Re}\,}
\def\Im{\mathrm{Im}\,}
\def\Ad{\mathrm{Ad}}
\def\End{\mathrm{End}}

\def\Hom{\mathrm{Hom}}
\def\Id{\mathrm{Id}}
\def\ad{\mathrm{ad}}
\def\pr{\mathrm{pr}}
\def\tr{\mathrm{tr}\,}
\def\spec{\mathrm{spec}\,}
\def\trdeg{\mathrm{trdeg}\,}

\def\supp{\mathop{\rm supp}}
\def\Lie{\mathop{\rm Lie}}
\def\Ind{\mathrm{Ind}}
\def\dotvar{\, \cdot\,}
\def\diag{\mathrm{diag}}

\def\1{\mathbf{1}}

\def\tds{\mathrm{tds}}
\def\mc{\mathrm{mc}}

\def\bs{\backslash}

\def\spn{\mathrm{span}}
\def\eucl{\mathrm{eucl}}
\def\rank{\mathrm{rank}}
\def\Cen{\mathfrak{Z}}
\def\Sym{\mathrm{Sym}}
\def\Gr{\mathrm{Grass}}

\def\bp{{}^\backprime}

\def\nc{\mathrm{nc}}
\def\nm{\mathrm{nm}}
\def\open{\mathrm{open}}
\def\Diff{\mathbb{D}}
\def\ev{\mathrm{ev}}
\def\Pl{\mathrm{Pl}}
\def\Iwasawa{a}
\def\hor{\mathrm{hor}}

 \title{The most continuous part of the Plancherel decomposition for a real spherical space}
 \author{Job J.~Kuit, Eitan Sayag}
 \predate{}
 \postdate{}
 \date{}        

\begin{document}
 \maketitle
\begin{abstract}
In this article we give a precise description of the Plancherel decomposition of the most continuous part of $L^{2}(Z)$ for a real spherical homogeneous space $Z$. Our starting point is the recent construction of Bernstein morphisms by Delorme, Knop, Kr{\"o}tz and Schlichtkrull. The most continuous part decomposes into a direct integral of unitary principal series representations. We give an explicit construction of the $H$-invariant functionals on these principal series. We show that for generic induction data the multiplicity space equals the full space of $H$-invariant functionals. Finally, we determine the inner products on the multiplicity spaces by refining the Maa\ss-Selberg relations.
\end{abstract}
 \tableofcontents

\section{Introduction}
In this paper we provide a complete description of the most continuous part of the Plancherel decomposition for a unimodular real spherical homogeneous space.

Let $Z:=G/H$, where $G=\underline{G}(\R)$ is the group of real points of a connected reductive algebraic group $\underline{G}$ defined over $\R$ and $H=\underline{H}(\R)$ the set of real points of an algebraic subgroup $\underline{H}$ of $\underline{G}$.
We assume that $Z$ is unimodular and hence admits a positive $G$-invariant Radon measure $\mu_{Z}$.
The basic problem in harmonic analysis on $Z$ is to obtain an explicit description of the Plancherel decomposition of the regular representation of $G$ on $L^{2}(Z,\mu_{Z})$ into a direct integral of irreducible unitary representations of $G$.

In the case that the group $G$ is considered as a homogeneous space of $G\times G$ such an explicit description of the Plancherel decomposition is found in the celebrated work of Harish-Chandra \cite{Harish-Chandra_HarmonicAnalyisOnRealReductiveGroupsI},
\cite{HarishChandra_HarmonicAnalysisOnRealReductiveGroupsII}, \cite{HarishChandra_HarmonicAnalysisOnRealReductiveGroupsIII}.
For real reductive symmetric spaces it was obtained by Delorme in \cite{DelormeFormuleDePlancherelPourLEspaceSymmetriques} and independently by Van den Ban and Schlichtkrull in \cite{vdBanSchlichtkrull_PlancherelDecompositionI}, \cite{vdBanSchlichtkrull_PlancherelDecompositionII}.

Recall that the space $Z=G/H$ is called symmetric in case $H$ is an open subgroup of the fixed point subgroup $G^{\sigma}$ for an involutive automorphism $\sigma:G \to G$.
The representations of $G$ occurring in the Plancherel decomposition of a reductive symmetric space $Z$ split into finitely many {\em series} according to the (class of) parabolic subgroup $P\subseteq G$ from which they are induced. The relevant parabolic subgroups of $G$ are the so-called $\sigma$-parabolics, namely those parabolic subgroups $P$ for which $P$ and $\sigma(P)$ are opposite to each other. With a Langlands decomposition $P=M_{P}A_{P}N_{P}$, with $\sigma(A_{P})=A_{P}$, the part attached to $P$ has the form of a direct integral of generalized principal series representations.
More specifically, these are induced representations $\Ind_{P}^{G}(\xi\otimes\lambda\otimes\1)$,
where $\xi$ is in the discrete series of representations for the symmetric space $M_{P}/M_{P} \cap H$ and $\lambda$ is a unitary character of $\fa_{P}=\Lie(A_{P})$ that vanishes on $\fa_{P} \cap \fh$. The part corresponding to the minimal $\sigma$-parabolic subgroup $Q$ is called the most continuous part of $L^{2}(Z)$. The Plancherel decomposition of the most continuous part was determined for real reductive symmetric spaces in the work of Van den Ban and Schlichtkrull in \cite{vdBanSchlichtkrull_MostContinuousPart}. This was based on the earlier works of Van den Ban on invariant linear functionals \cite{vdBan_PrincipalSeriesI}, \cite{vdBan_PrincipalSeriesII}. The most continuous part of $L^{2}(Z)$ decomposes as
$$
L_{\mc}^{2}(Z)
    \simeq \widehat{\bigoplus_{\xi\in\widehat{M}_{Q}}}\int_{i(\fa_{Q}/\fa_{Q}\cap\fh)_{+}^{*}}^{\oplus} V^{*}(\xi)\otimes\Ind_{\overline{Q}}^{G}(\xi\otimes\lambda\otimes \1)\,d\lambda,
$$
where $d\lambda$ is the Lebesgue measure on $i(\fa_{Q}/\fa_{Q}\cap\fh)^{*}$ and $i(\fa_{Q}/\fa_{Q}\cap\fh)_{+}^{*}$ is a fundamental domain for the stabilizer of $(\fa_{Q}/\fa_{Q}\cap \fh)^{*}$ in the Weyl group. The multiplicity spaces $V^{*}(\xi)$ are independent of $\lambda$. Moreover, $V^{*}(\xi)$ is non-zero only for finite dimensional unitary representations $\xi$ of $M_{Q}$.

A homogeneous space $Z$ is called real spherical if a minimal parabolic subgroup $P$ of $G$ admits an open orbit in $Z$.
All real reductive symmetric spaces are real spherical. A remarkable property of the class of real spherical spaces is the fact that all irreducible smooth representations of $G$ admit a finite dimensional space of $H$-invariant functionals by \cite[Theorem C]{KobayashiOshima_FiniteMultiplicitiyTheorems} and \cite{KrotzSchlichtkrull_MultiplicityBoundsAndSubrepresentationTheorem}. This property makes harmonic analysis on real spherical spaces suitable for developing Plancherel theory.

\medbreak

In this paper we determine the Plancherel decomposition of the most continuous part of $L^{2}(Z)$ for a real spherical space $Z$, thus generalizing the main result of \cite{vdBanSchlichtkrull_MostContinuousPart}.

Our starting point is the recent work of Delorme, Knop, Kr{\"o}tz and Schlichtkrull \cite{DelormeKnopKrotzSchlichtkrull_PlancherelTheoryForRealSphericalSpacesConstructionOfTheBernsteinMorphisms}.
Their construction of Bernstein morphisms allows to decompose $L^{2}(Z)$
into finitely many {\em blocks} of representations, each attached to a so-called boundary degeneration of $Z$. The block for the most degenerate of these boundary degenerations we call the most continuous part of $L^{2}(Z)$.
We show that, as in the symmetric case, the most continuous part decomposes into a direct integral of principal series representations. To determine the Plancherel decomposition of the most continuous part we construct linear functionals on these principal series. For generic parameters our construction provides a basis for the space of $H$-invariant linear functionals.
We then show the key result that all $H$-invariant functionals are tempered and the wave packets constructed using these functionals are square integrable.
Finally, by refining the Maa{\ss}-Selberg relations of \cite{DelormeKnopKrotzSchlichtkrull_PlancherelTheoryForRealSphericalSpacesConstructionOfTheBernsteinMorphisms}, we obtain a complete description of the inner product on the multiplicity spaces. This yields the full description of the most continuous part of $L^{2}(Z)$.

Assuming that the twisted discrete series conjecture from  \cite[(1.3)]{KrotzKuitOpdamSchlichtkrull_InfinitesimalCharactersOfDiscreteSeriesForRealSphericalSpaces} holds, the most continuous part of $L^{2}(Z)$ exhausts $L^{2}(Z)$ in case $Z$ is a complex spherical space, i.e., in case $G$ and $H$ are both complex. Thus, our construction is expected to yield the full Plancherel formula for complex spherical spaces.

\subsection{The most continuous part via Bernstein morphisms}
To describe the most continuous part $L^{2}(Z)_{\mc}$ of $L^{2}(Z)$ and the results in this article we recall some important invariants of the real spherical homogeneous space $Z$, boundary degenerations of $Z$, twisted discrete series representations and finally the Bernstein morphism, relating $L^{2}(Z)$ to twisted discrete series representations for boundary degenerations of $Z$.

We fix a minimal parabolic subgroup $P$ and a well chosen (with respect to $H$) Langlands decomposition $P=MAN$.
Inside the Lie algebra $\fa$ of $A$ one finds the compression cone, which is an open cone $\cC$ whose closure is finitely generated and contains $\fa_{\fh}:=\fa\cap \fh$. The cone $\overline{\cC}/\fa_{\fh}\subseteq \fa/\fa_{\fh}$ serves as a fundamental domain for a finite reflection group $W_{Z}$, called the little Weyl group of $Z$. Attached to the little Weyl group is a root system $\Sigma_{Z}$, called the spherical root system. The faces of the cone $\overline{\cC}$ are parameterized by subsets of $\Sigma_{Z}$, i.e., the sets
$$
\cF_{I}
:=\overline{\cC}\cap \bigcap_{\alpha\in I}\ker(\alpha)
\qquad(I\subseteq\Sigma_{Z})
$$
are precisely the faces of $\overline{\cC}$.

In \cite{KnopKrotz_ReductiveGroupActions} a smooth $G$-equivariant compactification $\widehat{\underline{Z}}(\R)$ of $\underline{Z}(\R)$ was constructed. For every $I\subseteq\Sigma_{Z}$ and $X$ contained in the relative interior of $\cF_{I}$ the limit
$$
z_{I}
:=\lim_{t\to\infty}\exp(tX) H\in\widehat{\underline{Z}}(\R)
$$
exists and does not depend on the choice of $X$. The stabilizer of $z_{I}$ is a real algebraic subgroup of $G$, and hence equals the set of real points of an algebraic subgroup $\widehat{\underline{H}}_{I}$ of $\underline{G}$ defined over $\R$. We note that $A_{I}:=\exp(\spn \cF_{I})$ is a subgroup of  $\widehat{\underline{H}}_{I}(\R)$. We set $\widehat{\underline{Z}}_{I}:=\underline{G}/\widehat{\underline{H}}_{I}$. Now $\widehat{\underline{Z}}(\R)$ admits a stratification in $G$-manifolds of the form $\widehat{\underline{Z}}_{I}(\R)$ where
$I\subseteq\Sigma_{Z}$. In the case where $Z$ admits a wonderful compactification $
\widehat{\underline{Z}}$, one has
$$
\widehat{\underline{Z}}(\R)
=\underline{Z}(\R)\cup\bigcup_{I\subseteq\Sigma_{Z}}\widehat{\underline{Z}}_{I}(\R).
$$
In the general case there is a need to use multiple copies of $G$-manifolds of the form $\widehat{\underline{Z}}_{I}(\R)$.

We use these spaces to define the boundary degenerations.
The group $\widehat{H}_{I}(\R)$ acts on the normal space of $\widehat{\underline{Z}}_{I}(\R)$ at the point $z_{I}$. The kernel of this representation on the normal space is a normal real algebraic subgroup of $\widehat{H}_{I}(\R)$, i.e., there exists a normal algebraic subgroup $\underline{H}_{I}$ of $\widehat{\underline{H}}_{I}$ so that the kernel of the isotropy action of $\widehat{H}_{I}(\R)$ on the normal space of $\widehat{\underline{Z}}_{I}(\R)$ at $z_{I}$ is equal to $\underline{H}_{I}(\R)$. The quotient $\widehat{\underline{H}}_{I}(\R)/\underline{H}_{I}(\R)$ is abelian. Its identity component is equal to $A_{I}/(A_{I}\cap H)$.

We define the algebraic varieties
$$
\underline{Z}_{I}
:=\underline{G}\cdot z_{I}
\qquad(I\subseteq\Sigma_{z}).
$$
These varieties are called boundary degenerations of $\underline{Z}$. The manifold $\underline{Z}_{I}(\R)$ is a finite union of homogeneous spaces for $G$, each of which is real spherical and is unimodular. The group $A_{I}$ acts from the right on $\underline{Z}_{I}(\R)$. The kernel of this action is $A_{H}:=\exp(\fa_{\fh})$.

The right-action of $A_{I}$ on $\underline{Z}_{I}(\R)$ induces a right-action on $L^{2}\big(\underline{Z}_{I}(\R)\big)$, which commutes with the left-regular representation of $G$. The decomposition of $L^{2}\big(\underline{Z}_{I}(\R)\big)$ with respect to the right-action of $A_{I}$ yields a disintegration in unitary $G$-modules
$$
L^{2}\big(\underline{Z}_{I}(\R)\big)
=\int_{\rho+i(\fa_{I}/\fa_{\fh})^{*}}^{\oplus}L^{2}\big(\underline{Z}_{I}(\R),\chi\big)\,d\chi.
$$
Here $\rho\in (\fa/\fa_{\fh})^{*}$ is an element so that the sections of the line bundle
$$
\underline{Z}_{I}(\R)\times_{A_{I}}\C_{\chi}\to \underline{Z}_{I}(\R)/A_{I}.
$$
with $\chi\in\rho+i(\fa_{I}/\fa_{\fh})^{*}$ are half-densities,
$L^{2}\big(\underline{Z}_{I}(\R),\chi\big)$ is the space of square integrable sections of this line bundle and $d\chi$ is the Lebesgue measure on $\rho_{Q}+i(\fa_{I}/\fa_{\fh})^{*}$.

The irreducible subrepresentations of $L^{2}_{\tds}\big(\underline{Z}_{I}(\R),\chi\big)$ for any $\chi\in\rho_{Q}+i(\fa_{I}/\fa_{\fh})^{*}$ are called twisted discrete series representations.
Let $L^{2}_{\tds}\big(\underline{Z}_{I}(\R),\chi\big)$ be the closure of the span of all irreducible subrepresentations of $L^{2}\big(\underline{Z}_{I}(\R),\chi\big)$. The spaces $L^{2}_{\tds}\big(\underline{Z}_{I}(\R),\chi\big)$ depend measurably on the character $\chi$.
We define
$$
L^{2}_{\tds}\big(\underline{Z}_{I}(\R)\big)
:=\int_{\rho_{Q}+i(\fa_{I}/\fa_{\fh})^{*}}^{\oplus}L^{2}_{\tds}\big(\underline{Z}_{I}(\R),\chi\big)\,d\chi.
$$

The main result of \cite{DelormeKnopKrotzSchlichtkrull_PlancherelTheoryForRealSphericalSpacesConstructionOfTheBernsteinMorphisms} is the construction of a map
\begin{equation}\label{eq Def B}
B:\bigoplus_{I\subseteq\Sigma_{Z}}L^{2}_{\tds}\big(\underline{Z}_{I}(\R)\big)\to L^{2}\big(\underline{Z}(\R)\big)
\end{equation}
with the following properties: $B$ is $G$-equivariant, surjective, isospectral, and for every $I\subseteq\Sigma_{Z}$ the restriction
\begin{equation}\label{eq Def B_I}
B_{I}
:=B|_{L^{2}_{\tds}\big(\underline{Z}_{I}(\R)\big)}
\end{equation}
is a sum of partial isometries. The existence of such a map goes back to ideas of Bernstein and hence $B$ is called the Bernstein morphism.
The Bernstein morphism was first constructed by Sakellaridis and Venkatesh for $p$-adic spherical spaces in \cite{SakellaridisVenkatesh_PeriodsAndHarmonicAnalysisOnSphericalVarieties}.

We mention here that for the case that $G$ is split and under the assumption of a conjecture on the nature of twisted discrete series representations, Delorme determined the kernel of the Bernstein morphism in \cite{Delorme_ScatteringOperators}. The kernel is described by so-called scattering operators.

In this article we focus on  $I=\emptyset$. The subspace
$$
L^{2}_{\mc}(Z)
:=\Im\big(B_{\emptyset}\big)\cap L^{2}(Z)
$$
decomposes in the largest continuous families of representations. Therefore, $L^{2}_{\mc}(Z)$ is called the most continuous part of $L^{2}(Z)$. For general $I\subseteq \Sigma_{Z}$  very little is known about the twisted discrete series of representations for $\underline{Z}_{I}(\R)$. This is different for $I=\emptyset$ because the degeneration $\underline{Z}_{\emptyset}(\R)$ can be described rather explicitly.

The boundary degeneration $\underline{Z}_{\emptyset}(\R)$ equals a finite union of copies of one real spherical homogeneous space for $G$ which we denote by $Z_{\emptyset}=G/H_{\emptyset}$. To be more precise, the copies of $Z_{\emptyset}$ in $\underline{Z}_{\emptyset}(\R)$ are parameterized by the open $P$-orbits in $\underline{Z}(\R)$.

The local structure theorem of \cite{KnopKrotzSchlichtkrull_LocalStructureTheorem} applied to the spherical space $Z$, provides an adapted parabolic subgroup $Q\subseteq G$ and Langlands decomposition $Q=M_{Q}A_{Q}N_{Q}$ with $A_{Q}\subseteq A$. Let $\overline{Q} =M_{Q}A_{Q}\overline{N}_{Q}$ be the opposite parabolic. For a reductive symmetric space $Q$ is the minimal $\sigma$-parabolic subgroup. Now the space $Z_{\emptyset}$ can be explicitly described as
$$
Z_{\emptyset}
=G/H_{\emptyset},
\qquad
H_{\emptyset}
=(M_{Q} \cap H)(A \cap H)\overline{N}_{Q}.
$$
In this case $A_{\emptyset}=A$.
The fact that the subgroup $H_{\emptyset}$ satisfies
$$
\overline{N}_{Q}\subseteq H_{\emptyset}\subseteq \overline{Q}
$$
makes decomposing $L^{2}(Z_{\emptyset})$ into a direct integral of irreducible unitary representation of $G$ easy.
Indeed employing induction by stages we obtain
$$
L^{2}(Z_{\emptyset})
=\Ind_{H_{\emptyset}}^{G}(\1)
=\Ind_{\overline{Q}}^{G}\Big(\Ind_{H_{\emptyset}}^{\overline{Q}}(\1)\Big).
$$
Moreover,
$$
\Ind_{\overline{Q}}^{G}\Big(\Ind_{H_{\emptyset}}^{\overline{Q}}(\1)\Big)
=L^{2}\big(M_{Q}/M_{Q}\cap H\big)\otimes L^{2}\big(A/A\cap H\big)
$$
and hence
$$
L^{2}(Z_{\emptyset})
    \simeq \widehat{\bigoplus_{\xi\in\widehat{M}_{Q}}}\int_{i(\fa/\fa_{\fh})^{*}_{+}}^{\oplus} \cM_{\emptyset,\xi}\otimes\Ind_{\overline{Q}}^{G}(\xi\otimes\lambda\otimes \1)\,d\lambda.
$$
Here $d\lambda$ is the Lebesgue measure on $i(\fa/\fa_{\fh})^{*}$ and $i(\fa/\fa_{\fh})^{*}_{+}$ is a fundamental domain for the stabilizer of $i(\fa/\fa_{\fh})^{*}$ in the Weyl group.
The space $\cM_{\emptyset,\xi}$ is the so-called multiplicity space attached to the representation $\Ind_{\overline{Q}}^{G}(\xi \otimes \lambda\otimes \1)$. It is independent of $\lambda$ and is only non-zero for finite dimensional unitary representations of $M_{Q}$. It follows from this description of the Plancherel decomposition that all irreducible unitary representations occurring in $L^{2}(Z_{\emptyset})$ belong to the twisted discrete series of representation for $Z_{\emptyset}$.
Furthermore, the twisted discrete series for $Z_{\emptyset}$ consists of the principal series representations of the form $\Ind_{\overline{Q}}^{G}(\xi \otimes \lambda)$ with $\lambda\in i(\fa/\fa_{\fh})^{*}$ and $(\xi, V_{\xi})$ a finite dimensional unitary representation of $M_{Q}$.

Invoking the formal properties of the Bernstein maps described above, we obtain a decomposition of $L_{\mc}^{2}(Z)$ as
$$
L_{\mc}^{2}(Z)
    \simeq \widehat{\bigoplus_{\xi\in\widehat{M}_{Q,\mathrm{fu}}}}\int_{i(\fa/\fa_{\fh})^{*}}^{\oplus} \cM_{\xi,\lambda}\otimes\Ind_{\overline{Q}}^{G}(\xi\otimes\lambda\otimes \1)\,d\lambda.
$$
Here $\widehat{M}_{Q,\mathrm{fu}}$ denotes the set of equivalence classes of finite dimensional unitary representation of $M_{Q}$.
A precise description of the Plancherel decomposition of $L^{2}_{\mc}(Z)$ amounts to the determination of the multiplicity spaces $\cM_{\xi,\lambda}$ with their inner product structure.

The elements of the multiplicity spaces can be interpreted as $H$-invariant continuous linear functionals on the space of smooth vectors for principal series representations $\Ind_{\overline{Q}}^{G}(\xi\otimes\lambda\otimes \1)$. As such, they can be studied as $V^{*}_{\xi}$-valued distributions on $Z$.

\subsection{Main results}
\label{Subsectoin Introduction - Main results}
Recall that we denote the adapted parabolic subgroup by $Q$.
To formulate our results concerning the multiplicity spaces for the representation $\Ind_{\overline{Q}}^{G}(\xi\otimes\lambda\otimes \1)$ and the Plancherel decomposition of $L^{2}_{\mc}(Z)$, we need some preparation. For the proofs in the text it is more convenient to work with  $V^{*}_{\xi}$-valued distributions rather than functionals. However, for clarity of exposition we state our results here in terms of continuous linear functionals.

More is known about $P$-orbits in $Z$ than about $\overline{Q}$ or $Q$-orbits. Therefore, instead of representations induced from $\overline{Q}$, we rather first consider representations induced from the minimal parabolic subgroup $P$. To describe the connection between the relevant representations induced from $\overline{Q}$ and representations induced from $P$, we fix a finite dimensional unitary representation $(\xi, V_{\xi})$ of $M_{Q}$.
Such a representation is necessarily trivial on the connected subgroup of $M_{Q}$ with Lie algebra equal to the sum of all non-compact simple ideals in the Lie algebra of $M_{Q}$. Therefore, for $\lambda\in \fa_{Q,\C}^{*}$ the representation $\Ind_{Q}^{G}(\xi\otimes\lambda\otimes 1)$ is a subrepresentation of $\Ind_{P}^{G}(\xi\big|_{M}\otimes\lambda+\rho_{P,Q}\otimes1)$, where $\rho_{P,Q}$ is the half sum of all roots of $\fa$ that occur in $P\cap M_{Q}$. Moreover, for generic $\lambda\in \fa_{Q,\C}^{*}$ the representations $\Ind_{\overline{Q}}^{G}(\xi\otimes\lambda\otimes \1)$ and $\Ind_{Q}^{G}(\xi\otimes\lambda\otimes 1)$ are equivalent.

We write $\cH_{\overline{Q},\xi,\lambda}$, $\cH_{Q,\xi,\lambda}$ and $\cH_{P,\xi,\lambda}$ for the spaces of smooth vectors for the representations $\Ind_{\overline{Q}}^{G}(\xi\otimes\lambda\otimes 1)$, $\Ind_{Q}^{G}(\xi\otimes\lambda\otimes 1)$ and $\Ind_{P}^{G}(\xi\big|_{M}\otimes\lambda+\rho_{P,Q}\otimes1)$, respectively. Now for generic $\lambda\in\fa_{Q,\C}^{*}$
$$
\cH_{\overline{Q},\xi,\lambda}\simeq\cH_{Q,\xi,\lambda}\subseteq\cH_{P,\xi,\lambda}.
$$
Our concern is with $H$-invariant continuous linear functionals on $\cH_{\overline{Q},\xi,\lambda}\simeq \cH_{Q,\xi,\lambda}$.
It is a remarkable fact that every such functional on $\cH_{Q,\xi,\lambda}$ is obtained by restricting an $H$-fixed continuous linear functional on $\cH_{P,\xi,\lambda}$.  The geometry of the orbits makes it more convenient to first determine the $H$-fixed continuous functionals on $\cH_{P,\xi,\lambda}$ and with that those on the  $\cH_{\overline{Q},\xi,\lambda}$, rather than considering functionals on  $\cH_{\overline{Q},\xi,\lambda}$ directly.

The analysis of $H$-fixed continuous linear functionals on $\cH_{P,\xi,\lambda}$ requires a closer study of the $P$-orbits in $Z$. We now discuss some aspects of this.
For $z \in Z$ we denote by $H_{z}$ the stabilizer of $z$ in $G$ and by $\fh_{z}=\Lie(H_{z})$ its Lie algebra. For every element $X \in \fa$  the limit
\begin{equation}\label{eq Intro fh_{z,X} def}
\fh_{z,X}:=\lim_{t \to \infty}\Ad\big(\exp(tX)\big)\fh_{z}
\end{equation}
exists in the Grassmannian manifold. Let $\cO$ be a $P$-orbit in $Z$. The subspace $\fa_{\cO}:=\fa \cap \fh_{z,X}$ with $z\in\cO$ and $X\in \fa^{-}$ is an invariant of $\cO$ as it is independent of the choices of $z$ and $X$. This allows us to define the rank of $\cO$ by
$$
\rank(\cO)
=\dim(\fa/\fa_{\cO}).
$$
For every open $P$-orbit $\cO$ we have $\fa_{\cO}=\fa_{\fh}$. The rank of each open orbit is therefore the same; this is an invariant of $Z$ called the rank of $Z$.
The rank of any $P$-orbit is bounded by $\rank(Z)$ and an orbit is called of maximal rank if $\rank(\cO)=\rank(Z)$. The set of maximal rank orbits is in general strictly larger than the set of open orbits. For example, in the group case every $P$-orbit is of maximal rank. See Example \ref{Ex Limits in group case}.
For our purposes the set of maximal rank orbits $\cO$ with $\fa_{\cO}=\fa_{\fh}$ is of great importance. We denote this set by $(P\bs Z)_{\fa_{\fh}}$. For many real spherical spaces the set of open $P$-orbits does not exhaust  $(P\bs Z)_{\fa_{\fh}}$. This is for example the case for $Z=G/\overline{N}_{P}$ and $Z=\mathrm{SO}(5,\C)/\mathrm{GL}(2,\C)$.

For any $H$-fixed continuous linear functional $\ell$ on $\cH_{P,\xi,\lambda}$ one can naturally attach a $V_{\xi}^{*}$-valued distribution $\mu_{\ell}$ on $Z$ that is left-$P$ equivariant.
For such distributions on $Z$ we denote by $(P \bs Z)_{\ell}$ the set of $P$-orbits in $\supp(\mu_{\ell})$ that are open in the relative topology of $\supp(\mu_{\ell})$.
Our first result is a strong restriction on the support of the distributions $\mu_{\ell}$ when the induction parameter $\lambda$ is generic.
Furthermore, we show that these distributions do not admit transversal derivatives.
More precisely, we prove the following:

\begin{IntroThm}[Theorems \ref{Thm condition on orbits for given generic lambda}, \ref{Thm bound on transversal degree} \& \ref{Thm Construction on max rank orbits for Q}]\label{Thm Intro 1}
There exists a finite union $\cS$ of hyperplanes in $(\fa/\fa_{\fh})^{*}$ so that for $\lambda\in (\fa/\fa_{\fh})^{*}_{\C}$ with $\Im\lambda\notin\cS$ the $H$-fixed continuous linear functionals $\ell$ on $\cH_{P,\xi,\lambda}$ satisfy the following.
\begin{enumerate}[(i)]
\item\label{Thm Intro 1 - item 1}
Only maximal rank orbits with $\fa_{\cO}=\fa_{\fh}$ contribute to $(P\bs Z)_{\ell}$, i.e.,
$$
(P \bs Z)_{\ell}
\subseteq (P\bs Z)_{\fa_{\fh}}.
$$
\item\label{Thm Intro 1 - item 2}
For each orbit $\cO\in (P\bs Z)_{\ell}$ there exists a representative $x_{\cO}\in G$ and an $\eta_{\cO}\in (V_{\xi}^{*})^{M_{Q}\cap x_{\cO}Hx_{\cO}^{-1}}$ so that for every $f\in \cH_{P,\xi,\lambda}$ with
$$
\supp(f)\cap\bigcup_{\cO \in (P \bs Z)_{\mu}} \partial \cO
=\emptyset
$$
we have the formula
$$
\ell(f)
=\sum_{\cO \in (P \bs Z)_{\ell}} \int_{(x_{\cO}^{-1}Px_{\cO} \cap H) \bs H} \big\langle\eta_{\cO},f(x_{\cO}h)\big\rangle dh,
$$
where $dh$ denotes an $H$-invariant Radon measure on $(x_{\cO}^{-1}Px_{\cO} \cap H) \bs H$.
In particular, the distribution $\mu_{\ell}$ attached to the functional $\ell$ does not admit any transversal derivatives.
\item\label{Thm Intro 1 - item 3}
Every non-zero $H$-fixed continuous linear functional $\ell$ on $\cH_{P,\xi,\lambda}$ restricts to a non-zero $H$-fixed continuous linear functional on $\cH_{Q,\xi,\lambda}$. In fact, the restriction map
$$
\Hom_{H}\big(\cH_{P,\xi,\lambda},\C\big)\to \Hom_{H}\big(\cH_{Q,\xi,\lambda},\C\big)
$$
is an isomorphism.
\end{enumerate}
\end{IntroThm}

The next result concerns the actual construction of $H$-invariant continuous functionals attached to maximal rank orbits. First, for each $\cO\in (P\bs Z)_{\fa_{\fh}}$ we carefully choose a representative $x_{\cO}\in G$, see Section \ref{Subsection Construction - Description}.
Given an orbit $\cO\in (P\bs Z)_{\fa_{\fh}}$ there exists a shifted open cone $\Gamma_{\cO}\subseteq (\fa/\fa_{\fh})^{*}$ so that for all $\lambda\in (\fa/\fa_{\fh})^{*}_{\C}$ with $\Re\lambda+\rho_{P,Q}\in\Gamma_{\cO}$ the integrals
\begin{equation}\label{eq Intro integral}
\ell_{\xi,\lambda,\eta}(f):=\int_{(x_{\cO}^{-1}Px_{\cO} \cap H) \bs H} \big\langle\eta,f(x_{\cO}h)\big\rangle dh
\end{equation}
are absolutely convergent for all $\eta\in (V_{\xi}^{*})^{M_{Q}\cap x_{\cO}Hx_{\cO}^{-1}}$ and $f\in \cH_{P,\xi,\lambda}$. Moreover, when viewed as $V_{\xi}^{*}$-valued distributions on $Z$, each family $\lambda\mapsto \ell_{\xi,\lambda,\eta}$ extends to a meromorphic family with parameter $\lambda\in (\fa/\fa_{\fh})^{*}_{\C}$.
We set
$$
V^{*}(\xi)
:=\bigoplus_{\cO\in (P\bs Z)_{\fa_{\fh}}}(V_{\xi}^{*})^{M_{Q}\cap x_{\cO}Hx_{\cO}^{-1}}.
$$
Note that $V^{*}(\xi)$ is finite dimensional.
We thus obtain a map
$$
V^{*}(\xi)\to\Hom_{H}\big(\cH_{P,\xi,\lambda},\C\big);
\quad \eta\mapsto \ell_{\xi,\lambda,\eta}
$$
with meromorphic dependence on $\lambda\in (\fa/\fa_{\fh})^{*}_{\C}$. After suitably normalizing these functionals using amongst other things the long intertwining operator, see (\ref{eq Def mu^circ}), we arrive at a map
$$
\ell^{\circ}_{\xi,\lambda}:V^{*}(\xi)\to \Hom_{H}\big(\cH_{\overline{Q},\xi,\lambda},\C\big),
$$
which is an isomorphism for generic $\lambda$. More precisely, the following hold.

\begin{IntroThm}[Theorem \ref{Thm Description D'(overline Q:xi:lambda)^H} \& Corollary \ref{Cor mu^circ holomorphic on imaginary axis}]
\label{Thm Intro 2}\,
\begin{enumerate}[(i)]
\item\label{Thm Intro 2 - item 1}
For every $\eta\in V^{*}(\xi)$ the map $\lambda\mapsto \ell^{\circ}_{\xi,\lambda}(\eta)$, considered as a family of  $V_{\xi}^{*}$-valued distributions on $Z$, is meromorphic on $(\fa/\fa_{\fh})^{*}_{\C}$.
\item\label{Thm Intro 2 - item 2}
For every $\eta\in V^{*}(\xi)$ the map $\lambda\mapsto \ell^{\circ}_{\xi,\lambda}(\eta)$ is holomorphic on an open neighborhood of $i(\fa/\fa_{\fh})^{*}$.
\item\label{Thm Intro 2 - item 3}
There exists a finite union $\cS$ of proper subspaces of $(\fa/\fa_{\fh})^{*}$ so that $\ell^{\circ}_{\xi,\lambda}$ is an isomorphism for $\lambda\in (\fa/\fa_{\fh})^{*}_{\C}$ with $\Im\lambda\notin\cS$.
\end{enumerate}
\end{IntroThm}

We now come to the determination of the multiplicity spaces.
Each multiplicity space $\cM_{\xi,\lambda}$ is naturally identified with a subspace of $\Hom_{H}(\cH_{\overline{Q},\xi,\lambda},\C)$. However, an $H$-fixed continuous linear functional $\ell$ on $\cH_{\overline{Q},\xi,\lambda}$ can only be contained in $\cM_{\xi,\lambda}$ if the generalized matrix coefficients with $\ell$ are almost contained in $L^{2}(Z)$. To be more precise, a functional $\ell $ can only contribute if it is tempered, i.e., if all generalized matrix coefficients with $\ell$ define tempered functions on $Z$.

\begin{IntroThm}[Theorem \ref{Thm temperedness} \& Theorem \ref{Thm Wave packets are L^2} and its Corollary \ref{Cor Multiplicity space is V^*(xi)}]\label{Thm Intro 3}
For $\lambda\in i(\fa/\fa_{\fh})^{*}$ outside of a finite union of proper subspaces of $i(\fa/\fa_{\fh})^{*}$ every $H$-fixed continuous linear functional on $\cH_{\overline{Q},\xi,\lambda}$ is tempered. In fact, for almost every $\lambda\in i(\fa/\fa_{\fh})^{*}$ we have
$$
\cM_{\xi,\lambda}=\Hom_{H}\big(\cH_{\overline{Q},\xi,\lambda},\C\big).
$$
\end{IntroThm}

To state the main result of the article, Theorem \ref{Thm Plancherel Theorem}, we define the Fourier transform for a smooth function $\phi$ with compact support on $Z$
$$
\Ft(\phi)(\xi,\lambda)
\in V^{*}(\xi)\otimes\cH_{\overline{Q},\xi,\lambda}
\simeq \Hom_{\C}\big(V^{*}(\xi^{\vee}),\cH_{\overline{Q},\xi,\lambda}\big)
$$
by
$$
\Ft(\phi)(\xi,\lambda)\eta:=\int_{Z}\phi(gH) \big(g\cdot \ell_{\xi^{\vee},-\lambda}^{\circ}\big)(\eta)\, dgH
\qquad\big(\eta\in V^{*}(\xi^{\vee})\big).
$$
On $V^{*}(\xi)$ there is a natural inner product induced by the inner product on $V_{\xi}$. We normalize this inner product by a factor of $\dim (V_{\xi})$.

\begin{IntroThm}[Theorem \ref{Thm Plancherel Theorem}]\label{Thm Intro 4}
Let $i(\fa/\fa_{\fh})^{*}_{+}$ be a fundamental domain for the stabilizer of $i(\fa/\fa_{\fh})^{*}$ in the Weyl group.
Then the Fourier transform $f\mapsto\Ft f$ extends to a unitary isomorphism
$$
L^{2}_{\mc}(Z)
    \to \widehat{\bigoplus_{\xi\in\widehat{M}_{Q,\mathrm{fu}}}}\int_{i(\fa/\fa_{\fh})^{*}_{+}}^{\oplus} V^{*}(\xi)\otimes\Ind_{\overline{Q}}^{G}(\xi\otimes\lambda\otimes \1)\,d\lambda.
$$
\end{IntroThm}

It is known that for generic $\lambda\in i(\fa/\fa_{\fh})^{*}$ the representations $\Ind_{\overline{Q}}^{G}(\xi\otimes\lambda\otimes \1)$ are irreducible. Moreover, if $\xi,\xi'\in \widehat{M}_{Q,\mathrm{fu}}$ and $\lambda,\lambda'\in i(\fa/\fa_{\fh})_{+}^{*}$ are generic the representations $\Ind_{\overline{Q}}^{G}(\xi\otimes\lambda\otimes \1)$ and $\Ind_{\overline{Q}}^{G}(\xi'\otimes\lambda'\otimes \1)$ are not equivalent if $(\xi,\lambda)\neq(\xi',\lambda')$.
Therefore, the above decomposition of $L^{2}_{\mc}(Z)$ is the Plancherel decomposition.

\subsection{Methods of Proof and structure of the article}
After setting up our notation in Section \ref{Section Setup and notation}, we begin in Section \ref{Section Orbits of max rank} with the study of $P$-orbits in $Z$. There are two main results. The first is Theorem \ref{Thm structure theorem for wPw^(-1) cdot z}, which is a structure theorem for maximal rank orbits. It is a generalization of a structure theorem of Brion, \cite[Proposition 6 \& Theorem 3]{Brion_OrbitClosuresOfSphericalSubgroupsInFlagVarieties}, for complex spherical spaces. Theorem \ref{Thm structure theorem for wPw^(-1) cdot z} is of crucial importance for our construction of $H$-fixed distributions.
The second main result in Section \ref{Section Orbits of max rank} is Theorem \ref{Thm properties of W-action}. We define an equivalence relation on the $P$-orbits of maximal rank. We then show that the Weyl group $W$ of the root system of $\fa$ in $\fg$ naturally acts on the set of equivalence classes. This action is transitive. The set of open orbits forms one equivalence class; its stabilizer is the little Weyl group $W_{Z}$. This result was first obtained by Knop in \cite{Knop_OnTheSetOfOrbitsForABorelSubgroup} for complex spaces and by Knop and Zhgoon in \cite{KnopZhgoon_ComplexityOfActionsOverPerfectFields} for spherical spaces defined over a field of characteristic $0$. Their results are more general than ours, but our description of the action of $W$ is tailor made for the way we apply it. The $W$-action is applied at several places, most notably for the precise choice of the representatives $x_{\cO}$ for the $P$-orbits in $(P\bs Z)_{\fa_{\fh}}$.
Our approach to $P$-orbits on $Z$ differs substantially from the techniques used by Brion, Knop and Zhgoon. The main tool for our considerations is the limit subalgebra $\fh_{z,X}$ from (\ref{eq Intro fh_{z,X} def}). Previously we used an analysis of these limit subalgebras to give a construction of the little Weyl group in \cite{KuitSayag_OnTheLittleWeylGroupOfARealSphericalSpace}. We heavily rely on the results from that article for the two main results in Section \ref{Section Orbits of max rank}.

In Section \ref{Section Distribution vectors} we set up a dictionary between invariant functionals on the smooth vectors of a principal series representation $\Ind_{S}^{G}(\xi\otimes\lambda\otimes 1)$ induced from a parabolic subgroup $S$ and a space $\cD'(S:\xi:\lambda)$ of $S$-equivariant $V_{\xi}^{*}$-valued distributions. Even though the exposition of the results is easier in the language of functionals, as in Section \ref{Subsectoin Introduction - Main results}, we find it easier to work with distributions and that is what we chose to do throughout the article. In Section \ref{Subsection Distribution vectors - Intertwining operators} we describe the action of the intertwining operators on those distributions as we were not able to locate such a description in the literature. The proofs for these results are relegated to Appendix B. As was explained in Section \ref{Subsectoin Introduction - Main results}, it is easier to work with the minimal parabolic subgroup $P$, rather than directly with $\overline{Q}$ or $Q$. To facilitate this, we make a comparison between inductions from different parabolic subgroups, when the induction data allows, in Sections \ref{Subsection Distribution vectors - Comparison} -- \ref{Subsection Distribution vectors - Comparison between P and Q}.

In Section \ref{Section Support and transdeg} we prove the first two assertions in Theorem \ref{Thm Intro 1}. For reductive symmetric spaces the restrictions on the support and transversal derivatives of $H$-fixed distributions in $\cD'(P:\xi:\lambda)$ are obtained using  Bruhat's theory which he developed in his thesis. See \cite[Theorem 5.1]{vdBan_PrincipalSeriesI} and \cite[Th{\'e}or{\`e}me 1 in Section 3.3]{CarmonaDelorme_BaseMeromorpheDeVecteursDistributionsHInvariants}. This approach relies heavily on precise knowledge of all $P$-orbits in $Z$. For reductive symmetric spaces the $P$-orbits are very well understood; a complete description of the $P$-orbits has been given by Matsuki in \cite{Matsuki_OrbitsOfSymmetricSpacesUnderActionOfMinimalParabolic} and \cite{Matsuki_OrbitsOnSymmetricSpacesUnderActionOfParabolicSubgroups}. Unfortunately, for real spherical spaces such a description is not available. We therefore resort to a different method, namely principal asymptotics, which is a technical tool from \cite[Theorem 5.1]{KrotzKuitOpdamSchlichtkrull_InfinitesimalCharactersOfDiscreteSeriesForRealSphericalSpaces}. The method of principal asymptotics can be considered as the analogue of the limit subalgebras (\ref{eq Intro fh_{z,X} def}) for $H$-fixed distributions in $\cD'(P:\xi:\lambda)$. Given such a distribution $\mu$, an orbit $\cO\in (P\bs Z)_{\mu}$, a point $z\in\cO$ and a sufficiently regular $X\in\fa^{-}$, the principal asymptotics of $\mu$ is a distribution $\mu_{z,X}$ defined on a left-$P$ invariant open neighborhood of $e$ in $G$ that is left $P$-equivariant and right invariant under the limit subalgebra $\fh_{z,X}$. These last distributions are easier to analyse. An immediate corollary is that the imaginary part of $\lambda$ must vanish on $\fa_{\cO}$, which implies the first assertion in Theorem \ref{Thm Intro 1}, see Theorem \ref{Thm condition on orbits for given generic lambda}. Moreover, $\mu$ has transversal derivatives on $\cO$ if and only if $\mu_{z,X}$ has transversal derivatives. The proof for the second assertion is now essentially reduced to the case of $\overline{N}_{P}$-invariant distributions in $\cD'(P:\xi:\lambda)$. For the latter distributions the absence of transversal derivatives for generic $\lambda$ is proved by an analysis of the action of the center of $\cU(\fg)$ in Theorem \ref{Thm bound on transversal degree}.

In Section \ref{Section Construction} we construct the $H$-invariant distributions in $\cD'(P:\xi:\lambda)$. By considering powers of matrix coefficients of finite dimensional $H$-spherical representations, one easily sees that the integrals (\ref{eq Intro integral}) are absolutely convergent if $\Re\lambda$ is in a certain shifted cone. We thus find holomorphic families of $H$-fixed distributions with family parameter $\lambda$. We then use the technique of Bernstein and Sato to extend these families to meromorphic families. This method is well known; it was for example used before by Olafsson in \cite[Theorem 5.1]{Olafsson_FourierAndPoissonTransformationAssociatedToASemisimpleSymmetricSpace}, Brylinski and Delorme \cite[Proposition 4]{BrylinskiDelorme_VecteursDistributionsH-Invariants} and Frahm \cite[Theorem 3.3]{Frahm}.

For symmetric spaces there are other ways to obtain the meromorphic extension. We mention here two methods of Van den Ban: \cite[Theorem 5.10]{vdBan_PrincipalSeriesI} using intertwining operators and \cite[Theorem 9.1]{vdBan_PrincipalSeriesII} using translation functors. The second method of Van den Ban is arguably the best since it provides a rather explicit functional equation. See also \cite[Th{\'e}or{\`e}me 2]{CarmonaDelorme_BaseMeromorpheDeVecteursDistributionsHInvariants} where this method was used by Carmona and Delorme. In our setting neither the method based on intertwining operators, nor the method based on translation functors is straightforwardly applicable since both require that the only orbits contributing in Theorem \ref{Thm Intro 1} (\ref{Thm Intro 1 - item 1}) are the open orbits. For real spherical spaces with the wavefront property, e.g. reductive symmetric spaces, only the open orbits contribute. We give a short proof of this in Appendix A.

To construct the $H$-invariant distributions in $\cD'(P:\xi:\lambda)$ on non-open $P$-orbits of maximal rank we use an idea from \cite[Theorem 7.1]{vdBanKuit_NormalizationsOfEisensteinIntegrals}.
We mention here that a similar construction for $p$-adic spherical spaces was done by Sakellaridis in \cite[Section 4]{Sakellaridis_UnramifiedSpectrumOfSphericalVarietiesOverPAdicField}.
The applicability of the idea heavily relies on the structure of maximal rank orbits. For reductive symmetric spaces this is readily obtained from the rich structure theory that exists for these spaces. For real spherical spaces the necessary assertions were proven using our methods concerning limits subalgebras in Theorem \ref{Thm structure theorem for wPw^(-1) cdot z}. Every maximal rank orbit $\cO$ is contained in an open $P'$-orbit $\cO'$  for a certain minimal parabolic subgroup $P'$. Moreover, $\cO'$ decomposes as a family of orbits of a unipotent subgroup of $P'$ parameterized by the points in $\cO$. This geometric decomposition translates on the level of distributions to a decomposition of the distributions we constructed before on open orbits into the application of a standard intertwining operator on a distribution supported on $\overline{\cO}$. The outcome of this analysis is a construction of $H$-invariant distributions $\mu$ in $\cD'(P:\xi:\lambda)$ with $(P\bs Z)_{\mu}=\{\cO\}$ by applying the inverse of a standard intertwining operator to a $H$-invariant distribution in $\cD'(P':\xi:\lambda)$ constructed on an open orbit. As a corollary we find that the $H$-invariant distributions in $\cD'(P:\xi:\lambda)$ fit into meromorphic families with family parameter $\lambda$. All this is described in Proposition \ref{Prop construction on max rank orbits}. With the rather explicit formulas for the distributions we obtain from Proposition \ref{Prop construction on max rank orbits} it is then shown that the distributions constructed on $P$-orbits in $(P\bs Z)_{\fa_{\fh}}$ actually are $Q$-equivariant, which establishes assertion (\ref{Thm Intro 1 - item 3}) in Theorem \ref{Thm Intro 1}. See Theorem \ref{Thm Construction on max rank orbits for Q}.

By combining Theorem \ref{Thm condition on orbits for given generic lambda}, Theorem \ref{Thm bound on transversal degree} and Theorem \ref{Thm Construction on max rank orbits for Q} we obtain in Theorem \ref{Thm Description D'(Z,Q:xi:lambda)} a full description of $\cD'(Q:\xi:\lambda)^{H}$ for generic $\lambda\in (\fa/\fa_{\fh})^{*}_{\C}$. The remainder of Section \ref{Section Construction} is devoted to a description of the action of the normalizer of $H$ in $A$ on $\cD'(Q:\xi:\lambda)$ and a proper normalization of the families of distributions we constructed. The latter results in the assertions (\ref{Thm Intro 2 - item 1}) and (\ref{Thm Intro 2 - item 3}) in Theorem \ref{Thm Intro 2}.

In section \ref{Section Wave packets} we prove Theorem \ref{Thm Intro 3}. There are two main results. The first is Theorem \ref{Thm temperedness}, which asserts that for generic $\lambda\in i(\fa/\fa_{\fh})^{*}$ all $H$-fixed distributions in $\cD'(\overline{Q}:\xi:\lambda)$ are tempered. The proof begins with an a priori estimate, which is then improved in a recursive process. For reductive symmetric spaces this was done by Van den Ban in \cite[Section 18]{vdBan_PrincipalSeriesII} using a technique of Wallach from \cite[Theorem 4.3.5]{Wallach_RealReductiveGroupsI}. For real spherical spaces this method is not easily applicable. This is due to the lack of a good polar decomposition. Instead we adapt the techniques developed in \cite{DelormeKrotzSouaifi_ConstantTerm} for the construction of the constant term map. In \cite{DelormeKrotzSouaifi_ConstantTerm} only tempered eigenfunctions are considered. However the techniques can be applied to non-tempered eigenfunctions as well and then used to improve estimates and prove temperedness.

Once we have established the temperedness of the distributions, we move on to the second main result in section \ref{Section Wave packets}: the square integrability of wave packets in Theorem \ref{Thm Wave packets are L^2}. The proof is similar to the analogous result for reductive symmetric spaces by Van den Ban, Carmona and Delorme in \cite{vdBanCarmonaDelorme_PaquetsDOndesDansLEspaceDeSchwartzDUnEspaceSymmetrique}.
Also this result heavily relies on the constant term map.
An important consequence, Corollary \ref{Cor Multiplicity space is V^*(xi)}, is that for almost every $\lambda\in i(\fa/\fa_{\fh})^{*}$ the multiplicity space $\cM_{\xi,\lambda}$ is identical to $\cD'(\overline{Q}:\xi:\lambda)^{H}$, and hence can be identified with $V^{*}(\xi)$ in view of Theorem \ref{Thm Intro 2} (\ref{Thm Intro 2 - item 3}).

In Section \ref{Section Most continuous part} we prove the Plancherel decomposition of $L^{2}_{\mc}(Z)$.  The abstract Plancherel decomposition provides $V^{*}(\xi)$ for almost every $\lambda\in i(\fa/\fa_{\fh})^{*}$ with an inner product.
To prove Theorem \ref{Thm Intro 4} it remains to show that this inner product is independent of $\lambda$ and up to a factor of $\dim(V_{\xi})$ equal to the inner product induced from the one on $V_{\xi}$. This we do in Section \ref{Section Most continuous part}. We first prove the required identity for the space $Z=Z_{\emptyset}$ by a direct computation in Theorem \ref{Thm Plancherel Theorem horospherical case}. The result for $Z_{\emptyset}$ can in view of the Maa\ss-Selberg relations \cite[Theorem 9.6]{DelormeKnopKrotzSchlichtkrull_PlancherelTheoryForRealSphericalSpacesConstructionOfTheBernsteinMorphisms} be used to determine the inner products for $Z$ itself. In order to apply the Maa\ss-Selberg relations we have to determine the constant terms of all distributions. We give explicit formulas in Proposition \ref{Prop Const mu formula} and \ref{Prop const formula}. If $Z$ has the wavefront property, then the Maa\ss-Selberg relations from \cite[Theorem 9.6]{DelormeKnopKrotzSchlichtkrull_PlancherelTheoryForRealSphericalSpacesConstructionOfTheBernsteinMorphisms} suffice to determine the inner product on $V^{*}(\xi)$. For the group case, and more generally for reductive symmetric spaces, this was done in \cite[Sections 14 \& 15]{DelormeKnopKrotzSchlichtkrull_PlancherelTheoryForRealSphericalSpacesConstructionOfTheBernsteinMorphisms}. For general real spherical spaces a refinement of the Maa\ss-Selberg relations is needed. This refinement is obtained in Corollaries \ref{Cor Orthogonal decomposition V^*(xi)} and \ref{Cor Maass-Selberg relations}. For the proof we construct suitable $G$-invariant differential operators on $Z$ using Knop's Harish-Chandra homomorphism from \cite{Knop_HarishChandraHomomorphism}. We then determine the Plancherel decomposition of $L^{2}_{\mc}(Z)$ in Theorem \ref{Thm Plancherel Theorem}.

Assertion (\ref{Thm Intro 1 - item 2}) in Theorem \ref{Thm Intro 1} is an easy corollary to Theorem \ref{Thm Plancherel Theorem}. For reductive symmetric spaces this was proven in \cite[Theorem 1]{vdBanSchlichtkrull_FourierTransformOnASemisimpleSymmetricSpace}. Finally, we provide in Corollary \ref{Cor Formula scattering operators} an explicit form for the so-called scattering operators introduced by \cite{Delorme_ScatteringOperators} in the case $G$ is a split real reductive group. Our formulas are written in terms of the standard intertwining operators acting on $H$-fixed linear functionals.

\subsection{Acknowledgments}
We thank Bernhard Kr{\"o}tz, Henrik Schlichtkrull and Erik van den Ban for various discussions on the subject matter of this paper.
This article was written over several years. While working on the sections 3 to 6 the first author was funded by the Deutsche Forschungsgemeinschaft Grant 262362164.

\section{Setup and notation}
\label{Section Setup and notation}

Groups are indicated by capital roman letters. Their Lie algebras are denoted by the corresponding lower-case fraktur letter.

Let $\underline{G}$ be a connected reductive algebraic group defined over $\R$ and set $G:=G(\R)$. Let $\underline{H}$ be an algebraic subgroup of $\underline{G}$ defined over $\R$ and set $H:=\underline{H}(\R)$.
We write $Z=G/H$.  If $z\in Z$, then the stabilizer subgroup of $Z$ is indicated by $H_{z}$ and its Lie algebra by $\fh_{z}$.

We set $G_{\C}:=\underline{G}(\C)$ and $H_{\C}:=\underline{H}(\C)$. If $E$ is a real vector space, then we write $E_{\C}$ for its complexification.

Throughout the article we fix a minimal parabolic subgroup $P$ of $G$ and a Langlands decomposition $P=MAN$. We assume that $Z$ is real spherical, i.e., there exists an open $P$-orbit in $Z$.
We further assume that $Z$ is unimodular. In view of \cite[Lemma 12.7]{DelormeKnopKrotzSchlichtkrull_PlancherelTheoryForRealSphericalSpacesConstructionOfTheBernsteinMorphisms} the space $Z$ is quasi-affine.

Let $\theta$ be a Cartan involution of $G$ so that $A$ is $\theta$-stable. We denote the corresponding involution on the Lie algebra $\fg$ by $\theta$ as well and write $K$ for the fixed point subgroup of $\theta$. Note that $K$ is a maximal compact subgroup of $G$.

If $Q$ is a parabolic subgroup of $G$, then we write $N_{Q}$ for the unipotent radical of $Q$ and $\overline{N}_{Q}$ for the unipotent radical $\theta N_{Q}$ of the opposite parabolic subgroup $\theta Q$.

We write $\Sigma$ for the root system of $\fa$ in $\fg$. If $Q$ is a parabolic subgroup containing $A$, then we define $\Sigma(Q)$ to be subset of $\Sigma$ of roots $\alpha$ so that the root space $\fg_{\alpha}$ is contained in $\fn_{Q}$. We define $\rho_{Q}$ to be the element of $\fa^{*}$ given by
$$
\rho_{Q}(X)
=\frac{1}{2}\tr \ad(X)\big|_{\fn_{Q}}.
$$
Further, we write $\fa^{-}$ for the open negative Weyl chamber with respect to $\Sigma(P)$.

We fix an $\Ad(G)$-invariant bilinear form $B$ on $\fg$ so that $-B(\dotvar,\theta \dotvar)$ is positive definite. For $E\subseteq \fg$, we define
$$
E^{\perp}
=\big\{X\in\fg:B(X,E)=\{0\}\big\}.
$$

For the notation for function spaces we follow the book of Schwartz \cite{Schwartz_Distributions_I}. In particular, spaces of compactly supported smooth, smooth and Schwartz functions are denoted by $\cD$, $\cE$ and $\cS$ respectively. Their strong duals are as usual indicated by a ${}'$.

If $N$ is a connected and simply connected subgroup of $G$ so that its Lie algebra $\fn$ is a nilpotent subalgebra of $\fg$, then we equip $N$ with the Haar-measure $dn$ given by the pull-back of the Lebesgue measure on $\fn$ along the exponential map. This we do in particular for the group $A$ and the unipotent radicals $N_{Q}$ of parabolic subgroups $Q$. Every compact subgroup we equip with the normalized Haar measure. We do this in particular for the groups $K$ and $M$. We normalize the Haar measure $dg$ on $G$ so that
$$
\int_{G}\phi(g)\,dg
=\int_{K}\int_{A}\int_{N_{P}}a^{2\rho_{P}}\phi(kan)\,dn\,da\,dk
\qquad\big(\phi\in \cD(G)\big).
$$
In view of the Local structure theorem, see Proposition \ref{Prop Local structure theorem}, there exists a parabolic subgroup $Q$ containing $A$ and a point $z\in Z$, so that $P\cdot z$ is open and
$$
N_{Q}\times M/(M\cap H_{z})\times A/(A\cap H_{z})\to P\cdot z;\quad(n,m,a)\mapsto nma\cdot z
$$
is a diffeomorphism. Let $\fa_{0}=\fa\cap (\fa\cap\fh_{z})^{\perp}$ and $A_{0}=\exp(\fa_{0})$. Then the group $MA_{0}N_{Q}$ is unimodular. We  normalize the invariant Radon measure on $Z$ by
$$
\int_{Z}\phi(z)\,dz
=\int_{N_{Q}}\int_{M}\int_{A_{0}}\phi(nma\cdot z)\,da\,dm\,dn
\qquad\big(\phi\in \cD(Z)\big).
$$
The Haar measure on $H$ we normalize by requiring that
$$
\int_{G}\phi(g)\,dg
=\int_{Z}\int_{H}\phi(gh)\,dh\,dgH
\qquad\big(\phi\in \cD(G)\big).
$$
Finally, we normalize the Lebesgue measure $i(\fa/\fa\cap\fh_{z})^{*}$ so that
$$
\phi(e)
=\int_{i(\fa/\fa\cap\fh_{z})^{*}}\int_{A/(A\cap H_{z})}\phi(a)a^{\lambda}\,d\lambda
\qquad\Big(\phi\in\cD\big(A/(A\cap H_{z})\big)\Big).
$$

\section{$P$-Orbits of maximal rank}
\label{Section Orbits of max rank}

\subsection{The local structure theorem}
\label{Subsection Orbits of maximal rank - LST}
In this section we give a reformulation of the local structure theorem, which follows from {\cite[Theorem 2.3]{KnopKrotzSchlichtkrull_LocalStructureTheorem} and its constructive proof.

\begin{Prop}\label{Prop Local structure theorem}
There exists a parabolic subgroup $Q$ with $P\subseteq Q$,  a Levi decomposition $Q=L_{Q}N_{Q}$ with $A\subseteq L_{Q}$, and for every open $P$-orbit $\cO$ in $Z$ a point $z\in \cO$ so that the following assertions hold.
\begin{enumerate}[(i)]
 \item\label{Prop Local structure theorem - item 0}
            $Q\cdot \cO=\cO$.
 \item\label{Prop Local structure theorem - item 1}
          $Q\cap H_{z}=L_{Q}\cap H_{z}$.
  \item\label{Prop Local structure theorem - item 2}
          The map
          $$
          N_{Q}\times L_{Q}/L_{Q}\cap H_{z}\to Z;
            \quad\big(n,l(L_{Q}\cap H_{z})\big)\mapsto nl\cdot z
          $$
          is a diffeomorphism onto $\cO$.
  \item\label{Prop Local structure theorem - item 3}
          The sum $\fl_{Q,\nc}$ of all non-compact simple ideals in $\fl_{Q}$ is contained in $\fh_{z}$.
  \item\label{Prop Local structure theorem - item 4}
          There exists an $X\in\fa\cap\fh_{z}^{\perp}$ so that $L_{Q}=Z_{G}(X)$ and $\alpha(X)>0$ for all $\alpha\in\Sigma(Q)$.
\end{enumerate}
\end{Prop}

\begin{Rem}\label{Rem L_Q roots}
The existence of an $X\in\fa\cap\fh_{z}^{\perp}$ with $L_{Q}=Z_{G}(X)$ has the following consequence. Let $\alpha\in\Sigma$. Then $\fg_{\alpha}\subseteq\fl_{Q}$ if and only if $\alpha^{\vee}\in\fa\cap\fh_{z}$.
\end{Rem}

\subsection{Adapted points}
\label{Subsection Orbits of maximal rank - Adapted points}
We now recall the notion of adapted points and some relevant results from \cite{KuitSayag_OnTheLittleWeylGroupOfARealSphericalSpace}.

Following \cite{KuitSayag_OnTheLittleWeylGroupOfARealSphericalSpace} we say that a point $z\in Z$ is adapted (to the Langlands decomposition $P=MAN$) if the following two conditions are satisfied.
\begin{enumerate}[(i)]
\item\label{Def Adapted points - item 1}
        $P\cdot z$ is open in $Z$, i.e., $\fp+\fh_{z}=\fg$,
\item\label{Def Adapted points - item 3}
        There exists an $X\in\fa\cap\fh_{z}^{\perp}$ so that $Z_{\fg}(X)=\fl_{Q}$.
\end{enumerate}
See Definition 3.3 and Remark 3.4 (b) in \cite{KuitSayag_OnTheLittleWeylGroupOfARealSphericalSpace}. It follows from Proposition \ref{Prop Local structure theorem} that every open $P$-orbit in $Z$ contains an adapted point. Moreover, the set of adapted points is $MA$-stable.

The Lie subalgebra $\fa\cap\fh_{z}$ is the same for all adapted points $z$ by \cite[Corollary 3.17]{KuitSayag_OnTheLittleWeylGroupOfARealSphericalSpace}. We denote this subalgebra by $\fa_{\fh}$ and refer to the dimension of $\fa/\fa_{\fh}$ as the rank of $Z$.

Adapted points have several of the properties that are listed in the local structure theorem, Proposition \ref{Prop Local structure theorem}. The following proposition is a combination of Proposition 3.6 and Remark 3.7 (b) in \cite{KuitSayag_OnTheLittleWeylGroupOfARealSphericalSpace}.

\begin{Prop}\label{Prop LST holds for adapted points}
Let $z\in Z$ be adapted. Then the following hold.
\begin{enumerate}[(i)]
  \item\label{Prop LST holds for adapted points - item 1}
            $Q\cap H_{z}=L_{Q}\cap H_{z}$,
  \item\label{Prop LST holds for adapted points - item 2}
            $\fl_{Q,\nc}\subseteq\fh_{z}$
  \item\label{Prop LST holds for adapted points - item 3}
          The map
          \begin{align*}
          M/(M\cap H_{z})\times \fa/\fa_{\fh}&\to L_{Q}/(L_{Q}\cap H_{z});\\
          \big(m(M\cap H_{z}), X+\fa_{\fh}\big)&\mapsto m\exp(X)(L_{Q}\cap H_{z})
          \end{align*}
          is a diffeomorphism.
  \item\label{Prop LST holds for adapted points - item 4}
          The map
          $$
          N_{Q}\times L_{Q}/(L_{Q}\cap H_{z})\to Z;
            \quad\big(n,l(L_{Q}\cap H_{z})\big)\mapsto nl\cdot z
          $$
          is a diffeomorphism onto $P\cdot z$.
\end{enumerate}
\end{Prop}

The adapted points in a given open $P$-orbit are up to $MA$-translation parameterized by $Q$-regular elements in $\fa\cap\fa_{\fh}^{\perp}$, i.e., by elements $X\in \fa\cap\fa_{\fh}^{\perp}$ so that $Z_{\fg}(X)=\fl_{Q}$. The following proposition follows directly from \cite[Proposition 3.12]{KuitSayag_OnTheLittleWeylGroupOfARealSphericalSpace}.

\begin{Prop}\label{Prop parameterization adapted points}
Let $\cO$ be an open $P$-orbit in $Z$. Let $X\in \fa\cap\fa_{\fh}^{\perp}$.
If $Z_{\fg}(X)=\fl_{Q}$, then there exists an adapted point $z\in \cO$ so that $X\in\fh_{z}^{\perp}$. Moreover, if $z'\in\cO$ is another adapted point so that $X\in\fh_{z'}^{\perp}$, then there exist $m\in M$ and $a\in A$ so that $z'=ma\cdot z$.
\end{Prop}

\subsection{Limits of subspaces}\label{Subsection Orbits of max rank - Limits of subspaces}
In this section we discuss limits of subspaces of $\fg$ in the Grassmannian and their properties.

\medbreak

For $k\in\N$ let $\Gr(\fg,k)$ be the Grassmannian of $k$-dimensional subspaces of the Lie algebra $\fg$.

We say that an element $X\in\fa$ is {\it order-regular} if
$$
\alpha(X)\neq \beta(X)
$$
for all $\alpha,\beta\in \Sigma$ with $\alpha\neq \beta$.

If $X\in\fa$ is order-regular, then in particular $\alpha(X)\neq-\alpha(X)$ and therefore $\alpha(X)\neq 0$ for every $\alpha\in\Sigma$. This implies that order-regular elements in $\fa$ are regular. The name order-regular refers to the fact that every order-regular element $X\in\fa$ determines a linear order $\geq$ on $\Sigma$ by setting
$$
\alpha\geq\beta
\quad\text{if and only if}\quad
\alpha(X)\geq\beta(X)
$$
for $\alpha,\beta\in\Sigma$.

The following proposition is taken from \cite[Lemma 4.1]{KrotzKuitOpdamSchlichtkrull_InfinitesimalCharactersOfDiscreteSeriesForRealSphericalSpaces} and \cite[Proposition 5.2]{KuitSayag_OnTheLittleWeylGroupOfARealSphericalSpace}.

\begin{Prop}\label{Prop Limits of subspaces}
Let $E\in \Gr(\fg,k)$ and let $X\in\fa$. The limit
$$
E_{X}:=\lim_{t\to\infty}\Ad\big(\exp(tX)\big)E,
$$
exists in the Grassmannian $\Gr(\fg,k)$. If $\lambda_{1}<\lambda_{2}<\dots<\lambda_{n}$ are the eigenvalues and $p_{1},\dots,p_{n}$ the corresponding projections onto the eigenspaces $V_{i}$ of $\ad(X)$, then $E_{X}$ is given by
\begin{equation}\label{eq formula for E_X}
E_{X}
=\bigoplus_{i=1}^{n}p_{i}\big(E\cap \bigoplus_{j=1}^{i}V_{j}\big).
\end{equation}
The following hold.
\begin{enumerate}[(i)]
\item\label{Prop Limits of subspaces - item 1} If $E$ is a Lie subalgebra of $\fg$, then $E_{X}$ is a Lie subalgebra of $\fg$.
\item\label{Prop Limits of subspaces - item 2} If $X\in\fa$ is order-regular, then $E_{X}$ is $\fa$-stable.
\item\label{Prop Limits of subspaces - item 3} Let $\cR\subseteq\fa$ be a connected component of the set of order-regular elements in $\fa$. If $X\in\overline{\cR}$ and $Y\in\cR$, then $\big(E_{X}\big)_{Y}=E_{Y}$. In particular, if $X,Y\in\cR$, then $E_{X}=E_{Y}$.
\item\label{Prop Limits of subspaces - item 4}If  $g,g'\in G$ and
$$
\lim_{t\to\infty}\exp(tX)g\exp(-tX)
=g',
$$
then
$$
\big(\Ad(g)E\big)_{X}
=\Ad(g')E_{X}
$$
\item\label{Prop Limits of subspaces - item 5}
Let $E_{\C,X}$ be the limit of $\Ad\big(\exp(tX)\big)E_{\C}$ for $t\to\infty$ in the Grassmannian of $k$-dimensional complex subspaces in the complexification $\fg_{\C}$ of $\fg$. Then
$$
E_{\C,X}
=(E_{X})_{\C}.
$$
\end{enumerate}
\end{Prop}

We note that if $X$ is not order-regular, then $E_{X}$ need not be stable under the action of $\fa$, even if $X$ is regular.

For $z\in Z$ and $X\in\fa$ we define
$$
\fh_{z,X}
:=(\fh_{z})_{X}.
$$

\subsection{Compression cone}\label{Subsection Orbits of max rank - Compression cone}
We may and will assume that the point $e H\in G/H=Z$ is adapted. We define
$$
\fh_{\emptyset}
:=(\fl_{Q}\cap\fh)\oplus\overline{\fn}_{Q}.
$$
For $z\in Z$, we define the cone
$$
\cC_{z}
:=\{X\in\fa:\fh_{z,X}=\Ad(m)\fh_{\emptyset} \text{ for some }m\in M\}.
$$
By \cite[Proposition 6.5]{KuitSayag_OnTheLittleWeylGroupOfARealSphericalSpace} the cones $\cC_{z}$ are the same for all adapted points $z\in Z$. We therefore may define
$$
\cC
:=\cC_{z},
$$
where $z$ is any adapted point in $Z$. We call $\cC$ the compression cone of $Z$.

In the following proposition we list some of the properties of the compression cone from \cite[Section 6]{KuitSayag_OnTheLittleWeylGroupOfARealSphericalSpace}.

\begin{Prop}\label{Prop properties compression cone}
\,
\begin{enumerate}[(i)]
\item\label{Prop properties compression cone - item 1} Let $z\in Z$. If $P\cdot z$ is not open, then $\cC_{z}=\emptyset$. If $P\cdot z$ is open, then $\fa^{-}\subseteq\cC_{z}\subseteq\cC$.
\item\label{Prop properties compression cone - item 2} $\cC=\cC+\fa_{\fh}$.
\item\label{Prop properties compression cone - item 3} $\overline{\cC}$ is a finitely generated cone.
\end{enumerate}
\end{Prop}

The edge of $\overline{\cC}$ we denote by $\fa_{E}$, i.e.,
\begin{equation}\label{eq def a_E}
\fa_{E}
:=\overline{\cC}\cap -\overline{\cC}.
\end{equation}
Note that $\fa_{\fh}\subseteq\fa_{E}$, but that in general $\fa_{E}$ may be strictly larger.
We recall from \cite[Section 6]{KnopKrotzSayagSchlichtkrull_SimpleCompactificationsAndPolarDecomposition} that $Z$ is called wavefront if
$$
\cC
=\fa^{-}+\fa_{\fh}.
$$
For wavefront spaces $Z$ we have $\fa_{E}=\fa_{\fh}$. All reductive symmetric spaces are wavefront, i.e., all spaces $G/H$ with $H$ an open subgroup of the fixed point subgroup of an involutive automorphism of $G$.

\subsection{Rank of a $P$-orbit}\label{Subsection Orbits of max rank - Rank of orbit}

In this section we define the rank of a $P$-orbit in $Z$.  We begin with a lemma.

\begin{Lemma}\label{Lemma fa cap fh_(pz,X) independent of p,X}
Let $z\in Z$. The set $\fa\cap\fh_{p\cdot z,X}$ is the same for all $p\in P$ and $X\in\fa^{-}$.
\end{Lemma}

\begin{proof}
Let $X\in \fa^{-}$ and $p\in P$. Further, let $p_{\fa}:\fp\to\fa$ be the projection along the decomposition $\fp=\fm\oplus\fa\oplus\fn_{P}$.
In view of (\ref{eq formula for E_X})
$$
\fa\cap\fh_{p\cdot z,X}
=p_{\fa}\big(\fh_{p\cdot z}\cap\fp\big).
$$
Note that $p_{\fa}$ is invariant under the adjoint action of $P$ on $\fp$.
As
$$
\fh_{p\cdot z}\cap\fp
=\Ad(p)\fh_{z}\cap\fp
=\Ad(p)\big(\fh_{z}\cap\fp\big),
$$
it follows that
$$
\fa\cap\fh_{p\cdot z,X}
=p_{\fa}\big(\fh_{z}\cap\fp\big).
$$
The right-hand side is independent of $p$ and $X$.
\end{proof}

Let $\cO$ be a $P$-orbit in $Z$.
Lemma \ref{Lemma fa cap fh_(pz,X) independent of p,X} allows us to define the set
$$
\fa_{\cO}
:=\fa\cap\fh_{z,X}
$$
where $z$ is any point $\cO$ and $X$ is any element in $\fa^{-}$.
We call the dimension of $\fa/\fa_{\cO}$ the rank of the orbit $\cO$.

\begin{Rem}
If $\cO$ is an open $P$-orbit, then it follows from Proposition \ref{Prop properties compression cone} (\ref{Prop properties compression cone - item 2}) that
$$
\fa_{\cO}
=\fa_{\fh}.
$$
\end{Rem}

\subsection{$P$-Orbits of maximal rank}\label{Subsection Orbits of max rank - Orbits of max rank}
The main result in this section is the following proposition, which will be crucial in this article.

\begin{Prop}\label{Prop limits of max rank orbits are conjugates of h_empty}
Let $\cO\in P\bs Z$. Then
$$
\rank(\cO)\leq\rank (Z).
$$
Let $X\in\fa^{-}$ be order-regular and let $z\in\cO$.
Then  $\rank(\cO)=\rank(Z)$ if and only if there exists a  $w\in N_{G}(\fa)$ so that
\begin{equation}\label{eq h_(z,X)=Ad(w)h_empty}
\fh_{z,X}=\Ad(w)\fh_{\emptyset}.
\end{equation}
If (\ref{eq h_(z,X)=Ad(w)h_empty}) holds, then
\begin{equation}\label{eq a_O=Ad(w)a_h}
\fa_{\cO}
=\Ad(w)\fa_{\fh}
\end{equation}
and there exists an open $P$-orbit $\cO'$ in $Z$ so that
$$
w^{-1}\cdot \cO
\subseteq \cO'.
$$
\end{Prop}

We say that a $P$-orbit $\cO$ in $Z$ is of maximal rank if $\rank(\cO)=\rank(Z)$.

\begin{Rem}\label{Rem Orbits of max rank}
\,
\begin{enumerate}[(a)]
\item\label{Rem Orbits of max rank - item 1}
Fix a $P$-orbit $\cO$ of maximal rank, a point $z\in \cO$ and an order-regular element $X\in\fa^{-}$. The element $w\in N_{G}(\fa)$ in (\ref{eq h_(z,X)=Ad(w)h_empty}) is not unique.  It follows from \cite[Lemma 10.3]{KuitSayag_OnTheLittleWeylGroupOfARealSphericalSpace} that the stabilizer of $\fh_{\emptyset}$ in $N_{G}(\fa)$ is equal to $N_{L_{Q}}(\fa)$. Therefore, the equality (\ref{eq h_(z,X)=Ad(w)h_empty}) only determines the coset $w N_{L_{Q}}(\fa)\in N_{G}(\fa)/N_{L_{Q}}(\fa)$.
The element $w$ may be chosen so that
\begin{equation}\label{eq wSigma^+ cap -Sigma^+=wSigma(Q) cap -Sigma}
\Ad^{*}(w)\Sigma(P)\cap \big(-\Sigma(P)\big)
=\Ad(w)^{*}\Sigma(Q)\cap \big(-\Sigma(P)\big).
\end{equation}
To see this, consider the group $L_{Q}'=(L_{Q}\cap H)A$. $L_{Q}'$ is reductive and normalizes $\fh_{\emptyset}$.
Since $P\cap wL_{Q}'w^{-1}$ and $w(P\cap L_{Q}')w^{-1}$ are both minimal parabolic subgroups of $wL_{Q}'w^{-1}$ containing $A$, there exists a $v\in N_{wL_{Q}'w^{-1}}(\fa)$ so that
$$
vw(P\cap L_{Q}')w^{-1}v^{-1}
=P\cap wL_{Q}w^{-1}
=P\cap vwL_{Q}'w^{-1}v^{-1}.
$$
Let $w'=vw$. Then
$$
w'N_{P}w'^{-1}
=w'\big((N_{P}\cap L_{Q}')N_{Q}\big)w'^{-1}
=(N_{P}\cap w'L_{Q}w'^{-1})w'N_{Q}w'^{-1}.
$$
If now $w$ is replaced by $w'$, then it follows that both (\ref{eq h_(z,X)=Ad(w)h_empty}) and (\ref{eq wSigma^+ cap -Sigma^+=wSigma(Q) cap -Sigma}) hold.
\item\label{Rem Orbits of max rank - item 2}
Fix an order-regular element $X\in\fa^{-}$ and a $P$-orbit $\cO$ in $Z$ of maximal rank. The element $w\in N_{G}(\fa)$ in (\ref{eq h_(z,X)=Ad(w)h_empty}) depends on the choice of the point $z\in \cO$. Indeed, it follows from Proposition \ref{Prop Limits of subspaces} (\ref{Prop Limits of subspaces - item 4}) that
$$
\fh_{man\cdot z,X}
=\Ad(m)\fh_{z,X}
\qquad(m\in M,a\in A,n\in N_{P}).
$$
Therefore,
$$
\fh_{man\cdot z,X}
=\Ad(mw)\fh_{\emptyset}
\qquad(m\in M,a\in A,n\in N_{P})
$$
if $z\in\cO$ satisfies (\ref{eq h_(z,X)=Ad(w)h_empty}).
Note that the coset $w Z_{G}(\fa)\in N_{G}(\fa)/Z_{G}(\fa)=W$ is independent of $z\in\cO$.
\item\label{Rem Orbits of max rank - item 3}
Fix a $P$-orbit $\cO$ in $Z$ of maximal rank and a point $z\in\cO$.
The element $w\in N_{G}(\fa)$  in (\ref{eq h_(z,X)=Ad(w)h_empty}) depends on the choice of the order-regular element $X\in\fa^{-}$, as can be seen in the following example.

\begin{Ex}\label{Ex Limits in group case}
Assume that $G=\bp G\times \bp G$ for an reductive group $\bp G$, and $H=\diag(\bp G)$. Let $\bp P$ be a minimal parabolic subgroup of $\bp G$ with Langlands decomposition $\bp P=\bp M\bp A\bp N$ and let $\bp \overline{P}=\bp M\bp A\bp \overline{N}$ be opposite to $\bp P$. We write $P$ for the minimal parabolic subgroup $\bp P\times\bp \overline{P}$ of $G$.
Let $\cR$ be a set of representatives for the Weyl group of $\bp G$ in $N_{\bp G}(\bp \fa)$. Then $\cR$ is in bijection with $P\bs G/H$ via the map
$$
\cR\to P\bs G/H;\quad w\mapsto \cO(w):=P(e,w)H.
$$
Now fix $w\in\cR$ and let $z$ be the point $(e,w) H$ in $\cO(w)$. Let $X_{1},X_{2}\in\bp \fa^{-}(\bp P)$. We assume that $X:=(X_{1},-X_{2})\in\fa^{-}$ is order-regular. Then for all $\alpha,\beta\in\Sigma(\bp \fg,\bp \fa)$
$$
\alpha(X_{1})\neq \beta(X_{2}).
$$
The limit subalgebra $\fh_{z,X}$ is equal to
$$
\{(Y,\Ad(w)Y):Y\in\bp\fm\oplus\bp\fa\}
\oplus\bigoplus_{\substack{\alpha\in\Sigma(\bp\fg,\bp\fa)\\ \alpha(X_{1})>-w\cdot\alpha(X_{2})}}\big(\bp\fg_{\alpha}\times\{0\}\big)
\oplus\bigoplus_{\substack{\alpha\in\Sigma(\bp\fg,\bp\fa)\\ \alpha(X_{1})<-w\cdot\alpha(X_{2})}}\big(\{0\}\times\bp\fg_{w\cdot\alpha}\big).
$$
From this formula it follows that every $P$-orbit in $G/H$ is of maximal rank.

If $w=e$, so that $\cO(w)$ is the open $P$ orbit in $G/H$, then
$$
\fh_{z,X}
=\diag(\bp\fm\oplus\bp\fa)
\oplus\big(\bp\overline{\fn}\times\{0\}\big)
\oplus\big(\{0\}\times\bp\fn\big)
=:\fh_{\emptyset}
$$
is independent of the choice of $X$. For other orbits $\fh_{z,X}$ does depend on $X$. We illustrate this by considering the most extreme case: the closed $P$-orbit in $G/H$. Let $w\in\cR$ represent the longest Weyl group element, so that $\cO(w)$ is the closed orbit.
Every choice of $X_{1}$ and $X_{2}$ corresponds to a unique positive system $\bp\Sigma^{+}$ of $\Sigma(\bp\fg,\bp\fa)$ satisfying.
\begin{equation}\label{eq inequality on roots for ex group case}
\alpha(X_{1})>-w\cdot\alpha(X_{2})
\qquad(\alpha\in\bp\Sigma^{+}).
\end{equation}
Vice versa, given a positive system $\bp\Sigma^{+}$ of $\Sigma(\bp\fg,\bp\fa)$, we may choose $X_{1}$ and $X_{2}$ so that (\ref{eq inequality on roots for ex group case}) holds. If (\ref{eq inequality on roots for ex group case}) is satisfied, then
$$
\fh_{z,X}
=\{(Y,\Ad(w)Y):Y\in\bp\fm\oplus\bp\fa\}
\oplus\bigoplus_{\alpha\in\bp\Sigma^{+}}\big(\bp\fg_{\alpha}\times\{0\}\big)
\oplus\bigoplus_{\alpha\in\bp\Sigma^{+}}\big(\{0\}\times\bp\fg_{\alpha}\big).
$$
In this case there exists a $v\in N_{\bp G}(\bp\fa)$ so that
$$
\fh_{z,X}
=\Ad(v,wv)\fh_{\emptyset}.
$$
Note that $v$ is a representative for the element of the Weyl group mapping $\Sigma(\bp\overline{P})$ to $\bp\Sigma^{+}$.
\end{Ex}
\end{enumerate}
\end{Rem}

For the proof of Proposition \ref{Prop limits of max rank orbits are conjugates of h_empty} we need a slight strengthening of \cite[Lemma 10.8]{KuitSayag_OnTheLittleWeylGroupOfARealSphericalSpace}.

\begin{Lemma}\label{Lemma limits in Ad(W)fh_emptyset determined by fa_z,X cap fa}
Let $z\in Z$ and let $X\in\fa$ be order regular. Then $\dim(\fa_{\fh})\geq \dim(\fh_{z,X}\cap\fa)$ if and only if there exists a $w\in N_{G}(\fa)$ so that $\fh_{z,X}=\Ad(w)\fh_{\emptyset}$. In that case
$$
\fa\cap\fh_{z,X}
=\Ad(w)\fa_{\fh}.
$$
\end{Lemma}

The proof for the lemma is essentially the same as the proof of \cite[Lemma 10.8]{KuitSayag_OnTheLittleWeylGroupOfARealSphericalSpace}; in the proof the equality $\dim(\fa_{\fh})= \dim(\fh_{z,X}\cap\fa)$ can straightforwardly be replaced  by the inequality $\dim(\fa_{\fh})\geq \dim(\fh_{z,X}\cap\fa)$.

\begin{proof}[Proof of Proposition \ref{Prop limits of max rank orbits are conjugates of h_empty}]
Let $z\in \cO$ and $X\in\fa^{-}$.
If $\rank(\cO)\geq\rank (Z)$, then $\dim(\fh_{z,X}\cap \fa)=\dim(\fa_{\cO})\leq \dim (\fa_{\fh})$. By Lemma \ref{Lemma limits in Ad(W)fh_emptyset determined by fa_z,X cap fa} there exists a $w\in N_{G}(\fa)$ so that $\fh_{z,X}=\Ad(w)\fh_{\emptyset}$. Moreover, if $\fh_{z,X}=\Ad(w)\fh_{\emptyset}$ for some $w\in N_{G}(\fa)$, then
$$
\fa_{\cO}
=\fh_{z,X}\cap\fa
=\Ad(w)\fh_{\emptyset}\cap\fa
=\Ad(w)\fa_{\fh},
$$
and hence $\rank(\cO)=\rank(Z)$.

It remains to prove the existence of an open $P$-orbit $\cO'$ in $Z$ so that $w\cdot \cO\subset\cO'$. We first prove that $w^{-1}\cdot z$ lies in an open $P$-orbit.
As $\fg=\fp+\fh_{\emptyset}$ we have
$$
\fh_{z,X}+\Ad(w)\fp
=\Ad(w)\big(\fh_{\emptyset}+\fp\big)
=\fg.
$$
It follows that for sufficiently large $t>0$ we have
$$
\Ad\big(\exp(tX)\big)\fh_{z}+\Ad(w)\fp
=\fg.
$$
As $\Ad(w) \fp$ and $\fg$ are both $A$-stable, it follows that
$$
\fh_{z}+\Ad(w)\fp
=\fg.
$$
Therefore, $wPw^{-1}\cdot z$ is open in $Z$, and hence $Pw^{-1}\cdot z$ is open in $Z$.

Now set $\cO'=Pw^{-1}\cdot z$. Let $n\in N_{P}$. In view of Proposition \ref{Prop Limits of subspaces} (\ref{Prop Limits of subspaces - item 4}) we have
$$
\fh_{n\cdot z,X}=\Ad(w)\fh_{\emptyset}.
$$
By the argument above, the $P$-orbit $Pw^{-1}n\cdot z$ is open. It follows that $w^{-1}N_{P}\cdot z$ is contained in the union of all open $P$-orbits in $Z$. As $w^{-1}N\cdot z$ is connected, intersects with $\cO'$ and the boundary of $\cO'$ only contains non-open $P$-orbits, it follows that $w^{-1}N\cdot z$ is contained in $\cO'$.
Moreover, since $MA$ is a normal subgroup of $N_{G}(\fa)$ we have
$$
Pw^{-1}man\cdot z
=Pw^{-1}n\cdot z
=\cO'
$$
for all $m\in M$, $a\in A$ and $n\in N_{P}$. This proves the last assertion.
\end{proof}

\subsection{Weakly adapted points}
Let $X\in \fa$. If $\cO$ is a $P$-orbit of maximal rank, then we say that $X$ is $\cO$-regular if $X\in\fa\cap\fa_{\cO}^{\perp}$ and $\alpha(X)\neq 0$ for all roots $\alpha\in\Sigma$ that do not vanish on $\fa\cap\fa_{\cO}^{\perp}$.
We say that a point $z\in Z$ is weakly adapted (to the Langlands decomposition $P=MAN$) if the following two conditions are satisfied.
\begin{enumerate}[(i)]
\item\label{Def Weakly dapted points - item 1}
        The $P$-orbit $\cO=P\cdot z$ is of maximal rank.
\item\label{Def Weakly adapted points - item 3}
        There exists an $\cO$-regular element in $\fa\cap\fh_{z}^{\perp}$.
\end{enumerate}
Note that an adapted point $z\in Z$ is also weakly adapted.

The weakly adapted points in a given maximal rank $P$-orbit admit a similar parametrization as the adapted points in Proposition \ref{Prop parameterization adapted points}.

\begin{Prop}\label{Prop parameterization weakly adapted points}
Let $\cO\in P\bs Z$ be of maximal rank. The following hold.
\begin{enumerate}[(i)]
\item\label{Prop parameterization weakly adapted points - item 1}
    For every $\cO$-regular element $X\in\fa$ there exists a weakly adapted point $z\in \cO$ so that $X\in\fh_{z}^{\perp}$. Moreover, if $z'\in\cO$ is another adapted point so that $X\in\fh_{z'}^{\perp}$, then there exist $m\in M$ and $a\in A$ so that $z'=ma\cdot z$.
\item\label{Prop parameterization weakly adapted points - item 2} Let $z\in \cO$ be weakly adapted and $w\in N_{G}(\fa)$. If there exists an $X\in\fa^{-}$ so that $\fh_{z,X}=\Ad(w)\fh_{\emptyset}$,  then $w^{-1}\cdot z$ is adapted.
\end{enumerate}
\end{Prop}

For the proof of the proposition we need the following lemma. We write $p_{\fa}:\fg\to\fa$ for the projection onto $\fa$ along the root space decomposition.

\begin{Lemma}\label{Lemma a-proj of w-part of h_z}
Let $\cO$ be a $P$-orbit of maximal rank and let $z\in\cO$. Let $w\in N_{G}(\fa)$ be so that $\fh_{z,X}=\Ad(w)\fh_{\emptyset}$ for some order-regular element $X\in\fa^{-}$. Then
$$
p_{\fa}\big((\fp+\Ad(w)\fq)\cap\fh_{z}\big)
=\fa_{\cO}.
$$
\end{Lemma}

\begin{proof}
It follows from (\ref{eq formula for E_X}) that
$$
p_{\fa}(\fp\cap\fh_{z})
=\fa\cap\fh_{z,X}
=\fa_{\cO},
$$
where $X$ is any element in $\fa^{-}$.
Therefore,
$$
\fa_{\cO}
\subseteq p_{\fa}\big((\fp+\Ad(w)\fq)\cap\fh_{z}\big).
$$

We move on to the other inclusion.
Let $Y\in \fp+\Ad(w)\fq$ and assume that $Y\in\fh_{z}$. We will prove that $p_{\fa}(Y)\in\fa_{\cO}$. We decompose $Y$ as
$$
Y
=Y_{\fp}+Y_{0}+Y_{-},
$$
where $Y_{\fp}\in\fp$, $Y_{0}\in\Ad(w)\fl_{Q,\nc}$ and $Y_{-}\in\Ad(w)\fn_{Q}\cap\overline{\fn}_{P}$.

In view of Proposition \ref{Prop limits of max rank orbits are conjugates of h_empty} the $P$-orbit $Pw^{-1}\cdot z$ is open.  Therefore, there exists a $n\in N_{Q}$ so that the connected subgroup $L_{Q,\nc}$ with Lie algebra $\fl_{Q,\nc}$ is contained in $H_{nw^{-1}\cdot z}$. It follows that $\Ad(wn^{-1})\fl_{Q,\nc}\subseteq \fh_{z}$. Note that $\Ad(w)\fl_{Q,\nc}\subseteq \Ad(wn^{-1})\fl_{Q,\nc}+\Ad(w)\fn_{Q}$. Therefore, there exists a $Y'\in \Ad(w)\fn_{Q}$ so that $Y_{0}+Y'\in\Ad(wn^{-1})\fl_{Q,\nc}\subseteq \fh_{z}$. Note that $p_{\fa}(Y_{0}+Y')\in\Ad(w)\fa_{\fh}=\fa_{\cO}$. By subtracting  $Y_{0}+Y'$ from $Y$ we may thus without loss of generality assume that $Y_{0}=0$, i.e.,
$$
Y
=Y_{\fp}+Y_{-}.
$$
Let $X\in\fa^{-}$ be order-regular and satisfy $\fh_{z,X}=\Ad(w)\fh_{\emptyset}$. The line $(\R  Y)_{X}$ is $\fa$-stable and contained in $\Ad(w)\fh_{\emptyset}$. Note that $Y_{-}$ is a linear combination of eigenvectors of $\ad(X)$ with strictly positive eigenvalues, whereas $Y_{\fp}$ is a linear combination of eigenvectors with non-positive eigenvalues. It follows that $Y_{-}=0$ as $(\R Y)_{X}$ would otherwise be a line in $\Ad(w)\fn_{Q}\cap\overline{\fn}_{P}$, which would be in contradiction with the fact that $(\R Y)_{X}$ is contained in $\Ad(w)\fh_{\emptyset}$. Now $Y\in\fp\cap\fh_{z}$ and hence $p_{\fa}(Y)\in p_{\fa}(\fp\cap\fh_{z}) =\fa_{\cO}$.
\end{proof}

\begin{proof}[Proof of Proposition \ref{Prop parameterization weakly adapted points}]
Let $z_{0}\in\cO$ and let $X\in\fa$ be $\cO$-regular.
Let $X'\in \fa^{-}$ be order-regular. By Proposition \ref{Prop limits of max rank orbits are conjugates of h_empty} there exists a $w\in N_{G}(\fa)$ so that $\fh_{z_{0},X'}=\Ad(w)\fh_{\emptyset}$.
It follows from Lemma \ref{Lemma a-proj of w-part of h_z} that
$$
X
\in\big((\fp+\Ad(w)\fq)\cap\fh_{z_{0}}\big)^{\perp}
=(\fn_{P}\cap\Ad(w)\fn_{Q})+\fh_{z_{0}}^{\perp}.
$$
In particular, there exists a $Y\in \fn_{P}\cap\Ad(w)\fn_{Q}$ so that $X+Y\in\fh_{z_{0}}^{\perp}$. Since $X$ is $\cO$-regular, it follows from (\ref{eq a_O=Ad(w)a_h}) that $\alpha(X)\neq 0$ for every $\alpha\in\Sigma$ so that $\alpha|_{\fa\cap\Ad(w)\fa_{\fh}^{\perp}}\neq 0$. As the roots of $\fa$ in $\fn_{P}\cap\Ad(w)\fn_{Q}$ do not vanish on $\fa\cap\Ad(w)\fa_{\fh}^{\perp}$, this implies that there exists a $n\in N_{P}\cap wN_{Q}w^{-1}$ so that  $\Ad(n)X=X+Y$.
Set $z=n^{-1}\cdot z_{0}$. Then
$$
X
\in \Ad(n^{-1})\fh_{z_{0}}^{\perp}
=\fh_{z}^{\perp}.
$$
This proves the first assertion in (\ref{Prop parameterization weakly adapted points - item 1}).

We move on to the second assertion in (\ref{Prop parameterization weakly adapted points - item 1}).
Let $z'\in\cO$ be another point so that $X\in\fh_{z'}^{\perp}$. By Proposition \ref{Prop limits of max rank orbits are conjugates of h_empty} the points $w^{-1}\cdot z$ and $w^{-1}\cdot z'$ lie in the same open $P$-orbit. Moreover,
$$
\Ad(w^{-1})X\in \fh_{w^{-1}\cdot z}^{\perp}\cap\fh_{w^{-1}\cdot z'}^{\perp}.
$$
By Proposition \ref{Prop parameterization adapted points}
$$
w^{-1}\cdot z'
\in MAw^{-1}\cdot z.
$$
As $MA$ is a normal subgroup of $N_{G}(\fa)$, it follows that
$$
z'
\in wMAw^{-1}\cdot z
= MA\cdot z.
$$
This concludes the proof of (\ref{Prop parameterization weakly adapted points - item 1}).

It remains to prove (\ref{Prop parameterization weakly adapted points - item 2}). Assume that $z$ is weakly adapted and there exists an $X\in\fa^{-}$ so that $\fh_{z,X}=\Ad(w)\fh_{\emptyset}$. By Proposition \ref{Prop limits of max rank orbits are conjugates of h_empty} the $P$-orbit through $w^{-1}\cdot z$ is open. Moreover, as $\fa\cap\fh_{z}^{\perp}$ contains $\cO$-regular elements, the set
$$
\fa\cap \fh_{w^{-1}\cdot z}^{\perp}
=\Ad(w^{-1})\big(\fa\cap\fh_{z}^{\perp}\big)
$$
contains elements $X$ so that $\alpha(X)\neq 0$ for all $\Sigma(Q)$. It follows that $w^{-1}\cdot z$ is adapted.
\end{proof}

\subsection{Structure of orbits of maximal rank}\label{Subsection Orbits of max rank - Structure of max rank orbit}

\begin{Prop}\label{Prop decomp of N}
Let $\cO\in P\bs Z$ be of maximal rank and $w\in N_{G}(\fa)$. Assume that there exist a point $z\in \cO$ and an order-regular element $X\in\fa^{-}$ so that $$
\fh_{z,X}
=\Ad(w)\fh_{\emptyset}.
$$
Then for every $y\in\cO$ the multiplication map
\begin{equation}\label{eq decomp of N}
(N_{P}\cap wN_{Q}w^{-1})\times(N_{P}\cap H_{y})\to N_{P};
\quad(n,n_{H})\mapsto nn_{H}
\end{equation}
is a diffeomorphism.
\end{Prop}

\begin{proof}
We prove the assertion first for $y=z$.
It follows from (\ref{eq formula for E_X}) in Proposition \ref{Prop Limits of subspaces} that
$$
(\fn_{P}\cap \fh_{z})_{X}
=\fn_{P}\cap \fh_{z,X}
=\fn_{P}\cap \Ad(w)\fh_{\emptyset}.
$$
Since
$$
\fg
=\Ad(w)\fn_{Q}\oplus(\Ad(w)\fh_{\emptyset}+\fm+\fa),
$$
and this decomposition is compatible with the root space decomposition of $\fg$, it follows that
$$
\fn_{P}
=(\fn_{P}\cap \Ad(w)\fn_{Q})\oplus(\fn_{P}\cap \fh_{z})_{X}.
$$
Hence, for sufficiently large $t>0$
$$
\fn_{P}
=(\fn_{P}\cap \Ad(w)\fn_{Q})\oplus\Ad\big(\exp(tX)\big)(\fn_{P}\cap \fh_{z})
$$
As $\fn_{P}$ and $\fn_{P}\cap \Ad(w)\fn_{Q}$ are both $\fa$-stable, it follows, that
$$
\fn_{P}
=(\fn_{P}\cap \Ad(w)\fn_{Q})\oplus(\fn_{P}\cap \fh_{z})
$$
and thus (\ref{eq decomp of N}) is a local diffeomorphism onto an open neighborhood of $e$ in $N_{P}$.
It remains to show that (\ref{eq decomp of N}) is a bijection.

The intersection $(N_{P}\cap wN_{Q}w^{-1})\cap(N_{P}\cap H_{z})$ is an algebraic subgroup of $N_{P}$ of dimension $0$. The only such subgroup is the trivial one. Therefore, (\ref{eq decomp of N}) is injective.

By \cite[Theorem 2]{Rosenlicht_OnQuotientVarietiesAndTheAffineEmbeddingOfCertainHomogeneousSpaces} both $N_{P}\cdot z$ and $(N_{P}\cap wN_{Q}w^{-1})\cdot z$ are closed submanifolds of $Z$. Since the image of (\ref{eq decomp of N}) is open in $N_{P}$, the set $(N_{P}\cap wN_{Q}w^{-1})\cdot z$ is a relatively open subset of $N_{P}\cdot z$. Hence, $(N_{P}\cap wN_{Q}w^{-1})\cdot z$ is open and closed in $N_{P}\cdot z$. As $N_{P}\cdot z$ is connected, it follows that
\begin{equation}\label{eq (N_P cap wN_Qw^(-1)) cdot z = N_P cdot z}
(N_{P}\cap wN_{Q}w^{-1})\cdot z
=N_{P}\cdot z.
\end{equation}
From this we conclude that (\ref{eq decomp of N}) is surjective and this concludes the proof of the proposition for $z=y$.

Let now $y\in \cO$ and let $m\in M$, $a\in A$ and $n\in N_{P}$ be such that $y=man\cdot z$. The identity (\ref{eq (N_P cap wN_Qw^(-1)) cdot z = N_P cdot z}) shows that we may choose $n\in N_{P}\cap wN_{Q}w^{-1}$. Since the groups $(N_{P}\cap wN_{Q}w^{-1})$ and $N_{P}$ are normalized by $MA(N_{P}\cap wN_{Q}w^{-1})$, the assertion in the proposition now follows from the case $z=y$.
\end{proof}

\begin{Thm}\label{Thm structure theorem for wPw^(-1) cdot z}
Let $\cO\in P\bs Z$ be of maximal rank and $z\in \cO$ weakly adapted. Let $w\in N_{G}(\fa)$ be so that
$$
\fh_{z,X}=\Ad(w)\fh_{\emptyset}
$$
for some order-regular $X\in\fa^{-}$.
Then $\fa_{\cO}=\Ad(w)\fa_{\fh}\subseteq\fh_{z}$. Moreover,  $Pw^{-1}\cdot z$ is open and the maps
\begin{align}
\label{eq local structure thm for Pw^(-1) cdot z}
    N_{Q}\times M/(M\cap H_{w^{-1}\cdot z})\times A/\exp(\fa_{\fh})\to Pw^{-1}\cdot z;&\quad(n,m,a)\mapsto nmaw^{-1}\cdot z\\
\nonumber\\
\label{eq local structure thm for O}
    (N_{P}\cap wN_{Q}w^{-1})\times M/(M\cap H_{z})\times A/\exp(\fa_{\cO})\to \cO;&\quad(n,m,a)\mapsto nma\cdot z\\
\nonumber\\
\label{eq decomp of wPw^(-1) wrt O}
    (\overline{N}_{P}\cap wN_{Q}w^{-1})\times \cO\to wPw^{-1}\cdot z;&\quad(n,x)\mapsto n\cdot x
\end{align}
are diffeomorphisms.
\end{Thm}

\begin{Rem}
The diffeomorphism (\ref{eq local structure thm for O}) may be viewed as a structure theorem for a $P$-orbit of maximal rank. For complex spherical spaces this structure theorem was first proven by Brion in \cite[Proposition 6 \& Theorem 3]{Brion_OrbitClosuresOfSphericalSubgroupsInFlagVarieties}. For our purposes the diffeomorphism (\ref{eq decomp of wPw^(-1) wrt O}) will be of particular  importance for the construction of distributions in Section \ref{Section Construction}.
\end{Rem}

\begin{proof}[Proof of Theorem \ref{Thm structure theorem for wPw^(-1) cdot z}]
In view of Proposition \ref{Prop parameterization weakly adapted points} (\ref{Prop parameterization weakly adapted points - item 2}) the $P$-orbit through $w^{-1}\cdot z$ is open and $w^{-1}\cdot z$ is adapted. The map (\ref{eq local structure thm for Pw^(-1) cdot z}) is a diffeomorphism by Proposition \ref{Prop LST holds for adapted points}.

It follows from Proposition \ref{Prop limits of max rank orbits are conjugates of h_empty} that
$$
w^{-1}\cdot\cO
\subseteq Pw^{-1}\cdot z.
$$
In view of Proposition \ref{Prop decomp of N}
$$
\cO
=(N_{P}\cap wN_{Q}w^{-1})MA\cdot z,
$$
and hence
$$
w^{-1}\cdot \cO
=(w^{-1}N_{P}w\cap N_{Q})MAw^{-1}\cdot z
$$
Since (\ref{eq local structure thm for Pw^(-1) cdot z}) is a diffeomorphism, the map
$$
(w^{-1}N_{P}w\cap N_{Q})\times M/(M\cap H_{w^{-1}\cdot z})\times A/\exp(\fa_{\fh})\to w^{-1}\cO;\quad(n,ma)\mapsto nmaw^{-1}\cdot z\\
$$
is a diffeomorphism. As $\fa_{\cO}=\Ad(w)\fa_{\fh}$, this implies that (\ref{eq local structure thm for O}) is a diffeomorphism.

Finally, the map (\ref{eq decomp of wPw^(-1) wrt O}) is a diffeomorphism since (\ref{eq local structure thm for Pw^(-1) cdot z}), (\ref{eq local structure thm for O}) and the product map
$$
(N_{P}\cap wN_{Q}w^{-1})\times (\overline{N}_{P}\cap wN_{Q}w^{-1})\to wN_{Q}w^{-1}
$$
are diffeomorphisms.
\end{proof}

\subsection{Admissible points and the little Weyl group}\label{Subsection Orbits of max rank - Admissible points and little Weyl group}
Following \cite[Definition 10.1]{KuitSayag_OnTheLittleWeylGroupOfARealSphericalSpace} we call a point $z\in Z$ admissible if it is adapted and if for every order-regular element $X\in\fa$ there exists a $w\in N_{G}(\fa)$ so that $\fh_{z,X}=\Ad(w)\fh_{\emptyset}$. By \cite[Proposition 10.4]{KuitSayag_OnTheLittleWeylGroupOfARealSphericalSpace} the set of admissible points is open and dense in the set of adapted points in $Z$ (with respect to the subspace topology). In particular, every open $P$-orbit in $Z$ contains an admissible point. We may and will assume that the point $eH\in G/H\in Z$ is admissible.

We define the groups
\begin{equation}\label{eq def cN}
\cN
:=N_{G}(\fa)\cap N_{G}(\fa_{\fh})
\end{equation}
and
\begin{equation}\label{eq def cZ}
\cZ
:=\{w\in N_{G}(\fa):\Ad(w)\fh_{\emptyset}=\Ad(m)\fh_{\emptyset}\text{ for some }m\in M\}
=N_{L_{Q}}(\fa)
\end{equation}
Note that $\cZ$ is a normal subgroup of $\cN$.
For an admissible point $z\in Z$ we set
\begin{equation}\label{eq def cW}
\cW
:=\{w\in N_{G}(\fa):\fh_{z,X}=\Ad(wm)\fh_{\emptyset} \text{ for some }X\in\fa\text{ and }m\in M\}.
\end{equation}
By \cite[Proposition 10.4]{KuitSayag_OnTheLittleWeylGroupOfARealSphericalSpace} the set $\cW$ does not depend on the choice of the admissible point. Furthermore, by \cite[Theorem 11.1]{KuitSayag_OnTheLittleWeylGroupOfARealSphericalSpace} it is a subgroup of $\cN$.  The quotient group
$$
W_{Z}
:=\cW/\cZ
$$
is equal to the little Weyl group of $Z$ as defined in \cite{KnopKrotz_ReductiveGroupActions}. The little Weyl group acts on $\fa/\fa_{\fh}$ as a finite reflection group. The set $\overline{\cC}/\fa_{\fh}$ is a fundamental domain for this action.

Let $\fa_{E}$ be the edge of $\overline{\cC}$, i.e.,
$$
\fa_{E}
:=\overline{\cC}\cap -\overline{\cC}.
$$
The little Weyl group acts trivially on $\fa_{E}/\fa_{\fh}$. See \cite[Lemma 12.1]{KuitSayag_OnTheLittleWeylGroupOfARealSphericalSpace}. Moreover, by \cite[Proposition 10.3]{KnopKrotz_ReductiveGroupActions} and \cite[Theorem 12.2]{KuitSayag_OnTheLittleWeylGroupOfARealSphericalSpace} the little Weyl group is the Weyl group of a root system in $(\fa/\fa_{E})^{*}$. This root system is called the spherical root system. We will indicate it by $\Sigma_{Z}$.

For our purposes the following characterization of $\cW$ is important.

\begin{Prop}\label{Prop characterization cV}
Let $z\in Z$ be admissible and let $w\in N_{G}(\fa)$. Then $Pw^{-1}\cdot z$ is open if and only if $w\in \cW$. In that case $w^{-1}\cdot z$ is admissible.
\end{Prop}

\begin{proof}
The assertion follows from \cite[Proposition 7.2]{KuitSayag_OnTheLittleWeylGroupOfARealSphericalSpace} and the equivariance of limits of subalgebras
$$
\fh_{w^{-1}\cdot z, X}
=\Ad(w^{-1})\fh_{z,\Ad(w)X}
\qquad(X\in\fa).
$$
\end{proof}

If $Z$ is wavefront, then the $\cW=\cN$ and the little Weyl group is equal to $W_{Z}=\cW/\cZ$. See Proposition \ref{Prop wavefront} in Appendix A.

\subsection{Weakly admissible points}\label{Subsection Orbits of max rank - Weakly admissible}
We call a point $z\in Z$ weakly admissible if it is weakly adapted and for every order-regular element $X\in\fa$ there exists a $w\in N_{G}(\fa)$ so that $\fh_{z,X}=\Ad(w)\fh_{\emptyset}$. Note that every admissible point is weakly admissible.

\begin{Prop}\label{Prop weakly admissible is W-stable}
Let $z\in Z$. If $z$ is weakly admissible, then $w\cdot z$ is weakly admissible for every $w\in N_{G}(\fa)$.
\end{Prop}

\begin{Rem}
As the set of admissible points is (relatively) open and dense in the set of adapted points, it follows from the proposition and Proposition \ref{Prop parameterization weakly adapted points} (\ref{Prop parameterization weakly adapted points - item 2}) that the set of weakly admissible points is open and dense (with respect to the subspace topology) in the set of weakly adapted points.
\end{Rem}

\begin{proof}[Proof of Proposition \ref{Prop weakly admissible is W-stable}]
Assume that $z$ is weakly admissible and $w\in N_{G}(\fa)$.
If $X\in \fa$ is order-regular then $\Ad(w)X$ is also order-regular. Therefore, there exists a $w'\in N_{G}(\fa)$ so that $\fh_{z,\Ad(w^{-1})X}=\Ad(w')\fh_{\emptyset}$, and hence
$$
\fh_{w\cdot z,X}
=\Ad(w)\fh_{z,\Ad(w^{-1})X}
=\Ad(ww')\fh_{\emptyset}.
$$
By Proposition \ref{Prop limits of max rank orbits are conjugates of h_empty} $\cO:=Pw\cdot z$ has maximal rank. It remains to prove that $ w\cdot z$ is weakly adapted.

Let $\cO'=P\cdot z$. Since $z$ is weakly adapted, we have $\fa_{\cO'}\subseteq\fh_{z}$ by Theorem \ref{Thm structure theorem for wPw^(-1) cdot z}, and hence $\Ad(w)\fa_{\cO'}\subseteq\fh_{w\cdot z}$. Let $X\in \fa^{-}$. Then
$$
\Ad(w)\fa_{\cO'}
\subseteq \fa\cap\fh_{w\cdot z,X}
=\fa_{\cO}.
$$
As both $\fa_{\cO'}$ and $\fa_{\cO}$ are conjugate to $\fa_{\fh}$, these two spaces are of equal dimension. Therefore, $\Ad(w)\fa_{\cO'}=\fa_{\cO}$.
Since $z$ is weakly adapted, there exists $\cO'$-regular elements in $\fa\cap\fh_{z}^{\perp}$. It follows that
there exist $\cO$-regular elements in $\fa\cap\fh_{w\cdot z}^{\perp}=\Ad(w)\big(\fa\cap\fh_{z}^{\perp}\big)$. This proves that $w\cdot z$ is weakly adapted.
\end{proof}

\begin{Prop}\label{Prop equivalence for h_(z,X)=Ad(w)h_empty}
Let $z\in Z$ be weakly admissible, let $X\in\fa$ be order-regular and let $w\in N_{G}(\fa)$. Then $\fh_{z,X}=\Ad(wm)\fh_{\emptyset}$ for some $m\in M$ if and only if $Pw^{-1}\cdot z$ is open and $X\in \Ad(w)\cC$.
\end{Prop}

\begin{proof}
We have
$$
\Ad(w)^{-1}\fh_{z,X}
=\fh_{w^{-1}\cdot z,\Ad(w^{-1})X}.
$$
By Proposition \ref{Prop weakly admissible is W-stable} the point $w^{-1}\cdot z $ is weakly admissible.
In view of Proposition \ref{Prop properties compression cone} the limit subalgebra $\fh_{w^{-1}\cdot z,\Ad(w^{-1})X}$ is equal to $\Ad(m)\fh_{\emptyset}$ for some $m\in M$ if and only if $w^{-1}\cdot z$ is open and $\Ad(w^{-1})X\in\cC$.
\end{proof}

\subsection{An action of the Weyl group}\label{Subsection Orbits of max rank - Weyl group action}
We write $(P\bs Z)_{\max}$ for the subset of $P\bs Z$ consisting of all $P$-orbits in $Z$ of maximal rank and $(P\bs Z)_{\open}$ for set of all open $P$-orbits in $Z$.

\begin{Prop}\label{Prop limits same for some X, then same for all X}
Let $\cO_{1},\cO_{2}\in(P\bs Z)_{\max}$, and let $z_{1}\in\cO_{1}$ and $z_{2}\in\cO_{2}$ be weakly admissible. Let further $X_{1},X_{2}\in\fa$ be order-regular, and let $m\in M$. If $\fh_{z_{1},X_{1}}=\Ad(m)\fh_{z_{2},X_{1}}$, then  $\fh_{z_{1},X_{2}}=\Ad(m)\fh_{z_{2},X_{2}}$.
\end{Prop}

\begin{proof}
Let $w\in N_{G}(\fa)$ be so that $\fh_{z_{1},X_{1}}=\Ad(w)\fh_{\emptyset}$. By Proposition \ref{Prop parameterization weakly adapted points} (\ref{Prop parameterization weakly adapted points - item 2}) the point $w^{-1}\cdot z_{1}$ is adapted. In view of Proposition \ref{Prop weakly admissible is W-stable} it is also weakly admissible, and hence $w^{-1}\cdot z_{1}$ is admissible.  Moreover,
$$
\fh_{w^{-1}\cdot z_{1},\Ad(w^{-1})X_{1}}
=\Ad(w^{-1})\fh_{z_{1},X_{1}}
=\fh_{\emptyset},
$$
and hence $\Ad(w^{-1})X_{1}\in\cC$. Let $v\in \cW$ be so that $\Ad(w^{-1})X_{2}\in\Ad(v)\cC$. Then $v^{-1}w^{-1}\cdot z_{1}$ is admissible by Proposition \ref{Prop characterization cV}. As $\Ad(v^{-1}w^{-1})X_{2}\in\cC$, it follows that there exists an $m_{1}\in M$ so that
$$
\fh_{z_{1},X_{2}}
=\Ad(wv)\fh_{v^{-1}w^{-1}\cdot z,\Ad(v^{-1}w^{-1})X_{2}}
=\Ad(m_{1}wv)\fh_{\emptyset}.
$$
In the same way we find
$$
\fh_{z_{2},X_{2}}
=\Ad(m_{2}wv)\fh_{\emptyset}
$$
for some $m_{2}\in M$. Now
$$
\fh_{z_{1},X_{2}}
=\Ad(m_{1}m_{2}^{-1})\fh_{z_{2},X_{2}}.
$$
The latter is equal to $\Ad(m)\fh_{z_{2},X_{2}}$ if and only if
\begin{equation}\label{eq fm condition}
\fm\cap \fh_{z_{1},X_{2}}
=\Ad(m)\big(\fm\cap\fh_{z_{2},X_{2}}\big).
\end{equation}
It thus suffices to prove the latter.

We claim that
\begin{equation}\label{eq fm-components}
\fm\cap \fh_{z, X}
=\fm\cap\fh_{z}
\end{equation}
for all weakly adapted points $z\in Z$ and all order-regular elements $X\in \fa$. To prove the claim we first consider an adapted point $z\in Z$. Then (\ref{eq fm-components}) follows from Proposition \ref{Prop Local structure theorem} if $X\in \fa^{-}$.  Since the limit $\fh_{z,X}$ is the same for all $X\in \cC$, (\ref{eq fm-components}) also holds for $X\in \cC$. If $X\in \fa$ is any order-regular element, then there exists a $u\in \cW$ so that $\Ad(u)X\in\cC$. Then
$$
\fm\cap \fh_{z, X}
=\fm\cap \Ad(u^{-1})\fh_{u\cdot z,\Ad(u)X}
=\Ad(u^{-1})\big(\fm\cap \fh_{u\cdot z,\Ad(u)X}\big).
$$
By Proposition \ref{Prop characterization cV} the point $u\cdot z$ is adapted. Therefore,
$$
\Ad(u^{-1})\big(\fm\cap \fh_{u\cdot z,\Ad(u)X}\big)
=\Ad(u^{-1})\big(\fm\cap \fh_{u\cdot z}\big)
=\fm\cap \fh_{z}.
$$
This proves (\ref{eq fm-components}) in case $z\in Z$ is adapted. Let now $z\in Z$ be weakly adapted. Then there exists a $u\in N_{G}(\fa)$ so that $u\cdot z$ is adapted. In that case
$$
\fm\cap \fh_{z, X}
=\Ad(u^{-1})\big(\fm\cap \fh_{u\cdot z,\Ad(u)X}\big)
=\Ad(u^{-1})\big(\fm\cap \fh_{u\cdot z}\big)
=\fm\cap \fh_{z}.
$$
This proves the claim (\ref{eq fm-components}).

The required identity (\ref{eq fm condition}) follows from (\ref{eq fm-components}) as $\fh_{z_{1},X_{1}}=\Ad(m)\fh_{z_{2},X_{1}}$ and hence
$$
\fm\cap \fh_{z_{1},X_{2}}
=\fm\cap\fh_{z_{1},X_{1}}
=\Ad(m)\big(\fm\cap\fh_{z_{2},X_{1}}\big)
=\Ad(m)\big(\fm\cap \fh_{z_{2},X_{2}}\big).
$$
\end{proof}

If $\cO\in P\bs Z$ and $X\in \fa^{-}$, then up to $M$-conjugacy the limits $\fh_{z,X}$ do not depend on the point $z\in \cO$,  see Remark \ref{Rem Orbits of max rank} (\ref{Rem Orbits of max rank - item 2}). In view of Proposition \ref{Prop limits same for some X, then same for all X} we may thus define an equivalence relation $\sim$ on $(P\bs Z)_{\max}$ by requiring that
$$
\cO_{1}\sim\cO_{2}
$$
if and only if for a given order-regular element $X\in\fa$  there exists weakly admissible points $z_{1}\in \cO_{1}$ and $z_{2}\in \cO_{2}$  so that
$$
\fh_{z_{1},X}=\fh_{z_{2},X}.
$$
The equivalence relation does not depend on the choice of the order-regular element $X\in\fa$.

\begin{Rem}\,
\begin{enumerate}[(a)]
\item
If $\cO\in P\bs Z$, then by Proposition \ref{Prop properties compression cone} the limit subalgebra $\fh_{z,X}$ for a given order-regular element $X\in\fa^{-}$ does not depend on $z\in \cO$ up to $M$-conjugacy. Moreover, $\fh_{m\cdot z,X}=\Ad(m)\fh_{z,X}$ for every $m\in M$. Therefore, two $P$-orbits $\cO_{1},\cO_{2}$ of maximal rank are equivalent if and only if there exists an order-regular elements $X\in\fa^{-}$, points $z_{1}\in\cO_{1}$ and $z_{2}\in\cO_{2}$, and an $m\in M$ so that $\fh_{z_{1},X}=\Ad(m)\fh_{z_{2},X}$.
\item
Let $X\in \fa^{-}$ be order-regular. If $z\in Z$, then by Proposition \ref{Prop properties compression cone} the limit subalgebra $\fh_{z,X}$ is an $M$-conjugate of $\fh_{\emptyset}$ if and only if $P\cdot z$ is open. Therefore, if $\cO_{1}\in P\bs Z$ is open and $\cO_{2}\in (P\bs Z)_{\max}$, then $\cO_{1}\sim\cO_{2}$ if and only if $\cO_{2}$ is open. In particular, the set $(P\bs Z)_{\open}$ of all open $P$-orbits in $Z$ forms an equivalence class.
\end{enumerate}
\end{Rem}

We denote the equivalence classes of $\sim$ by $[\cdot]$ and recall the subgroup $\cW$ of $N_{G}(\fa)$ from (\ref{eq def cW}).

\begin{Thm}\label{Thm properties of W-action}
For any $v\in N_{G}(\fa)$ and any $\cO\in (P\bs Z)_{\max}$, the equivalence class $[Pv\cdot z]$ is independent of the choice of the weakly admissible point $z\in \cO$. For $w\in W$ and  $\cO\in (P\bs Z)_{\max}$ we may thus set
$$
w\cdot [\cO]=[Pv\cdot z],
$$
where $v\in N_{G}(\fa)$ is any representative of $w$ and $z\in\cO$ is any weakly admissible point. The map
$$
W\times (P\bs Z)_{\max}/_{\sim}\to(P\bs Z)_{\max}/_{\sim}
$$
thus obtained defines an action of $W$ on $(P\bs Z)_{\max}/_{\sim}$. This action has the following properties.
\begin{enumerate}[(i)]
\item\label{Thm properties of W-action - item 1} $W$ acts transitively on $(P\bs Z)_{\max}/_{\sim}$.
\item\label{Thm properties of W-action - item 2} The stabilizer of the equivalence class $(P\bs Z)_{\open}$ is equal to $\cW/MA$.
\item\label{Thm properties of W-action - item 3} Let $w\in W$ and let $v\in N_{G}(\fa)$ be a representative for $w$. If $X\in \fa\cap\fa_{\fh}^{\perp}$ satisfies $Z_{\fg}(X)=\fl_{Q}$, and $z_{1},\dots, z_{n}$ is a set of admissible points representing the open $P$-orbits in $Z$ with
    $$
    X\in\fa\cap\fh_{z_{i}}^{\perp}
    \qquad(1\leq i\leq n),
    $$
    then
    $$
    Pw\cdot z_{i}\neq Pw\cdot z_{j}
    \qquad (1\leq i<j\leq n)
    $$
    and
    $$
    w\cdot(P\bs Z)_{\open}
    =\{Pw\cdot z_{i}:1\leq i\leq n\}.
    $$
    In particular, the cardinalities of the equivalence classes are all equal, i.e, for every  $\cO\in(P\bs Z)_{\max}$ the cardinality of $[\cO]$ is equal to the number of open $P$-orbits in $Z$.
\item\label{Thm properties of W-action - item 4} If $w\in W$ and $\cO\in (P\bs Z)_{\max}$, then $\fa_{\cO'}=\Ad(w)\fa_{\cO}$ for every $\cO'\in w\cdot[\cO]$.
\end{enumerate}
\end{Thm}

The action of $W$ on $(P\bs Z)_{\max}$ lifts to an action of $N_{G}(\fa)$. In later sections we will use interchangeably the actions of $W$ and $N_{G}(\fa)$ on $(P\bs Z)_{\open}$ and use the same notation without further indication.

\begin{Rem}\label{Remark properties of W-action}\,
\begin{enumerate}[(a)]
\item\label{Remark properties of W-action - item 1}
Let $\underline{P}\subseteq \underline{G}$ be a minimal parabolic subgroup defined over $\R$  and let $\underline{Z}$ be an algebraic $\underline{G}$-variety.  Assume that $\underline{Z}$ is real spherical, i.e., that $\underline{P}$ admits an open orbit in $\underline{Z}$.  In \cite{KnopZhgoon_ComplexityOfActionsOverPerfectFields} Knop and Zhgoon constructed an action of $W$ on the set of $\underline{P}$-orbits $\underline{\cO}$ in $\underline{Z}$  with the property that $\underline{\cO}(\R)\neq\emptyset$. If $Z$ is an open $G$-orbit in $\underline{Z}(\R)$, then each equivalence class in $(\underline{P}(\R)\bs Z)_{\max}$ corresponds to one $\underline{P}$-orbit $\underline{\cO}$ of maximal rank  in $\underline{Z}$  with the property that $\underline{\cO}(\R)\neq\emptyset$. The construction of Knop and Zhgoon then coincides with the $W$-action on $(P\bs Z)_{\max}/\sim$ from Theorem \ref{Thm properties of W-action}.
\item\label{Remark properties of W-action - item 2}
If $P$ admits only one open orbit in $Z$, then each equivalence class of $\sim$ consists of precisely one $P$-orbit in $Z$. The above action then yields a transitive action of $W$ on $(P\bs Z)_{\max}$. This is in particular the case if $G$ and $H$ are both complex groups; the latter action then coincides with Knop's $W$-action on $P\bs Z$ from \cite{Knop_OnTheSetOfOrbitsForABorelSubgroup} restricted to the set of maximal rank orbits.
\item\label{Remark properties of W-action - item 3}
Assume that in each open $P$-orbit in $Z$ the set of adapted points is precisely equal to one $MA$-orbit, equivalently if there exists a $z\in Z$ so that $P\cdot z$ is open and $\fa\cap\fa_{\fh}^{\perp}\subseteq\fh_{z}^{\perp}$. Then every adapted point is admissible. If $\cO$ is an open orbit in $Z$, $z\in \cO$ is adapted, and $w=vMA\in W$, then the orbit
\begin{equation}\label{eq w cdot O=Pv cdot z}
w\cdot \cO
:=Pv\cdot z
\end{equation}
does not depend on the choice of $z$.  In view of Proposition \ref{Prop weakly admissible is W-stable} the map
\begin{equation}\label{eq W-action of max rank orbits}
W\times(P\bs Z)_{\max}\to (P\bs Z)_{\max};\quad(w,\cO)\mapsto w\cdot \cO
\end{equation}
defines an action of $W$ on $(P\bs Z)_{\max}$. This action is a refinement of the action of $W$ on $(P\bs Z)_{\max}/\sim$ defined in Theorem \ref{Thm properties of W-action}. The condition that in each open $P$-orbit in $Z$ the set of adapted points forms one $MA$-orbit is in particular satisfied in case $Z$ is symmetric, i.e., in case $H$ is an open subgroup of the fixed point subgroup of an involutive automorphism of $G$.
\end{enumerate}
\end{Rem}

\begin{proof}[Proof of Theorem \ref{Thm properties of W-action}]
Recall from Proposition \ref{Prop weakly admissible is W-stable} that the set of weakly admissible points is stable under $N_{G}(\fa)$. In particular, if $\cO\in (P\bs Z)_{\max}$ and $z\in\cO$ is weakly admissible, then $Pw\cdot z\in (P\bs Z)_{\max}$ for all $w\in N_{G}(\fa)$.

Let $\cO\in(P\bs Z)_{\max}$ and let $z_{1},z_{2}\in\cO$ be weakly admissible. Let further $w\in N_{G}(\fa)$. We claim that
\begin{equation}\label{eq Pwz_1 sim Pwz_2}
Pw\cdot z_{1}\sim Pw\cdot z_{2}.
\end{equation}

By Remark \ref{Rem Orbits of max rank} (\ref{Rem Orbits of max rank - item 2}) there exists an $m\in M$ so that for all $X\in\fa^{-}$ we have  $\fh_{z_{1},X}=\Ad(m)\fh_{z_{2},X}$. After replacing $z_{2}$ by $m\cdot z_{2}$ we may assume that $\fh_{z_{1},X}=\fh_{z_{2},X}$ for all $X\in\fa^{-}$.
By Proposition \ref{Prop limits same for some X, then same for all X} the limits $\fh_{z_{1},X}$ and $\fh_{z_{2},X}$ are equal for all order-regular $X\in\fa$. Fix now an order-regular $X\in\fa$ and let $w\in N_{G}(\fa)$.  Then $\Ad(w^{-1})X$ is order-regular, and hence
$$
\fh_{w\cdot z_{1},X}
=Ad(w)\fh_{z_{1},\Ad(w^{-1})X}
=Ad(w)\fh_{z_{2},\Ad(w^{-1})X}
=\fh_{w\cdot z_{2},X}.
$$
This proves the claim (\ref{eq Pwz_1 sim Pwz_2}).

From the claim it follows that for a given weakly adapted point $z_{0}$ and $v\in N_{G}(\fa)$ we have $[Pv\cdot z]=[Pv\cdot z_{0}]$ for all weakly adapted point $z\in P\cdot z_{0}$. This proves the first assertion in the theorem and we thus obtain an action of $W$ on $(P\bs Z)_{\max}$.

We move on to prove the listed properties of this action.
It follows from Proposition \ref{Prop limits of max rank orbits are conjugates of h_empty} that the action is transitive, and from Proposition \ref{Prop characterization cV} that the stabilizer of the equivalence class of open $P$-orbits is equal to $\cW/MA$.

We move on to prove (\ref{Thm properties of W-action - item 3}).
Let $w\in W$, let $v\in N_{G}(\fa)$ be a representative for $w$ and let $X\in \fa\cap\fa_{\fh}^{\perp}$ satisfy $Z_{\fg}(X)=\fl_{Q}$. Every open $P$-orbit admits by Proposition \ref{Prop parameterization adapted points} an admissible point $z$ with $X\in \fa\cap\fh_{z}^{\perp}$. This point is unique up to translation by an element in $MA$. Let $z_{1},\dots, z_{n}$ be a set of adapted point representing the open $P$-orbits in $Z$ and assume that
$$
X\in \fa\cap\fh_{z_{i}}^{\perp}\qquad (1\leq i\leq n).
$$
Then $Pv\cdot z_{i}\in w\cdot(P\bs Z)_{\open}$ for all $1\leq i\leq n$. Moreover, the points $v\cdot z_{1},\dots, v\cdot z_{n}$ are weakly admissible and
$$
\Ad(w)X\in\fa\cap\fh_{v\cdot z_{i}}^{\perp}
\qquad(1\leq i\leq n).
$$
If $Pv\cdot z_{i}=Pv\cdot z_{j}$ for some $1\leq i<j\leq n$, then it follows from Proposition \ref{Prop parameterization weakly adapted points} (\ref{Prop parameterization weakly adapted points - item 1}) that $z_{i}\in MA\cdot z_{j}$, which leads to a contradiction. We conclude that the orbits $Pv\cdot z_{1},\dots Pv\cdot z_{n}$ are pairwise distinct.
It now suffices to show that the number of orbits in $w\cdot (P\bs Z)_{\open}$ does not exceed the number of open orbits.

Let $z\in Z$ be any point so that $[P\cdot z]=w\cdot(P\bs Z)$ and let $Y\in \fa^{-}$. There exists a $u\in N_{G}(\fa)$ so that $\fh_{z,Y}=\Ad(u)\fh_{\emptyset}$. By Proposition \ref{Prop limits of max rank orbits are conjugates of h_empty} the assignment
\begin{equation}\label{eq O mapsto Pv^-1 cdot O}
\cO\mapsto Pu^{-1}\cdot \cO
\end{equation}
maps $w\cdot (P\bs Z)_{\open}$ to $(P\bs Z)_{\open}$.
We claim that this map is injective.

Let $\cO,\cO'\in w\cdot (P\bs Z)_{\open}$ be so that $Pu^{-1}\cdot \cO=Pu^{-1}\cdot \cO'$. By Proposition \ref{Prop parameterization weakly adapted points} (\ref{Prop parameterization weakly adapted points - item 1}) there exist weakly adapted points $z\in\cO$ and $z'\in \cO'$ so that $\Ad(u)X\in\fh_{z}^{\perp}$ and $\Ad(u)X\in\fh_{z'}^{\perp}$. Now $u^{-1}\cdot z$ and $u^{-1}\cdot z'$ are adapted points in the same open orbit. It follows from Proposition \ref{Prop parameterization weakly adapted points} (\ref{Prop parameterization weakly adapted points - item 1}) that $u^{-1}\cdot z\in MAu^{-1}\cdot z'$. This implies that $z\in MA\cdot z'$, and hence
$$
\cO
=P\cdot z
=P\cdot z'
=\cO'.
$$
This proves the injectivity of (\ref{eq O mapsto Pv^-1 cdot O}) and hence (\ref{Thm properties of W-action - item 3}).

Finally, we prove (\ref{Thm properties of W-action - item 4}). Let $\cO\in (P\bs Z)_{\max}$ and $w\in W$. If $z\in \cO$ is weakly admissible, then for every order-regular element $X\in\fa$ we have $\fa_{\cO}=\fa\cap\fh_{z,X}$. Since
$$
\fa\cap\fh_{w\cdot z,X}
=\Ad(w)\big(\fa\cap\fh_{z,\Ad(w^{-1}X)}\big)
=\Ad(w)\fa_{\cO}
\qquad(X\in\fa \text{ order-regular})
$$
it follows that $\fa_{Pw\cdot z}=\Ad(w)\fa_{\cO}$.
This proves (\ref{Thm properties of W-action - item 4}).
\end{proof}

\section{Distribution vectors of principal series representations}
\label{Section Distribution vectors}
\subsection{Basic Definitions}
\label{Section Distribution vectors - basic definitions}

For a parabolic subgroup $S$ of $G$ with Langlands decomposition $S=M_{S}A_{S}N_{S}$ and a representation $\xi$ of $M_{S}$ on a Hilbert space $V_{\xi}$ and $\lambda\in\fa_{S,\C}^{*}$, we define $C^{\infty}(S:\xi:\lambda)$ to be the space smooth vectors in the principal series representation induced from $S$ with induction data $\xi\otimes\lambda\otimes 1$, i.e., the space of smooth $V_{\xi}$-valued functions $f$ on $G$ with the property that
$$
f(manx)
=a^{\lambda+\rho_{S}}\xi(m)f(x)
\qquad(m\in M_{S},a\in A_{S},n\in N_{S},x\in G).
$$

Recall that $K$ is a maximal compact subgroup of $G$. The pairing
$$
C^{\infty}(S:\xi^{\vee}:-\lambda)\times C^{\infty}(S:\xi:\lambda)\to\C;
\quad(\chi,f)\mapsto \int_{K}\big(\chi(k),f(k)\big)\,dk
$$
is non-degenerate and $G$-equivariant. We thus obtain a $G$-equivariant inclusion
$$
C^{\infty}(S:\xi^{\vee}:-\lambda)\hookrightarrow C^{\infty}(S:\xi:\lambda)'.
$$

For a smooth manifold $\cM$ and a Hilbert space $V$ we define $\cE(\cM,V)$ to be the vector space of all smooth functions $\cM\to V$, and $\cD(\cM,V)$ to be the subspace of $\cE(\cM,V)$ consisting of all functions with compact support.
We write $\cD'(\cM,V)$ for the continuous dual of $\cD(\cM,V)$. Note that in case $\cM$ is an open subset of $G$, there is a natural injection $\cE(\cM,V^{*})\hookrightarrow\cD'(\cM,V)$ (using the Haar measure on $G$ to identify densities with functions).

Let $L^{\vee}$ and $R^{\vee}$ be the contragredients of the left-regular representation $L$ and the right-regular representation $R$, respectively.
We define $\cD'(S:\xi:\lambda)$ to be the subspace of $\cD'(G,V_{\xi})$ consisting of all distributions $\mu$ such that
\begin{equation}\label{eq L(man)mu=a^lambda xi(m) mu}
L^{\vee}(man)\mu
=a^{\lambda-\rho_{S}}\xi^{\vee}(m^{-1})\mu
\qquad(m\in M_{S}, a\in A_{S},n\in N_{S}).
\end{equation}

Let $V$ be a Hilbert space. We write $\cD'(G,V)^{H}$ for the subspace of $\cD'(G,V)$ of distributions that are invariant under the right-regular representation of $H$ on $\cD'(G,V)$, i.e.,
$$
\cD'(G,V)^{H}
=\{\mu\in\cD'(G,V):R^{\vee}(h)\mu=\mu \text{ for all }h\in H\}.
$$
If $\phi\in \cD(G,V)$, then in view of the identification $Z=G/H$ the function
$$
gH\mapsto \int_{H}\phi(gh)\,dh
$$
defines an element of $\cD(Z,V)$. The map $\cD(G,V)\to\cD(Z,V)$ thus obtained is continuous. Moreover, the induced map
\begin{equation}\label{eq Identification D'(Z) simeq D'(G)^H}
\cD'(Z,V)\to\cD'(G,V)^{H};\quad\mu\mapsto\Big(\phi\mapsto \mu\Big(\int_{H}\phi(\dotvar h)\,dh\Big)\Big)
\end{equation}
is a topological isomorphism. We will use this isomorphism to identify $\cD'(Z,V)$ with $\cD'(G,V)^{H}$.
Finally, we define the space
$$
\cD'(Z,S:\xi:\lambda)
=\cD'(Z,V_{\xi})\cap\cD'(S:\xi:\lambda).
$$

\subsection{Distribution vectors versus functionals}
\label{Subsection Distribution vectors - Distribution vectors versus functionals}

Let $S=M_{S}A_{S}N_{S}$, $(\xi,V_{\xi})$ and $\lambda\in\fa_{S,\C}^{*}$ be as before.
In this section we compare the spaces $C^{\infty}(S:\xi:\lambda)'$ and $\cD'(S:\xi:\lambda)$. We follow for this the analysis in \cite[Section 2.3]{CarmonaDelorme_BaseMeromorpheDeVecteursDistributionsHInvariants}.

Let $\psi_{0}\in \cD(G)$ satisfy
$$
\int_{M_{S}}\int_{A_{S}}\int_{N_{S}}a^{2\rho_{S}}\psi_{0}(man x)\,dn\,da\,dm
=1
\qquad(x\in G).
$$
One may for instance take $\psi_{0}\in \cD(G)$ to be right $K$-invariant and satisfying
$$
\int_{G}\psi_{0}(x)\Iwasawa_{S}^{2\rho_{S}}(x)\,dx
=1,
$$
where $\Iwasawa_{S}:G\to A_{S}$ is the map given by
$$
x\in N_{S}\Iwasawa_{S}(x)M_{S}K\qquad(x\in G).
$$
For $\mu\in\cD'(S:\xi:\lambda)$, let $\omega^{S}_{\xi,\lambda}\mu\in C^{\infty}(S:\xi:\lambda)'$ be given by
$$
\big(\omega^{S}_{\xi,\lambda}\mu\big)(f)
=\mu(\psi_{0}f)
\qquad\big(f\in C^{\infty}(S:\xi:\lambda)\big).
$$
The map
\begin{equation}\label{eq Def omega^S_(xi,lambda)}
\omega^{S}_{\xi,\lambda}:\cD'(S:\xi:\lambda)\to C^{\infty}(S:\xi:\lambda)'
\end{equation}
we thus obtain is a topological isomorphism; it is easily seen that the map
$$
\theta^{S}_{\xi,\lambda}: C^{\infty}(S:\xi:\lambda)'\to\cD'(S:\xi:\lambda),
$$
which for $\eta\in C^{\infty}(S:\xi:\lambda)'$ and $\phi\in \cD(G,V_{\xi})$ is given by
\begin{equation}\label{eq def theta}
\big(\theta^{S}_{\xi,\lambda}\eta\big)(\phi)
=\eta\Big(x\mapsto\int_{M_{S}}\int_{A_{S}}\int_{N_{S}} a^{-\lambda+\rho_{S}}\xi(m^{-1}) \phi(manx)\,dn\,da\,dm\Big),
\end{equation}
is the inverse of $\omega^{S}_{\xi,\lambda}$. In particular it follows that $\omega^{S}_{\xi,\lambda}$ does not depend on the choice of the function $\psi_{0}$.
Note that $\theta^{S}_{\xi,\lambda}$ intertwines the representation $\pi_{S:\xi:\lambda}^{\vee}$ on $C^{\infty}(S:\xi:\lambda)'$ with $R^{\vee}$ on $\cD'(S:\xi:\lambda)$.
The restriction of $\omega^{S}_{\xi,\lambda}$ to $\cD'(S:\xi:\lambda)^{H}$ is a $G$-equivariant isomorphism to the space of $H$-fixed functionals on $C^{\infty}(S:\xi:\lambda)$.

\subsection{Intertwining operators}
\label{Subsection Distribution vectors - Intertwining operators}
Let $S=M_{S}A_{S}N_{S}$, $(\xi,V_{\xi})$ and $\lambda\in\fa_{S,\C}^{*}$ be as before.
For $u\in \cU(\fg)$ we set
$$
p_{S,\xi,\lambda,u}:\cE(G,V_{\xi})\to[0,\infty];
\quad\phi\mapsto\int_{G}\|\Iwasawa_{S}(x)^{-\lambda+\rho_{S}}\,R(u)\phi(x)\|_{\xi}\,dx
$$
and endow the space
$$
\cV_{S,\xi,\lambda}
:=\big\{\phi\in \cE(G,V_{\xi}):p_{S,\xi,\lambda,u}(\phi)<\infty\text{ for every }u\in\cU(\fg)\big\},
$$
with the Fr\'echet topology induced by the seminorms $p_{S,\xi,\lambda,u}$. Note that $\cD(G,V_{\xi})\subseteq\cV_{S,\xi,\lambda}$.
Further, for two parabolic subgroups $S_{1}$ and $S_{2}$ of $G$ with $A_{S_{1}}=A_{S_{2}}=A_{S}$, we write
$$
A(S_{2}:S_{1}:\xi:\lambda):C^{\infty}(S_{1}:\xi:\lambda)\to C^{\infty}(S_{2}:\xi:\lambda)
$$
for the standard Knapp-Stein intertwining operators and define
$$
\cA(S_{2}:S_{1}:\xi:\lambda)
:=\theta^{S_{2}}_{\xi,\lambda}\circ A(S_{1}:S_{2}:\xi:\lambda)^{*}\circ\omega^{S_{1}}_{\xi,\lambda}.
$$
We assume that $A_{S}\subseteq A$ and identify $\fa_{S,\C}^{*}$ with the annihilator of $\fm_{S}\cap \fa$ in $\fa_{\C}^{*}$.

\begin{Prop}\label{Prop int formula for a(S_2:S_1:xi:lambda)}
Let $S_{1}$, $S_{2}$, $\xi$ and $\lambda$ be as above.
The following diagram commutes.
$$
\xymatrix{
   \cD'(S_{1}:\xi:\lambda) \ar[rr]^{\cA(S_{2}:S_{1}:\xi:\lambda)} \ar@<-2pt>[dd]_{\omega^{S_{1}}_{\xi,\lambda}}
        && \cD'(S_{2}:\xi:\lambda) \ar@<-2pt>[dd]_{\omega^{S_{2}}_{\xi,\lambda}} \\
\\
   C^{\infty}(S_{1}:\xi:\lambda)'\ar@<-2pt>[uu]_{\theta^{S_{1}}_{\xi,\lambda}} \ar[rr]^{A(S_{1}:S_{2}:\xi:\lambda)^{*}}
        && C^{\infty}(S_{2}:\xi:\lambda)' \ar@<-2pt>[uu]_{\theta^{S_{2}}_{\xi,\lambda}}
}
$$
Assume that $\lambda\in\fa_{S,\C}^{*}$ satisfies
$$
\langle\Re\lambda,\alpha\rangle>0\qquad\big(\alpha\in\Sigma(\fa:S_{2})\cap-\Sigma(\fa:S_{1})\big),
$$
Then for every $\phi\in \cV_{S_{2},\xi,\lambda}$ and every $x\in G$ the integral
$$
\int_{N_{S_{2}}\cap \overline{N}_{S_{1}}}\phi(nx)\,dx
$$
is absolutely convergent and the function $\int_{N_{S_{2}}\cap \overline{N}_{S_{1}}}\phi(n\dotvar)\,dx$ thus obtained is an element of $\cV_{S_{1},\xi,\lambda}$. Moreover, the map
$$
\cV_{S_{2},\xi,\lambda}\to\cV_{S_{1},\xi,\lambda};\quad\phi\mapsto\int_{N_{S_{2}}\cap \overline{N}_{S_{1}}}\phi(n\dotvar)\,dn
$$
is continuous.
Finally, if $\mu\in\cD'(S_{1}:\xi:\lambda)$, then $\mu$ extends to a continuous linear functional on $\cV_{S_{1},\xi,\lambda}$, and
the distribution $\cA(S_{2}:S_{1}:\xi:\lambda)\mu\in\cD'(S_{2}:\xi:\lambda)$ is given by
\begin{equation}\label{eq a(S_2:S_1)mu}
\big[\cA(S_{2}:S_{1}:\xi:\lambda)\mu\big](\phi)
=\mu\Big(\int_{N_{S_{2}}\cap\overline{N}_{S_{1}}}\phi(n\dotvar)\,dn\Big)
\qquad\big(\phi\in \cV_{S_{2},\xi,\lambda}\big).
\end{equation}
\end{Prop}

For the proof of the proposition we refer to Appendix B.
\medbreak

We define an inner product $\langle\cdot,\cdot\rangle_{S,\xi,\lambda}$ on $C^{\infty}(S:\xi:\lambda)$ by
\begin{equation}\label{eq inner product on C^infty(S:xi:lambda)}
\langle \phi,\psi\rangle_{S,\xi,\lambda}
=\int_{K}\langle\phi(k),\psi(k)\rangle_{\xi}\,dk
\qquad\big(\phi,\psi\in C^{\infty}(S:\xi:\lambda)\big),
\end{equation}
where $\langle\cdot,\cdot\rangle_{\xi}$ is the inner product on $V_{\xi}$. We consider parabolic subgroup $S_{1}$ and $S_{2}$ with $A_{S_{1}}=A_{S_{2}}=A_{S}\subseteq A$ as before. The adjoint of the intertwining operator $A(S_{2}:S_{1}:\xi:\lambda)$ with respect to (\ref{eq inner product on C^infty(S:xi:lambda)}) is given by
$$
A(S_{2}:S_{1}:\xi:\lambda)^{\dagger}
=A(S_{1}:S_{2}:\xi:-\overline{\lambda})
$$
The composition $A(S_{2}:S_{1}:\xi:\lambda)\circ A(S_{1}:S_{2}:\xi:\lambda)$ is an intertwining operator from $C^{\infty}(S_{1}:\xi:\lambda)$ to itself. It is therefore given by multiplication by a scalar. As in \cite[\S 14.6]{Knapp_RepresentationTheoryOfSemisimpleGroups} we choose a meromorphic function on $\fa_{S,\C}^{*}$
$$
\lambda\mapsto \gamma(S_{1}:S_{2}:\xi:\lambda)
$$
that is real and non-negative on $\fa_{S}^{*}$ and satisfies the identity of meromorphic operators
$$
A(S_{2}:S_{1}:\xi:\lambda)\circ A(S_{1}:S_{2}:\xi:\lambda)
=\gamma(S_{2}:S_{1}:\xi:\lambda)\gamma(S_{1}:S_{2}:\xi:\lambda)\Id.
$$
We may choose these functions so that for all $\lambda\in\fa_{S,\C}^{*}$
$$
\gamma(S_{1}:S_{2}:\xi:\lambda)
=\overline{\gamma(S_{2}:S_{1}:\xi:-\overline{\lambda})},
$$
$$
\gamma(S_{1}:S_{2}:\xi:\lambda)
=\gamma(S_{1}:S_{2}:\xi':\lambda)
\qquad(\xi'\simeq \xi,\lambda\in \fa_{S,\C}^{*})
$$
for every $\xi'$ equivalent to $\xi$,
and
$$
\gamma\big(vS_{1}v^{-1}:vS_{2}v^{-1}:v\cdot \xi:\Ad^{*}(v)\lambda\big)
=\gamma(S_{1}:S_{2}:\xi:\lambda)
$$
for every $v\in N_{G}(\fa_{S})$.
Here $v\cdot \xi$ is the representation of $M_{S}$ with representation space $V_{\xi}$ given by
$$
(v\cdot\xi)(m)
=\xi(v^{-1}mv)
\qquad(m\in M_{S}).
$$
If we normalize the intertwining operators with these $\gamma$-functions as
$$
A^{\circ}(S_{1}:S_{2}:\xi:\lambda)
:=\frac{1}{\gamma(S_{1}:S_{2}:\xi:\lambda)}A(S_{1}:S_{2}:\xi:\lambda),
$$
then we obtain the identities
$$
A^{\circ}(S_{3}:S_{1}:\xi:\lambda)\\
=A^{\circ}(S_{3}:S_{2}:\xi:\lambda)\circ A^{\circ}(S_{2}:S_{1}:\xi:\lambda)
$$
for all parabolic subgroups $S_{1},S_{2},S_{3}$ with $A_{S_{1}}=A_{S_{2}}=A_{S_{3}}=A_{S}\subseteq A$, $\lambda\in(\fa_{S})^{*}_{\C}$ and unitary representations $\xi$ of $M_{S_{1}}=M_{S_{2}}=M_{S_{3}}=M_{S}$.
In particular,
$$
A^{\circ}(S_{1}:S_{2}:\xi:\lambda)^{-1}
=A^{\circ}(S_{2}:S_{1}:\xi:\lambda),
$$
and hence the operator $A^{\circ}(S_{1}:S_{2}:\xi:\lambda)$ is unitary if $\lambda\in i\fa_{S}^{*}$.

For $v\in N_{G}(\fa_{S})$ we define the intertwining operator
$$
\cI_{v}(S:\xi:\lambda):\cD'(S:\xi:\lambda)\to\cD'(S:v\cdot\xi:\Ad^{*}(v)\lambda)
$$
by
$$
\cI_{v}(S:\xi:\lambda)
:=L^{\vee}(v)\circ \cA(v^{-1}Sv:S:\xi:\lambda)
$$
and the corresponding normalized intertwining operator
$$
\cI^{\circ}_{v}(S:\xi:\lambda):\cD'(S:\xi:\lambda)\to\cD'(S:v\cdot\xi:\Ad^{*}(v)\lambda)
$$
by
$$
\cI^{\circ}_{v}(S:\xi:\lambda)
:=\frac{1}{\gamma(v^{-1}Sv:S:\xi:\lambda)}\cI_{v}(S:\xi:\lambda).
$$
We note that
$$
\cI^{\circ}_{v}(S:w\cdot\xi:w\cdot\lambda)\circ \cI^{\circ}_{w}(S:\xi:\lambda)
=\cI^{\circ}_{vw}(S:\xi:\lambda)
\qquad\big(v,w\in N_{G}(\fa_{S})\big).
$$
The family of operators $\lambda\mapsto \cI^{\circ}_{v}(S:\xi:\lambda)$ is meromorphic.
There exists a locally finite union $\cH$ of complex affine hyperplanes in $\fa_{S,\C}^{*}$ of the form $\{\lambda\in\fa_{S,\C}^{*}:\lambda(\alpha^{\vee})=c\}$ for some $\alpha\in \Sigma(S)$ and $c\in\R$, so that for all unitary representations $\xi$ of $M_{S}$ the poles of the families $\lambda\mapsto \cI_{v}(S:\xi:\lambda)$ and  $\lambda\mapsto \cI^{\circ}_{v}(S:\xi:\lambda)$ lie on $\cH$.

\subsection{Comparison between induction from different parabolic subgroups}\label{Subsection Distribution vectors - Comparison}

Let $S$ and $T$ be parabolic subgroups of $G$ and assume that $S\subseteq T$. Let $S=M_{S}A_{S}N_{S}$ $T=M_{T}A_{T}N_{T}$ be Langlands decompositions of $S$ and $T$, respectively, and assume that $A_{T}\subseteq A_{S}\subseteq A$. Observe that $S \cap M_{T}$ is a parabolic subgroup
of $M_{T}$. Moreover, $\fa=\fa_{T}\oplus(\fm_{T}\cap\fa)$.
We identify $\fa_{T}^{*}$ as a subspace $\fa^{*}$ by extending the functionals by $0$ on $\fm_{T}\cap\fa$.
Note that
$$
\rho_{S \cap M_{T}}
=\rho_{S}-\rho_{T}.
$$

Let $(\xi, V_{\xi})$ be a representation of $M_{T}$ on a Hilbert space and assume that
$$
M_{T}\cap N_{S}\subseteq\ker(\xi).
$$
Let further $\lambda\in\fa_{T,\C}^{*}$.
There is a natural $M_{T}$-equivariant embedding
$$
i: \xi \hookrightarrow \Ind_{M_{T} \cap S}^{M_{T}}(\xi|_{M_{S}} \otimes \rho_{T}-\rho_{S} \otimes 1),
$$
see \cite[Lemma 4.4]{vdBan_PrincipalSeriesI}. Concretely, the map $i$ from $V_{\xi}$ into
the space
$C^\infty(M_{T}\cap S: \xi|_{M_{S}}: \rho_{T}-\rho_{S})$
of smooth vectors  for the principal series
representation on the right-hand side is given by
$$
i(v)(m_T) = \xi(m_T) v,\qquad (v \in V_{\xi}, \;m_T \in M_{T}).
$$
Let $\lambda\in\fa_{T,\C}^{*}$.
Using induction by stages, we obtain a $G$-equivariant embedding
$$
\Ind_{T}^G (\xi \otimes \lambda\otimes 1)
\hookrightarrow \Ind_{S}^G (\xi|_{M_{S}} \otimes (\lambda +\rho_{T}-\rho_{S}) \otimes 1).
$$
On the level of smooth vectors this results in a $G$-equivariant embedding
$$
i_{\xi,\lambda}^{\#}: C^\infty(T : \xi : \lambda) \hookrightarrow C^\infty(S: \xi|_{M_{S}} : \lambda +\rho_{T}-\rho_{S}),
$$
which is the natural inclusion map. Note that $i_{\xi^{\vee},-\lambda}^{\#}$ extends to a continuous inclusion
$$
\cD'(T:\xi:\lambda)\hookrightarrow\cD'(S:\xi|_{M_{S}}:\lambda-\rho_{T}+\rho_{S}).
$$
Using the isomorphisms from Section \ref{Subsection Distribution vectors - Distribution vectors versus functionals}, we now arrive at the following result.

\begin{Prop}\label{Prop inclusion of functionals}
Let $\lambda\in\fa_{T,\C}^{*}$ and let $(\xi, V_{\xi})$ be a representation of $M_{T}$ on a Hilbert space and assume that $M_{T}\cap N_{S}\subseteq\ker(\xi)$.
There exists a $G$-equivariant injective map
$$
C^{\infty}(T:\xi:\lambda)'\hookrightarrow C^{\infty}(S:\xi|_{M_{S}}:\lambda-\rho_{T}+\rho_{S})'
$$
so that
$$
\xymatrix{
   \cD'(T:\xi:\lambda) \ar@{^{(}->}[rr] \ar@<-2pt>[dd]_{\omega^{T}_{\xi,\lambda}}
        && \cD'(S:\xi|_{M_{S}}:\lambda-\rho_{T}+\rho_{S}) \ar@<-2pt>[dd]_{\omega^{S}_{\xi|_{M_{S}},\lambda-\rho_{T}+\rho_{S}}} \\
\\
   C^{\infty}(T:\xi:\lambda)'\ar@<-2pt>[uu]_{\theta^{T}_{\xi,\lambda}} \ar@{^{(}->}[rr]
        && C^{\infty}(S:\xi|_{M_{S}}:\lambda-\rho_{T}+\rho_{S})' \ar@<-2pt>[uu]_{\theta^{S}_{\xi|_{M_{S}},\lambda-\rho_{T}+\rho_{S}}}
}
$$
is a commuting diagram.
\end{Prop}

\subsection{$L_{Q,\nc}$-spherical representations of $M_{Q}$}
\label{Subsection Distribution vectors - M_Q representations}
Let $Q=M_{Q}A_{Q}N_{Q}$ be the Langlands decomposition of $Q$ with $A_{Q}\subseteq A$. Then $L_{Q}=M_{Q}A_{Q}$.
We recall that $\fl_{Q,\nc}$ is the sum of the simple ideals of non-compact type in $\fl_{Q}$. We write $L_{Q,\nc}$ for the connected subgroup of $L_{Q}$ with Lie algebra $\fl_{Q,\nc}$. Note that $L_{Q,\nc}\subseteq M_{Q}$.

We first look at a few properties of $M_{Q}$ and $L_{Q,\nc}$ which will be needed in this and later sections.

\begin{Lemma}\label{Lemma decomposition of M_Q}
\ \begin{enumerate}[(i)]
\item\label{Lemma decomposition of M_Q -  item 1}
$L_{Q,\nc}$ is a closed normal subgroup of $M_{Q}$.
\item\label{Lemma decomposition of M_Q - item 2}
$
M_{Q} = M L_{Q,\nc} \simeq M \times_{M\cap L_{Q,\nc}} L_{Q,\nc}.
$
\item\label{Lemma decomposition of M_Q - item 3}
The group $M_{L,\nc}$ acts trivially on $M_{Q} / M_{L,\nc}$.
\item\label{Lemma decomposition of M_Q - item 4}
Let $z\in Z$ be a weakly adapted point and $w\in N_{G}(A)$ so that $\fh_{z,X}=\Ad(w)\fh_{\emptyset}$ for some order-regular element $X\in\fa^{-}$.  Then $wL_{Q,\nc}w^{-1}\subseteq M_{Q}\cap H_{z}$. If  $w$ normalizes $\fa_{\fh}$, or equivalently if $\fa_{P\cdot z}=\fa_{\fh}$, then $L_{Q,\nc}\subseteq M_{Q}\cap H_{z}$.
\end{enumerate}
\end{Lemma}

\begin{proof}
Since $\fm_{Q}$ is reductive, there exists an ideal $\fm_{Q,c}$ complementary to $\fl_{Q,\nc}$. The group $L_{Q,\nc}$ is equal to the connected component of $Z_{M_{Q}}(\fm_{Q,c})$ and therefore $L_{Q,\nc}$ is closed. As $\fm_{Q} = \fm + \fl_{Q,\nc}$ the set $ML_{Q,\nc}$ is open. Moreover, since $M$ is compact and $L_{Q,\nc}$ is closed,  $ML_{Q,\nc}$ is also closed. From the fact that $M$ intersects with every connected component of $M_{Q}$  assertion (\ref{Lemma decomposition of M_Q - item 2}) now follows.

The subalgebra $\fl_{Q,\nc}$ is an $M$-stable ideal of $\fm_{Q}$. Assertion (\ref{Lemma decomposition of M_Q -  item 1}) therefore follows from (\ref{Lemma decomposition of M_Q - item 2}).

Since  $L_{Q,\nc}$ is normal in $M_{Q}$, it acts trivially on the quotient $M_{Q}/L_{Q,\nc}$, and hence (\ref{Lemma decomposition of M_Q - item 3}) follows.

Finally we prove (\ref{Lemma decomposition of M_Q - item 4}). Let $\cO$ be a $P$-orbit in $Z$ of maximal rank and let $z\in \cO$ be weakly admissible. We select a regular element $X\in\fa^{-}$. Then there exists a $w\in N_{G}(\fa)$ so that $\fh_{z,X}=\Ad(w)\fh_{\emptyset}$. By Proposition \ref{Prop parameterization weakly adapted points} (\ref{Prop parameterization weakly adapted points - item 2}) the point $w^{-1}\cdot z$ is adapted. Therefore, $\fl_{Q,\nc}\subseteq\fh_{w^{-1}\cdot z}=\Ad(w^{-1})\fh_{z}$, and hence $\Ad(w)\fl_{Q,\nc}\subseteq\fh_{z}$. The assertion now follows as $L_{Q,\nc}$ is connected. By Remark \ref{Rem L_Q roots} the roots of $\fa$ in $\fl_{Q,\nc}$ are precisely those roots that vanish on $\fa\cap \fa_{\fh}^{\perp}$. If $w$ normalizes $\fa_{\fh}$, then it follows that $w$ normalizes $\fl_{Q,\nc}$ and hence $L_{Q,\nc}$.
\end{proof}

Given a continuous representation of $M_{Q}$ in a Fr\'echet space $V$, we denote its space of
smooth vectors by $V^\infty$ and equip it with the structure of a continuous Fr\'echet $M_{Q}$-module in the usual way.  The continuous linear
dual we denote by ${V^\infty}'$.

\begin{Cor}\label{Cor irreducible restriction to M}
Let $(\xi, V_{\xi})$ be an irreducible continuous representation of $M_{Q}$ in
a Fr\'echet space $V$ such that
$$
({V_{\xi}^{\infty}}')^{L_{Q,\nc}} \neq 0.
$$
Then $\xi|_{L_{Q,\nc}}$ is trivial and $\xi|_M$ is irreducible.
In particular, $\xi$ is finite dimensional and unitarizable. In particular this is the case if
$$
({V_{\xi}^{\infty}}')^{H_{z}}\neq\{0\}.
$$
for some weakly adapted point $z$ in a $P$-orbit $\cO$ in $Z$ with $\fa_{\cO}=\fa_{\fh}$.
\end{Cor}

\begin{proof}
The proof is the same as the one for \cite[Corollary 4.4]{vdBanKuit_NormalizationsOfEisensteinIntegrals}. For convenience we give it here.
Let $\eta\in ({V_{\xi}^{\infty}}')^{L_{Q,\nc}}$. If $\eta\neq 0$, then there is a unique injective
continuous linear $M_{Q}$-equivariant map $j : V^{\infty} \to \cE(M_{Q}/L_{Q,\nc})$ such that $j^{*}\delta = \eta$,
with $\delta$ denoting the Dirac measure of $M_{Q}/L_{Q,\nc}$ at $eL_{Q,\nc}$. It follows from Lemma \ref{Lemma decomposition of M_Q} (\ref{Lemma decomposition of M_Q - item 3}) that $L_{Q,\nc}$ acts trivially $\cE(M_{Q}/L_{Q,\nc})$ and hence on $V^{\infty}$. We conclude that $L_{Q,\nc}\subseteq\ker(\xi)$. By application of Lemma \ref{Lemma decomposition of M_Q} (\ref{Lemma decomposition of M_Q - item 2}) it follows that $\xi|_{M}$ is irreducible. The final assertion follows from Lemma \ref{Lemma decomposition of M_Q} (\ref{Lemma decomposition of M_Q - item 4}).
\end{proof}

Let $\widehat M_{Q,\mathrm{fu}}$ be the set of equivalence classes of finite dimensional irreducible unitary representations of $M_{Q}$.

\begin{Cor}\label{Cor M_Q,fu simeq M reps with trivial restriction to M cap L_Q,nc}
Every representation in $\widehat M_{Q,\mathrm{fu}}$ restricts to the trivial representation on $L_{Q,\nc}$. The restriction map $\xi \mapsto \xi|_{M_{Q}}:= \xi|_M$ induces an injection
$$
\widehat M_{Q,\mathrm{fu}}  \hookrightarrow \widehat M.
$$
The image of this injection equals
$$
\big\{[\xi]\in \widehat{M}:\xi\big|_{M\cap L_{Q,\nc}}\text{ is trivial}\big\}.
$$
\end{Cor}

\begin{proof}
Since the $L_{Q,\nc}$ is connected semisimple of the non-compact type, the restriction of a representation from $\widehat M_{Q,\mathrm{fu}}$ to $L_{Q,\nc}$ is trivial. The remaining assertions follow from Lemma \ref{Lemma decomposition of M_Q}.
\end{proof}

\subsection{Comparison between induction from $P$ and $Q$}\label{Subsection Distribution vectors - Comparison between P and Q}

The following proposition follows directly from Corollary \ref{Cor irreducible restriction to M} and the comparison of induction from different parabolic subgroups in Section \ref{Subsection Distribution vectors - Comparison}.

\begin{Prop}
Let $\xi$ be a representation of $M_{Q}$ on a Hilbert space $V_{\xi}$ and $\lambda\in\fa_{Q,\C}^{*}$. Assume that
$$
({V_{\xi}^{\infty}}')^{L_{Q,\nc}}\neq\{0\}.
$$
Then $\xi$ is finite dimensional, $\xi|_M$ is irreducible and
$$
\cD'(Q:\xi:\lambda)
\subseteq\cD'(P:\xi|_{M}:\lambda+\rho_{P}-\rho_{Q}).
$$
Moreover, there exists a natural inclusion
$$
C^{\infty}(Q:\xi:\lambda)'\hookrightarrow C^{\infty}(P:\xi|_{M}:\lambda+\rho_{P}-\rho_{Q})'
$$
so that
$$
\xymatrix{
   \cD'(Q:\xi:\lambda) \ar@{^{(}->}[rr] \ar@<-2pt>[dd]_{\omega^{Q}_{\xi,\lambda}}
        && \cD'(P:\xi|_{M}:\lambda+\rho_{P}-\rho_{Q}) \ar@<-2pt>[dd]_{\omega^{P}_{\xi|_{M},\lambda+\rho_{P}-\rho_{Q}}} \\
\\
   C^{\infty}(Q:\xi:\lambda)'\ar@<-2pt>[uu]_{\theta^{Q}_{\xi,\lambda}} \ar@{^{(}->}[rr]
        && C^{\infty}(P:\xi|_{M}:\lambda+\rho_{P}-\rho_{Q})' \ar@<-2pt>[uu]_{\theta^{P}_{\xi|_{M},\lambda+\rho_{P}-\rho_{Q}}}
}
$$
is a commuting diagram. In particular this is the case if
$$
({V_{\xi}^{\infty}}')^{H_{z}}\neq\{0\}.
$$
for some weakly adapted point $z$ contained in a $P$-orbit $\cO$ of maximal rank in $Z$ with $\fa_{\cO}=\fa_{\fh}$.
\end{Prop}

The following lemma describes $\cD'(Q:\xi:\lambda)$ as a subspace of $\cD'(P:\xi|_{M}:\lambda+\rho_{P}-\rho_{Q})$.

\begin{Lemma}\label{Lemma characterization of D'(Q:sigma:lambda)}
Let $\sigma\in\widehat{M}$ be so that $\sigma|_{M\cap L_{Q,\nc}}$ is trivial, and let $\lambda\in\fa_{Q,\C}^{*}$.
Let $\xi$ be the representation of $M_{Q}$ so that $\xi|_{L_{Q,\nc}}$ is trivial and $\xi|_{M}=\sigma$.
If $\mu\in\cD'(P:\sigma:\lambda+\rho_{P}-\rho_{Q})$ satisfies
$$
L^{\vee}(\overline{n})\mu=\mu
\qquad(n\in M_{Q}\cap \overline{N}_{P}),
$$
then $\mu\in\cD'(Q:\xi:\lambda)$.
\end{Lemma}

\begin{proof}
Let $G_{\mu}$ be the closed subgroup of $G$ consisting of elements $g\in G$ so that
$$
L^{\vee}(g)\mu=\mu.
$$
Since $M_{Q}=ML_{Q,\nc}$, see Lemma \ref{Lemma decomposition of M_Q}, it suffices to prove that $L_{Q,\nc}\subseteq G_{\mu}$.
The latter follows from the assumptions as $L_{Q,\nc}$ is the smallest closed subgroup of $G$ containing $M_{Q}\cap N_{P}$ and $M_{Q}\cap \overline{N}_{P}$.

\end{proof}

\section{Support and transversal degree }\label{Section Support and transdeg}

Throughout this section we fix $\sigma\in\widehat{M}$ and $\lambda\in\fa_{\C}^{*}$.

\subsection{Transversal degree}\label{Subsection Support and transdeg - Transversal degree}
Let $\cM$ be a smooth submanifold of $G$ and let $U$ be an open subset of $G$.

We fix a set of smooth vector fields $v_{1},\dots,v_{n}$ on $U$ so that at every point $y\in \cM\cap U$
$$
T_{y}G=\R v_{1}(y)\oplus\cdots\oplus \R v_{n}(y)\oplus T_{y}\cM.
$$
For a multi-index $\beta$ in $n$-variables, let $\partial^{\beta}$ be the differential operator $C^{\infty}(U:V)\to C^{\infty}(U:V)$ given by
$$
\partial^{\beta}\phi
=\underbrace{v_{1}\cdots v_{1}}_{\beta_{1} \text{ times}} \cdots \underbrace{v_{n}\cdots v_{n}}_{\beta_{n} \text{ times}}(\phi)
\qquad\big(\phi\in C^{\infty}(U:V)\big).
$$

Let $\mu\in\cD'(G,V)$ and assume that $\supp\mu=\cM\cap U$. It follows from \cite[p. 102]{Schwartz_Distributions_I} that for every multi-index $\beta$ there exists a distribution $\mu_{\beta}\in\cD'(\cM\cap U,V)$ such that for all $\phi\in \cD(U,V)$
$$
\mu(\phi)
=\sum_{\beta}\mu_{\beta}\big(\partial^{\beta}\phi\big).
$$
This decomposition of $\mu$ is unique. Let $k_{U}=\max\{|\beta|:\mu_{\beta}\neq 0\}$. The {\em transversal degree of $\mu$ at a point $y\in \cM$}, is defined to be the minimum of the numbers $k_{U}$, where $U$ runs over all neighborhoods of $y$ in $G$.
The transversal degree is independent of the choice of the vector fields $v_{i}$.

For a distribution $\mu\in\cD'(Z,P:\sigma:\lambda)$ let $(P\bs Z)_{\mu}$ be the set of $\cO\in P\bs Z$ with the property that there exists an open neighborhood $U$ of $\cO$ in $G$ such that
$$
\supp\mu\cap U
= \cO.
$$

The proof for the following proposition can be found in Remark 5.2 in \cite{KrotzKuitOpdamSchlichtkrull_InfinitesimalCharactersOfDiscreteSeriesForRealSphericalSpaces}.

\begin{Prop}\label{Prop support condition}
Let $\mu\in\cD'(Z,P:\sigma:\lambda)$. Then
$$
\supp\mu
=\overline{\bigcup_{\cO\in (P\bs Z)_{\mu}}\cO}.
$$
\end{Prop}

Let $\mu\in\cD'(Z,P:\sigma:\lambda)$ and let $\cO\in (P\bs Z)_{\mu}$. Then the transversal degree of $\mu$ at $z\in\cO$ does not depend on $z\in\cO$, see \cite[Lemma 5.5]{KrotzKuitOpdamSchlichtkrull_InfinitesimalCharactersOfDiscreteSeriesForRealSphericalSpaces}. Therefore, we may define the {\em transversal degree of $\mu$ at the orbit $\cO$} to be the transversal degree of $\mu$ at any point $z\in\cO$. We write $\trdeg_{\cO}(\mu)$ for the transversal degree of $\mu$ at $\cO$.

\subsection{Principal asymptotics}\label{Subsection Support and transdeg - Principal asymptotics}

Let $X\in\fa^{-}$ be order-regular and let $z\in Z$.
We define the $\fa$-stable subalgebra
$$
\overline\fn_{z,X}:=\fh_{z,X}\cap \overline{\fn}_{P}
$$
and write $\overline{N}_{z,X}$ for the connected subgroup of $G$ with Lie algebra equal to $\overline\fn_{z,X}$.
Let $\overline{\fn}_{X}^{z}$ be an $\fa$-stable complementary subspace to $\overline\fn_{z,X}$ in $\overline{\fn}_{P}$, so that
$$
\overline{\fn}_{P}=\overline\fn_{z,X}\oplus\overline\fn_{X}^{z}.
$$
We define $\Sigma(\overline{\fn}_{X}^{z};\fa)$ to be the set of roots of $\fa$ in $\overline{\fn}_{X}^{z}$ and $\rho_{\cO,X}\in\fa^{*}$ by setting
$$
\rho_{\cO,X}(Y)
=\frac{1}{2}\tr\big(\ad(Y)\big|_{\overline{\fn}_{z,X}}\big)
\qquad(Y\in\fa).
$$
Note that for $m\in M$, $a\in A$ and $n\in N_{P}$ we have
$$
\overline{\fn}_{man\cdot z,X}
=\Ad(m)\overline{\fn}_{z,X},
$$
and thus $\Sigma(\overline{\fn}_{X}^{z};\fa)$ and $\rho_{\cO,X}$ only depend on $\cO$, not on the choice for $z\in\cO$.

Let $e_{1},\dots, e_{n}$ be a basis of $\overline\fn_{X}^{z}$ consisting of joint eigenvectors for the action of $\ad(\fa)$ on $\overline\fn_{X}^{z}$.
For a multi-index $\beta$, let $\kappa_{\beta}\in\N_{0}\Sigma(\overline{\fn}_{X}^{z};\fa)$ be the $\fa$-weight of $e_{1}^{\beta_{1}}\cdots e_{n}^{\beta_{n}}\in\cU(\overline{\fn}_{P})$, where $\cU(\overline{\fn}_{P})$ denotes the universal enveloping algebra of $\overline{\fn}_{P}$. We write $\partial^{\beta}$ for the differential operator on $P\overline{N}$ that for $\phi\in\cE(P\overline{N},V)$ is given by
$$
\big(\partial^{\beta}\phi\big)(p\overline{n})
:=\frac{\partial^{\beta_{1}}}{\partial x_{1}^{\beta_{1}}}\cdots\frac{\partial^{\beta_{n}}}{\partial x_{n}^{\beta_{n}}}
    \phi\big(p\exp(\sum_{i=1}^{n}x_{i}e_{i})\overline n\big)\Big|_{x_{i}=0}
\qquad(p\in P, \overline{n}\in\overline{N}).
$$

The following theorem was proven in \cite{KrotzKuitOpdamSchlichtkrull_InfinitesimalCharactersOfDiscreteSeriesForRealSphericalSpaces}; see Theorem 5.1 and its proof and Corollary 5.3. We formulate the results here using distributions instead of functionals, for which we use the identifications in Section \ref{Section Distribution vectors}.

\begin{Prop}\label{Prop principal asymptotics}
Let $\mu\in\cD'(Z,P:\sigma:\lambda)$ and let $\cO\in (P\bs Z)_{\mu}$. We fix a point $z\in\cO$ and identify $\mu$ with an $H_{z}$-invariant distribution in $\cD'(P:\sigma:\lambda)$ as in section \ref{Section Distribution vectors}, for which we, with abuse of notation, also write $\mu$.
Let $X\in \fa^{-}$ be order regular and satisfy
\begin{equation}\label{eq omega_beta(X) neq omega_gamma(X)}
\kappa_{\beta}(X)\neq\kappa_{\gamma}(X)
\end{equation}
for any two multi-indices $\beta,\gamma$ with $|\beta|,|\gamma|\leq \trdeg_{\cO}(\mu)$ and $\kappa_{\beta}\neq\kappa_{\gamma}$.
Then there exist a left-$P$-invariant open neighborhood $\Omega$ of $e$ in $G$, a $\kappa\in\N_{0}\Sigma(\overline{\fn}_{X}^{z};\fa)$, and a unique non-zero distribution $\mu_{z,X}\in\cD'(\Omega,V_{\sigma})$, so that
$$
\lim_{t\to\infty}e^{t \big(\lambda+\rho_P+2\rho_{\cO,X}-\kappa\big)(X)}R^{\vee}\big(\exp(tX)\big)\mu
=\mu_{z,X}\,.
$$
Here the limit is with respect to weak-$*$ topology on $\cD'(\Omega,V_{\sigma})$.
The distribution $\mu_{z,X}$ is given by the following.
For every multi-index $\beta$ with $\kappa_{\beta}=\kappa$ there exists a $c_{\beta}\in V_{\sigma}^{*}$ such that for all $\phi\in\cD(\Omega,V_{\sigma})$
$$
\mu_{z,X}(\phi)
=\sum_{\substack{\beta\\ \kappa_{\beta}=\kappa}}
\int_{M}\int_{A}\int_{N_{P}}\int_{\overline{N}_{z,X}} a^{-\lambda+\rho_{P}}
\Big(\sigma^{\vee}(m) c_{\beta},\partial^{\beta}\phi(man\overline{n})\Big)\,d\overline{n}\,dn\,da\,dm.
$$
Finally, $\mu_{z,X}$ has the following properties.
\begin{enumerate}[(i)]
\item \label{Prop principal asymptotics - item 1}
    $L^{\vee}(man)\mu_{z,X}=a^{\lambda-\rho_{P}}\sigma^{\vee}(m^{-1})\mu_{z,X}$ for every $m\in M$, $a\in A$ and $n\in N_{P}$.
\item \label{Prop principal asymptotics - item 2}
    $R^{\vee}(Y)\mu_{z,X}=\big(-\lambda-\rho_P-2\rho_{\cO,X}+\kappa\big)(Y)\mu_{z,X}$ for every $Y\in\fa$.
\item \label{Prop principal asymptotics - item 3}
    $R^{\vee}(Y)\mu_{z,X}=0$ for every $Y\in\fh_{z,X}$.
\item \label{Prop principal asymptotics - item 5}
    The following are equivalent:
    \begin{enumerate}[(a)]
    \item $\trdeg_{\cO}(\mu)\neq 0$
     \item the transversal degree of $\mu_{z,X}$ (w.r.t. the submanifold $P\overline{N}_{z,X}\cap \Omega$ of $G$) at any point in $\supp\mu_{z,X}$ is non-zero.
     \item $\kappa\neq 0$.
     \end{enumerate}
\end{enumerate}
\end{Prop}

\begin{Cor}\label{Cor condition on lambda|fa_O}
Let $\mu\in\cD'(Z,P:\sigma:\lambda)$ and let $\cO\in (P\bs Z)_{\mu}$. Let $X\in \fa^{-}$ satisfy (\ref{eq omega_beta(X) neq omega_gamma(X)}). Then
\begin{equation}\label{eq lambda restricted to a_x}
\lambda\big|_{\fa_{\cO}}
\in \big(-\rho_P -2\rho_{\cO,X}+\N_{0}\Sigma(\overline{\fn}_{X}^{z};\fa)\big)\big|_{\fa_{\cO}}.
\end{equation}
Moreover, $\trdeg_{\cO}(\mu)\neq 0$ if and only if there exists a non-zero element $\kappa\in\N_{0}\Sigma(\overline{\fn}_{X}^{z};\fa)$ so that
$$
\lambda\big|_{\fa_{\cO}}
\in\big(-\rho_P -2\rho_{\cO,X}+\kappa\big)\big|_{\fa_{\cO}}.
$$
\end{Cor}

\begin{Rem}\label{Rem rho_O,X for maximal rank}
If $\cO$ is of maximal rank, then $\rho_{\cO,X}$ can be explicitly determined.
Let $z\in\cO$ and $X\in\fa^{-}$.  By Proposition \ref{Prop limits of max rank orbits are conjugates of h_empty} there exists a $w\in N_{G}(\fa)$ so that $\fh_{z,X}=\Ad(w)\fh_{\emptyset}$. In view of Remark \ref{Rem Orbits of max rank} (\ref{Rem Orbits of max rank - item 1}) we may choose $w$ so that (\ref{eq wSigma^+ cap -Sigma^+=wSigma(Q) cap -Sigma}) is satisfied. The latter guarantees that
$$
\Ad(w)(\overline{\fn}_{P}\cap\fl_{Q})\subseteq \overline{\fn}_{P},
$$
and hence
$$
\overline{\fn}_{z,X}
=\Ad(w)\overline{\fn}_{P}\cap\overline{\fn}_{P},
\quad
\overline{\fn}_{X}^{z}
=\Ad(w)\fn_{P}\cap\overline{\fn}_{P}
=\Ad(w)\fn_{Q}\cap\overline{\fn}_{P}.
$$
It follows that
$$
\rho_{\cO,X}
=-\frac{1}{2}\rho_{wPw^{-1}}-\frac{1}{2}\rho_{P}.
$$
In particular (\ref{eq lambda restricted to a_x}) can be rewritten as
$$
\lambda\big|_{\fa_{\cO}}
\in \Big(\Ad^{*}(w)\rho_{P}+\N_{0}\Big(-\Sigma(P)\cap\Ad^{*}(w)\Sigma(Q)\Big)\Big)\Big|_{\fa_{\cO}}.
$$
\end{Rem}

\begin{proof}[Proof of Corollary \ref{Cor condition on lambda|fa_O}]
The functional $\rho_{\cO,X}$ only depends on the connected component of the set of order-regular elements in $\fa$ in which $X$ is chosen. Every connected component of the set of order-regular elements in $\fa$ contains elements $X$ that satisfy $\kappa_{\beta}(X)\neq\kappa_{\gamma}(X)$ for any two multi-indices $\beta,\gamma$ with $|\beta|,|\gamma|\leq \trdeg_{\cO}(\mu)$ and $\kappa_{\beta}\neq\kappa_{\gamma}$. The claim therefore follows from (\ref{Prop principal asymptotics - item 2}) and (\ref{Prop principal asymptotics - item 3}) in Theorem \ref{Prop principal asymptotics}.
\end{proof}

For $\cO\in P\bs Z$ we set
$$
\cH_{\cO}
:=\big\{\lambda\in \fa_{\C}^{*}:\lambda\big|_{\fa_{\cO}}\in\frac{1}{2}\Z\Sigma\big|_{\fa_{\cO}}\big\}
$$
and we define
$$
\cH_{\nm}
:=\bigcup_{\substack{\cO\in P\bs Z\\\rank(\cO)<\rank(Z)}}\cH_{\cO}.
$$
Here $\nm$ stands for not maximal.
Then $\cH_{\nm}$ is a locally finite set of complex affine subspaces in $\fa_{\C}^{*}$ of codimension at least $1$.
Note that if $\lambda\in\cH_{\cO}$, then $\Im(\lambda)\in(\fa/\fa_{\cO})^{*}$.

\begin{Thm}\label{Thm condition on orbits for given generic lambda}
Let $\sigma\in \widehat{M}$ and $\lambda\in\fa_{\C}^{*}$.
The following assertions hold true.
\begin{enumerate}[(i)]
\item Let $\mu\in \cD'(Z,P:\sigma:\lambda)$. If $\cO\in(P\bs Z)_{\mu}$, then $\lambda\in\cH_{\cO}$.
\item If $\lambda\notin\cH_{\nm}$, then for all $\mu\in \cD'(Z,P:\sigma:\lambda)$
    $$
    (P\bs Z)_{\mu}
    \subseteq(P\bs Z)_{\max}.
    $$
\end{enumerate}
\end{Thm}

\begin{proof}
The assertions follow directly from Corollary \ref{Cor condition on lambda|fa_O}.
\end{proof}

\subsection{$P$-Orbits of maximal rank and transversal degree}\label{Subsection Support and transdeg - Orbits of max rank and transversal degree}
Let $\Cen(\fg)$ be the center of the universal enveloping algebra $\cU(\fg)$ of $\fg$ and let $\ft$ be a maximal abelian subalgebra in $\fm$. We fix a positive system $\Sigma_{\fm}^{+}$ of the root system of $i\ft$ in $\fm$ and write $\rho_{\fm}$ for the corresponding half-sum of positive roots. We further write $W_{\C}$ for the Weyl group of the root system $\Sigma_{\C}$ of $(\fa+i\ft)_{\C}$ in $\fg_{\C}$.
Let $\gamma$ be the Harish-Chandra homomorphism $\gamma:\Cen(\fg)\to\Sym\big((\fa\oplus i\ft)_{\C}\big)^{W_{\C}} \simeq \C\big[(\fa\oplus i \ft)^{*}\big]^{W_{\C}}$.

\begin{Prop}\label{Prop bound on transversal degree}
Let $\lambda\in\fa_{\C}^{*}$ and $\sigma\in\widehat{M}$. Further, let $\mu\in\cD'(Z,P:\sigma:\lambda)$ and let $\cO\in(P\bs Z)_{\mu}$. Assume that $\cO$ is of maximal rank. Let $\Lambda_{\sigma}\in i\ft^{*}$ be the highest-weight of $\sigma$.
If $\trdeg_{\cO}(\mu)\neq 0$, then there exists a non-zero element $\nu\in \N_{0}\Sigma(P)$ and a dominant $\Sigma_{\fm}$-integral weight $\Lambda\in i\ft^{*}$ so that
$$
W_{\C}\cdot (\lambda+\Lambda_{\sigma}+\rho_{\fm})
=W_{\C}\cdot \big(\lambda+\Lambda+\rho_{\fm}+\nu\big).
$$
Moreover, if $X_{\nu}\in\fa$ satisfies $B(X_{\nu},\dotvar)=\nu$, then $X_{\nu}\notin \fa_{\cO}$.
\end{Prop}

\begin{proof}
Let $z\in\cO$ and $X\in\fa^{-}$ be as in Proposition \ref{Prop principal asymptotics}.  By Proposition \ref{Prop limits of max rank orbits are conjugates of h_empty} there exists a $w\in N_{G}(\fa)$ so that $\fh_{z,X}=\Ad(w)\fh_{\emptyset}$. In view of Remark \ref{Rem Orbits of max rank} (\ref{Rem Orbits of max rank - item 1}) we may choose $w$ so that (\ref{eq wSigma^+ cap -Sigma^+=wSigma(Q) cap -Sigma}) is satisfied. Then
$$
\rho_{\cO,X}
=-\frac{1}{2}\rho_{wPw^{-1}}-\frac{1}{2}\rho_{P},
$$
see Remark \ref{Rem rho_O,X for maximal rank}.
In view of Proposition \ref{Prop principal asymptotics} there exists a left-$P$-invariant open neighborhood $\Omega$ of $e$ in $G$, a $\kappa\in\Sigma(\Ad(w)\fn_{P}\cap\overline{\fn}_{P};\fa)$ and non-zero distribution $\mu_{z,X}\in\cD'(\Omega,V_{\sigma})$ so that (\ref{Prop principal asymptotics - item 1}) -- (\ref{Prop principal asymptotics - item 5}) in Proposition \ref{Prop principal asymptotics} hold. In particular,
\begin{equation}\label{eq eta_(z,X) annihilated by Ad(w)ofn}
R^{\vee}(Y)\mu_{z,X}
=0
\qquad(Y\in \Ad(w)\overline{\fn}_{P}).
\end{equation}
and
\begin{equation}\label{eq equivariance of eta_(z,X) for max rank orbits}
R^{\vee}(Y)\mu_{z,X}
=\big(-\lambda+\rho_{wPw^{-1}}+\kappa\big)(Y)\mu_{z,X}
\qquad(Y\in\fa).
\end{equation}
Moreover, for every multi-index $\beta$ with $\kappa_{\beta}=\kappa$ there exists a unique $c_{\beta}\in V_{\sigma}^{*}$ such that for all $\phi\in\cD(\Omega,V_{\sigma})$
\begin{equation}\label{eq formula for eta_(z,X)}
\mu_{z,X}(\phi)
=\sum_{\substack{\beta\\ \kappa_{\beta}=\kappa}}
    \int_{M}\int_{A}\int_{N_{P}}\int_{w\overline{N}_{P}w^{-1}\cap\overline{N}_{P}} a^{-\lambda+\rho_{P}}
\Big(\sigma^{\vee}(m) c_{\beta},\partial^{\beta}\phi(man\overline{n})\Big)\,d\overline{n}\,dn\,da\,dm.
\end{equation}
The representation $\cD'(P:\xi:\lambda)$ admits an infinitesimal character, which via the Harish-Chandra isomorphism is identified with $-(\lambda+\Lambda_{\sigma}+\rho_{\fm})$. Therefore,
$$
R^{\vee}(u)\mu
=\gamma(u)\big(-(\lambda+\Lambda_{\sigma}+\rho_{\fm})\big)\mu
\qquad\big(u\in\Cen(\fg)\big).
$$
Since the elements of $\Cen(\fg)$ commutate with the adjoint action of $G$, we find for all $u\in\Cen(\fg)$
\begin{align}
\gamma(u)\big(-(\lambda+\Lambda_{\sigma}+\rho_{\fm})\big)\mu_{z,X}
\nonumber&=\lim_{t\to\infty}e^{t \big(\lambda+\rho_P+2\rho_{\cO,X}-\kappa\big)(X)}R^{\vee}\big(\exp(tX)\big)R^{\vee}(u)\mu\\
\label{eq inf char mu_(z,X)}&=R^{\vee}(u)\mu_{z,X}\,.
\end{align}
We will prove the proposition by computing $R^{\vee}(u)\mu_{z,X}$ using the formula (\ref{eq formula for eta_(z,X)}) for $\mu_{z,X}$.
For this we first look at the action of $M$ on $\mu_{z,X}$.

Let $m_{0}\in M$. By  (\ref{eq formula for eta_(z,X)}) we have for every $\phi\in\cD(\Omega,V_{\sigma})$
\begin{align*}
&R^{\vee}(m_{0})\mu_{z,X}(\phi)\\
&=\sum_{\substack{\beta\\ \kappa_{\beta}=\kappa}}
   \int_{M}\int_{A}\int_{N_{P}} \int_{w\overline{N}_{P}w^{-1}\cap\overline{N}_{P}} a^{-\lambda+\rho_{P}}
\Big(\sigma^{\vee}(m) c_{\beta},\partial^{\beta}\big(R(m_{0}^{-1})\phi\big)(man\overline{n})\Big)\,d\overline{n}\,dn\,da\,dm.
\end{align*}
The Haar-measure on each of the groups $w\overline{N}_{P}w^{-1}\cap\overline{N}_{P}$,  $N_{P}$ and $A$ is invariant under conjugation by $m_{0}$, and hence the right-hand side is equal to
\begin{align*}
&\sum_{\substack{\beta\\ \kappa_{\beta}=\kappa}}
    \int_{M}\int_{A}\int_{N_{P}}\int_{w\overline{N}_{P}w^{-1}\cap\overline{N}_{P}} a^{-\lambda+\rho_{P}}\\
    &\quad\times
        \frac{\partial^{\beta_{1}}}{\partial x_{1}^{\beta_{1}}}\cdots\frac{\partial^{\beta_{n}}}{\partial x_{n}^{\beta_{n}}}
        \bigg(\sigma^{\vee}(mm_{0})c_{\beta}, \phi\Big(man\exp\big(\sum_{i=1}^{n}x_{i}\Ad(m_{0})e_{i}\big)\overline n\Big)\bigg)\Big|_{x_{i}=0}
        \,d\overline{n}\,dn\,da\,dm.
\end{align*}

For multi-indices $\beta$ and $\beta'$, let $\chi^{\beta}_{\beta'}:M\to\C$ be determined by
$$
\partial^{\beta}\big(\psi\circ C_{m}\big)
=\sum_{\beta'}\chi^{\beta}_{\beta'}(m)\partial^{\beta'}\psi
\qquad\big(\psi\in C^{\infty}(wN_{P}w^{-1}\cap\overline{N}_{P}), m\in M\big).
$$
Here $C_{m}$ denotes conjugation by $m$.
If we denote $\trdeg_{\cO}(\mu)$ by $k$ and write $S$ for the set of multi-indices of length at most $k$, then the representation $\chi$ of $M$ on $\C^{S}$ that for all multi-indices $\beta\in S$ is given by
$$
\big(\chi(m)v\big)_{\beta}
=\sum_{\beta'}\chi^{\beta}_{\beta'}(m)v_{\beta'}
\qquad\big(m\in M, v=(v_{\beta'})_{\beta'\in S}\in \C^{S}\big)
$$
is isomorphic to the adjoint representation of $M$ on $\bigoplus_{l=0}^{k}\big(\Ad(w)\fn_{P}\cap\overline{\fn}_{P}\big)^{\otimes l}$.

Let $c\in \C^{S}\otimes V_{\sigma}^{*}\simeq (V_{\sigma}^{*})^{S}$ be the element of which the $\beta$'th component is equal to $c_{\beta}$ for each $\beta\in S$. Then for all $\phi\in \cD(\Omega, V_{\sigma})$
\begin{align}
\label{eq R^vee(m)mu_(z,X)} R^{\vee}(m_{0})\mu_{z,X}(\phi)
&=\sum_{\substack{\beta\\ \kappa_{\beta}=\kappa}}
    \int_{M}\int_{A}\int_{N_{P}}\int_{w\overline{N}_{P}w^{-1}\cap\overline{N}_{P}}
     a^{-\lambda+\rho_{P}}\\
\nonumber  &\qquad\qquad\times\bigg(\sigma^{\vee}(m)
    \Big(\big(\chi\otimes\sigma^{\vee}\big)(m_{0})c\Big)_{\beta},
    \partial^{\beta}\phi(man\overline{n})\bigg)
    \,d\overline{n}\,dn\,da\,dm.
\end{align}
The center $\Cen(\fg)$ is contained in
$$
\Cen(\fm\oplus\fa)\oplus\cU(\fg)\Ad(w)\overline{\fn}_{P}
=\big(\Cen(\fm)\otimes\Sym(\fa)\big)\oplus\cU(\fg)\Ad(w)\overline{\fn}_{P}.
$$
Let $u\in \Cen(\fg)$. There exist $v_{\fm,1},\dots,v_{\fm,k}\in\Cen(\fm)$ and $v_{\fa,1},\dots,v_{\fa,k}\in \Sym(\fa)\simeq \C[\fa^{*}]$  so that
$$
u-\sum_{j=1}^{k}v_{\fm,j}\otimes v_{\fa,j}\in\cU(\fg)\Ad(w)\overline{\fn}_{P}.
$$
We may assume that $\delta_{j}:=v_{\fa,j}\big(-\lambda+\rho_{wPw^{-1}}+\kappa\big)$ is for every $1\leq j\leq k$ either equal to $0$ or to $1$.
In view of (\ref{eq eta_(z,X) annihilated by Ad(w)ofn}),  (\ref{eq equivariance of eta_(z,X) for max rank orbits}) and (\ref{eq R^vee(m)mu_(z,X)}) we have for all $\phi\in\cD(\Omega,V_{\sigma})$
\begin{align*}
R^{\vee}(u)\mu_{z,X}(\phi)
&=\sum_{j=1}^{k}\delta_{j}\sum_{\substack{\beta\\ \kappa_{\beta}=\kappa}}
    \int_{M}\int_{A}\int_{N_{P}}\int_{w\overline{N}_{P}w^{-1}\cap\overline{N}_{P}}
     a^{-\lambda+\rho_{P}}\\
\nonumber  &\qquad\qquad\times\bigg(\sigma^{\vee}(m)
    \Big(\big(\chi\otimes\sigma^{\vee}\big)(v_{\fm,j})c\Big)_{\beta},
    \partial^{\beta}\phi(man\overline{n})\bigg)
    \,d\overline{n}\,dn\,da\,dm.
\end{align*}
Let $\Xi_{\chi\otimes\sigma^{\vee}}\subseteq(i\ft)^{*}$ be the set of lowest weights of $\chi\otimes\sigma^{\vee}$ and let $\gamma_{\fm}:\Cen(\fm)\to\Sym(\ft_{\C})\simeq\C[\ft^{*}]$ be the Harish-Chandra homomorphism for $\fm$. Then $\big(\chi\otimes\sigma^{\vee}\big)(v_{\fm,j})$ acts diagonalizably on $\C^{S}\otimes V_{\sigma}^{*}$ with eigenvalues $\gamma_{\fm}(v_{\fm,j})(\eta-\rho_{\fm})$ for $\eta\in \Xi_{\chi\otimes\sigma^{\vee}}$.
From (\ref{eq inf char mu_(z,X)}) it follows that $R^{\vee}(u)\mu_{z,X}$ is a multiple of $\mu_{z,X}$. It follows from the uniqueness of the element $c$ that there exists an $\eta\in\Xi_{\chi\otimes\sigma^{\vee}}$ so that
$$
\delta_{j}\big(\chi\otimes\sigma^{\vee}\big)(v_{\fm,j})c
=\delta_{j}\gamma_{\fm}(v_{\fm,j})(\eta-\rho_{\fm})c
\qquad(1\leq j\leq k).
$$
Therefore,
\begin{align*}
R^{\vee}(u)\mu_{z,X}
&=\sum_{j=1}^{k}\delta_{j}\gamma_{\fm}(v_{\fm,j})(\eta-\rho_{\fm})\mu_{z,X}\\
&=\sum_{j=1}^{k}\Big(v_{\fa,j}\big(-\lambda+\rho_{wPw^{-1}}+\kappa\big)\Big)\Big(\gamma_{\fm}(v_{\fm,j})(\eta-\rho_{\fm})\Big)\mu_{z,X}\\
&=\gamma(u)\big(-\lambda+\kappa+\eta-\rho_{\fm}\big)\mu_{z,X},
\end{align*}
and hence by (\ref{eq inf char mu_(z,X)})
$$
\gamma(u)\big(-\lambda-\Lambda_{\sigma}-\rho_{\fm}\big)\mu_{z,X}
=\gamma(u)\big(-\lambda+\kappa+\eta-\rho_{\fm})\mu_{z,X}.
$$
As $\mu_{z,X}\neq0$ and this identity holds for all $u\in\Cen(\fg)$, the first assertion now follows with $\nu=-\kappa$ and $\Lambda=-\eta$.
The second assertion follows from the Remarks \ref{Rem L_Q roots} and \ref{Rem rho_O,X for maximal rank}.
\end{proof}

\begin{Thm}\label{Thm bound on transversal degree}
Let $\lambda\in\fa_{\C}^{*}$, $\sigma\in\widehat{M}$ and $\mu\in\cD'(Z,P:\sigma:\lambda)$. Let  $\cO\in (P\bs Z)_{\mu}$ be of maximal rank. The following assertions hold true.
\begin{enumerate}[(i)]
\item\label{Thm bound on transversal degree - item 1} $\Im\lambda\in(\fa/\fa_{\cO})^{*}$. Assume that $\Im\lambda$ is regular in the sense that if $w\in W_{\C}$ stabilizes $\Im\lambda$, then $w$ normalizes $\fa_{\cO}+i\ft$ and acts trivially on $(\fa+i\ft)/(\fa_{\cO}+i\ft)$. Then
    $$
    \trdeg_{\cO}(\mu)=0.
    $$
\item\label{Thm bound on transversal degree - item 2} If $\trdeg_{\cO}(\mu)=0$ and  $v\in N_{G}(\fa)$ satisfies (\ref{eq wSigma^+ cap -Sigma^+=wSigma(Q) cap -Sigma}) and $\fh_{z,X}=\Ad(v)\fh_{\emptyset}$ for some $z\in \cO$ and order-regular element $X\in \fa^{-}$, then
   $$
   \Re\lambda
   \in \rho_{vPv^{-1}}+(\fa/\fa_{\cO})^{*}
   =\Ad^{*}(v)\big(\rho_{P}+(\fa/\fa_{\fh})^{*}\big)
   $$
\end{enumerate}
\end{Thm}

\begin{proof}
Let $\cO\in (P\bs Z)_{\mu}$ be of maximal rank and assume that $\trdeg_{\cO}(\mu)\neq 0$. Let $\Lambda_{\sigma}\in i\ft^{*}$ be the highest-weight of $\sigma$. By Proposition \ref{Prop bound on transversal degree} there exists a non-zero $\nu\in \N_{0}\Sigma(P)$, a dominant $\Sigma_{\fm}$-integral element $\Lambda\in i\ft^{*}$, and a $w\in W_{\C}$ so that
$$
w\cdot \big(\lambda+\Lambda_{\sigma}+\rho_{\fm}\big)
=\lambda+\Lambda+\rho_{\fm}+\nu.
$$
Moreover, the element $X_{\nu}\in\fa$ so that $B(X_{\nu},\dotvar)=\nu$ is not contained in $\fa_{\cO}$.

Note that $W_{\C}$ stabilizes the real subspace $(\fa\oplus i\ft)^{*}$ of $(\fa\oplus i\ft)_{\C}^{*}$. Therefore,
\begin{equation}\label{eq equation lambda_R}
w\cdot \big(\Re\lambda+\Lambda_{\sigma}+\rho_{\fm}\big)
=\Re\lambda+\Lambda+\rho_{\fm}+\nu
\end{equation}
and
\begin{equation}\label{eq equation lambda_I}
w\cdot \Im\lambda
=\Im\lambda.
\end{equation}
Furthermore, it follows from Corollary \ref{Cor condition on lambda|fa_O} that $\Im\lambda\in(\fa/\fa_{\cO})^{*}$.
Now assume that $\Im\lambda$ satisfies the regularity condition stated in  (\ref{Thm bound on transversal degree - item 1}). In view of (\ref{eq equation lambda_I}) the element $w$ normalizes $\fa_{\cO}+i\ft$ and acts trivially on $\big((\fa+i\ft)/(\fa_{\cO}+i\ft)\big)^{*}$. Let $\fa_{\cO}^{\perp}$ be the Killing orthocomplement of $\fa_{\cO}$ in $\fa$. Then $(\fa+i\ft)/(\fa_{\cO}+i\ft)$ is identified with $\fa_{\cO}^{\perp}$ via the Killing form and hence $w$ acts trivially on $\fa_{\cO}^{\perp}$.
It follows from (\ref{eq equation lambda_R}) that
$$
(\Re\lambda+\nu)\big|_{\fa_{\cO}^{\perp}}
=w\cdot(\Re\lambda+\Lambda_{\sigma}+\rho_{\fm})\big|_{\fa_{\cO}^{\perp}}
=\Re\lambda\big|_{\fa_{\cO}^{\perp}}
$$
and thus $\nu|_{\fa_{\cO}^{\perp}}=0$. This is in contradiction with $X_{\nu}\notin \fa_{\cO}$. We thus conclude that $\trdeg_{\cO}(\mu)= 0$.
This proves (\ref{Thm bound on transversal degree - item 1}).

Assertion (\ref{Thm bound on transversal degree - item 2}) follows from Corollary \ref{Cor condition on lambda|fa_O} and Remark \ref{Rem rho_O,X for maximal rank}.
\end{proof}

\section{Construction and properties of $H$-fixed distribution vectors}\label{Section Construction}
\subsection{$H$-spherical finite dimensional representations}\label{Subsection Construction - Finite dimensional reps}
We write $Z_{\C}$ for the complexification $\underline{G}(\C)/\underline{H}(\C)$ of $Z$. Note that $Z$ naturally embeds into $Z_{\C}$. We further write $\C[Z]^{(P)}$ for the multiplicative monoid of functions $f:Z\to\C$ so that
\begin{enumerate}[(a)]
\item there exists a non-zero regular function $\phi$ on $Z_{\C}$ so that $f=\phi\big|_{Z}$,
\item there exists a $\nu\in\fa^{*}$ so that
$$
f(man\cdot z)
=a^{\nu}f(z)
\qquad(m\in M, a\in A, n\in N).
$$
\end{enumerate}
It follows from \cite[Lemma 5.6]{KnopKrotzSayagSchlichtkrull_SimpleCompactificationsAndPolarDecomposition} that $\C[Z]^{(P)}$ is finitely generated.

For every function $f\in\C[Z]^{(P)}$ there exists a finite dimensional representation $(\pi,V)$, an $H$-fixed vector $v_{H}\in V$ and a $MN$-fixed vector $v^{*}\in V^{*}$ for the contragredient representation $\pi^{\vee}$ of $\pi$ such that $f$ is the matrix-coefficient of $v_{H}$ and $v^{*}$. If $\pi$ has lowest weight $\nu\in\fa^{*}$, then $v^{*}$ is a highest weight vector of $\pi^{\vee}$ with weight $-\nu$. Note that for $m\in M$, $a\in A$, $n\in N_{P}$, $g\in G$ and $h\in H$
\begin{equation}\label{eq f(mangh)=a^nu f(g)}
f(mangh)
=v^{*}\big(\pi(mangh)v_{H}\big)
=a^{\nu}v^{*}\big(\pi(g)v_{H}\big)
=a^{\nu}f(g).
\end{equation}
We define $\Lambda$ to be the monoid of $\fa$-weights $\nu$ that occur in $\C[Z]^{(P)}$, i.e., $\Lambda$ is the monoid of lowest $\fa$-weights of  finite dimensional representations $\pi$ with $V^{H}\neq\{0\}$ and $(V^{*})^{MN}\neq\{0\}$.
It follows from (\ref{eq f(mangh)=a^nu f(g)}) that
$$
\Lambda
\subseteq(\fa/\fa_{\fh})^{*}.
$$
The rank of the lattice generated by $\Lambda$ is equal to $\rank(Z)$, see the proof of \cite[Proposition 3.13]{KnopKrotzSayagSchlichtkrull_SimpleCompactificationsAndPolarDecomposition}.

We define the submonoid $\C[Z]^{(P)}_{+}$ of $\C[Z]^{(P)}$ by
$$
\C[Z]^{(P)}_{+}
:=\Big\{f\in\C[Z]^{(P)}:f^{-1}(\C\setminus\{0\})=\bigcup_{\substack{\cO\in P\bs Z\\\cO \text{ open}}}\cO\Big\}.
$$
By \cite[Lemma 3.6]{KnopKrotzSayagSchlichtkrull_SimpleCompactificationsAndPolarDecomposition} the set $\C[Z]^{(P)}_{+}$ is non-empty.
Furthermore, we write $\Lambda_{+}$ for the submonoid of $\Lambda$ corresponding to $\C[Z]^{(P)}_{+}$.
We note that
$\C[Z]^{(P)}_{+}\C[Z]^{(P)}=\C[Z]^{(P)}_{+}$, and hence $\Lambda_{+}+\Lambda=\Lambda_{+}$. As $\C[Z]^{(P)}_{+}$ is non-empty, the submonoid $\Lambda_{+}$ has full rank.

\begin{Lemma}\label{Lemma Occurence of M-reps in regular functions}
Let $z\in Z$ be adapted, let $\sigma\in\widehat{M}$ and let $\eta\in (V_{\sigma}^{*})^{M\cap H_{z}}$. Then there exists a $\nu\in(\fa/\fa_{\fh})^{*}$ and a regular function $f_{\eta}:Z\to V_{\sigma}^{*}$ so that $f_{\eta}(z)=\eta$ and
$$
f_{\eta}(man\cdot z')=a^{\nu}\sigma^{\vee}(m)f_{\eta}(z')
\qquad(m\in M, a\in A, n\in N_{P}, z'\in Z).
$$
\end{Lemma}

\begin{proof}
We may assume that $(V_{\sigma}^{*})^{M\cap H_{z}}\neq \{0\}$.
It suffices to prove the existence of a regular function $\phi:Z_{\C}\to\C$ so that $\phi$ is $N_{P,\C}$-invariant and $\sigma^{\vee}$ occurs as a direct summand in the representation of $M$ generated by $\phi$. Let $z\in Z$ be adapted. We define the algebras
$$
\cA
:=\big\{\phi:Z_{\C}\to\C:\phi \text{ is regular and }\phi(n\cdot z')=\phi(z')\text{ for all }n\in N_{P_{\C}}, z'\in Z_{\C}\big\}
$$
and
$$
\cA_{M}
:=\big\{M/(M\cap H_{z})\ni m\mapsto \phi(m\cdot z): \phi\in \cA\big\}
$$
In order to prove the existence of a function $\phi$ with the properties mentioned above, it suffices to prove that $\cA_{M}$ is dense in $C\big(M/(M\cap H_{z})\big)$. For this we use the Stone-Weierstrass theorem.

Note that $\cA_{M}$ is a subalgebra of $C\big(M/(M\cap H_{z})\big)$, is closed under complex conjugation and contains the unit, i.e., the constant function $1$. By the Stone-Weierstrass theorem $\cA_{M}$ is dense in $C\big(M/(M\cap H_{z})\big)$ if $\cA_{M}$ separates points in $M/(M\cap H_{z})$.
For the latter it suffices to prove that $\cA$ separates points in $M\cdot z\subseteq Z$.

Let $m_{1},m_{2}\in M$ and assume that $m_{1}\cdot z\neq m_{2}\cdot z$. By \cite[Theorem 2]{Rosenlicht_OnQuotientVarietiesAndTheAffineEmbeddingOfCertainHomogeneousSpaces} the $N_{P}$-orbits $N_{P}m_{1}\cdot z$ and $N_{P}m_{2}\cdot z$ are closed in $Z$.
The space $N_{P}\bs Z$ is isomorphic to the quasi-affine space $G/N_{P}\times_{\diag(G)}Z$ and
$$
\cA
\simeq \C[G/N_{P}\times Z]^{\diag(G)}.
$$
Therefore, also $\cD_{1}:=\diag(G)\cdot \big(e N_{P}\times m_{1}\cdot z\big)$ and $\cD_{2}:=\diag(G)\cdot \big(e N_{P}\times m_{2}\cdot z\big)$ are closed.
It is an straightforward corollary of the main result in \cite{Birkes_OrbitsOfLinearAlgebraicGroups} that then $\cD_{1}$ and $\cD_{2}$ are also Zariski closed.
Let $\cI_{1}$ and $\cI_{2}$ be the ideals of $\C[G/N_{P}\times Z]$ of functions vanishing on $\cD_{1}$ and $\cD_{2}$, respectively.  As $m_{1}\cdot z\neq m_{2}\cdot z$ and $z$ is adapted, it follows from the local structure theorem, Proposition \ref{Prop LST holds for adapted points} that $\cD_{1}$ and $\cD_{2}$ are disjoint. Together with the fact that $\cD_{1}$ and $\cD_{2}$ are Zariski-closed this implies that
$$
\C[G/N_{P}\times Z]
=\cI_{1}+\cI_{2}.
$$
Since $\cD_{1}$ and $\cD_{2}$ are $\diag(G)$-orbits, the ideals $\cI_{1}$ and $\cI_{2}$ are $\diag(G)$-stable.
As $\diag(G)$ is reductive, it follows that
$$
\C[G/N_{P}\times Z]^{\diag(G)}
=\cI_{1}^{\diag(G)}+\cI_{2}^{\diag(G)}.
$$
In particular, there exist a $\phi_{1}\in \cI_{1}^{\diag(G)}$ and a $\phi_{2}\in \cI_{2}^{\diag(G)}$ so that $\phi_{1}+\phi_{2}$ is the constant function $1$. Now $\phi_{2}(m_{1}\cdot z)=1$ and $\phi_{2}(m_{2}\cdot z)=0$. We thus conclude that $\cA$ separates points in $M\cdot z\subseteq Z$.
\end{proof}

\subsection{Construction on open $P$-orbits}\label{Subsection Construction - construction on open orbits}
For an adapted point $z\in Z$,  $\lambda\in \rho_{P}+(\fa/\fa_{\fh})_{\C}^{*}$, a finite dimensional representation $(\sigma,V_{\sigma})$ of $M$, and $\eta\in (V_{\sigma}^{*})^{M\cap H_{z}}$ we define the function
$$
\epsilon_{z}(P:\sigma:\lambda:\eta):Z\to V_{\sigma}^{*}
$$
by
$$
\left\{
  \begin{array}{ll}
    \epsilon_{z}(P:\sigma:\lambda:\eta)(nma\cdot z)=a^{-\lambda+\rho_{P}}\sigma^{\vee}(m)\eta, & (n\in N_{Q}, a\in A, m\in M); \\
    \epsilon_{z}(P:\sigma:\lambda:\eta)(y)=0, & (y\notin P\cdot z).
  \end{array}
\right.
$$
We note that in view of Proposition \ref{Prop LST holds for adapted points} this function is well defined.

Let $\Gamma$ be the cone in $(\fa/\fa_{\fh})^{*}$ generated by $\Lambda$, i.e.,
$$
\Gamma
=\sum_{\lambda\in\Lambda}\R_{\geq 0}\lambda.
$$
Since $\Lambda$ has full rank, the interior of $\Gamma$ is non-empty.

\begin{Prop}\label{Prop construction on open orbit}
Let $z\in Z$ be adapted.
Let $(\sigma,V_{\sigma})$ be a finite dimensional unitary representation of $M$. Assume that $\eta$ is a non-zero $M\cap H_{z}$-fixed vector in $V_{\sigma}^{*}$. For $\lambda\in\rho_{P}-\Gamma+i(\fa/\fa_{\fh})^{*}$ the function $\epsilon_{z}(P:\sigma:\lambda:\eta)$ is measurable and bounded on every compact subset of $G$.
\end{Prop}

\begin{proof}
Let $\cO=P\cdot z$. Note that $\cO$ is open.
The function $\epsilon_{z}(P:\sigma:\lambda:\eta)$ is continuous outside of the set $\partial(\cO)$, which has measure $0$ in $G$. Therefore, $\epsilon_{z}(P:\sigma:\lambda:\eta)$ is measurable. Now let $\lambda\in\rho_{P}-\Gamma+i(\fa/\fa_{\fh})^{*}$. Let $f_{1},\dots,f_{r}$ be a set of generators of $\C[Z]^{(P)}$ with $f_{i}(z)=1$ , and let $\lambda_{1},\dots,\lambda_{r}$ be the corresponding set of generators of $\Lambda$. Let $\nu_{1},\dots,\nu_{r}\in\C$ with $\Re\nu_{i}\geq 0$ be such that
$$
\rho_{P}-\lambda
=\sum_{i=1}^{r}\nu_{i}\lambda_{i}.
$$
Then
$$
\epsilon_{z}(P:\sigma:\lambda:\eta)
=\Big(\prod_{j=1}^{r}f_{j}^{\nu_{j}}\Big)\epsilon_{z}(P:\sigma:\rho_{P}:\eta).
$$
Therefore,
$$
\|\epsilon_{z}(P:\sigma:\lambda:\eta)(x)\|_{\sigma}
=\Big(\prod_{j=1}^{r}|f_{j}(x)|^{\Re\nu_{j}}\Big)\|\epsilon_{z}(P:\sigma:\rho_{P},\eta)(x)\|_{\sigma}.
$$
As
$$
\|\epsilon_{z}(P:\sigma:\rho_{P},\eta)(x)\|_{\sigma}
=\left\{
   \begin{array}{ll}
     \|\eta\|_{\sigma}, & (x\in \cO) \\
     0, & (x\notin\cO),
   \end{array}
 \right.
$$
the function $\epsilon_{z}(P:\sigma:\rho_{P},\eta)$ is bounded.
Since the functions $f_{j}$ are continuous, it follows that $\epsilon_{z}(P:\sigma:\lambda:\eta)$ is bounded on every compact subset of $G$.
\end{proof}

For every adapted point $z\in Z$,  finite dimensional unitary representation $(\sigma,V_{\sigma})$ of $M$,  non-zero $M\cap H_{z}$-fixed vector $\eta$ in $V_{\sigma}^{*}$, and $\lambda\in\rho_{P}-\Gamma+i(\fa/\fa_{\fh})^{*}$, the function $\epsilon_{z}(P:\sigma:\lambda:\eta)$ defines in view of Proposition \ref{Prop construction on open orbit} a distribution $\mu_{z}(P:\sigma:\lambda:\eta)$ in $\cD'(Z,P:\sigma:\lambda)$ given by
\begin{equation}\label{eq Def mu_z}
\mu_{z}(P:\sigma:\lambda:\eta):\cD(Z,V_{\sigma})\to\C;
\quad\phi\mapsto \int_{Z}\Big(\epsilon_{z}(P:\sigma:\lambda:\eta)(x),\phi(x)\Big)\,dx.
\end{equation}
It follows from Proposition \ref{Prop LST holds for adapted points}, that for all $\phi\in \cD(Z,V_{\sigma})$
\begin{align}\label{eq Formula mu_z}
&\mu_{z}(P:\sigma:\lambda:\eta)(\phi)\\
\nonumber&\qquad=\int_{N_{Q}}\int_{M/M\cap H_{z}}\int_{A/A\cap H_{z}}
    a^{-\lambda+\rho_{P}-2\rho_{Q}}\Big(\sigma^{\vee}(m)\eta,\phi(nma\cdot z)\Big)\,dm\,dn\,da.
\end{align}
It is easily seen that for a given adapted point $z\in Z$, finite dimensional representation $\sigma$ of $M$, and $M\cap H_{z}$-fixed vector $\eta$ in $V_{\sigma}^{*}$, the family $\rho_{P}-\Gamma+i(\fa/\fa_{\fh})^{*}\ni\lambda\mapsto\mu_{z}(P:\sigma:\lambda:\eta)$ is a holomorphic family of distributions in $\cD'(Z,V_{\sigma})$.
We will show that this family extends meromorphically to all of $(\fa/\fa_{\fh})_{\C}^{*}$. To do so we use the theorem of Bernstein and Sato.
Our proof is similar to that of \cite[Theorem 5.1]{Olafsson_FourierAndPoissonTransformationAssociatedToASemisimpleSymmetricSpace}.

\begin{Prop}\label{Prop meromorphic continuation on open orbit}
Let $z\in Z$ be adapted.
Let $(\sigma,V_{\sigma})$ be a finite dimensional unitary representation of $M$ and let $\eta$ be a non-zero $M\cap H_{z}$-fixed vector in $V_{\sigma}^{*}$. The family
$$
\rho_{P}-\Gamma+i(\fa/\fa_{\fh})^{*}\ni\lambda \mapsto \mu_{z}(P:\sigma:\lambda:\eta)
$$
of distributions in $\cD'(Z,V_{\sigma})$ defined in (\ref{eq Def mu_z}) is holomorphic and extends to a meromorphic family on $\rho_{P}+(\fa/\fa_{\fh})_{\C}^{*}$. There exists a locally finite union $\cH$ of complex affine hyperplanes in $(\fa/\fa_{\fh})_{\C}^{*}$ of the form
\begin{equation}\label{eq Form of hyperplanes}
\{\lambda\in(\fa/\fa_{\fh})^{*}_{\C}: \lambda(X)=a\}\quad\text{for some }X\in\fa\text{ and }a\in\R,
\end{equation}
so that the poles of the family $\lambda\mapsto \mu_{z}(P:\sigma:\lambda:\eta)$ lie on $\rho_{P}+\cH$.  For $\lambda\in\rho_{P}+ (\fa/\fa_{\fh})_{\C}^{*}$ outside of the set $\rho_{P}+\cH$ the distribution $\mu_{z}(P:\sigma:\lambda:\eta)$ thus obtained is contained in $\cD'(Z,P:\sigma:\lambda)$.
\end{Prop}

\begin{proof}
Let $\nu_{1},\dots,\nu_{r}\in \Lambda$ be a basis of $(\fa/\fa_{\fh})^{*}$ and let $f_{1},\dots,f_{r}\in \C[Z]^{(P)}$ be so that
$$
f_{j}(man\cdot z)=a^{\nu_{j}}
\qquad(1\leq j\leq r, m\in M, a\in A, n\in N_{P}).
$$
Note that each $f_{j}$ is real valued and thus $f_{j}^{2}$ is non-negative.
For $\lambda\in \sum_{j=1}^{r}\R_{\geq 0}\nu_{j}+i(\fa/\fa_{\fh})^{*}$, we define
$$
\varphi^{\lambda}
:=\prod_{j=1}^{r}\big(f_{j}^{2}\big)^{u_{j}}:Z\to\C,
$$
where $u_{j}\in \C$ is determined by
$$
\lambda
=2\sum_{j=1}^{r}u_{j}\nu_{j}.
$$
By the theorem of Bernstein (see \cite[Appendice A]{BrylinskiDelorme_VecteursDistributionsH-Invariants}) there exists for every $1\leq j\leq r$ a  polynomial function  $b_{j}$ on $(\fa/\fa_{\fh})_{\C}^{*}$ and a differential operator $D_{j}$ on $Z$ with coefficients in $\C[Z][\lambda]$, so that for all $\lambda\in\sum_{j=1}^{r}\R_{\geq 1}\nu_{j}+i(\fa/\fa_{\fh})^{*}$
\begin{equation}\label{eq Functional equation phi_lambda}
D_{j}\varphi^{\lambda}
=b_{j}(\lambda)\varphi^{\lambda-\nu_{j}}.
\end{equation}
Furthermore, there exists a locally finite union $\cH'$ of complex affine hyperplanes of the form (\ref{eq Form of hyperplanes})  in $(\fa/\fa_{\fh})_{\C}^{*}$ so that the zero's of the polynomials $b_{j}$ are contained in $\cH'$.

We now write $\cO$ for the open $P$-orbit $P\cdot z$ and $\1_{\cO}$ for its characteristic function. Let $n$ be the maximum of the degrees of the differential operators $D_{j}$. Let further $f_{0}\in\C[Z]^{(P)}_{+}$ and let $\gamma\in\Lambda_{+}$ be its weight.
Then for all $\lambda\in (n+1)\gamma+\sum_{j=1}^{r}\R_{\geq0}\nu_{j}$ we have
$$
\varphi^{\lambda}
=f_{0}^{n+1}\varphi^{\lambda-(n+1)\gamma}.
$$
Since $f_{0}$ vanishes on $\partial\cO$ and $\varphi^{\lambda-(n+1)\gamma}$ is continuous, it follows that $\varphi^{\lambda}\1_{\cO}$ is at least $n$ times continuously differentiable. From (\ref{eq Functional equation phi_lambda}) it then follows that for all $\lambda\in  (n+1)\gamma+\sum_{j=1}^{r}\R_{\geq 1}\nu_{j}$ and every $1\leq j\leq r$
$$
D_{j}\Big(\varphi^{\lambda}\1_{\cO}\Big)
=\Big(D_{j}\varphi^{\lambda}\Big)\1_{\cO}
=b_{j}(\lambda)\varphi^{\lambda-\nu_{j}}\1_{\cO}.
$$
By means of this functional equation the family
$$
\lambda\mapsto \varphi^{\cO,\lambda}:=\varphi^{\lambda}\1_{\cO}
$$
can be extended to a meromorphic family of distributions on $Z$. The poles of this family lie on $\cH'$.

Let $\sigma\in\widehat{M}$ and $\eta\in (V_{\sigma}^{*})^{M\cap H_{z}}$. By Lemma \ref{Lemma Occurence of M-reps in regular functions} there exists a regular function $f_{\eta}:Z\to V_{\sigma}^{*}$ and a $\nu\in(\fa/\fa_{\fh})^{*}$ so that
$$
f_{\eta}(man\cdot z)
=a^{\nu}\sigma^{\vee}(m)\eta
\qquad(m\in M, a\in A, n\in N_{P}).
$$
Now for $\lambda\in \rho_{P}-\nu-\sum_{j=1}^{r}\R_{\geq0}\nu_{j}+i(\fa/\fa_{\fh})^{*}$
$$
\epsilon_{z}(P:\sigma:\lambda:\eta)
=\varphi^{\cO,\rho_{P}-\lambda-\nu} f_{\eta}.
$$
It follows that for these $\lambda $ the distribution $\mu_{z}(P:\sigma:\lambda:\eta)$ is given by
$$
\mu_{z}(P:\sigma:\lambda:\eta)(\phi)
=\varphi^{\cO,\rho_{P}-\lambda-\nu}\Big(\big(f_{\eta},\phi\big)\Big)
\qquad\big(\phi\in \cD(Z,V_{\sigma})\big),
$$
where $\big(f_{\eta},\phi\big)$ is short-hand notation for the function $Z\to\C$, $z'\mapsto\big(f_{\eta}(z'),\phi(z')\big)$. As $\lambda\mapsto\varphi^{\cO,\rho_{P}-\lambda-\nu}$ is a meromorphic family of distributions, it follows that $\mu_{z}(P:\sigma:\lambda:\eta)$ extends to a meromorphic family of distributions in $\cD'(Z,V_{\sigma})$. The poles of this family lie on $\rho_{Q}-\nu-\cH'$. Moreover,  $\mu_{z}(P:\sigma:\lambda:\eta)$ is contained in $\cD'(Z,P:\sigma:\lambda)$ for all $\lambda\in\rho_{P}+(\fa/\fa_{\fh})_{\C}^{*}$ outside of the poles of the family as this is true for all $\lambda$ in the open subset $\rho_{P}-\Gamma+i(\fa/\fa_{\fh})^{*}$ of $\rho_{P}+(\fa/\fa_{\fh})_{\C}^{*}$.
\end{proof}

\subsection{Construction on $P$-orbits of maximal rank}\label{Subsection Construction - Construction on max rank orbits}

\begin{Lemma}\label{Lemma a(wPw^-1:P:sigma:lambda) is isomorphism}
Let $w\in N_{G}(\fa)$ be so that (\ref{eq wSigma^+ cap -Sigma^+=wSigma(Q) cap -Sigma}) holds. Let $\sigma\in\widehat{M}$. Then for $\lambda$ in $\Ad(w)^{*}\big(\rho_{P}+(\fa/\fa_{\fh})_{\C}^{*}\big)$  outside of a locally finite set of complex affine hyperplanes of the form
\begin{equation}\label{eq Form of hyperplanes IV}
\{\lambda\in\Ad(w)^{*}\big(\rho_{P}+(\fa/\fa_{\fh})_{\C}^{*}\big): \lambda(X)=c\}\quad\text{for some }X\in\fa\text{ and }a\in\R
\end{equation}
the intertwining operator
$$
\cA(wPw^{-1}:P:\sigma:\lambda):\cD'(P:\sigma:\lambda)\to\cD'(wPw^{-1}:\sigma:\lambda)
$$
is an isomorphism.
\end{Lemma}

\begin{proof}
Let $l$ be the length of $w$ and let $P=P_{0},\cdots, P_{l}=wPw^{-1}$ be a sequence of minimal parabolic subgroups so that $A\subseteq P_{j}$ and $P_{j}$ and $P_{j+1}$ are adjacent for every $j$. For $0\leq j< l$ let $\alpha_{j}\in\Sigma(\fa)$ be the reduced root such that
$$
\Sigma(P_{j+1},\fa)\cap \Sigma(\overline{P_{j}},\fa)
\subseteq \{\alpha_{j},2\alpha_{j}\}.
$$
The rank one standard intertwining operators $A(P_{j}:P_{j+1}:\sigma:\lambda)$ are isomorphisms for $\lambda$ in $\rho_{P}+(\fa/\fa_{\cO})_{\C}^{*}$ outside a locally finite union $\cH_{j}$ of complex affine hyperplanes of the form $\{\lambda\in \fa^{*}_{\C}:\lambda(\alpha_{j}^{\vee})=c\}$ with $c\in \Q$.
See for example \cite[Proposition B.1]{KrotzKuitOpdamSchlichtkrull_InfinitesimalCharactersOfDiscreteSeriesForRealSphericalSpaces}. From \cite[Th\'eor\`eme 1]{Schiffmann_IntegralesDEntrelacementEtFonctionsDeWhittaker} it now follows that $A(P:wPw^{-1}:\sigma:\lambda)$ is an isomorphism for
$$
\lambda\in\Big(\rho_{P}+(\fa/\fa_{\cO})_{\C}^{*}\Big)\setminus\bigcup_{j=0}^{l-1}\cH_{j}.
$$
The same then holds for $A(P:wPw^{-1}:\sigma:\lambda)^{*}$ and $\cA(wPw^{-1}:P:\sigma:\lambda)$.

Since (\ref{eq wSigma^+ cap -Sigma^+=wSigma(Q) cap -Sigma}) is assumed to hold, we have
$$
\{\alpha_{j}:0\leq j<l\}
\subseteq\Sigma(wPw^{-1},\fa)\cap\Sigma(\overline{P},\fa)
=\Sigma(wQw^{-1},\fa)\cap\Sigma(\overline{P},\fa).
$$
In view of Remark \ref{Rem L_Q roots} $\alpha_{j}^{\vee}\notin \Ad(w)\fa_{\fh}$ for all $j$. Therefore, the intersection of $\cH_{j}$ with $\Ad(w)^{*}\big(\rho_{P}+(\fa/\fa_{\fh})_{\C}^{*}\big)$ is for every $1\leq j\leq l-1$ a locally finite union of affine hyperplanes of the form  (\ref{eq Form of hyperplanes IV}).
\end{proof}

We now come to construction of distributions on maximal rank orbits.

\begin{Prop}\label{Prop construction on max rank orbits}
Let $w\in N_{G}(\fa)$ be so that (\ref{eq wSigma^+ cap -Sigma^+=wSigma(Q) cap -Sigma}) holds. Let further $\cO\in w\cdot (P\bs Z)_{\open}$ and let $z\in \cO$ be weakly adapted.
Now the $wPw^{-1}$-orbit $wPw^{-1}\cdot z$ is open in $Z$ and $z$ is adapted to $wPw^{-1}=MAwN_{P}w^{-1}$.
Let $(\sigma,V_{\sigma})$ be a finite dimensional unitary representation of $M$ and let $\eta$ be a $M\cap H_{z}$-fixed vector in $V_{\sigma}^{*}$.
The assignment
\begin{equation}\label{eq mu_(O_w)=a(wPw^(-1):P)mu_O}
\lambda\mapsto \mu_{z}(P:\sigma:\lambda:\eta):=\cA(wPw^{-1}:P:\sigma:\lambda)^{-1}\mu_{z}(wPw^{-1}:\sigma:\lambda:\eta)
\end{equation}
defines a meromorphic family on $\rho_{wPw^{-1}}+(\fa/\fa_{\cO})_{\C}^{*}=\Ad(w)^{*}\big(\rho_{P}+(\fa/\fa_{\fh})_{\C}^{*}\big)$ of distributions in $\cD'(Z,P:\sigma:\lambda)$.
These distributions have the following properties.
\begin{enumerate}[(i)]
\item\label{Prop construction on max rank orbits - item 2}  If
$$
\lambda\in
\rho_{wPw^{-1}}-\Ad^{*}(w)\Gamma+i(\fa/\fa_{\cO})^{*}
=\Ad(w)^{*}\big(\rho_{P}-\Gamma+i(\fa/\fa_{\fh})^{*}\big),
$$
then the distribution $\mu_{z}(P:\sigma:\lambda:\eta)$ is for $\phi\in\cD(G,V_{\sigma})$ given by the absolutely convergent integral
\begin{align}
\label{eq mu_cO identity}
&\mu_{z}(P:\sigma:\lambda:\eta)(\phi)\\
\nonumber
&\qquad=\int_{N_{P}\cap wN_{Q}w^{-1}}\int_{M/M\cap H_{z}}\int_{A/A\cap H_{z}}a^{-\lambda+\Ad^{*}(w)\rho_{P}-2\Ad^{*}(w)\rho_{Q}}\\
\nonumber
    &\qquad\qquad\qquad\qquad\qquad\qquad\qquad\qquad\qquad\times
        \Big(\sigma^{\vee}(m)\eta,\phi(nma\cdot z)\Big)\,da\,dm\,dn.
\end{align}
\item\label{Prop construction on max rank orbits - item 3}
There exists a locally finite union $\cH$ of complex affine hyperplanes of the form
\begin{equation}\label{eq Form of hyperplanes II}
\{\lambda\in(\fa/\fa_{\cO})^{*}_{\C}: \lambda(Y)=a\}\quad\text{for some }Y\in\fa\text{ and }a\in\R
\end{equation}
in $(\fa/\fa_{\cO})_{\C}^{*}$, so that the poles of the family $\lambda\mapsto \mu_{z}(P:\sigma:\lambda:\eta)$ lie on $\rho_{wPw^{-1}}+\cH$.
\item\label{Prop construction on max rank orbits - item 4}
For every $\lambda\in\rho_{wPw^{-1}}+\Big((\fa/\fa_{\cO})_{\C}^{*}\setminus\cH\Big)$  and $\eta\in (V_{\sigma}^{*})^{M\cap H_{z}}\setminus \{0\}$ we have
$$
\supp \mu_{z}(P:\sigma:\lambda:\eta)=\overline{\cO}.
$$
\item\label{Prop construction on max rank orbits - item 1}
Up to scaling the distributions $\mu_{z}(P:\sigma:\lambda:\eta)$ do not depend on the choice of $w$, i.e., if $w'\in N_{G}(\fa)$ satisfies (\ref{eq wSigma^+ cap -Sigma^+=wSigma(Q) cap -Sigma}) and $w\cdot(P\bs Z)_{\open}=w'\cdot(P\bs Z)_{\open}$,
then there exists a $c>0$ so that
$$\mu_{z}(P:\sigma:\lambda:\eta)
=c\cA(w'P{w'}^{-1}:P:\sigma:\lambda)^{-1}\mu_{z}(w'P{w'}^{-1}:\sigma:\lambda:\eta)
$$
as a meromorphic identity on $\rho_{wPw^{-1}}+(\fa/\fa_{\cO})_{\C}^{*}$.
\end{enumerate}
\end{Prop}

\begin{Rem}
For reductive symmetric spaces the distributions $\mu_{z}(Q:\xi:\lambda:\eta)$ from Proposition \ref{Prop construction on max rank orbits} were constructed in \cite{vdBanKuit_NormalizationsOfEisensteinIntegrals}. The proof uses the same crucial point: the geometric decomposition (\ref{eq decomp of wPw^(-1) wrt O}) in Theorem \ref{Thm structure theorem for wPw^(-1) cdot z} translates to a decomposition of a distribution $\mu_{z}(P:\sigma:\lambda:\eta)$ for an adapted point $z$, as constructed in Section \ref{Subsection Construction - construction on open orbits}, into an intertwining operator and a distribution that transforms under a conjugate minimal parabolic subgroup $wPw^{-1}$ and is supported on the closure of a non-open $wPw^{-1}$-orbit. In \cite{vdBanKuit_NormalizationsOfEisensteinIntegrals} the objects under consideration are $H$-invariant functionals on $C^{\infty}(P:\sigma:\lambda)$; we consider here distributions in $\cD'(Z,P:\sigma:\lambda)$. The resulting analysis is formally the same. However, we choose here to look at distributions rather than functionals since in this way we can avoid working with densities.
\end{Rem}

\begin{proof}
In view of Proposition \ref{Prop meromorphic continuation on open orbit} the family (\ref{eq mu_(O_w)=a(wPw^(-1):P)mu_O}) is a meromorphic family of distributions in $\cD'(Z,P:\sigma:\lambda)$. It follows from Proposition \ref{Prop meromorphic continuation on open orbit} and Lemma \ref{Lemma a(wPw^-1:P:sigma:lambda) is isomorphism} that the poles of this family lie on a locally finite union of complex affine hyperplanes of the form (\ref{eq Form of hyperplanes II}). This proves (\ref{Prop construction on max rank orbits - item 3}).

We move on to (\ref{Prop construction on max rank orbits - item 2}).
Let $\lambda\in\Ad(w)^{*}\big(\rho_{P}-\Gamma+i(\fa/\fa_{\fh})^{*}\big)$ and $\phi\in\cD(Z,V_{\sigma})$.
By Proposition \ref{Prop limits of max rank orbits are conjugates of h_empty} the $wPw^{-1}$-orbit $wPw^{-1}\cdot z $ is open. Moreover, if $P=MAN_{P}$ is replaced by $wPw^{-1}=MA(wN_{P}w^{-1})$, then the point $z$ is adapted. Therefore, it follows from Proposition \ref{Prop construction on open orbit} that the integral
\begin{align*}
&\int_{wN_{Q}w^{-1}}\int_{M/M\cap H_{z}}\int_{A/A\cap H_{z}}\int_{K}|\big(\eta,\phi(knma\cdot z)\big)|\,dk \, a^{-\Re\lambda+\Ad^{*}(w)\rho_{P}-2\Ad^{*}(w)\rho_{Q}}\,da\,dm\,dn
\end{align*}
is absolutely convergent.
The product map
$$
(wN_{Q}w^{-1}\cap \overline{N}_{P})\times(wN_{Q}w^{-1}\cap N_{P})\to wN_{Q}w^{-1};
\quad(\overline{n},n)\mapsto \overline{n}n
$$
is a diffeomorphism with Jacobean equal to the constant function $1$. Therefore, we may replace the integral over $wN_{Q}w^{-1}$ by a repeated integral, the first over $wN_{Q}w^{-1}\cap N_{P}$ and the second over $wN_{Q}w^{-1}\cap \overline{N}_{P}$.
For $\phi\in \cD(Z,V_{\sigma})$ we set
\begin{align*}
&\chi_{z}(P:\sigma:\lambda:\eta)(\phi)\\
&\qquad:=\int_{N_{P}\cap wN_{Q}w^{-1}}\int_{M/M\cap H_{z}}\int_{A/A\cap H_{z}}a^{-\lambda+\Ad^{*}(w)\rho_{P}-2\Ad^{*}(w)\rho_{Q}}\\
\nonumber
    &\qquad\qquad\qquad\qquad\qquad\qquad\qquad\qquad\times
        \Big(\sigma^{\vee}(m)\eta,\phi(nma\cdot z)\Big)\,da\,dm\,dn.
\end{align*}
It follows from Fubini's theorem that the integral  $\chi_{z}(P:\sigma:\lambda:\eta)(L_{g}\phi)$ is absolutely convergent for almost every $g\in G$, and the resulting function $I(\phi):g\mapsto \chi_{z}(P:\sigma:\lambda:\eta)(L_{g}\phi)$ is locally integrable on $G$.
We claim that the integral is absolutely convergent for every $g\in G$ and that $I(\phi)$ is smooth. Indeed, in view of \cite[Th\'eor\`eme 3.1]{DixmierMalliavin_FactorisationsDeFunctionsEtDeVecteursIndefinimentDifferentiables} we may write $\phi$ is a finite sum of convolutions $\psi*\chi$ with $\psi\in \cD(G)$ and $\chi\in\cD(Z,V_{\sigma})$.
It follows from the above analysis that the integral
$$
\int_{G}\psi(y)I(\chi)(y^{-1}g)\,dy
$$
is absolutely convergent for every $g\in G$. Moreover, it depends smoothly on $g$ and by Fubini's theorem it is equal to $I(\psi*\chi)(g)$. This proves the claim.
It is now easily seen that $\Ad(w)^{*}\big(\rho_{P}-\Gamma+i(\fa/\fa_{\fh})^{*}\big)\ni\lambda\mapsto \chi_{z}(P:\sigma:\lambda:\eta)$ defines a holomorphic family of distributions in $\cD'(Z,V_{\sigma})$.

We claim that $\chi_{z}(P:\sigma:\lambda:\eta)\in\cD'(Z,P:\sigma:\lambda)$. To prove the claim, we first note that
$$
L^{\vee}(ma)\chi_{z}(P:\sigma:\lambda:\eta)
=a^{\lambda-\rho_{P}}\sigma^{\vee}(m^{-1})\chi_{z}(P:\sigma:\lambda:\eta)
$$
for every $m\in M$ and $a\in A$. To prove the claim it thus suffices to show that
\begin{equation}\label{eq L(N)mu=mu}
L^{\vee}(n)\chi_{z}(P:\sigma:\lambda:\eta)
=\chi_{z}(P:\sigma:\lambda:\eta)
\qquad(n\in N_{P}).
\end{equation}
Let $M_{0}$ be a submanifold of $M$ so that
$$
M_{0}\to M/(M\cap H_{z});
\quad m_{0}\mapsto m_{0}(M\cap H_{z})
$$
is a diffeomorphism onto an open and dense subset of $M/(M\cap H_{z})$ and let $d\mu$ be the pull back of the invariant measure on $M/(M\cap H_{z})$ along this map. Let further $A_{0}$ be a closed subgroup of $A$ so that
$$
A_{0}\to A/(A\cap wH_{z}w^{-1});
\quad a_{0}\mapsto a_{0}(A\cap wH_{z}w^{-1})
$$
is a diffeomorphism. For every $p\in P$ the map
\begin{equation}\label{eq N cap wN_Q w^-1 diffeomorphic to N/N cap H}
N_{P}\cap wN_{Q}w^{-1}\to N_{P}/(N_{P}\cap H_{p\cdot z});
\quad n\mapsto n(N_{P}\cap H_{p\cdot z})
\end{equation}
is a diffeomorphism by Proposition \ref{Prop decomp of N}.
We normalize the $N_{P}$-invariant measure $d_{p\cdot z}\nu$ on $N_{P}/N_{P}\cap H_{p\cdot z}$ so that its pull-back along (\ref{eq N cap wN_Q w^-1 diffeomorphic to N/N cap H}) is the Haar measure on $N_{P}\cap wN_{Q}w^{-1}$.
After changing the order of integration we get for all $\phi\in\cD(G,V_{\sigma})$
\begin{align}
\nonumber&\chi_{z}(P:\sigma:\lambda:\eta)(\phi)\\
\label{eq alternative formula chi}&\quad=\int_{M_{0}}\int_{A_{0}}\int_{N_{P}/(N_{P}\cap H_{ma\cdot z})}a^{-\lambda+\Ad^{*}(w)\rho_{P}-2\Ad^{*}(w)\rho_{Q}}\\
\nonumber &\quad\qquad\qquad\qquad\qquad\qquad\qquad\qquad\times
    \Big(\sigma^{\vee}(m)\eta,\phi(nma\cdot z)\Big)\,d\nu_{ma\cdot z}(n)\,da\,d\mu(m).
\end{align}
The identity (\ref{eq L(N)mu=mu}) follows from the invariance of the measures on the homogeneous spaces $N_{P}/N_{P}\cap H_{p\cdot z}$.
We have thus proven the claim that $\chi_{z}(P:\sigma:\lambda:\eta)\in\cD'(Z,P:\sigma:\lambda)$.

We move on to show that $\chi_{z}(P:\sigma:\lambda:\eta)=\mu_{z}(P:\sigma:\lambda:\eta)$.
By (\ref{eq wSigma^+ cap -Sigma^+=wSigma(Q) cap -Sigma}) and (\ref{eq a(S_2:S_1)mu}) we have for all $\phi\in\cD(G,V_{\sigma})$
\begin{align*}
&\big[\cA(wPw^{-1}:P:\sigma:\lambda)\chi_{z}(P:\sigma:\lambda:\eta)\big](\phi)\\
&\qquad=\int_{\overline{N}_{P}\cap wN_{Q}w^{-1}}\int_{N_{P}\cap wN_{Q}w^{-1}}\int_{M/M\cap H_{z}}\int_{A/A\cap H_{z}}
    a^{-\lambda+\Ad^{*}(w)\rho_{P}-2\Ad^{*}(w)\rho_{Q}}\\
    &\qquad\qquad\qquad\qquad\qquad\qquad\qquad\qquad\qquad\times
    \Big(\sigma^{\vee}(m)\eta,\phi(\overline{n}nma\cdot z)\Big)\,da\,dm\,dn\,d\overline{n}\\
&\qquad=\int_{wN_{Q}w^{-1}}\int_{M/M\cap H_{z}}\int_{A/A\cap H_{z}}
    a^{-\lambda+\Ad^{*}(w)\rho_{P}-2\Ad^{*}(w)\rho_{Q}}\\
    &\qquad\qquad\qquad\qquad\qquad\qquad\qquad\qquad\qquad\times
    \Big(\sigma^{\vee}(m)\eta,\phi(nma\cdot z)\Big)\,da\,dm\,dn.
\end{align*}
The right-hand side is equal to $\mu_{z}(wPw^{-1}:\sigma:\lambda:\eta)(\phi)$, and hence
$$
\cA(wPw^{-1}:P:\sigma:\lambda)\chi_{z}(P:\sigma:\lambda:\eta)
=\cA(wPw^{-1}:P:\sigma:\lambda)\mu_{z}(P:\sigma:\lambda:\eta).
$$
It follows that (\ref{eq mu_cO identity}) holds for $\lambda\in \Ad(w)^{*}\big(\rho_{P}-\Gamma+i(\fa/\fa_{\fh})^{*}\big)$ for which the intertwining operator $\cA(wPw^{-1}:P:\sigma:\lambda)$ is an isomorphism. In view of Lemma \ref{Lemma a(wPw^-1:P:sigma:lambda) is isomorphism} this is the case for $\lambda$ outside of a locally finite union of hyperplanes. Since $\chi_{z}(P:\sigma:\lambda:\eta)$ depends holomorphically and $\mu_{z}(P:\sigma:\lambda:\eta)$ meromorphically on $\lambda$, the identity (\ref{eq mu_cO identity}) holds in fact for all $\lambda\in\Ad(w)^{*}\big(\rho_{P}-\Gamma+i(\fa/\fa_{\fh})^{*}\big)$.

We move on to prove (\ref{Prop construction on max rank orbits - item 4}).
Assume that $\eta\neq 0$. From (\ref{eq mu_cO identity}) it follows that $\supp \mu_{z}(P:\sigma:\lambda:\eta)\subseteq \overline{\cO}$. Since the support of $\mu_{z}(P:\sigma:\lambda:\eta)$ is a union of $P$-orbits in $Z$, it suffices to prove that the restriction of $\mu_{z}(P:\sigma:\lambda:\eta)$ to the open subset $Z\setminus\partial\cO$ is non-zero. The right-hand side of  (\ref{eq mu_cO identity}) defines for every $\lambda\in\Ad(w)^{*}\big(\rho_{P}-\Gamma+i(\fa/\fa_{\fh})^{*}\big)$ a non-zero distribution on $Z\setminus\partial\cO$. Moreover, the dependence on $\lambda$ is holomorphic, and hence the right-hand side of  (\ref{eq mu_cO identity}) defines a holomorphic family of distributions on $Z\setminus\partial\cO$ with family parameter $\lambda\in\Ad(w)^{*}\big(\rho_{P}-\Gamma+i(\fa/\fa_{\fh})^{*}\big)$. As this family coincides on a non-empty open subset of $\Ad(w)^{*}\big(\rho_{P}-\Gamma+i(\fa/\fa_{\fh})^{*}\big)$ with the meromorphic family
$$
\lambda\mapsto \mu_{z}(P:\sigma:\lambda:\eta)\big|_{Z\setminus\partial\cO},
$$
it follows that $\mu_{z}(P:\sigma:\lambda:\eta)\big|_{Z\setminus\partial\cO}\neq 0$ for all $\lambda$ for which $\mu_{z}(P:\sigma:\lambda:\eta)$ is defined. This proves (\ref{Prop construction on max rank orbits - item 4}).

Finally we prove (\ref{Prop construction on max rank orbits - item 1}). Let $w'\in N_{G}(\fa)$ satisfy (\ref{eq wSigma^+ cap -Sigma^+=wSigma(Q) cap -Sigma}) and $w\cdot(P\bs Z)_{\open}=w'\cdot(P\bs Z)_{\open}$.
Because of meromorphical continuation, it suffices to prove the uniqueness for $\lambda$ in the open subset $\Ad^{*}(w)\big(\rho_{P}-\Gamma+i(\fa/\fa_{\fh})^{*}\big)$ of $\Ad^{*}(w)\big(\rho_{P}+(\fa/\fa_{\fh})_{\C}^{*}\big)$. For these $\lambda$ the distribution $\mu_{z}(P:\sigma:\lambda:\eta)$ is given by the right-hand side of (\ref{eq alternative formula chi}).
It follows from Proposition \ref{Prop decomp of N} that there exists a $\gamma:M\times A\to\R_{>0}$ so that for every $\psi\in \cD(Z)$, $m\in M$ and $a\in A$
$$
\int_{N_{P}/(N_{P}\cap H_{ma\cdot z})}\psi(nma\cdot z)\,d\nu_{ma\cdot z}(n)
=\gamma(m,a)\int_{N_{P}\cap w'N_{Q}{w'}^{-1}}\psi(n'ma\cdot z)\,dn'.
$$
The function
$$
M\times A\to\R_{>0};\quad (m,a)\mapsto \frac{\gamma(m,a)}{\gamma(e,e)}
$$
is a character of $M\times A$. Therefore, there exists a $c>0$ and a $\nu\in \fa^{*}$ so that
$$
\gamma(m,a)=ca^{\nu}
\qquad(m\in M, a\in A).
$$
It follows that,
\begin{align*}
&\mu_{z}(P:\sigma:\lambda:\eta)(\phi)\\
&\qquad=c\int_{N_{P}\cap w'N_{Q}{w'}^{-1}}\int_{M/M\cap H_{z}}\int_{A/A\cap H_{z}}
   a^{-\lambda+\Ad^{*}(w)\rho_{P}-2\Ad^{*}(w)\rho_{Q}+\nu}\\
    &\qquad\qquad\qquad\qquad\qquad\qquad\qquad\qquad\qquad\times
        \Big(\sigma^{\vee}(m)\eta,\phi(nma\cdot z)\Big)\,da\,dm\,dn.
\end{align*}
Since $w'$ satisfies (\ref{eq wSigma^+ cap -Sigma^+=wSigma(Q) cap -Sigma}), we have $w'N_{P}w'^{-1}\cap \overline{N}_{P}=w'N_{Q}w'^{-1}\cap \overline{N}_{P}$.
Now for every $\phi\in \cD'(G,V_{\sigma})$
\begin{align*}
&\big[\cA(w'Pw'^{-1}:P:\sigma:\lambda)\mu_{z}(P:\sigma:\lambda:\eta)\big](\phi)\\
\quad&=c\int_{w'N_{Q}w'^{-1}\cap \overline{N}_{P}}\int_{N_{P}\cap w'N_{Q}{w'}^{-1}}\int_{M/M\cap H_{z}}\int_{A/A\cap H_{z}}
   a^{-\lambda+\Ad^{*}(w)\rho_{P}-2\Ad^{*}(w)\rho_{Q}+\nu}\\
\nonumber
    &\qquad\qquad\qquad\qquad\qquad\qquad\qquad\qquad\qquad\times
        \Big(\sigma^{\vee}(m)\eta,\phi(nma\cdot z)\Big)\,da\,dm\,dn\\
\quad&=c\int_{w'N_{Q}w'^{-1}}\int_{M/M\cap H_{z}}\int_{A/A\cap H_{z}}
   a^{-\lambda+\Ad^{*}(w)\rho_{P}-2\Ad^{*}(w)\rho_{Q}+\nu}\\
\nonumber
    &\qquad\qquad\qquad\qquad\qquad\qquad\qquad\qquad\qquad\times
        \Big(\sigma^{\vee}(m)\eta,\phi(nma\cdot z)\Big)\,da\,dm\,dn.
\end{align*}
Since $\cA(w'Pw'^{-1}:P:\sigma:\lambda)\mu_{z}(P:\sigma:\lambda:\eta)$ is a distribution in $\cD'(w'Pw'^{-1}:\sigma:\lambda)$, $\nu$ must satisfy
$$
-\lambda+\Ad^{*}(w)\rho_{P}-2\Ad^{*}(w)\rho_{Q}+\nu
=-\lambda+\Ad^{*}(w')\rho_{P}-2\Ad^{*}(w')\rho_{Q}.
$$
Thus, in view of (\ref{eq Formula mu_z}), we have
$$
\cA(w'Pw'^{-1}:P:\sigma:\lambda)\mu_{z}(P:\sigma:\lambda:\eta)
=c\mu_{z}(w'Pw'^{-1}:\sigma:\lambda:\eta).
$$
This concludes the proof of (\ref{Prop construction on max rank orbits - item 1}).
\end{proof}

Let $\cO\in P\bs Z$ be of maximal rank and let $z\in \cO$ be weakly adapted. Assume that $\fa_{\cO}=\fa_{\fh}$.  Let $\xi\in\widehat{M}_{Q}$. We recall from Corollary \ref{Cor irreducible restriction to M} that if $({V^{\infty}}')^{H_{z}}\neq \{0\}$, then $\xi|_{L_{Q,\nc}}$ is trivial, $\xi$ is finite dimensional and unitarizable and $\xi|_{M}$ is irreducible.
For $\xi\in\widehat{M}_{Q}$ and $\eta\in (V_{\xi}^{*})^{M_{Q}\cap H_{z}}$ we define the meromorphic family of distributions
$$
\mu_{z}(Q:\xi:\lambda:\eta)
:=\mu_{z}(P:\xi|_{M}:\lambda+\rho_{P}-\rho_{Q}:\eta)
$$
with family parameter $\lambda\in(\fa/\fa_{\fh})_{\C}^{*}$.

\begin{Thm}\label{Thm Construction on max rank orbits for Q}
For every $z$, $\xi$, $\eta$ as above, the assignment
$$
(\fa/\fa_{\fh})_{\C}^{*}\ni\lambda\to\mu_{z}(Q:\xi:\lambda:\eta)
$$
defines a meromorphic family of distributions in $\cD'(Z,Q:\xi:\lambda)$. The poles of the family lie on a locally finite union of complex affine hyperplanes of the form
$$
\{\lambda\in(\fa/\fa_{\fh})^{*}_{\C}: \lambda(Y)=a\}\quad\text{for some }Y\in\fa\setminus \fa_{\fh}\text{ and }a\in\R.
$$
Let $\cO\in P\bs Z$ satisfy $\fa_{\cO}=\fa_{\fh}$ and let $z\in \cO$ be weakly adapted. Let $w\in N_{G}(\fa)$ be so that (\ref{eq wSigma^+ cap -Sigma^+=wSigma(Q) cap -Sigma}) holds and $[\cO]=w\cdot(P\bs Z)_{\open}$.
If
$$
\lambda\in
\rho_{wQw^{-1}}-\Ad^{*}(w)\Gamma+i(\fa/\fa_{\fh})^{*}
=\Ad(w)^{*}\big(\rho_{Q}-\Gamma+i(\fa/\fa_{\fh})^{*}\big),
$$
then the distribution $\mu_{z}(Q:\xi:\lambda:\eta)$ is for $\phi\in\cD(G,V_{\xi})$ given by the absolutely convergent integral
\begin{align}
\label{eq mu_cO(Q) identity}
&\mu_{z}(Q:\xi:\lambda:\eta)(\phi)\\
\nonumber
&\qquad=\int_{N_{Q}\cap wN_{Q}w^{-1}}\int_{M/M\cap H_{z}}\int_{A/A\cap H_{z}}a^{-\lambda-\Ad^{*}(w)\rho_{Q}}
        \Big(\xi^{\vee}(m)\eta,\phi(nma\cdot z)\Big)\,da\,dm\,dn.
\end{align}
\end{Thm}

\begin{proof}
We first prove that $\mu_{z}(Q:\xi:\lambda:\eta)\in \cD'(Z,Q:\xi:\lambda)$.
By Proposition \ref{Prop construction on max rank orbits} we have $\mu_{z}(Q:\xi:\lambda:\eta)\in\cD'(P:\xi|_{M}:\lambda+\rho_{P}-\rho_{Q})$.
In view of Lemma \ref{Lemma characterization of D'(Q:sigma:lambda)} and meromorphicity it suffices to show that $\mu_{z}(Q:\xi:\lambda:\eta)$ is left-$M_{Q}\cap \overline{N}_{P}$-invariant for $\lambda\in\rho_{wQw^{-1}}-\Ad^{*}(w)\Gamma+i(\fa/\fa_{\fh})^{*}$.

The fact that $\fa_{\cO}=\fa_{\fh}$ implies $w\in \cN=N_{G}(\fa)\cap N_{G}(\fa_{\fh})$. In view of Remark \ref{Rem L_Q roots} the element $w$ normalizes $L_{Q,\nc}$ and hence also $M_{Q}$. Since (\ref{eq wSigma^+ cap -Sigma^+=wSigma(Q) cap -Sigma}) is assumed to hold, $w$ even normalizes $M_{Q}\cap \overline{N}_{P}$.
The point $w^{-1}\cdot z$ is adapted. Therefore,
\begin{equation}\label{eq L_Q,nc subseteq H_z}
L_{Q,\nc}
=wL_{Q,\nc}w^{-1}
\subseteq wH_{w^{-1}z}w^{-1}
=H_{z}.
\end{equation}
In particular,
$$
M_{Q}\cap \overline{N}_{P}\subseteq H_{z}.
$$

We claim that $M_{Q}\cap \overline{N}_{P}$ normalizes $N_{P}\cap wN_{Q}w^{-1}$. In fact, $M_{Q}$ has this property. As $N_{P}=(N_{P}\cap M_{Q}) N_{Q}$, we have
\begin{equation}\label{eq N_P cap wN_Q w^-1=N_Q cap wN_Qw^-1}
N_{P}\cap wN_{Q}w^{-1}=N_{Q}\cap wN_{Q}w^{-1}.
\end{equation}
The claim is now proven by observing that  $M_{Q}$ normalizes $N_{Q}$ and hence also $wN_{Q}w^{-1}$.

As $M$ and $A$ normalize $M_{Q}\cap \overline{N}_{P}$, $M_{Q}\cap\overline{N}_{P}$ normalizes $N_{P}\cap wN_{Q}w^{-1}$, it is follows from (\ref{eq L_Q,nc subseteq H_z}) and (\ref{eq mu_cO identity}) that $\mu_{z}(Q:\xi:\lambda:\eta)$ is left-$M_{Q}\cap \overline{N}_{P}$-invariant.
This concludes the proof that $\mu_{z}(Q:\xi:\lambda:\eta)\in \cD'(Z,Q:\xi:\lambda)$.

The remaining assertions follow directly from Proposition \ref{Prop construction on max rank orbits} and (\ref{eq N_P cap wN_Q w^-1=N_Q cap wN_Qw^-1}).
\end{proof}

\begin{Rem}
Let $\cO\in P\bs Z$ be of maximal rank and  $z\in\cO$ weakly adapted. If $\fa_{\cO}=\fa_{\fh}$, then there exists a positive $H_{z}$-invariant Radon measure on $(H_{z}\cap Q)\bs H_{z}$. To prove this it suffices to show that $H_{z}\cap Q$ is unimodular. The modular character of $H_{z}\cap Q$ is given by
$$
\Delta:H_{z}\cap Q\to\R_{>0};
\quad h \mapsto \big|\det\big(\Ad(h)\big|_{\fh_{z}\cap\fq}\big)\big|
$$
Let $w\in \cN$ be so that (\ref{eq wSigma^+ cap -Sigma^+=wSigma(Q) cap -Sigma}) and $[\cO]=w\cdot(P\bs Z)_{\open}$.
In view of (\ref{eq L_Q,nc subseteq H_z}), Proposition \ref{Prop decomp of N} and Theorem \ref{Thm structure theorem for wPw^(-1) cdot z} we have
\begin{align*}
H_{z}\cap Q
&=L_{Q,\nc}(H_{z}\cap P)
=L_{Q,\nc}(M\cap H_{z})(A\cap H)(N_{P}\cap H_{z})\\
&=L_{Q,\nc}(M\cap H_{z})(A\cap H)(N_{Q}\cap H_{z}).
\end{align*}
Since $L_{Q,\nc}$ is semisimple, $M\cap H_{z}$ is compact and $N_{Q}\cap H_{z}$ is unipotent, the restriction of $\Delta$ to each of these three subgroups is trivial. It thus remains to show that the restriction to $A\cap H$ is trivial as well.

As $A\cap H$ is contained in the center of $L_{Q}$, this groups centralizes $L_{Q,\nc}(M\cap H_{z})(A\cap H)$. Therefore,
$$
\Delta(a)
=\big|\det\big(\Ad(a)\big|_{\fn_{Q}\cap\fh_{z}}\big)\big|
\qquad(a\in A\cap H_{z}).
$$
It follows from Proposition \ref{Prop decomp of N} and (\ref{eq N_P cap wN_Q w^-1=N_Q cap wN_Qw^-1}) that the multiplication map
$$
(N_{Q}\cap wN_{Q}w^{-1})\times(N_{Q}\cap H_{z})\to N_{Q};
\quad(n,n_{H})\mapsto nn_{H}
$$
is a diffeomorphism. Moreover, this map is $A\cap H_{z}$-equivariant, and hence
$$
\Delta(a)
=\Big|\frac{\det\big(\Ad(a)\big|_{\fn_{Q}}\big)}{\det\big(\Ad(a)\big|_{\fn_{Q}\cap\Ad(w)\fn_{Q}}\big)}\Big|
=\frac{a^{2\rho_{Q}}}{a^{\rho_{Q}+\Ad^{*}(w)\rho_{Q}}}
\qquad(a\in A\cap H_{z}).
$$
As $A\cap H_{z}$ is connected, $w$ centralizes this subgroup. It follows that $\Delta$ is the trivial character and hence that $H_{z}\cap Q$ is unimodular. This concludes the proof of the claim that $(H_{z}\cap Q)\bs H_{z}$ admits a positive $H_{z}$-invariant Radon measure.

The invariant measure allows to describe the distributions $\mu_{z}(Q:\xi:\lambda:\eta)$ as a functional on $C^{\infty}(Q:\xi:\lambda)$. To do so, let $g\in G$ be so that $gH=z$. We use (\ref{eq Identification D'(Z) simeq D'(G)^H}) to identify $\cD'(Z,Q:\xi:\lambda)$ with $\cD'(Q:\xi:\lambda)^{H}$. Recall the map $\omega_{\xi,\lambda}^{Q}$ from (\ref{eq Def omega^S_(xi,lambda)}). A straightforward computation shows that for a suitable normalization of the invariant measure on $(H\cap g^{-1}Qg)\bs H$
$$
\Big(\omega_{\xi,\lambda}^{Q}\mu_{z}(Q:\xi:\lambda:\eta)\Big)(f)
=\int_{(H\cap g^{-1}Qg)\bs H}\big( \eta, f(gh)\big)\,dh
\qquad\big(f\in C^{\infty}(Q:\xi:\lambda)\big).
$$

The distributions $\mu_{z}(P:\sigma:\lambda:\eta)$ from Proposition \ref{Prop construction on max rank orbits} with $z$ a weakly adapted point contained in a $P$-orbit $\cO$ in $Z$ with $\fa_{\cO}\neq \fa_{\fh}$, can be similarly description as functionals on $C^{\infty}(P:\sigma:\lambda)$. However, in this generality not all homogeneous spaces $(H_{z}\cap P)\bs H_{z}$ admit positive $H_{z}$-invariant Radon measures. To remedy this, one has to consider  the elements in $C^{\infty}(P:\sigma:\lambda)$ as smooth densities; see \cite[Lemma 3.1]{vdBanKuit_NormalizationsOfEisensteinIntegrals}.
\end{Rem}

\subsection{A description of $\cD'(Q:\xi:\lambda)^{H}$}\label{Subsection Construction - Description}
In this section we give a precise description of $\cD'(Q:\xi:\lambda)^{H}$ for a finite dimensional unitary representation $\xi$ of $M_{Q}$ and $\lambda\in i(\fa/\fa_{\fh})^{*}$ outside of a union of finitely many proper subspaces. We identify distributions on $Z$ with $H$-invariant distributions on $G$ via the map (\ref{eq Identification D'(Z) simeq D'(G)^H}). We recall that the point $eH\in Z=G/H$ is chosen to be admissible.

We define
$$
(P\bs Z)_{\fa_{\fh}}
:=\{\cO\in(P\bs Z)_{\max}:\fa_{\cO}=\fa_{\fh}\}.
$$
First we choose a set of good representatives of the $P$-orbits in $(P\bs Z)_{\fa_{\fh}}$.

By \cite[Proposition 3.13]{KuitSayag_OnTheLittleWeylGroupOfARealSphericalSpace} every open $P$-orbit contains a point $x\cdot z$ with $x\in G\cap \exp(i\fa)H_{\C}$. For such a point the equality
$$
\fa\cap\Ad(x)\fh^{\perp}=\fa\cap \fh^{\perp}
$$
holds.
For every $\cO\in (P\bs Z)_{\open}$ we choose an $x_{\cO}\in G\cap \exp(i\fa)H_{\C}$ so that $x_{\cO}H$ is an adapted point in $\cO\subseteq Z=G/H$.
We may and will choose $x_{PH}$ to be $e$.

We recall the group $\cN=N_{G}(\fa)\cap N_{G}(\fa_{\fh})$ and its subgroup $\cW$ from (\ref{eq def cN}) and (\ref{eq def cW}).
For every $\cN/\cW$ we choose a  representative $v_{w}\in K\cap \cN$ as follows. In view of Theorem \ref{Thm properties of W-action} the set $\cN/\cW$ is in bijection with the set of equivalence classes of $P$-orbits in $(P\bs Z)_{\fa_{\fh}}$, i.e., the map
$$
\cN/\cW\to (P\bs Z)_{\fa_{\fh}}/_{\sim};
\quad v\cW\mapsto v\cdot (P\bs Z)_{\open}
$$
is a bijection.
We choose an order-regular element $X\in \fa^{-}$.
For $w\in \cN/\cW$ we now choose $v_{w}\in K\cap \cN$ so that
$$
\fh_{z,X}=\Ad(v_{w})\fh_{\emptyset}
$$
for some weakly adapted point in an $P$-orbit $\cO$ with $[\cO]=w\cdot (P\bs Z)_{\open}$. We note that this equation determines $v_{w}$ up to right-multiplication by an element from $\cZ\cap K$ and that $v_{w}\cZ$ is independent of the choice of $z$.
The elements $v_{w}$ do however depend on the choice of $X$. The crucial property of the $v_{w}$ is that they are representatives of the elements in $\cN/\cW$, i.e.,
$$
v_{w}\cW=w
\qquad (w\in \cN/\cW).
$$
We may and will choose the $v_{w}$ so that  they satisfy (\ref{eq wSigma^+ cap -Sigma^+=wSigma(Q) cap -Sigma}). The representative of $e\cW$ we choose to be $e$.

By Theorem \ref{Thm properties of W-action} (\ref{Thm properties of W-action - item 3}) and (\ref{Thm properties of W-action - item 4}) the points
$$
v_{w}x_{\cO'}H
\qquad\big(w\in \cN/\cW, \cO'\in(P\bs Z)_{\open}\big)
$$
form a set of weakly admissible representatives for the $P$-orbits in $(P\bs Z)_{\fa_{\fh}}$.
If $\cO\in (P\bs Z)_{\fa_{\fh}}$, then we write $x_{\cO}$ for $v_{w}x_{\cO'}$, where $w\in \cN/\cW$ and $\cO'\in(P\bs Z)_{\open}$ are so that $Pv_{w}x_{\cO'}H$.

For a finite dimensional unitary representation $(\xi, V_{\xi})$ of $M_{Q}$ we define the vector space
$$
V^{*}(\xi)
:=\bigoplus_{\cO\in (P\bs Z)_{\fa_{\fh}}}(V_{\xi}^{*})^{M_{Q}\cap x_{\cO}Hx_{\cO}^{-1}}
$$
and equip it with the inner product induced from the inner product on $V_{\xi}$.

\begin{Prop}\label{Prop M cap xHx=M cap H}
Let $\cO,\cO'\in (P\bs Z)_{\fa_{\fh}}$. If $\cO\sim\cO'$, then
$$
M_{Q}\cap x_{\cO}Hx_{\cO}^{-1}
=M_{Q}\cap x_{\cO'}Hx_{\cO'}^{-1}.
$$
In particular
$$
M_{Q}\cap x_{\cO}Hx_{\cO}^{-1}
=M_{Q}\cap H
\qquad\big(\cO\in (P\bs Z)_{\open}\big).
$$
\end{Prop}

\begin{proof}
Let $\cO\in (P\bs Z)_{\fa_{\fh}}$. Let $w\in \cN/\cW$ be so that $[\cO]=w\cdot(P\bs Z)_{\open}$. Then $x_{\cO}=v_{w}x_{\cO_{0}}$ for some open $P$-orbit $\cO_{0}$.
It follows from Remark \ref{Rem L_Q roots} that $\cN$ normalizes $M_{Q}$.
Therefore,
$$
M_{Q}\cap x_{\cO}Hx_{\cO}^{-1}
=v_{w}\big(M_{Q}\cap x_{\cO_{0}}Hx_{\cO_{0}}^{-1}\big)v_{w}^{-1}.
$$
Since $v_{w}$ only depends on the equivalence class $[\cO]$, and not on the particular orbit $\cO$ in it, it thus suffices to prove the assertion only for the open $P$-orbits $\cO$.

We have  $M_{Q}=ML_{Q,\nc}$  and $L_{Q,\nc}\subseteq x_{\cO}Hx_{\cO}^{-1}$. Hence it suffices to prove
$$
M\cap x_{\cO}Hx_{\cO}^{-1}
=M\cap H
\qquad\big(\cO\in (P\bs Z)_{\open}\big).
$$
Note that $M\cap x_{\cO}Hx_{\cO}^{-1}= M\cap x_{\cO}H_{\C}x_{\cO}^{-1}$. Let $t\in \exp(i\fa)$ and $h\in H_{\C}$ be so that $x_{\cO}=th$. Then
$$
M\cap x_{\cO}H_{\C}x_{\cO}^{-1}
=M\cap tH_{\C}t^{-1}
=M\cap H_{\C}.
$$
For the last equality we used that $t$ centralizes $M$. The assertion now follows as $M\cap H_{\C}=M\cap H$.
\end{proof}

As a corollary of the previous proposition we obtain that the group
$$
M_{Q,[\cO]}:=M_{Q}\cap x_{\cO}Hx_{\cO}^{-1}
$$
only depends on the equivalence class $[\cO]\in(P\bs Z)_{\fa_{\fh}}/\sim$, not on the choice of the $P$-orbit in $[\cO]$. Therefore,
$$
V^{*}(\xi)
=\bigoplus_{[\cO]\in(P\bs Z)_{\fa_{\fh}}/\sim}\Big((V_{\xi}^{*})^{M_{Q,[\cO]}}\Big)^{[\cO]}
$$
Here $V^{S}$ for a vector space $V$ and a finite set $S$ denotes the vector space of functions $S\to V$. We write $v_{s}$ for the $s$-component of a vector $v\in V^{S}$, i.e., $v_{s}=v(s)$.

\begin{Rem}
The space $V^{*}(\xi)$ will serve as the multiplicity space in the Plancherel decomposition for the principal series representations $\Ind_{\overline{Q}}^{G}(\xi\otimes\lambda\otimes 1)$ with $\lambda\in i(\fa/\fa_{\fh})^{*}$. In case $H$ is symmetric, i.e., $H$ is an open subgroup of the fixed point subgroup $G^{\sigma}$ of some involution $\sigma$ of $G$, much information about these multiplicity spaces has been given in \cite{vdBanSchlichtkrull_MultiplicitiesInPlancherelDecompositionForSemisimpleSymmetricSpace}.
If $H$ is equal to the full fixed point subgroup of an algebraic involution on $G$, then Proposition \ref{Prop M cap xHx=M cap H} coincides with \cite[Lemma 7]{vdBanSchlichtkrull_MultiplicitiesInPlancherelDecompositionForSemisimpleSymmetricSpace}.
For symmetric spaces we have $(P\bs Z)_{\fa_{\fh}}=(P\bs Z)_{\open}$, see Appendix A.
\end{Rem}

\begin{Thm}\label{Thm Description D'(Z,Q:xi:lambda)}
There exists a finite union $\cS$ of proper subspaces of $(\fa/\fa_{\fh})^{*}$ so that for every finite dimensional unitary representation  $(\xi,V_{\xi})$ of $M_{Q}$ and every $\lambda\in(\fa/\fa_{\fh})_{\C}^{*}$ with $\Im\lambda\notin\cS$ the map
$$
\mu(Q:\xi:\lambda):V^{*}(\xi)\to\cD'(Q:\xi:\lambda)^{H};\quad
    \eta\mapsto \sum_{\cO\in(P\bs Z)_{\fa_{\fh}}}
    \mu_{x_{\cO}H}(Q:\xi:\lambda:\eta_{\cO})
$$
is a linear isomorphism.
\end{Thm}

\begin{proof}
Let
$$
\cS_{1}
=(\fa/\fa_{\fh})^{*}\cap\bigcup_{\substack{\cO\in P\bs Z\\\fa_{\cO}\neq\fa_{\fh}}}(\fa/\fa_{\cO})^{*}
$$
If $\lambda\in (\fa/\fa_{\fh})_{\C}^{*}$ satisfies $\Im\lambda\notin \cS_{1}$, then we have in view of Theorem \ref{Thm condition on orbits for given generic lambda}
$$
(P\bs Z)_{\mu}
\subseteq\big\{\cO\in(P\bs Z)_{\max}:\fa_{\cO}=\fa_{\fh}\big\}
$$
for every $\mu\in \cD'(Z,Q:\xi:\lambda)$. Further, it follows from Theorem \ref{Thm bound on transversal degree} that there exists a finite union $\cS_{2}$ of proper subspaces of $(\fa/\fa_{\fh})^{*}$ so that
$$
\trdeg_{\cO}(\mu)=0
\qquad\big(\cO\in(P\bs Z)_{\mu}\big)
$$
for all $\lambda\in(\fa/\fa_{\fh})_{\C}^{*}$ with $\Im \lambda\notin \cS_{1}\cup\cS_{2}$ and all $\mu\in \cD'(Z,Q:\xi:\lambda)$.

In view of Theorem \ref{Thm Construction on max rank orbits for Q} there exist a finite union $\cS_{3}$ of hyperplanes in $(\fa/\fa_{\fh})^{*}$ so that the poles of the meromorphic family of maps $\lambda\mapsto \mu(Q:\xi:\lambda)$ lie in $\cS_{3}$. We set
$$
\cS
:=\cS_{1}\cup\cS_{2}\cup\cS_{3}
$$
We fix $\lambda\in(\fa/\fa_{\fh})_{\C}^{*}$ with $\Im\lambda\notin\cS$. Let now $\mu\in \cD'(Z,Q:\xi:\lambda)$. To prove the theorem, it suffices to show that $\mu$ is a sum of distributions $\mu_{x_{\cO}H}(Q:\xi:\lambda:\eta_{\cO})$ with $\cO\in (P\bs Z)_{\fa_{\fh}}$ and $\eta_{\cO}\in (V_{\xi}^{*})^{M\cap x_{\cO}Hx_{\cO}^{-1}}$.

The condition on $\lambda$ assures that
$$
(P\bs Z)_{\mu}
\subseteq(P\bs Z)_{\fa_{\fh}}
$$
and
$$
\trdeg_{\cO}(\mu)
=0
\qquad\big(\cO\in (P\bs Z)_{\mu}\big).
$$
If $\mu\neq 0$, then there exists an $\cO\in (P\bs Z)_{\mu}$. Since $\trdeg_{\cO}(\mu)=0$, there exists a distribution $\mu_{\cO}$ on $\cO$ so that $\mu$ on
$$
U
=Z\setminus\Big(\partial \cO\cup \bigcup_{\cO'\in (P\bs Z)_{\mu}\setminus\{\cO\}}\overline{\cO'}\Big)
$$
is given by
$$
\mu(\phi)
=\mu_{\cO}(\phi\big|_{\cO})
\qquad\big(\phi\in\cD(U,V_{\xi})\big).
$$
By \cite[Lemma 5.5]{KrotzKuitOpdamSchlichtkrull_InfinitesimalCharactersOfDiscreteSeriesForRealSphericalSpaces} $\mu_{\cO}$ is in fact given by integrating against a smooth function. Moreover, since $\mu$ is right-$H$-invariant, also $\mu_{\cO}$ is right-$H$-invariant. Likewise, $\mu_{\cO}$ inherits the left-$P$-equivariance from $\mu$. As $\cO$ is a $P$-orbit in $Z$, $\mu_{\cO}$ is fully determined by its value in any given point. In particular, we may evaluate $\mu_{\cO}$ in $z_{\cO}:=x_{\cO}H$. This results in a non-zero vector $\eta_{\cO}$ in $(V_{\xi}^{*})^{M_{Q}\cap H_{z_{\cO}}}=(V_{\xi}^{*})^{M_{Q}\cap x_{\cO}Hx_{\cO}^{-1}}$. Let $w\in N_{G}(\fa)$ satisfy (\ref{eq wSigma^+ cap -Sigma^+=wSigma(Q) cap -Sigma}) and $[\cO]=w\cdot(P\bs Z)_{\open}$. It follows from Theorem \ref{Thm structure theorem for wPw^(-1) cdot z} that $\mu(\phi)$ is for every $\phi\in\cD(U,V_{\xi})$ given by
\begin{align*}
\label{eq mu_cO identity}
&\int_{N_{P}\cap wN_{Q}w^{-1}}\int_{M/M\cap H_{z}}\int_{A/A\cap H_{z}}a^{-\lambda+\Ad^{*}(w)\rho_{Q}}         \Big(\xi^{\vee}(m)\eta_{\cO},\phi(nma\cdot z)\Big)\,da\,dm\,dn\\
&\quad=\mu_{z_{\cO}}(Q:\xi:\lambda:\eta_{\cO})(\phi).
\end{align*}
Hence $\mu':=\mu-\mu_{z_{\cO}}(Q:\xi:\lambda:\eta_{\cO})$ is a distribution in $\cD'(Z,Q:\xi:\lambda)$ with
$$
(P\bs Z)_{\mu'}
\subseteq \Big((P\bs Z)_{\mu}\setminus\{\cO\}\Big)\cup \{\cO'\in (P\bs Z)_{\fa_{\fh}}: \cO'\subseteq \partial\cO\}
$$
and
$$
\trdeg_{\cO'}\big(\mu'\big)=0
\qquad\big(\cO'\in(P\bs Z)_{\mu'}\big).
$$
We now replace $\mu$ by $\mu'$ and repeat this argument. After finitely many iterations of this process we obtain that $\mu$ is a sum of distributions  $\mu_{x_{\cO}H}(Q:\xi:\lambda:\eta_{\cO})$ with $\cO\in (P\bs Z)_{\fa_{\fh}}$ and $\eta_{\cO}\in (V_{\xi}^{*})^{M\cap x_{\cO}Hx_{\cO}^{-1}}$.
\end{proof}

From now on we fix  a finite union $\cS$ of proper subspaces of $(\fa/\fa_{\fh})^{*}$, so that the conclusions of Theorem \ref{Thm Description D'(Z,Q:xi:lambda)} hold and all intertwining operators $\cI_{v}(Q:\xi:\lambda)$ and $\cI_{v}^{\circ}(Q:\xi:\lambda)$ with $v\in \cN$, $\xi\in \widehat{M}_{Q,\mathrm{fu}}$ and $\Im\lambda\notin \cS$ are isomorphisms.

\subsection{Action of $A_{E}$ on $\cD'(Q:\xi:\lambda)^{H}$}
\label{Subsection Construction - Action of N_A(H)}
We recall the from (\ref{eq def a_E}) that $\fa_{E}$ denotes the edge of $\overline{\cC}$, i.e.,
$$
\fa_{E}
=\overline{\cC}\cap -\overline{\cC}.
$$
We write $A_{E}$ for the connected subgroup of $G$ with Lie algebra $\fa_{E}$, i.e.,
$$
A_{E}
:=\exp(\fa_{E}).
$$
The group $A_{E}$ normalizes the Lie algebra $\fh$. In fact, it follows from the theory of smooth compactifications of $Z$ that $A_{E}$ normalizes $H_{\C}$, and hence also $H$. See \cite[Theorem 4.1]{DelormeKnopKrotzSchlichtkrull_PlancherelTheoryForRealSphericalSpacesConstructionOfTheBernsteinMorphisms} where this is shown for an algebraic subgroup of $\underline{G}$ for which the identity component of the group of real points is equal to $A_{E}$.
Therefore, $A_{E}$ acts from the right on $Z=G/H$ by
$$
gH\cdot a:=gaH
\qquad \big(g\in G, a\in A_{E}\big).
$$
We now investigate the induced right action of $A_{E}$ on the spaces $\cD'(Q:\xi:\lambda)^{H}$.

If $\cO\in (P\bs Z)_{\fa_{\fh}}$, then we define
\begin{equation}\label{eq Def iota_O}
\iota_{\cO}:(V_{\xi}^{*})^{M_{Q,[\cO]}}\hookrightarrow V^{*}(\xi)
\end{equation}
to be the inclusion map determined by
$$
\big(\iota_{\cO}\eta\big)_{\cO'}
=\left\{
   \begin{array}{ll}
     \eta & (\cO=\cO'), \\
     0 & (\cO\neq \cO').
   \end{array}
 \right.
$$

\begin{Prop}\label{Prop Action N_A(H)}
Let $(\xi,V_{\xi})$ be a finite dimensional unitary representation of $M_{Q}$ and $\lambda\in(\fa/\fa_{\fh})_{\C}^{*}$ with $\Im\lambda\notin\cS$.
If $w\in \cN/\cW$ and $\cO\in w\cdot(P\bs Z)_{\open}$, then
$$
R^{\vee}(a)\circ\mu(Q:\xi:\lambda)\circ\iota_{\cO}
=a^{-\Ad^{*}(v_{w}^{-1})\lambda+\rho_{Q}}\mu(Q:\xi:\lambda)\circ\iota_{\cO}
\qquad\big(a\in A_{E}\big).
$$
\end{Prop}

We first prove a lemma.

\begin{Lemma}\label{Lemma x_O a=waw^(-1)x_O}
Let $w\in\cN/\cW$ and $\cO\in w\cdot (P\bs Z)_{\open}$.
Then
$$
x_{\cO}a
\in v_{w}av_{w}^{-1}x_{\cO}H
\qquad \big(a\in A_{E}\big).
$$
\end{Lemma}

\begin{proof}
By \cite[Lemma 12.1]{KuitSayag_OnTheLittleWeylGroupOfARealSphericalSpace} the little Weyl group $W_{Z}$, and hence also $\cW$,  acts trivially on $\fa_{E}/\fa_{\fh}$. Let $\cO'\in (P\bs Z)_{\open}$ be so that $x_{\cO}=v_{w}x_{\cO'}$. Further let $t\in \exp(i\fa)$ and $h\in H_{\C}$ be so that $x_{\cO'}=th$.  Now
$$
x_{\cO}a
=\big(v_{w}av_{w}^{-1}\big)v_{w}t\big(a^{-1}ha\big)
=\big(v_{w}av_{w}^{-1}\big)x_{\cO}h^{-1}\big(a^{-1}ha\big)
\qquad(a\in A_{E}).
$$
The assertion now follows as
$$
h^{-1}\big(a^{-1}ha\big)
\in H_{\C}\cap G
=H.
$$
\end{proof}

\begin{proof}[Proof of Proposition \ref{Prop Action N_A(H)}]
By meromorphic continuation, it suffices to prove the assertion only for
$$
\lambda\in
\Ad(v)^{*}\big(\rho_{Q}-\Gamma+i(\fa/\fa_{\fh})^{*}\big).
$$
For these $\lambda$ the distribution $\mu_{x_{\cO}H}(Q:\xi:\lambda:\eta)$ is given by (\ref{eq mu_cO(Q) identity}).

If $\fs$ is an $A_{E}$-stable subspace of $\fg$, then we write $\Delta_{\fs}$ for the character of $A_{E}$ given by
$$
\Delta_{\fs}(a)
=\big|{\det}_{\fs}\big(\Ad(a^{-1})\big|_{\fs}\big)\big|
\qquad\big(a\in A_{E}\big).
$$
If $\psi\in \cD(G)$ and $a\in A_{E}$, then
$$
\int_{H}\phi(ha^{-1})\,dh
=\Delta_{\fh}(a)\int_{H}\phi(a^{-1}h)\,dh.
$$
As $\fg=\fh\oplus \fn_{Q}\oplus\fm'\oplus\fa'$ for suitable subspaces $\fm'$ and $\fa'$ of $\fm$ and $\fa$, respectively, we have
$$
\Delta_{\fg}=\Delta_{\fh}\Delta_{\fn_{Q}}\Delta_{\fm'\oplus\fa'}.
$$
Since $G$ is reductive, the character $\Delta_{\fg}$ is trivial. As $A_{E}$ centralizes $\fm'$ and $\fa'$, also $\Delta_{\fm'\oplus\fa'}$ is trivial. Furthermore,
$$
\Delta_{\fn_{Q}}(a)
=a^{-2\rho_{Q}}
\qquad\big(a\in A_{E}\big).
$$
We thus conclude that
$$
\Delta_{\fh}(a)
=a^{2\rho_{Q}}
\qquad\big(a\in A_{E}\big).
$$
The assertion now follows from Lemma \ref{Lemma x_O a=waw^(-1)x_O},  (\ref{eq mu_cO(Q) identity}) and the invariance of the measure on $A/A\cap H$.
\end{proof}

\subsection{$B$-matrices}
\label{Subsection Construction - B-matrices}
We continue with the notation from the previous section.

The following is an immediate corollary of  Theorem \ref{Thm Description D'(Z,Q:xi:lambda)}.

\begin{Cor}\label{Cor B-matrix}
Let $\xi$ be a finite dimensional unitary representation of $M_{Q}$, $v\in \cN$, and $\lambda\in (\fa/\fa_{\fh})_{\C}^{*}$ with $\Im\lambda\notin\cS$. Then there exists a unique linear operator
$$
\cB_{v}(Q:\xi:\lambda) :V^{*}(\xi)\to V^{*}(v\cdot\xi)
$$
so that the diagram
\begin{equation}\label{eq B-matrix diagram}
\xymatrixcolsep{10pc}\xymatrixrowsep{5pc}\xymatrix{
   \cD'(Q:\xi:\lambda)^{H} \ar[r]^{\cI_{v}(Q:\xi:\lambda)}
        & \cD'(Q:v\cdot\xi:\Ad^{*}(v)\lambda)^{H}  \\
   V^{*}(\xi)\ar@<-2pt>[u]^{\mu(Q:\xi:\lambda)} \ar[r]^{\cB_{v}(Q:\xi:\lambda)}
       & V^{*}(v\cdot\xi)\ar@<-2pt>[u]_{\mu\big(Q:v\cdot \xi: \Ad^{*}(v)\lambda\big)}
}
\end{equation}
commutes.
\end{Cor}

A similar map was first introduced in \cite{vdBan_PrincipalSeriesI} in the setting of real reductive symmetric spaces, where it was called the $\cB$-matrix, hence our notation.

We will prove a few properties of $\cB$-matrices.
Recall the maps $\iota_{\cO}$ from (\ref{eq Def iota_O}) and the set of representatives $\{v_{w}:w\in \cN/\cW\}$ for $\cN/\cW$ in $\cN\cap K$ from Section \ref{Subsection Construction - Description}. If $w\in\cN$, then by slight abuse of notation we write $v_{w}$ for $v_{w\cW}$.
Every element $w\in\cN$ defines a bijection
\begin{equation}\label{eq def s_w}
s_{w}:(P\bs Z)_{\fa_{\fh}}\to(P\bs Z)_{\fa_{\fh}};
\quad \cO\mapsto Pwx_{\cO}H.
\end{equation}
Note that
$$
[s_{w}(\cO)]=w\cdot[\cO]
\qquad\big(w\in \cN, \cO\in (P\bs Z)_{\fa_{\fh}}\big),
$$
and hence
$$
M_{Q,[s_{w}\cO]}
=M_{Q,w\cdot[\cO]}
=wM_{Q,[\cO]}w^{-1}
\qquad\big(w\in\cN, \cO\in (P\bs Z)_{\fa_{\fh}}\big).
$$

\begin{Prop}\label{Prop Properties of B-matrix}
Let $\xi$ be a finite dimensional unitary representation of $M_{Q}$ and $\lambda\in (\fa/\fa_{\fh})^{*}_{\C}$ with $\Im\lambda\notin \cS$.
Let $v,w\in\cN$ and let $\cO\in w\cdot(P\bs Z)_{\open}$.
Let further $\eta\in (V_{\xi}^{*})^{M_{Q,[\cO]}}=(V_{\xi}^{*})^{M_{Q}\cap v_{w}Hv_{w}^{-1}}$.
Then $\cB_{v}(Q:\xi:\lambda)\circ \iota_{\cO}(\eta)$ satisfies the following assertions.
\begin{enumerate}[(i)]
\item\label{Prop Properties of B-matrix - item 1} If $vw\notin Z_{G}(\fa_{E}/\fa_{\fh})$, then
$$
\Big(\cB_{v}(Q:\xi:\lambda)\circ\iota_{\cO}(\eta)\Big)_{\cO'}=0
\qquad\big(\cO'\in (P\bs Z)_{\open}\big).
$$
\item\label{Prop Properties of B-matrix - item 2a} If $\dim\big(v^{-1}N_{Q}v\cap \overline{N_{Q}}\big)+\dim(\cO)<\dim (Z)$, then
$$
\Big(\cB_{v}(Q:\xi:\lambda)\circ\iota_{\cO}(\eta)\Big)_{\cO'}=0
\qquad\big(\cO'\in (P\bs Z)_{\open}\big).
$$
\item\label{Prop Properties of B-matrix - item 2b} If $\dim\big(v^{-1}N_{Q}v\cap \overline{N}_{Q}\big)+\dim(\cO)=\dim (Z)$ and $vw\notin \cW$, then
$$
\Big(\cB_{v}(Q:\xi:\lambda)\circ\iota_{\cO}(\eta)\Big)_{\cO'}
=0
\qquad\big(\cO'\in (P\bs Z)_{\open}\big).
$$
\item\label{Prop Properties of B-matrix - item 3} If $v=v_{w}^{-1}$, then
$$
\Big(\cB_{v}(Q:\xi:\lambda)\circ\iota_{\cO}(\eta)\Big)_{\cO'}
=\left\{
  \begin{array}{ll}
    \eta & \big(\cO'=s_{v}(\cO)\big), \\
\\
    0 & \big(\cO'\in (P\bs Z)_{\open}, \cO'\neq s_{v}(\cO)\big).
  \end{array}
\right.
$$
\end{enumerate}
\end{Prop}

\begin{proof}
Let $\mu=\cI_{v}(Q:\xi:\lambda)\circ\mu(Q:\xi:\lambda)\eta$.

By Proposition \ref{Prop Action N_A(H)} we have for all $a\in A_{E}$
$$
R^{\vee}(a)\mu
=\cI_{v}(Q:\xi:\lambda)\circ R^{\vee}(a)\circ\mu(Q:\xi:\lambda)(\eta)
=a^{-\Ad^{*}(v_{w}^{-1})\lambda+\rho_{Q}}\mu.
$$
Let $\eta'\in V^{*}(v\cdot \xi)$ be so that $\mu=\mu\big(Q:v\cdot\xi:\Ad^{*}(v)\lambda\big)\eta'$.
Then for all $a\in A_{E}$
$$
R^{\vee}(a)\mu
=\sum_{w'\in \cN/\cW}\sum_{\cO'\in w'\cdot(P\bs Z)_{\open}}
    a^{-\Ad^{*}(v_{w'}^{-1}v)\lambda+\rho_{Q}}\mu(Q:\xi:\lambda)\circ\iota_{\cO'}(\eta'_{\cO'}).
$$
Both identities are identities of meromorphic functions in the parameter $\lambda$. Therefore, the only terms in the sum on the right-hand side of the second identity that can be non-zero, are those for $w'\in \cN/\cW$ with $v_{w'}^{-1}vv_{w}\in Z_{G}(\fa_{E}/\fa_{\fh})$. Since $\cW$ is a subgroup of $Z_{G}(\fa_{E}/\fa_{\fh})$, see \cite[Lemma 12.1]{KuitSayag_OnTheLittleWeylGroupOfARealSphericalSpace}, the latter condition is equivalent to $v_{w'}^{-1}vw\in Z_{G}(\fa_{E}/\fa_{\fh})$. Assertion (\ref{Prop Properties of B-matrix - item 1}) now follows by taking $w'=e\cW$.

We move on to prove (\ref{Prop Properties of B-matrix - item 2a}) and  (\ref{Prop Properties of B-matrix - item 2b}). By meromorphic continuation it suffices to prove the assertion for $\lambda\in (\fa/\fa_{\fh})_{\C}^{*}$ for which the intertwining operator $\cA(v^{-1}Qv:Q:\xi:\lambda)$ is given by a convergent integral over $v^{-1}N_{Q}v\cap \overline{N}_{Q}$. Since $\supp\big(\mu(Q:\xi:\lambda)\eta\big)\subseteq\overline{\cO}$, we then have
$$
\supp(\mu)
=v\cdot\supp\big(\cA(v^{-1}Qv:Q:\xi:\lambda)\mu\big)
\subseteq v\cdot \overline{\big(v^{-1}N_{Q}v\cap \overline{N}_{Q}\big)\cdot \overline{\cO}}.
$$
If $\dim\big(v^{-1}N_{Q}v\cap \overline{N}_{Q}\big)+\dim(\cO)<\dim (Z)$, then the interior of the support of $\mu$ is empty. This proves (\ref{Prop Properties of B-matrix - item 2a}). Assume that $\dim\big(v^{-1}N_{Q}v\cap \overline{N}_{Q}\big)+\dim(\cO)=\dim (Z)$ and the support of $\mu$ contains an open $P$-orbit $\cO'$.
Then $v^{-1}\cdot\cO'\subseteq(v^{-1}N_{Q}v\cap \overline{N}_{Q})\cdot \cO$, and hence $v^{-1}\cdot \cO'\cap \cO\neq \emptyset$. Let $z\in v^{-1}\cdot\cO'\cap \cO$. It follows from Proposition \ref{Prop decomp of N} that
$$
v^{-1}N_{Q}v\to v^{-1}N_{Q}v\cdot z;
\quad n\mapsto n\cdot z
$$
is a diffeomorphism. Since
$$
(v^{-1}N_{Q}v\cap \overline{N}_{P})\times (v^{-1}N_{Q}v\cap N_{P})\to v^{-1}N_{Q}v;
\quad(\overline{n},n)\mapsto \overline{n}n
$$
is a diffeomorphism as well, we obtain that
$$
(v^{-1}N_{Q}v\cap N_{P})\to (v^{-1}N_{Q}v\cap N_{P})\cdot z;
\quad n\mapsto n\cdot z
$$
is a diffeomorphism.
We note that $v$ normalizes $M_{Q}$ in view of Remark \ref{Rem L_Q roots} as it is contained in $\cN$. Therefore,
$$
v^{-1}N_{Q}v\cap \overline{N}_{P}
=v^{-1}N_{Q}v\cap \overline{N}_{Q}.
$$
It follows from Proposition \ref{Prop decomp of N}, Theorem \ref{Thm structure theorem for wPw^(-1) cdot z} and the assumption on the dimension of $\cO$ that
\begin{align*}
\dim (N_{P}\cdot z)
&=\dim(N_{Q})-\dim (Z)+\dim (\cO)
=\dim(N_{Q})-\dim\big(v^{-1}N_{Q}v\cap \overline{N}_{Q}\big)\\
&=\dim\big(v^{-1}N_{Q}v\cap N_{P}\big)
=\dim\Big(\big(v^{-1}N_{Q}v\cap N_{P}\big)\cdot z\Big)
\end{align*}
This implies that $(v^{-1}N_{Q}v\cap N_{P})\cdot z$ is open in $N_{P}\cdot z$.
By \cite[Theorem 2]{Rosenlicht_OnQuotientVarietiesAndTheAffineEmbeddingOfCertainHomogeneousSpaces} both $(v^{-1}N_{Q}v\cap N_{P})\cdot z$ and $N_{P}\cdot z$ are closed. Therefore, $(v^{-1}N_{Q}v\cap N_{P})\cdot z$ is also closed in $N_{P}\cdot z$. Moreover, both are connected. We thus conclude
$$
(v^{-1}N_{Q}v\cap N_{P})\cdot z
=N_{P}\cdot z.
$$
In particular we see that $N_{P}\cdot z\subseteq v^{-1}\cdot \cO'\cap \cO$.
Since every $N_{P}$-orbit in $\cO$ contains a weakly admissible point, we may without loss of generality assume that $z$ is weakly admissible.

Now $v\cdot z$ is an admissible point in $\cO'$. Hence if $X\in \fa^{-}$ is order-regular, then there exists an $u\in \cW$ so that $\fh_{v\cdot z, \Ad(v)X}=\Ad(u)\fh_{\emptyset}$ . We now have
$$
\fh_{z,X}
=\Ad(v^{-1})\fh_{v\cdot z,\Ad(v)X}
=\Ad(v^{-1}u)\fh_{\emptyset}.
$$
This implies that $\cO\in v^{-1}u\cdot(P\bs Z)_{\open}$. By assumption $\cO\in w\cdot(P\bs Z)_{\open}$. Therefore, $u^{-1}vw$ stabilizes $(P\bs Z)_{\open}$. Since the stabilizer is equal to $\cW$ by Theorem \ref{Thm properties of W-action}(\ref{Thm properties of W-action - item 2}), we may conclude that $vw\in \cW$. This proves (\ref{Prop Properties of B-matrix - item 2b}).

Finally, we prove (\ref{Prop Properties of B-matrix - item 3}). Let $v= v_{w}^{-1}$. In view of  (\ref{eq mu_(O_w)=a(wPw^(-1):P)mu_O}) we have
$$
\mu
=L^{\vee}(v)\mu_{x_{\cO}H}(v^{-1}Pv:\xi:\lambda+\rho_{P}-\rho_{Q}:\eta).
$$
Using meromorphic continuation, (\ref{eq mu_cO identity}) and the fact that $v_{w}$ satisfies (\ref{eq wSigma^+ cap -Sigma^+=wSigma(Q) cap -Sigma}) we obtain
$$
\mu
=\mu_{vx_{\cO}H}\big(P:v\cdot\xi:\Ad^{*}(v)\lambda+\rho_{P}-\rho_{Q}:\eta\big)
=\mu\big(Q:v\cdot\xi:\Ad^{*}(v)\lambda\big)\circ\iota_{s_{v}(\cO)}(\eta).
$$
This proves  (\ref{Prop Properties of B-matrix - item 3}).
\end{proof}

We define the map
$$
\beta(\xi:\lambda):V^{*}(\xi)\to V^{*}(\xi)
$$
for $\eta\in V^{*}(\xi)$, $\cO\in (P\bs Z)_{\open}$ and $w\in \cN/\cW$ to be given by
\begin{equation}\label{eq Def beta}
\Big(\beta(\xi:\lambda)\eta\Big)_{s_{v_{w}}(\cO)}
=\frac{1}{\gamma(v_{w}\overline{Q}v_{w}^{-1}:\overline{Q}:\xi:\lambda)}\Big(\cB_{v_{w}^{-1}}(Q:\xi:\lambda)\eta\Big)_{\cO}.
\end{equation}
We will use $\beta(\xi:\lambda)$ for the normalization of the map $\mu(Q:\xi:\lambda)$ in the next section.

Let $\ev_{g}$ denote evaluation in a point $g\in G$.
Since
$$
\Big(\cB_{v_{w}^{-1}}(Q:\xi:\lambda)\eta\Big)_{\cO}
=\ev_{v_{w}x_{\cO}}\circ \cA(v_{w}Qv_{w}^{-1}:Q:\xi:\lambda)\circ\mu(Q:\xi:\lambda)
$$
depends meromorphically on $\lambda\in (\fa/\fa_{\fh})_{\C}^{*}$  for $w\in \cN/\cW$ and $\cO\in (P\bs Z)_{\open}$, the assignment
$$
(\fa/\fa_{\fh})_{\C}^{*}\to \End\big(V^{*}(\xi)\big);
\quad\lambda\mapsto \beta(\xi:\lambda)
$$
is a meromorphic function. Moreover,  $\lambda\mapsto\beta(\xi:\lambda)$ is holomorphic on $\{\lambda\in (\fa/\fa_{\fh})_{\C}^{*}:\Im\lambda\notin\cS\}$.

If we order the orbits in $(P\bs Z)_{\fa_{\fh}}$ by dimension and choose a basis of $V^{*}(\xi)$ subject to the decomposition
$$
V^{*}(\xi)
=\bigoplus_{\cO\in (P\bs Z)_{\fa_{\fh}}}(V_{\xi}^{*})^{M_{Q,[\cO]}},
$$
then in view of Proposition \ref{Prop Properties of B-matrix}(\ref{Prop Properties of B-matrix - item 2a} -- \ref{Prop Properties of B-matrix - item 3}) the matrix of $\beta(\xi:\lambda)$ with respect to this basis is upper triangular and the diagonal entries are reciprocals of $\gamma$-factors. It follows that $\beta(\xi:\lambda)$ is invertible. Since $\lambda\mapsto\beta(\xi:\lambda)$ is meromorphic, it follows from Cramer's rule that also
$$
(\fa/\fa_{\fh})_{\C}^{*}\to\End\big(V^{*}(\xi)\big);
\quad \lambda\mapsto \beta(\xi:\lambda)^{-1}
$$
is meromorphic. This observation has the following corollary.

\begin{Cor}
Let $\xi$ be a finite dimensional unitary representation of $M_{Q}$. For every $v\in\cN$ the $\cB$-matrix $\cB_{v}(Q:\xi:\lambda)$ depends meromorphically on $\lambda\in (\fa/\fa_{\fh})^{*}_{\C}$.
\end{Cor}

\begin{proof}
The map $\beta(\xi:\lambda)\circ\mu(Q:\xi:\lambda)^{-1}$ is for $\mu\in \cD'(Q:\xi:\lambda)^{H}$,  $w\in \cN/\cW$ and $\cO\in (P\bs Z)_{\open}$ given by
\begin{align*}
&\Big(\beta(\xi:\lambda)\circ\mu(Q:\xi:\lambda)^{-1}(\mu)\Big)_{s_{v_{w}}(\cO)}\\
&\qquad=\frac{1}{\gamma(v_{w}\overline{Q}v_{w}^{-1}:\overline{Q}:\xi:\lambda)}\ev_{v_{w}x_{\cO}}
        \circ\cA(v_{w}Qv_{w}^{-1}:Q:\xi:\lambda)(\mu).
\end{align*}
It follows that for a meromorphic family of distributions $\mu_{\lambda}\in \cD'(Q:\xi:\lambda)^{H}$ with family parameter $\lambda\in(\fa/\fa_{\fh})_{\C}^{*}$, the assignment $\lambda\mapsto\beta(\xi:\lambda)\circ\mu(Q:\xi:\lambda)^{-1}(\mu_{\lambda})$ is meromorphic. We apply this to
$$
\mu_{\lambda}
= \cI_{v}(Q:\xi:\lambda)\circ\mu(Q:\xi:\lambda)(\eta)
$$
for $v\in \cN$ and $\eta\in V^{*}(\xi)$ and thus conclude that
\begin{align*}
\cB_{v}(Q:\xi:\lambda)\eta
&=\mu(Q:\xi:\lambda)^{-1}\circ\cI_{v}(Q:\xi:\lambda)\circ\mu(Q:\xi:\lambda)(\eta)\\
&=\beta(\xi:\lambda)^{-1}\circ \Big(\beta(\xi:\lambda)\circ\mu(Q:\xi:\lambda)^{-1}\Big)(\mu_{\lambda})
\end{align*}
depends meromorphically on $\lambda$.
\end{proof}

\subsection{Normalization}
\label{Subsection Construction - Normalization}
We continue with the notation from the previous section.
For a finite dimensional unitary representation $\xi$ of $M_{Q}$ and $\lambda\in(\fa/\fa_{\fh})^{*}_{\C}$ with $\Im\lambda\notin \cS$, we normalize our distributions $\mu(Q:\xi:\lambda)\eta$ using the map $\beta(\xi:\lambda)$ from (\ref{eq Def beta}) by defining
\begin{equation}\label{eq Def mu^circ}
\mu^{\circ}(\xi:\lambda)
:=\cA\big(Q:\overline{Q}:\xi:\lambda\big)^{-1}\circ\mu(Q:\xi:\lambda)\circ\beta(\xi:\lambda)^{-1}
:V^{*}(\xi)\to \cD'(\overline{Q}:\xi:\lambda)^{H}.
\end{equation}
The reason for normalizing the distributions is to make sure that the composition of the constant term map and $\mu^{\circ}(\xi:\lambda)$ will have a desirable form,  see  Section \ref{Subsection Most continuous part - Constant Term}.

We end this section with a reformulation of Theorem \ref{Thm Description D'(Z,Q:xi:lambda)}.

\begin{Thm}\label{Thm Description D'(overline Q:xi:lambda)^H}
For every finite dimensional unitary representation  $(\xi,V_{\xi})$ of $M_{Q}$ and every $\lambda\in(\fa/\fa_{\fh})_{\C}^{*}$ with $\Im\lambda\notin\cS$ the map
\begin{equation}\label{eq Parametrization D'(overline Q:xi:lambda)^H}
\mu^{\circ}(\xi:\lambda):V^{*}(\xi)\to\cD'(\overline{Q}:\xi:\lambda)^{H}
\end{equation}
is a linear isomorphism. The assignment
$$
\lambda\mapsto \mu^{\circ}(\xi:\lambda)\eta
$$
defines for every $\eta\in V^{*}(\xi)$ a meromorphic family of distributions in $\cD'(G, V_{\xi})$ with family parameter $\lambda\in (\fa/\fa_{\fh})^{*}_{\C}$. The poles of the family lie on a locally finite union of complex affine hyperplanes of the form
\begin{equation}\label{eq Form of hyperplanes III}
\{\lambda\in(\fa/\fa_{\fh})^{*}_{\C}: \lambda(X)=a\}\quad\text{for some }X\in\fa\text{ and }a\in\R.
\end{equation}
\end{Thm}

\begin{proof}
The poles of standard intertwining operators, as well as the poles and zero's of $\gamma$-functions,  all lie on a locally finite union of complex affine hyperplanes of the form (\ref{eq Form of hyperplanes III}). The proposition now follows from Theorem \ref{Thm Description D'(Z,Q:xi:lambda)}.
\end{proof}

For future reference we record here that in view of the corollary we may and will equip $\cD'(\overline{Q}:\xi:\lambda)^{H}$ for $\lambda\in (\fa/\fa_{\fh})^{*}_{\C}$ with $\Im\lambda\notin \cS$  with an inner product so that the map (\ref{eq Parametrization D'(overline Q:xi:lambda)^H}) is an isometry.

\subsection{The horospherical case}
\label{Subsection Construction - Horospherical case}
We call the real spherical homogeneous space $Z$ horospherical if $\fa$ normalizes $\fh_{z}$ for one (and hence for all) adapted point $z\in Z$. We note that $Z$ is horospherical if and only if the compression cone $\cC$ equals $\fa$, which is equivalent to the little Weyl group of $Z$ being trivial.
In this case the stabilizer $H_{z}$ of an adapted point $z\in Z$ is given by
$$
H_{z}
=(L_{Q}\cap H_{z})\overline{N}_{Q}
=(M\cap H_{z})\exp(\fa_{\fh})L_{Q,\nc}\overline{N}_{Q},
$$
where $L_{Q,\nc}$ is the connected subgroup of $G$ with Lie algebra $\fl_{Q,\nc}$.

In this section we will further explicate the description of $\cD'(\overline{Q}:\xi:\lambda)^{H}$ from Theorem \ref{Thm Construction on max rank orbits for Q} and Theorem \ref{Thm Description D'(overline Q:xi:lambda)^H} under the assumption that $H$ is horospherical.

We first note that $\overline{N}_{P}\subseteq L_{Q,\nc}\overline{N}_{Q}\subseteq H_{z}$ for every adapted point $z$. As $P$ admits only one open orbit in $G/\overline{N}_{P}$, it follows that there exists precisely one open $P$-orbit $\cO_{0}$ in $Z$. Recall that the point $eH\in Z=G/H$ is assumed to be admissible. As in \cite[Example 3.5]{KuitSayag_OnTheLittleWeylGroupOfARealSphericalSpace} it is easily seen that the set of adapted points in $Z$ is equal to $MAH/H$. In particular, we are in the situation of Remark \ref{Remark properties of W-action} (\ref{Remark properties of W-action - item 2}) and (\ref{Remark properties of W-action - item 3}), and hence the Weyl group $W$ acts transitively on the set of $P$-orbits in $Z$ of maximal rank. This action is given by (\ref{eq w cdot O=Pv cdot z}) and (\ref{eq W-action of max rank orbits}). By Theorem \ref{Thm properties of W-action} (\ref{Thm properties of W-action - item 2}) the stabilizer of the open orbit is equal to $\cZ/MA$, where $\cZ=N_{L_{Q}}(\fa)$.

It follows from the Bruhat decomposition that each $P$-orbit in $Z$ is of the form $PwH$, with $w\in N_{G}(\fa)$. If $\cO=PwH$, then
$$
\fa_{\cO}
=\fa\cap \Ad(w)\fh
=\Ad(w)\fa_{\fh}.
$$
In particular, each orbit is of maximal rank. Therefore, the map
\begin{equation}\label{eq N/cV to P bs Z}
N_{G}(\fa)/\cZ\to P\bs Z;
\quad v\cZ\mapsto PvH
\end{equation}
is a bijection. We recall from (\ref{eq def cN}) that $\cN$ denotes the group $N_{G}(\fa)\cap N_{G}(\fa_{\fh})$. The image of $\cN/\cZ$ under the map (\ref{eq N/cV to P bs Z} ) is equal to the set $(P\bs Z)_{\fa_{\fh}}$ of $P$-orbits $\cO$ in $Z$ with $\fa_{\cO}=\fa_{\fh}$.
We complete the set $\{v_{w}:w\in\cN/\cW\}$ from Section \ref{Subsection Construction - Description} to a set of representatives $\fN$ of $\cN/\cZ$ in $\cN\cap K$. Then
$$
\fN\to (P\bs Z)_{\max};
\quad v\mapsto PvH
$$
is a bijection and the points $vH\in Z$ with $v\in\fN$ are weakly adapted. The $v\in\fN$ play the role of the elements $x_{\cO}\in G$ from section \ref{Subsection Construction - Description}.

Let $\xi$ be a finite dimensional unitary representation of $M_{Q}$. Then
\begin{equation}\label{eq Decomp V^*(xi) horospherical}
V^{*}(\xi)
=\bigoplus_{\substack{v\in\fN}}(V_{\xi}^{*})^{M_{Q}\cap vHv^{-1}}
\end{equation}
It follows from Proposition \ref{Prop Properties of B-matrix}(\ref{Prop Properties of B-matrix - item 1}) that the map $\beta(\xi:\lambda)$ is diagonal with respect to a basis of $V^{*}(\xi)$ subject to the decomposition (\ref{eq Decomp V^*(xi) horospherical}).
Now Proposition \ref{Prop Properties of B-matrix} (\ref{Prop Properties of B-matrix - item 3}) yields that $\mu^{\circ}(\xi:\lambda)$ for $\eta\in V^{*}(\xi)$ is given by
$$
\mu^{\circ}(\xi:\lambda)\eta
=\sum_{v\in \cN}\gamma(v^{-1}\overline{Q}v:\overline{Q}:\xi:\lambda) \cA(Q:\overline{Q}:\xi:\lambda)^{-1}
    \mu_{vH}(Q:\xi:\lambda:\eta_{v}).
$$
For $v\in\fN$ we write $\iota_{v}$ for the inclusion map
\begin{equation}\label{eq Def iota_v}
\iota_{v}:(V_{\xi}^{*})^{M_{Q}\cap vHv^{-1}}
\hookrightarrow V^{*}(\xi).
\end{equation}
For $v\in\cN$ we further write $\cI^{\circ}_{v}(\xi:\lambda)$ for the normalized intertwining operator $\cI^{\circ}_{v}(\overline{Q}:\xi:\lambda)$ from Section \ref{Subsection Distribution vectors - Intertwining operators}.

\begin{Cor}\label{Cor Description D'(Z,overline Q:xi:lambda) horospherical case}
Let $(\xi,V_{\xi})$ be a finite dimensional unitary representation of $M_{Q}$.
For all $\lambda\in(\fa/\fa_{\fh})_{\C}^{*}$ with $\Im\lambda\notin\cS$ the map
\begin{equation}\label{eq Isomorphism horospherical case}
\mu^{\circ}(\xi:\lambda): V^{*}(\xi)\to\cD'(\overline{Q}:\xi:\lambda)^{H};
\end{equation}
is a linear isomorphism. The assignment
$$
\lambda\mapsto \mu^{\circ}(\xi:\lambda)\eta
$$
defines for every $\eta\in V^{*}(\xi)$ a meromorphic family of distributions in $\cD'(G, V_{\xi})$ with family parameter $\lambda\in (\fa/\fa_{\fh})^{*}_{\C}$.
For all $v\in\fN$ the distributions $\mu\in\cD'\big(Q:v^{-1}\cdot\xi:\Ad^{*}(v^{-1})\lambda\big)^{H}$ are smooth in $e$ and may therefore be evaluated in $e$.  The inverse of (\ref{eq Isomorphism horospherical case}) is given by
\begin{align}\label{eq Inverse isomorphism horospherical case}
\nonumber\cD'(\overline{Q}:\xi:\lambda)^{H}&\to V^{*}(\xi);\\
\quad \mu &\mapsto
    \Big(\ev_{e}\circ \cA\big(Q:\overline{Q}:v^{-1}\cdot \xi:\Ad^{*}(v^{-1})\lambda\big)\circ \cI^{\circ}_{v^{-1}}(\xi:\lambda)(\mu)\Big)_{v\in\fN},
\end{align}
where $\ev_{g}$ denotes evaluation in a point $g\in G$.

Let $v\in \fN$.  Then
\begin{equation}\label{eq horospherical case a mu_v =a^(-v^-1 lambda+rho)}
R^{\vee}(a)\mu^{\circ}(\xi:\lambda)\circ\iota_{v}
=a^{-\Ad^{*}(v^{-1})\lambda+\rho_{Q}}\mu^{\circ}(\xi:\lambda)\circ\iota_{v}
\qquad(a\in A)
\end{equation}
and
\begin{equation}\label{eq transformation rule mu horospherical case}
\cI^{\circ}_{v}(\xi:\lambda)\circ\mu^{\circ}(\xi:\lambda)\circ\iota_{e}
=\mu^{\circ}\big(v\cdot\xi:\Ad^{*}(v)\lambda\big)\circ\iota_{v}.
\end{equation}
Finally, if $\lambda$ satisfies $\Re\lambda(\alpha^{\vee})>0$ for all $\alpha\in -\Sigma(Q)\cap \Sigma(vQv^{-1}) $, then for every $\eta\in (V_{\xi}^{*})^{M_{Q}\cap vHv^{-1}}$ the distribution $\mu^{\circ}(\xi:\lambda)(\iota_{v}\eta)$ is for $\phi\in \cD(G,V_{\xi})$ given by
\begin{align}\label{eq formula mu horospherical case}
&\gamma(Q:vQv^{-1}:\xi:\lambda)\mu^{\circ}(\xi:\lambda)(\iota_{v}\eta)(\phi)\\
\nonumber&\quad= \int_{\overline{N}_{Q}\cap vN_{Q}v^{-1}}\int_{M_{Q}}\int_{A}\int_{\overline{N}_{Q}}a^{-\lambda-\Ad^{*}(v)\rho_{Q}}
        \Big(\xi^{\vee}(m)\eta,\phi(nmav\overline{n})\Big)\,d\overline{n}\,da\,dm\,dn.
\end{align}
\end{Cor}

\begin{proof}
Except for the identities (\ref{eq Inverse isomorphism horospherical case}), (\ref{eq transformation rule mu horospherical case}) and (\ref{eq formula mu horospherical case}) all assertions follow directly from  Theorem \ref{Thm Construction on max rank orbits for Q}, Proposition \ref{Prop Action N_A(H)} and Theorem \ref{Thm Description D'(overline Q:xi:lambda)^H}.

Let $\eta\in V_{\xi}^{M_{Q}\cap H}$.
It follows from (\ref{eq mu_cO(Q) identity}) that $\mu(Q:\xi:\lambda)(\iota_{e}\eta)$ for $\phi\in \cD(G,V_{\xi}^{*})$ and $\lambda\in \rho_{Q}-\Gamma+i(\fa/\fa_{\fh})^{*}$ is given by
$$
\big(\mu(Q:\xi:\lambda)(\iota_{e}\eta)\big)(\phi)
=\int_{N_{Q}}\int_{M_{Q}}\int_{A}\int_{\overline{N}_{Q}}a^{-\lambda-\rho_{Q}}
        \Big(\xi^{\vee}(m)\eta,\phi(nma\overline{n})\Big)\,d\overline{n}\,da\,dm\,dn.
$$
In view of (\ref{eq a(S_2:S_1)mu}) in Proposition \ref{Prop int formula for a(S_2:S_1:xi:lambda)} and (\ref{eq Def mu^circ}) we have
\begin{equation}\label{eq mu^circ e component horospherical case}
\big(\mu^{\circ}(\xi:\lambda)(\iota_{e}\eta)\big)(\phi)
=\int_{M_{Q}}\int_{A}\int_{\overline{N}_{Q}}a^{-\lambda-\rho_{Q}}
        \Big(\xi^{\vee}(m)\eta,\phi(ma\overline{n})\Big)\,d\overline{n}\,da\,dm.
\end{equation}
The right-hand side is a convergent integral for all $\lambda\in (\fa/\fa_{\fh})^{*}_{\C}$ and depends holomorphically on $\lambda$. Therefore, the identity holds for all $\lambda\in i(\fa/\fa_{\fh})^{*}$.

Let $v\in \fN$ and $\eta\in V_{\xi}^{M_{Q}\cap H}$.
Since
\begin{align*}
&\cA\big(Q:\overline{Q}:v\cdot \xi:\Ad^{*}(v)\lambda\big)\circ\cI^{\circ}_{v}(\xi:\lambda)\\
&\quad=\frac{1}{\gamma\big(v^{-1}\overline{Q}v:\overline{Q}:\xi:\lambda\big)}
    \cA\big(Q:\overline{Q}:v\cdot\xi:\Ad^{*}(v)\lambda\big)\\
    &\qquad\qquad\qquad\qquad\qquad\qquad\qquad\qquad
    \circ \cA\big(\overline{Q}:v\overline{Q}v^{-1}:v\cdot\xi:\Ad^{*}(v)\lambda\big)\circ L^{\vee}(v)\\
&\quad=\gamma\big(v\overline{Q}v^{-1}:\overline{Q}:v\cdot\xi:\Ad^{*}(v)\lambda\big)
   \cA\big(Q:v\overline{Q}v^{-1}:v\cdot\xi:\Ad^{*}(v)\lambda\big)\circ  L^{\vee}(v),
\end{align*}
we have by (\ref{eq mu^circ e component horospherical case}) and (\ref{eq a(S_2:S_1)mu}) for $\phi\in\cD(G:V_{\xi}^{*})$
\begin{align*}
&\Big(\cA(Q:\overline{Q}:v\cdot \xi:\Ad^{*}(v)\lambda)\circ\cI^{\circ}_{v}(\xi:\lambda)\circ\mu^{\circ}(\xi:\lambda)(\iota_{e}\eta)\Big)(\phi)\\
&\quad=\gamma\big(v\overline{Q}v^{-1}:\overline{Q}:v\cdot\xi:\Ad^{*}(v)\lambda\big)\\
&\quad\times\int_{N_{Q}\cap vN_{Q}v^{-1}}\int_{M_{Q}}\int_{A}\int_{\overline{N}_{Q}}a^{-\Ad^{*}(v)\lambda-\Ad^{*}(v)\rho_{Q}}
        \Big((v\cdot\xi^{\vee})(m)\eta,\phi(nmav\overline{n})\Big)\,d\overline{n}\,da\,dm\,dn
\end{align*}
under the condition that $\lambda$ satisfies $\Re\lambda(\alpha^{\vee})>0$ for all $\alpha\in \Sigma(Q)\cap-\Sigma(v^{-1}Qv) $.
It follows from (\ref{eq mu_cO(Q) identity}) and (\ref{eq Def mu^circ}) that the right-hand side equals
\begin{align*}
&\gamma\big(v\overline{Q}v^{-1}:\overline{Q}:v\cdot\xi:\Ad^{*}(v)\lambda\big)
    \mu_{vH}\big(Q:v\cdot\xi:\Ad^{*}(v)\lambda:\eta\big)(\phi)\\
&\qquad=\Big(\cA\big(Q:\overline{Q}:v\cdot\xi:\Ad^{*}(v)\lambda\big)\circ\mu^{\circ}\big(v\cdot\xi:\Ad^{*}(v)\lambda\big)(\iota_{v}\eta)\Big)(\phi).
\end{align*}
By meromorphic continuation we obtain (\ref{eq transformation rule mu horospherical case}).

If $\lambda$ satisfies $\Re\lambda(\alpha^{\vee})>0$ for all $\alpha\in -\Sigma(Q)\cap \Sigma(vQv^{-1}) $, then (\ref{eq formula mu horospherical case}) follows from (\ref{eq a(S_2:S_1)mu}), (\ref{eq transformation rule mu horospherical case}) and (\ref{eq mu^circ e component horospherical case}).

Finally we move on to show (\ref{eq Inverse isomorphism horospherical case}).
Let $v,w\in\fN$ and $\eta\in (V_{\xi}^{*})^{M_{Q}\cap wHw^{-1}}$. We set
$$
\mu
:= \cA\big(Q:\overline{Q}:v^{-1}\cdot \xi:\Ad^{*}(v^{-1})\lambda\big)\circ \cI^{\circ}_{v^{-1}}(\xi:\lambda)\circ\mu^{\circ}(\xi:\lambda)(\iota_{w}\eta).
$$
By (\ref{eq transformation rule mu horospherical case}) we have
\begin{align*}
\mu
&= \cA\big(Q:\overline{Q}:v^{-1}\cdot \xi:\Ad^{*}(v^{-1})\lambda\big)\circ \cI^{\circ}_{v^{-1}w}\big(w^{-1}\xi:\Ad^{*}(w^{-1})\lambda\big)\\
&\qquad\qquad\qquad\qquad\qquad\qquad\qquad\qquad\circ\mu^{\circ}\big(w^{-1}\xi:\Ad^{*}(w^{-1})\lambda\big)(\iota_{e}\eta)\\
&= c
    \cA\big(Q:v^{-1}w\overline{Q}w^{-1}v:v^{-1}\cdot \xi:\Ad^{*}(v^{-1})\lambda\big)\circ L^{\vee}(v^{-1}w)\\
&\qquad\qquad\qquad\qquad\qquad\qquad\qquad\qquad\circ\mu^{\circ}\big(w^{-1}\xi:\Ad^{*}(w^{-1})\lambda\big)(\iota_{e}\eta)
\end{align*}
for some $c\in\C$.
By meromorphic continuation it follows from (\ref{eq formula mu horospherical case}) and (\ref{eq a(S_2:S_1)mu}) that
$$
\supp(\mu)
\subseteq \overline{N_{Q}v^{-1}w\overline{Q}},
$$
and hence
$e\in\supp(\mu)$ if and only if $v=w$. Now suppose $v=w$. Then $\mu$ is smooth on the open subset $vQ\overline{N}_{Q}$. We may therefore evaluate $\mu$ in $e$. Now
\begin{align*}
\ev_{e}(\mu)
&=\ev_{e}\circ \cA\big(Q:\overline{Q}:v^{-1}\cdot \xi:\Ad^{*}(v^{-1})\lambda\big)
    \circ\mu^{\circ}\big(v^{-1}\xi:\Ad^{*}(v^{-1})\lambda\big)(\iota_{e}\eta)\\
&=\ev_{e}\circ\mu\big(Q:v^{-1}\xi:\Ad^{*}(v^{-1})\lambda\big)(\iota_{e}\eta)
=\eta.
\end{align*}
This proves (\ref{eq Inverse isomorphism horospherical case}).
\end{proof}

\section{Temperedness, the constant term and wave packets}
\label{Section Wave packets}
We continue with the notation and choices from Section \ref{Subsection Construction - Description}.
If $V$ is a finite dimensional vector space, $\mu\in\cD'(G,V)$ and $\phi\in\cD(G,V)$, then we write $m_{\phi,\mu}$ for the matrix coefficient
$$
m_{\phi,\mu}:G\to \C;\quad g\mapsto \big(R^{\vee}(g)\mu\big)(\phi).
$$

\subsection{Temperedness}
\label{Subsection Wave packets - Temperedness}
The space $Z$ admits a polar decomposition, which was first given in \cite{KnopKrotzSayagSchlichtkrull_SimpleCompactificationsAndPolarDecomposition}. The following version is a slight reformulation from \cite[Proposition 8.6]{KuitSayag_OnTheLittleWeylGroupOfARealSphericalSpace}.

Recall that $K$ is a maximal compact subgroup so that $\fk$ is Killing perpendicular to $\fa$.

\begin{Prop}\label{Prop Polar decomposition}
There exists a compact subset $\Omega\subseteq G$ so that
\begin{equation}\label{eq Polar decomp}
G
=\bigcup_{\cO\in(P\bs Z)_{\open}}\Omega\exp(\overline{\cC}) x_{\cO}H.
\end{equation}
The set $\Omega$ can be chosen to be $\Omega=\bigcup_{j=1}^{r}f_{j}K_{j}$ with $r\in\N$, $ f_{j}\in G$ and the $K_{j}$  maximal compact subgroups of $G$.
\end{Prop}

Let $\xi$ be a finite dimensional unitary representation of $M_{Q}$ and $\lambda\in(\fa/\fa_{\fh})^{*}_{\C}$. We call a distribution $\eta\in\cD'(\overline{Q}:\xi:\lambda)^{H}$ tempered if there exists an $N\in\N_{0}$ and a continuous seminorm $p$ on $\cD(G:V_{\xi})$ so that for every $\phi\in\cD(G:V_{\xi})$
$$
|m_{\phi,\eta}(\omega \exp(X) x_{\cO})|
\leq e^{\rho_{Q}(X)}(1+\|X\|)^{N}p(\phi)
\qquad\big(\cO\in(P\bs Z)_{\open},\omega\in\Omega, X\in\overline{\cC}\big).
$$

We recall that we have equipped the spaces $\cD'(\overline{Q}:\xi:\lambda)^{H}$, with $\Im\lambda\notin \cS$,  with an inner product so that the map (\ref{eq Parametrization D'(overline Q:xi:lambda)^H}) is an isometry.

\begin{Thm}\label{Thm temperedness}
Let $\xi$ be a finite dimensional unitary representation of $M_{Q}$ and $\fC$ a compact subset of $\{\lambda\in (\fa/\fa_{\fh})^{*}_{\C}:\Im\lambda\notin i\cS\}$.
There exist an $N\in\N_{0}$ and a continuous seminorm $p$ on $\cD(G,V_{\xi})$ so that  for every $\lambda\in \fC$,  $\mu\in \cD'(\overline{Q}:\xi:\lambda)^{H}$, $\phi\in\cD(G,V_{\xi})$, $\cO\in(P\bs Z)_{\open}$, $\omega\in\Omega$ and $X\in \overline{\cC}$
$$
|m_{\phi,\mu}  (\omega \exp(X)x_{\cO} ) |
\leq \max_{w\in W} e^{\rho_{Q}(X)+\Re\Ad^{*}(w)\lambda(X)}(1+\|X\|)^{N}\|\mu\| p(\phi).
 $$
In particular, every distribution $\mu\in\cD'(\overline{Q}:\xi:\lambda)^{H}$ with $\lambda\in i(\fa/\fa_{\fh})^{*}\setminus i\cS$ is tempered.
\end{Thm}

The proof for the theorem is by induction on the faces $\cF$ of $\overline{\cC}$.
In the Sections \ref{Subsection Wave packets - A priori estimate} we give an a priori estimate which accomplishes the initial step of the induction. In   Section \ref{Subsection Wave packets - Boundary degenerations} we recall the notion of boundary degenerations and some of their properties. As the proof of the theorem relies heavily on the theory of the constant term map from \cite{DelormeKrotzSouaifi_ConstantTerm}, we have to recall the necessary definitions and results. We do so in Section \ref{Subsection Wave packets - Preparation}. Finally the proof of the theorem will be given in Section \ref{Subsection Wave packets - Proof of Theorem}.

\subsection{An a priori estimate}\label{Subsection Wave packets - A priori estimate}

For $\lambda\in(\fa/\fa_{\fh})^{*}_{\C}$ we define $\langle\lambda\rangle:\fa\to\R$ by
$$
\langle\lambda\rangle
:=\max\Big(\{0\}\cup \{\Re\Ad^{*}(w)\lambda:w\in W\}\Big).
$$

\begin{Lemma}\label{Lemma A priori estimate}
Let $\xi$ be a finite dimensional unitary representation of $M_{Q}$ and $\fC$ a compact subset of $\{\lambda\in (\fa/\fa_{\fh})^{*}_{\C}:\Im\lambda\notin\cS\}$.
There exists an $\zeta\in(\fa/\fa_{\fh})^{*}$ with $\zeta|_{\fa_{E}}=\rho_{Q}|_{\fa_{E}}$, and a continuous seminorm $p$ on $\cD(G,V_{\xi})$ so that  for every $\lambda\in \fC$,  $\mu\in \cD'(\overline{Q}:\xi:\lambda)^{H}$ and $\phi\in\cD(G,V_{\xi})$ we have
$$
|m_{\phi,\mu}  (\omega \exp(X)  ) |
\leq  e^{\zeta(X)+\langle\lambda\rangle(X)} \|\mu\| p(\phi)
    \qquad   \big(\omega\in\Omega, X\in\exp(\overline{\cC})\big).
$$
\end{Lemma}

\begin{proof}
Let $\fa_{0}$ be a complementary subspace to $\fa_{E}$ in $\fa$ .
Let $k\in\N$ and let $\cD_{k}'(\overline{Q}:\xi:\lambda)$ be the subspace of $\cD'(\overline{Q}:\xi:\lambda)$ of distributions of order at most $k$. We recall the maximal compact subgroup $K$ of $G$ and note that $\cD_{k}'(\overline{Q}:\xi:\lambda)$ is canonically isomorphic to the space $\cD_{k}'(K;V_{\xi})^{M}$ of left-$M$-invariant distributions in  $\cD'(K,V_{\xi})$ of order at most $k$.  Note that $\cD_{k}'(K;V_{\xi})^{M}$ is a Banach space. The same proof as for \cite[Lemma 10.1]{vdBan_PrincipalSeriesII} yields the existence of constants $C> 0$ and  $r > 0$, independent of $\lambda\in\fC$, such that for every $X\in \fa_{0}$, the operator $R\big(\exp(X)\big)$ maps  $\cD_{k}'(\overline{Q}:\xi:\lambda)$ to itself with operator norm
$$
\|R\big(\exp(X)\big)\|
\leq Ce^{r\|X\|}.
$$
We recall that the spherical root system $\Sigma_{Z}$ in $(\fa/\fa_{E})^{*}$ admits the image $\overline{\cC}/\fa_{E}$ of $\overline{\cC}$ under the projection $\fa\to\fa/\fa_{E}$ as a Weyl chamber. By taking a sum of positive roots, we find a functional $\zeta_{0}\in(\fa/\fa_{E})^{*}$ that is strictly positive on $\overline{\cC}\setminus \fa_{E}$. Note that $\zeta_{0}|_{\fa_{E}}=0$.

By Proposition \ref{Prop Action N_A(H)} we have for all $\mu\in\cD'(\overline{Q}:\xi:\lambda)^{H}$
$$
|m_{\phi,\mu}\big(g\exp(Y)\big) |
\leq \max_{w\in W}e^{\rho_{Q}(Y)+\Re\Ad^{*}(w)\lambda(Y)}|m_{\phi,\mu}  (g) |
\qquad   \big(g\in G, Y\in\fa_{E}\big).
 $$
We recall the isometry $\mu^{\circ}(\xi:\lambda)$ from Theorem \ref{Thm Description D'(overline Q:xi:lambda)^H}. Since the distributions $\mu^{\circ}(\xi:\lambda)\eta$ depend smoothly on $\lambda\in\fC$ and linearly on $\eta$, the assertion therefore follows, after rescaling $\zeta_{0}$ if necessary, with $\zeta=\zeta_{0}+\rho_{Q}$.
\end{proof}

\subsection{Boundary degenerations}\label{Subsection Wave packets - Boundary degenerations}
For the proof of Theorem \ref{Thm temperedness} we will use the theory of the constant term as developed in \cite{DelormeKrotzSouaifi_ConstantTerm}. In the theory certain degenerations of $Z$ play an important role. We recall here the necessary definitions and results.

The closure of the compression cone $\overline{\cC}$ is finitely generated and hence polyhedral as $-\cC^{\vee}$ is finitely generated.
We call a subset $\cF\subseteq\overline{\cC}$ a face of $\overline{\cC}$ if $\cF=\overline{\cC}$ or there exists a closed half-space $\cH$ so that
$$
\cF
=\overline{\cC}\cap\cH
\quad\text{and}\quad
\cC\cap\partial\cH=\emptyset.
$$
There exist finitely many faces of $\overline{\cC}$ and each face is polyhedral cone. For a face $\cF$ of $\overline{\cC}$ we define
$$
\fa_{\cF}
:=\spn(\cF)
$$
and denote the interior of $\cF$ in $\fa_{\cF}$  by $\cF^{\circ}$.

Let $z\in Z$ be adapted and let $\cF$ be a face of $\overline{\cC}$. By \cite[ Lemma 8.1]{KuitSayag_OnTheLittleWeylGroupOfARealSphericalSpace} the limits $\fh_{z,X}$ are the same for all $X$ in the interior of $\cF$. We may thus define
$$
\fh_{z,\cF}
:=\fh_{z,X},
$$
where $X$ is any element in the interior of $\cF$.

The following lemma  is \cite[Lemma 8.3]{KuitSayag_OnTheLittleWeylGroupOfARealSphericalSpace}.

\begin{Lemma}\label{Lemma properties fh_(z,cF)}
Let $z\in Z$ be adapted and let $\cF$ be a face of $\overline{\cC}$. The Lie algebra $\fh_{z,\cF}$ is a real spherical subalgebra of $\fg$. Moreover,
$$
N_{\fg}(\fh_{z,\cF})
=\fh_{z,\cF}+\fa_{\cF}+N_{\fm}(\fh_{z,\cF}).
$$
Finally,
$$
\fh_{z,\cF}\cap\fa
=\fa_{\fh}.
$$
\end{Lemma}

For an adapted point $z\in Z$ and a face $\cF$ of $\overline{\cC}$ we define $H_{z,\cF}$ to be the connected subgroup of $G$ with Lie algebra $\fh_{z,\cF}$.
Each subgroup $H_{z,\cF}$ equals the connected component of a group of real points of an algebraic subgroup of $\underline{G}$, namely the subgroups $H_{I,c}$ defined in \cite[Section 4.5]{DelormeKnopKrotzSchlichtkrull_PlancherelTheoryForRealSphericalSpacesConstructionOfTheBernsteinMorphisms}.
We write $Z_{z,\cF}$ for the homogeneous space $G/H_{z,\cF}$. These spaces are called the boundary degenerations of $Z$.
Since $\fa_{\cF}$ normalizes $\fh_{z,\cF}$, the group $A_{\cF}:=\exp(\fa_{\cF})$ normalizes $H_{z,\cF}$.

One boundary degeneration will be of particular interest to us when we come to Section \ref{Section Most continuous part}: the boundary degeneration for the face $\cF=\overline{\cC}$. If $z\in Z$ is an adapted point so that $M\cap H_{z}=M\cap H$, then $\fh_{z,X}=\fh_{\emptyset}$ for all $X\in \cC$. Therefore, the group $H_{z,\overline{\cC}}$ is in this case the connected component of the subgroup
$$
H_{\emptyset}
:=(L_{Q}\cap H)\overline{N}_{Q}.
$$

\subsection{Preparation for the proof of Theorem \ref{Thm temperedness}}
\label{Subsection Wave packets - Preparation}
The proof of Theorem \ref{Thm temperedness} relies heavily on the theory of the constant term as developed in \cite{DelormeKrotzSouaifi_ConstantTerm}. In this section we recall the necessary objects and results, which we will then use in the next section to prove the theorem. We first discuss the algebras of invariant differential operators on $Z$ and its boundary degenerations. We then give the differential equations satisfied by the matrix-coefficients. Finally, we introduce the notion of $\cF$-piece-wise linear functionals and construct an $\cF$-piece-wise linear functional $\beta_{\cF,\lambda}$, which will be used to improve the a priori estimate on the matrix-coefficients from Lemma \ref{Lemma A priori estimate}.

We fix an adapted point $z\in Z$. In this and the next section we will suppress the indices $z$ and simply write $Z_{\cF}$, $H_{\cF}$ and $\fh_{\cF}$ for $Z_{z,\cF}$, $H_{z,\cF}$ and $\fh_{z,\cF}$, respectively.

We now follow \cite[Section 5]{DelormeKrotzSouaifi_ConstantTerm}.

Let $\fb=\fm\oplus\fa\oplus\fn_{Q}$ and $\fb_{H}=(\fm\cap\fh_{z})\oplus\fa_{\fh}$. Now $\cU(\fb)\fb_{H}$ is a two-sided ideal of $\cU(\fb)$.
We recall from \cite[(5.4)]{DelormeKrotzSouaifi_ConstantTerm} that the rings $\Diff(Z)$ and $\Diff(Z_{\cF})$ of $G$-invariant differential operators on $Z$ and $Z_{\cF}$, respectively,  may be identified with subalgebras of $\cU(\fb)/\cU(\fb)\fb_{H}$.  By \cite[Lemma 5.2]{DelormeKrotzSouaifi_ConstantTerm} the limit
$$
\lim_{t\to\infty}\Ad\big(\exp(tX)\big)D
$$
exists for every $D\in \Diff(Z)$ and $X\in \cF^{\circ}$ and defines a $G$-invariant differential operator on $Z_{\cF}$. The limit does not depend on the choice of $X$. Moreover, the map
$$
\delta_{\cF}:\Diff(Z)\to\Diff(Z_{\cF});\quad D\mapsto \lim_{t\to\infty}\Ad\big(\exp(tX)\big)D
$$
is an injective algebra morphism.
The $\fa$-weights occurring in $\delta_{\cF}(D)-D$, with $D\in \Diff(Z)$, considered as an element of $\cU(\fb)/\cU(\fb)\fb_{H}$, are strictly negative on $\cF^{\circ}$ and the $\fa$-weights occurring in $\delta_{\cF}(D)$ are non-positive on $\overline{\cC}$.

Every element $u$ of the center $\cZ(\fg)$ of $\cU(\fg)$ determines a differential operator $D_{u}\in\Diff(Z_{\cF})$. We write $\Diff_{0}(Z_{\cF})$ for the image of $\cZ(\fg)$ in $\Diff(Z_{\cF})$.
By \cite[Lemma 5.6]{DelormeKrotzSouaifi_ConstantTerm} the ring  $\Diff(Z_{\cF})$ is finitely generated as a $\Diff_{0}(Z_{\cF})$-module. Let $V_{\cF}$ be a finite dimensional vector subspace of $\Diff(Z_{\cF})$ so that the linear map
$$
\Diff_{0}(Z_{\cF})\otimes V_{\cF}\to\Diff(Z_{\cF});\quad \sum_{j}D_{j}\otimes u_{j}\mapsto \sum_{j}D_{j}u_{j}
$$
is surjective.

For $\lambda\in (\fa/\fa_{\fh})^{*}_{\C}$ we write $\cI_{\lambda}$ and  $\cI_{\cF,\lambda}$ for the ideals of $\Diff(Z)$ and $\Diff(Z_{\cF})$, respectively, generated by the elements of the form $D_{u}-\chi_{\lambda}(u)$ with $u\in\cZ(\fg)$, where $\chi_{\lambda}:\cZ(\fg)\to\C$ is the infinitesimal character of $\cD'(\overline{Q}:\xi:\lambda)$. Now $\cI_{\cF,\lambda}=\delta_{\cF}(\cI_{\lambda})$. As
$$
\Diff_{0}(Z_{\cF})
=\C+\cI_{\cF,\lambda},
$$
we have
$$
\Diff(Z_{\cF})
=(\C+\cI_{\cF,\lambda})V_{\cF}.
$$
Since
\begin{equation}\label{eq I_F U= D(Z_F) I_F}
\cI_{\cF,\lambda}V_{\cF}
=\cI_{\cF,\lambda}\Diff_{0}(Z_{\cF})V_{\cF}
=\cI_{\cF,\lambda}\Diff(Z_{\cF})
=\Diff(Z_{\cF})\cI_{\cF,\lambda},
\end{equation}
we find
$$
\Diff(Z_{\cF})
=V_{\cF}+\Diff(Z_{\cF})\cI_{\cF,\lambda}.
$$

For every $\lambda\in (\fa/\fa_{\fh})^{*}_{\C}$ there exists a subspace $U_{\cF}$ of $V_{\cF}$ so that the sum
\begin{equation}\label{eq Assumption on B}
\Diff(Z_{\cF})
=U_{\cF}\oplus\Diff(Z_{\cF})\cI_{\cF,\lambda}
\end{equation}
is direct sum of vector spaces. As $V_{\cF}$ is finite dimensional and $V_{\cF}\cap \Diff(Z_{\cF})\cI_{\cF,\lambda}$ depends continuously on $\lambda$, the subspace $U_{\cF}$ can in fact be chosen locally uniformly with respect to $\lambda$, i.e.,  every $\lambda \in (\fa/\fa_{\fh})^{*}_{\C}$ has an open neighborhood $\fB$ in $(\fa/\fa_{\fh})^{*}_{\C}$ so that there exists a subspace $U_{\cF}$ of $V_{\cF}$ for which (\ref{eq Assumption on B}) holds for all $\lambda\in\fB$.

We define
$$
\rho_{\cF,\lambda}:\Diff(Z_{\cF})\to \End(U_{\cF})
$$
to be the map determined by
\begin{equation}\label{eq Def rho_F}
Du\in
\rho_{\cF,\lambda}(D)u+\Diff(Z_{\cF})\cI_{\cF,\lambda}
\qquad\big(D\in \Diff(Z_{\cF}), u\in \cU_{\cF}\big).
\end{equation}
Then $\rho_{\cF,\lambda}$ defines a representation of $\Diff(Z_{\cF})$ on $U_{\cF}$ which is isomorphic to the canonical representation of $\Diff(Z_{\cF})$ on $\Diff(Z_{\cF})/\Diff(Z_{\cF})\cI_{\cF,\lambda}$, and $\rho_{\cF,\lambda}$ depends polynomially on $\lambda$.

There exists a natural injective algebra homomorphism
\begin{equation}\label{eq S(a_F) hookrightarrow Diff(Z_F)}
S(\fa_{\cF})\hookrightarrow \Diff(Z_{\cF}); \quad X\mapsto D_{X},
\end{equation}
which is determined by
$$
D_{X}f(z)=\frac{d}{dt}f\big(z\cdot \exp(tX)\big)\big|_{t=0}
\qquad\big(X\in \fa_{\cF}, f\in \cE(Z), z\in Z\big).
$$
In view of (\ref{eq S(a_F) hookrightarrow Diff(Z_F)}) the $\Diff(Z_{\cF})$-representation $\rho_{\cF,\lambda}$ induces a Lie algebra homomorphism
$$
\Gamma_{\cF,\lambda}:\fa_{\cF} \to  \operatorname{End} (U_{\cF}^*);
\quad X\mapsto \rho_{\cF,\lambda}(D_{X})^{t}.
$$
We note that $\Gamma_{\cF,\lambda}$ depends polynomially on $\lambda$.

We now use this machinery to analyse the matrix-coefficients $m_{\phi,\mu}$ for $\mu\in\cD'(\overline{Q}:\xi:\lambda)^{H}$ and $\phi\in\cD(G,V_{\xi})$. To do so we define the map
$$
\Phi_{\mu,\phi}:A\to U_{\cF}^{*}
$$
by setting
$$
\Big(\Phi_{\mu,\phi}(a)\Big)(u)
=\big(R(u)m_{\phi,\mu}\big)(a)
\qquad(a\in A, u\in U_{\cF}).
$$
The function $\Phi_{\mu,\phi}$ satisfies the system of differential equations
\begin{equation}\label{eq ODE Phi}
\partial_{X}\Phi_{\mu,\phi}
= \Gamma_{\cF,\lambda}(X)  \Phi_{\mu,\phi} + \Psi_{\mu,\phi,X}   \qquad (X\in \fa_{\cF})
\end{equation}
with
$$
\Psi_{\mu,\phi,X}:A\to U_{\cF}^{*}
$$
given by
$$
\Psi_{\mu,\phi,X}(a)(u)
=\Big( R\big(Xu-\rho_{\cF,\lambda}(X)u\big) m_{\phi,\mu}\Big)(a)\qquad (a\in A, u\in U_{\cF}).
$$
As in \cite[Lemma 5.7]{DelormeKrotzSouaifi_ConstantTerm} we may solve the ordinary differential equation (\ref{eq ODE Phi}) and obtain that for all $a\in A$ and $X\in\fa_{\cF}$,
\begin{equation}\label{eq Formula Phi}
\Phi_{\mu,\phi}(a\exp(X))
=e^{\Gamma_{\cF,\lambda}(X)}\Phi_{\mu,\phi}(a)+\int_{0}^{1}e^{(1-s)\Gamma_{\cF,\lambda}(X)}\Psi_{\mu,\phi,X}\big(a\exp(sX)\big)\,ds.
\end{equation}

Let $\cQ_{\cF,\lambda}$ be the set of generalized $\fa_{\cF}$-weights of $\Gamma_{\cF,\lambda}$.
For $\nu\in\cQ_{\cF,\lambda}$ we define $E_{\nu}\in\End(\cU_{\cF}^{*})$ to be the projection onto the generalized eigenspace with eigenvalue $\nu$. We further define
$$
\Phi_{\mu,\phi}^{\nu}
:=E_{\nu}\circ\Phi_{\mu,\phi}
\qquad\big(\phi\in\cD(G,V_{\xi})\big).
$$
In view of (\ref{eq Formula Phi}) we have
$$
\Phi_{\mu,\phi}^{\nu}(a\exp(X))
=e^{\Gamma_{\cF,\lambda}(X)}\Phi_{\mu,\phi}^{\nu}(a)+\int_{0}^{1}E_{\nu}e^{(1-s)\Gamma_{\cF,\lambda}(X)}\Psi_{\mu,\phi,X}\big(a\exp(sX)\big)\,ds
$$
for all $\phi\in\cD(G,V_{\xi})$, $\nu\in\cQ_{\cF,\lambda}$, $a\in A$ and $X\in\fa_{\cF}$.

If $X\in\fa_{\cF}$ and $u\in U_{\cF}$, then in view of (\ref{eq Def rho_F}) and  (\ref{eq I_F U= D(Z_F) I_F}) the element $Xu-\rho_{\cF,\lambda}(X)u$ is contained in $V_{\cF}\cI_{\cF,\lambda}=V_{\cF}\delta_{\cF}(\cI_{\lambda})$. Therefore, if $u_{1},\dots, u_{n}$ is a basis of $V_{\cF}$, then there exist  bilinear maps
$$
\omega_{\cF,\lambda}^{i}:\fa_{\cF}\times U_{\cF}\to \cI_{\lambda}
\qquad(1\leq i\leq n)
$$
so that
$$
Xu-\rho_{\cF,\lambda}(X)u
=\sum_{i=1}^{n}u_{i}\delta_{\cF}\big(\omega_{\cF,\lambda}^{i}(X,u)\big)
\qquad\big(X\in \fa_{\cF}, u\in U_{\cF}\big).
$$
We denote by $\Xi_{\cF,\lambda}$ the finite set of all $\fa$-weights that occur in
$$
\big\{\delta_{\cF}\big(\omega_{\cF,\lambda}^{i}(X,u)\big)-\omega_{\cF,\lambda}^{i}(X,u):1\leq i\leq n, X\in \fa_{\cF}, u\in U_{\cF}\big\}.
$$
We recall from Section \ref{Subsection Orbits of max rank - Admissible points and little Weyl group} that the image of $\overline{\cC}$ under the projection $\fa\to\fa/\fa_{E}$ is a Weyl chamber of the spherical root system $\Sigma_{Z}$ in $(\fa/\fa_{E})^{*}$.  Let $\Sigma_{Z}^{+}$ be the positive system of $\Sigma_{Z}$ so that this Weyl chamber is the negative one.
We then define $\beta_{\cF,\lambda}$ on $\fa$ by
$$
\beta_{\cF,\lambda}(X)
:=\max_{\nu\in\Xi_{\cF,\lambda}\cup \Sigma_{Z}^{+}}\nu(X),\quad (X\in \fa).
$$
Because of the signs of the $\fa$-weights occurring in $\Diff(Z_{\cF})$ and elements of the form $\delta_{\cF}(D)-D$ with $D\in \Diff(Z)$, we have $\beta_{\cF,\lambda}\big|_{\overline{\cC}}\leq 0$ and $\beta_{\cF,\lambda}\big|_{\cF^{\circ}}< 0$. The maximum in the definition of $\beta_{\cF,\lambda}$ also runs over the set of positive roots in $\Sigma_{Z}$; this is to ensure that $\beta_{\cF,\lambda}$ vanishes on the edge of $\cF$. We do not actually need this, but we simply follow the definition in \cite[(5.23)]{DelormeKrotzSouaifi_ConstantTerm}. Note that $\beta_{\cF,\lambda}$ depends polynomially on $\lambda$.

\subsection{Proof of Theorem \ref{Thm temperedness}}
\label{Subsection Wave packets - Proof of Theorem}
We continue with the notation from the previous section.
We recall that the proof of Theorem \ref{Thm temperedness} is by induction on the faces $\cF$ of $\overline{\cC}$.

Let $\cF$ be a face of $\overline{\cC}$. We call a function $\zeta: \fa\to\R$ an $\cF$-piecewise linear functional on $\fa$ if it is piecewise linear and satisfies
$$
\zeta\big|_{\partial \cF}=\rho_{Q}\big|_{\partial\cF}.
$$
Here $\partial\cF$ denotes the union of the faces $\cF'$ of $\overline{\cC}$ with $\cF'\subsetneq\cF$.
Note that for any two $\cF$-piecewise linear functionals $\zeta$ and $\zeta'$, also the functions
$$
\fa\ni X\mapsto \max\big(\zeta(X), \zeta'(X)\big)
\quad\text{and}\quad
\fa\ni X\mapsto \min\big(\zeta(X), \zeta'(X)\big)
$$
are $\cF$-piecewise linear functionals on $\fa$. Moreover,
$$
\fa=\fa_{\cF}\times\big(\fa\cap \fa_{\cF}^{\perp}\big)\ni (X,Y)\mapsto \zeta(X)+\zeta'(Y)
$$
defines an $\cF$-piecewise linear functional as well.

For an $\cF$-piece-wise linear functional $\zeta$ we decompose $\cQ_{\cF,\lambda}$ as
$$
\cQ_{\cF,\lambda}
=\cQ_{\cF,\lambda}^{\zeta,+}\cup\cQ_{\cF,\lambda}^{\zeta, 0}\cup\cQ_{\cF,\lambda}^{\zeta, -},
$$
with
\begin{align*}
\cQ_{\cF,\lambda}^{\zeta,+}
&:=\big\{\nu\in \cQ_{\cF,\lambda}:\Re\nu(X)>\zeta(X)+\langle\lambda\rangle(X) \text{ for some } X\in\cF^{\circ}\big\},\\
\cQ_{\cF,\lambda}^{\zeta, -}
&:=\big\{\nu\in \cQ_{\cF,\lambda}:\Re\nu<\zeta+\langle\lambda\rangle\big|_{\cF^{\circ}}\big\},\\
\cQ_{\cF,\lambda}^{\zeta, 0}
&:=\cQ_{\cF,\lambda}\setminus\big(\cQ_{\cF,\lambda}^{\zeta,+}\cup\cQ_{\cF,\lambda}^{\zeta,-}\big)\\
&\:=\big\{\nu\in \cQ_{\cF,\lambda}: \Re\nu\leq\zeta+\langle\lambda\rangle\big|_{\cF}\\
&\qquad \qquad\qquad\text{ and }
    \Re\nu(X)=\zeta(X)+\langle\lambda\rangle(X)\text{ for some } X\in\cF^{\circ}\big\}.
\end{align*}
The following lemma is needed for the induction step.

\begin{Lemma}\label{Lemma improvement of estimate}
Let $\cF$ be a face of $\overline{\cC}$ and $\xi$ a finite dimensional unitary representation of $M_{Q}$. Let further $\fC$ be a compact subset of $\{\lambda\in (\fa/\fa_{\fh})^{*}_{\C}:\Im\lambda\notin\cS\}$ such that there exists a subspace $U_{\cF}$ of $V_{\cF}$ for which (\ref{eq Assumption on B}) holds for all $\lambda\in\fC$. Let $\zeta:\fC\times\fa\to\R$ be a continuous function so that $\zeta_{\lambda}:=\zeta(\lambda, \dotvar)$ is an $\cF$-piecewise linear functional on $\fa$ for all $\lambda\in \fC$.  Assume that there exists an $N\in \N_{0}$ and a continuous seminorm $p$ on $\cD(G,V_{\xi})$ so that for every $\lambda\in \fC$,  $\mu\in \cD'(\overline{Q}:\xi:\lambda)^{H}$ and $\phi\in\cD(G,V_{\xi})$
\begin{equation} \label{eq initial estimate}
\big|m_{\phi,\mu}\big(\omega \exp(X)\big)\big|
\leq  e^{\big(\zeta_{\lambda}+\langle\lambda\rangle\big)(X)}(1+\|X\|)^{N} \|\mu\| p(\phi)
\qquad\big(\omega\in\Omega, X\in\overline{\cC}\big).
\end{equation}
Then there exists an $N'\in\N_{0}$, a continuous semi-norm $p'$ on $\cD(G,V_{\xi})$,  and a continuous function $\zeta':\fC\times\fa\to\R$
so that
\begin{enumerate}[(i)]
\item $\zeta'_{\lambda}:=\zeta'(\lambda, \dotvar)$ is an $\cF$-piecewise linear functional for all $\lambda\in \fC$
\item
$\zeta_{\lambda}'\big|_{\cF}
=\max\Big(
    \big\{\zeta_{\lambda}+\frac{1}{2}\beta_{\cF,\lambda},\rho_{Q}\big\}
    \cup\big\{\Re\nu-\langle\lambda\rangle:\nu\in \cQ_{\cF,\lambda}^{\zeta_{\lambda},-}\big\}
    \Big)\Big|_{\cF}$  for all $\lambda\in\fC$
\item For all $\lambda\in\fC$, $\mu\in\cD'(\overline{Q}:\xi:\lambda)^{H}$ and $\phi\in\cD(G,V_{\xi})$
$$
\big|m_{\phi,\mu}\big(\omega\exp(X)\big)\big|
\leq e^{\big(\zeta_{\lambda}'+\langle\lambda\rangle\big)(X)}(1+\|X\|)^{N'}\|\mu\|p'(\phi)
\qquad   \big(\omega\in\Omega, X\in\overline{\cC}\big).
$$
\end{enumerate}
\end{Lemma}

\begin{proof}
We apply Lemma \ref{Lemma A priori estimate} to $Z_{\cF}$ instead of $Z$, and find that there exists an $N_{\cF}\in\N_{0}$, a continuous semi-norm $p_{\cF}$ on $\cD(G,V_{\xi})$ and a $\zeta_{\cF}\in \fa^{*}$ so that
$$
\zeta_{\cF}\big|_{\fa_{\cF}}
=\rho_{Q}\big|_{\fa_{\cF}}
$$
and for every $\mu_{\cF}\in \cD'(\overline{Q}:\xi:\lambda)^{H_{\cF}}$ and $\phi\in \cD(G,V_{\xi})$
\begin{equation}\label{eq Estimate matrix-coefficient for Z_F}
\big|m_{\phi,\mu_{\cF}}\big(\exp(X)\big)\big|
\leq e^{\big(\zeta_{\cF}+\langle\lambda\rangle\big)(X)}(1+\|X\|)^{N_{\cF}}\|\mu_{\cF}\|p_{\cF}(\phi)
\qquad(\lambda\in \fC, X\in\overline{\cC}).
\end{equation}
We define $\zeta':\fC\times\fa\to\R$ by requiring that for all $X\in\fa_{\cF}$ and $Y\in \fa\cap \fa_{\cF}^{\perp}$
\begin{align*}
&\zeta'(\lambda,X)
=\max\Big(
    \big\{\zeta_{\lambda}(X)+\frac{1}{2}\beta_{\cF,\lambda}(X),\rho_{Q}(X)\big\}
    \cup\big\{\Re\nu(X)-\langle\lambda\rangle(X):\nu\in \cQ_{\cF,\lambda}^{\zeta_{\lambda},-}\big\}
    \Big),\\
&\zeta'(\lambda,Y)
=\max\big(\zeta_{\lambda}(Y),\zeta_{\cF}(Y)\big),\\
&\zeta'(\lambda,X+Y)
=\zeta'(\lambda,X)+\zeta'(\lambda,Y).
\end{align*}
Observe that $\zeta'$ is continuous and $\zeta'(\lambda,\dotvar)$ defines for every $\lambda\in\fC$ an $\cF$-piecewise linear functional on $\fa$. Therefore, it suffices to prove the existence of an $N'\in\N_{0}$ and a continuous semi-norm $p'$ on $\cD(G,V_{\xi})$ such that for every  $\lambda\in\fC$,  $\mu\in\cD'(\overline{Q}:\xi:\lambda)^{H}$, $\phi\in\cD(G,V_{\xi})$ and $\nu\in\cQ_{\cF,\lambda}$
$$
\|\Phi_{\mu,\phi}^{\nu}\big(\exp(X)\big)\|
\leq e^{\big(\zeta_{\lambda}'+\langle\lambda\rangle\big)(X)}(1+\|X\|)^{N'}\|\mu\|p'(\phi)
\qquad(X\in\overline{\cC}).
$$
For this we follow \cite[Sections 5.3, 5.4, 6.2 \& 6.3]{DelormeKrotzSouaifi_ConstantTerm}.

If one uses the estimate (\ref{eq initial estimate}) instead of the tempered estimates, then in the same way as in the proof for \cite[Lemma 5.8]{DelormeKrotzSouaifi_ConstantTerm} it follows that there exists a continuous semi-norm $q$ on $\cD(G,V_{\xi})$ such that for all $\lambda\in\fC$,  $\mu\in\cD'(\overline{Q}:\xi:\lambda)^{H}$, and $\phi\in\cD(G,V_{\xi})$
\begin{equation}\label{eq Phi estimate - 1}
\| L(v) \Phi_{\mu,\phi}\big(\exp(X)\big)\|
\leq  e^{\big(\zeta_{\lambda}+\langle\lambda\rangle\big)(X)}(1+\|X\|)^{N}\|\mu\|q\big(L(v)\phi\big)
\qquad \big(v\in \cU(\fa), X\in\overline{\cC}\big),
\end{equation}
and for every compact subset $B\subseteq \fa$, there exists a constant $C>0$ so that
\begin{equation}\label{eq Psi estimate}
\| L(v) \Psi_{\mu,\phi,X}(\exp(Y))\|
\leq  Ce^{\big(\zeta_{\lambda}+\beta_{\cF,\lambda}+\langle\lambda\rangle\big)(Y)}(1+\|Y\|)^{N} \|\mu\|\|X\| q\big(L(v)\phi\big)
\end{equation}
for all $v\in \cU(\fa)$, $X\in \fa_{\cF}$ and $Y\in B+\overline{\cC}$.
Let
$$
E_{\nu}(\lambda, X)
:=e^{-\nu(X)}E_{\nu}\circ\Gamma_{\cF,\lambda}(X)
\qquad\big(\lambda\in\fC,\nu\in\cQ_{\cF,\lambda},X\in\fa_{\cF}\big).
$$
The same arguments as the ones for \cite[Lemma 5.9]{DelormeKrotzSouaifi_ConstantTerm} show the existence of constants $c>0$ and $n\in\N_{0}$ so that
\begin{equation}\label{eq E_nu estimate}
\|E_{\nu}(\lambda, X)\|
\leq c(1+\|X\|)^{n}
\qquad\big(\lambda\in\fC,\nu\in\cQ_{\cF,\lambda},X\in\fa_{\cF}\big).
\end{equation}
We define $\delta_{\lambda}:\fa_{\cF}\to[0,\frac{1}{2}]$ by
$$
\delta_{\lambda}(X)
=\min\bigg(\Big\{\frac{\Re\nu(X)-\zeta_{\lambda}(X)-\langle\lambda\rangle(X)}{\beta_{\cF,\lambda(X)}}:\nu\in\cQ_{\cF,\lambda}^{\zeta_{\lambda},-}\Big\}\cup\Big\{ \frac{1}{2}\Big\}\bigg)
\qquad(X\in\cF^{\circ}).
$$
In view of (\ref{eq Phi estimate - 1}), (\ref{eq Psi estimate}) and (\ref{eq E_nu estimate}) it suffices to prove that for every $\lambda\in\fC$,  $\mu\in\cD'(\overline{Q}:\xi:\lambda)^{H}$, $\phi\in\cD(G,V_{\xi})$ and $\nu\in\cQ_{\cF,\lambda}$ there exists a function $\Phi_{\mu,\phi}^{\nu,\infty}:A\to U_{\cF}^{*}$ so that
\begin{align}\label{eq Phi estimate - 2}
&\big\|\Phi_{\mu,\phi}^{\nu}\big(a\exp(tX)\big)-\Phi_{\mu,\phi}^{\nu,\infty}\big(a\exp(tX)\big)\big\|\\
\nonumber&\leq e^{\big(\zeta_{\lambda}+\delta_{\lambda}\beta_{\cF,\lambda}+\langle\lambda\rangle\big)(tX)}\bigg(
    \|E_{\nu}(\lambda, tX)\|\|\Phi_{\mu,\phi}(a)\|\\
\nonumber &\qquad\qquad+
    \int_{0}^{\infty}e^{-\big(\zeta_{\lambda}+\frac{1}{2}\beta_{\cF,\lambda}+\langle\lambda\rangle\big)(sX)}\big\|E_{\nu}\big(\lambda, (t-s)X\big)\big\|
        \big\|\Psi_{\mu,\phi,X}\big(a\exp(sX)\big)\big\|\,ds\bigg)
\end{align}
for all $a\in A$, $X\in\cF^{\circ}$ and $t\geq0$, and
\begin{equation}\label{eq Phi^infty estimate}
\big\|\Phi_{\mu,\phi}^{\nu,\infty}\big(\exp(X)\big)\big\|
\leq e^{\big(\zeta_{\cF}+\langle\lambda\rangle\big)(X)}(1+\|X\|)^{N_{\infty}}\|\mu\|q_{\infty}(\phi)
\qquad(X\in\overline{\cC})
\end{equation}
for some $N_{\infty}\in\N_{0}$ and a continuous semi-norm $q_{\infty}$ on $\cD(G,V_{\xi})$.

If $\nu\in\cQ_{\cF,\lambda}^{\zeta_{\lambda},+}$ or $\nu\in\cQ_{\cF,\lambda}^{\zeta_{\lambda},-}$, then we may take $\Phi_{\mu,\phi}^{\nu,\infty}=0$. The estimate (\ref{eq Phi estimate - 2}) is the analogue of \cite[Lemma 5.11]{DelormeKrotzSouaifi_ConstantTerm} and is obtained as in \cite[Corollary 5.16 \& Lemma 5.17]{DelormeKrotzSouaifi_ConstantTerm} and \cite[Lemma 5.18]{DelormeKrotzSouaifi_ConstantTerm}, respectively.

Let  $\nu\in\cQ_{\cF,\lambda}^{\zeta_{\lambda},0}$. If $Y\in\cF^{\circ}$ with
$\Re\nu(Y)> \zeta_{\lambda}(Y)+\beta_{\cF,\lambda}(Y)+\langle\lambda\rangle(Y)$,
then it follows as in \cite[Section 5.4.1]{DelormeKrotzSouaifi_ConstantTerm} that the limit
$$
\Phi_{\mu,\phi}^{\nu,\infty}(a)
:=\lim_{t\to\infty}e^{-t\Gamma_{\cF,\lambda}(Y)}\Phi_{\mu,\phi}^{\nu}\big(a\exp(tY)\big)
\qquad(a\in A)
$$
exists and is independent of the choice of $Y$. Note that such $Y$ exist because $\nu\in\cQ_{\cF,\lambda}^{\zeta_{\lambda},0}$
and $\beta_{\cF,\lambda}|_{\cF^{\circ}<0}$.
We first show that (\ref{eq Phi^infty estimate}) is satisfied.
For $\mu\in\cD'(\overline{Q}:\xi:\lambda)^{H}$ we define
$$
\mu_{\cF,\nu}:\cD(G,V_{\xi})\to\C;\quad \phi\mapsto \Phi_{\mu,\phi}^{\nu,\infty}(e)(1).
$$
From the definitions it is easily seen that $\mu_{\cF,\zeta_{\lambda}}$ is a distribution and belongs to $\cD'(\overline{Q}:\xi:\lambda)$. We claim that $\mu_{\cF,\nu}$ is right $H_{\cF}$-invariant.
To prove the claim we choose $Y\in\cF^{\circ}$ so that
$\zeta_{\lambda}(Y)+\langle\lambda\rangle(Y)=\Re\nu(Y)$.
From (\ref{eq Phi estimate - 2}) it follows that
$$
\lim_{t\to\infty}e^{-t\nu(Y)}\|\Phi_{\mu,\phi}^{\nu}\big(\exp(tY)\big)-\Phi_{\mu,\phi}^{\nu,\infty}\big(\exp(tY)\big)\|
=0.
$$
Moreover, if we use this $Y$ in the proof of \cite[Lemma 6.2]{DelormeKrotzSouaifi_ConstantTerm}, then we obtain
$$
\Phi_{\mu,\phi}^{\nu,\infty}\big(a\exp(X)\big)
=e^{\Gamma_{\cF,\lambda}(X)}\Phi_{\mu,\phi}^{\nu,\infty}(a)
\qquad(a\in A, X\in\fa_{\cF}).
$$
Now the claim follows with the same arguments as in the proof for \cite[Lemma 6.5(iii)]{DelormeKrotzSouaifi_ConstantTerm}.
The estimate (\ref{eq Phi^infty estimate})  follows from the claim and (\ref{eq Estimate matrix-coefficient for Z_F}).

We finish the proof by showing that (\ref{eq Phi estimate - 2}) holds also in this case.
If $\nu\in\cQ_{\cF,\lambda}^{\zeta_{\lambda},0}$ and $X\in\cF^{\circ}$ with
$\Re\nu(X)\leq \zeta_{\lambda}(X)+\frac{1}{2}\beta_{\cF,\lambda}(X)+\langle\lambda\rangle(X)$,
then the estimate follows from (\ref{eq Phi^infty estimate}) and the estimate on $\Phi_{\mu,\phi}^{\nu}\big(a\exp(tX)\big)$ that one obtains analogous to \cite[Lemma 5.17]{DelormeKrotzSouaifi_ConstantTerm}.
If $X\in \cF^{\circ}$ with
$$
\zeta_{\lambda}(X)+\frac{1}{2}\beta_{\cF,\lambda}(X)+\langle\lambda\rangle(X)
\leq \Re\nu(X)
\leq \zeta_{\lambda}(X)+\langle\lambda\rangle(X),
$$
then we find as in the proof for \cite[Lemma 5.19]{DelormeKrotzSouaifi_ConstantTerm}
that
$$
\Phi_{\mu,\phi}^{\nu,\infty}\big(a\exp(tX)\big)
=\Phi_{\mu,\phi}^{\nu}\big(a\exp(tX)\big)+\int_{t}^{\infty}E_{\nu}e^{(t-s)\Gamma_{\cF,\lambda}(X)}\Psi_{\mu,\phi,X}\big(a\exp(sX)\big)\,ds.
$$
Now (\ref{eq Phi estimate - 2}) follows from (\ref{eq Psi estimate}) and (\ref{eq E_nu estimate}).
(The fact that $\delta_{\lambda}$ is not a constant like in \cite[(5.37)]{DelormeKrotzSouaifi_ConstantTerm} is irrelevant for the proof.)
\end{proof}

\begin{proof}[Proof of Theorem \ref{Thm temperedness}]
We prove the theorem using the principle of induction on the faces $\cF$ of $\overline{\cC}$. In particular we will show that for every face $\cF$ of $\overline{\cC}$ there exists an $N\in \N_{0}$, a continuous seminorm $p$ on $\cD(G,V_{\xi})$ and a continuous function $\zeta:\fC\times \fa\to\R$
so that
\begin{enumerate}[(i)]
\item\label{eq Estimate for cF - item 1} $\zeta_{\lambda}:=\zeta(\lambda, \dotvar)$ is an $\cF$-piecewise linear functional for all $\lambda\in \fC$,
\item\label{eq Estimate for cF - item 2} $\zeta_{\lambda}\big|_{\cF}=\rho_{Q}\big|_{\cF}$,
\item\label{eq Estimate for cF - item 3} For all $\lambda\in\fC$, $\mu\in\cD'(\overline{Q}:\xi:\lambda)^{H}$ and $\phi\in\cD(G,V_{\xi})$
$$
\big|m_{\phi,\mu}\big(\omega\exp(X)\big)\big|
\leq e^{\big(\zeta_{\lambda}+\langle\lambda\rangle\big)(X)}(1+\|X\|)^{N}\|\mu\|p(\phi)
\qquad   \big(\omega\in\Omega, X\in\overline{\cC}\big).
$$
\end{enumerate}
Lemma \ref{Lemma A priori estimate} serves as the initial step with $\cF=\fa_{E}$.
Let $\cF$ be a face of $\overline{\cC}$ with $\cF\neq\fa_{E}$, and assume that for every face $\cF'$ of $\overline{\cC}$ with $\cF'\subsetneq \cF$  there exists an $N_{\cF'}\in \N_{0}$, a continuous seminorm $p_{\cF'}$ on $\cD(G,V_{\xi})$ and a continuous function $\zeta_{\cF'}:\fC\times \fa\to\R$
so that (\ref{eq Estimate for cF - item 1}) -- (\ref{eq Estimate for cF - item 3}) hold with $\cF'$ in place of $\cF$. Without loss of generality may assume that $\zeta_{\cF'}(\lambda,X)\geq \rho_{Q}(X)$ for all $\lambda\in\fC$ and $X\in \cF$.
We define
$$
\zeta_{0}:\fC\times\fa\to\R;\quad (\lambda,X)\mapsto \min_{\cF'\subsetneq\cF}\zeta_{\cF'}(\lambda,X)
$$
Then $\zeta_{0}$ is a continuous function with the property that $\zeta_{0}(\lambda,\dotvar)$ is an $\cF$-piecewise linear functional on $\fa$ and there exists an $N_{0}\in \N_{0}$ and a continuous seminorm $p_{0}$ on $\cD(G,V_{\xi})$ so that  for all $\lambda\in\fC$, $\mu\in\cD'(\overline{Q}:\xi:\lambda)^{H}$ and $\phi\in\cD(G,V_{\xi})$
$$
\big|m_{\phi,\mu}\big(\omega\exp(X)\big)\big|
\leq e^{\zeta_{0}(\lambda,X)+\langle\lambda\rangle(X)}(1+\|X\|)^{N_{0}}\|\mu\|p_{0}(\phi)
\qquad   \big(\omega\in\Omega, X\in\overline{\cC}\big).
$$
We use Lemma \ref{Lemma improvement of estimate} to improve this estimate.

After passing to a finite cover of $\fC$ of sufficiently small compact subsets of $\{\lambda\in (\fa/\fa_{\fh})^{*}_{\C}:\Im\lambda\notin\cS\}$, we may assume that $\fC$ satisfies the condition in Lemma \ref{Lemma improvement of estimate}.
We now apply Lemma \ref{Lemma improvement of estimate} repeatedly and obtain sequences $(N_{k})_{k\in\N_{0}}$, $(p_{k})_{k\in\N_{0}}$ and $(\zeta_{k})_{k\in\N_{0}}$ of natural numbers, continuous seminorms on $\cD(G,V_{\xi})$ and continuous functions on $\fC\times \fa\to\R$, respectively, so that the above assertions (\ref{eq Estimate for cF - item 1}) and (\ref{eq Estimate for cF - item 3}) hold with $N_{k}$, $p_{k}$ and $\zeta_{k}$ in place of $N$, $p$ and $\zeta$.
The sequence $(\zeta_{k})_{k\in\N_{0}}$ satisfies for $\lambda\in\fC$ and $X\in \cF$
\begin{align}\label{eq recurrence relation zeta_k}
&\zeta_{k+1}(\lambda,X)\\
\nonumber&\quad=\max\Big(
    \big\{\zeta_{k}(\lambda, X)+\frac{1}{2}\beta_{\cF,\lambda}(X),\rho_{Q}(X)\big\}
    \cup \big\{\Re\nu(X)-\langle\lambda\rangle(X):\nu\in \cQ_{\cF,\lambda}^{\zeta_{k}(\lambda,\dotvar),-}\big\}\Big).
\end{align}

We claim that there exists an $n\in\N$ so that $ \zeta_{k}(\lambda,\dotvar)|_{\cF} =\rho_{Q}|_{\cF}$ for every $k\geq n$ and $\lambda\in \fC$.
To see this, we first note that the subsequence $(\zeta_{k})_{k\in\N}$ is decreasing. This implies that the sets $\cQ_{\cF,\lambda}^{\zeta_{k}(\lambda,\dotvar),0}$ and $\cQ_{\cF,\lambda}^{\zeta_{k}(\lambda,\dotvar),-}$ are decreasing with $k$. Note that the cardinality of $\cQ_{\cF,\lambda}^{\zeta_{k}(\lambda,\dotvar),-}$ is bounded by the dimension of $U_{\cF}^{*}$.
Furthermore, the fact that $\beta_{\cF,\lambda}$ is a piece-wise linear functional that depends continuously on $\lambda$ implies that
$$
n':=\sup_{\lambda\in\fC,X\in \cF^{\circ}}-2\frac{\zeta_{0}(\lambda,X)-\rho_{Q}(X)}{\beta_{\cF,\lambda}(X)}<\infty
$$
The claim now follows from (\ref{eq recurrence relation zeta_k}) with $n=n'+\dim(U_{\cF}^{*})$.
The above assertions (\ref{eq Estimate for cF - item 1}) -- (\ref{eq Estimate for cF - item 3})  now follow with $N=N_{n}$, $p=p_{n}$ and $\zeta=\zeta_{n}$.
\end{proof}

\subsection{Constant term approximation}
\label{Subsection Wave packets - Constant term}
We now give a version of the constant term approximation (see \cite[Theorem 1.2]{DelormeKrotzSouaifi_ConstantTerm} and \cite[Theorem 6.2]{KrotzSayagSchlichtkrull_GeometricCountingOnWavefrontRealSphericalSpaces}) which is applicable to our setting.

\begin{Thm}\label{Thm constant term}
Let $\xi$ be a finite dimensional unitary representation of $M_{Q}$. Let further $z\in Z$ be admissible and let $\cF$ be a face of $\overline{\cC}$.
There exists an open neighborhood $\fU$ of  $i(\fa/\fa_{\fh})^{*}\setminus i\cS$ in $(\fa/\fa_{\fh})^{*}_{\C}$ and for every $\lambda\in \fU$ a linear map
$$
\Const_{z,\cF}(\xi:\lambda):\cD'(\overline{Q}:\xi:\lambda)^{H}\to\cD'(\overline{Q}:\xi:\lambda)^{H_{z,\cF}};
\quad \mu\mapsto \mu_{z,\cF}
$$
with the following properties.
\begin{enumerate}[(i)]
\item\label{Thm constant term - item 3}
For every $\eta\in V^{*}(\xi)$ the map
$$
\fU\to\cD'(\overline{Q}:\xi:\lambda)^{H_{z,\cF}};
\quad\lambda\mapsto \Const_{z,\cF}(\xi:\lambda)\circ\mu^{\circ}(\xi:\lambda)(\eta)
$$
is a holomorphic family of distributions.
\item\label{Thm constant term - item 1}
Let $x\in G$ be so that $z=xH\in G/H=Z$ and let $X\in \cF^{\circ}$. If $\lambda\in i(\fa/\fa_{\fh})^{*}\setminus i\cS$ and $\mu\in\cD'(\overline{Q}:\xi:\lambda)^{H}$, then
\begin{equation}\label{eq Constant term limit formula}
\lim_{t\to \infty} e^{-t\rho_{Q}(X)}\Big(R^{\vee}\big(\exp(tX)x\big)\mu-R^{\vee}\big(\exp(tX)\big)\mu_{z,\cF}\Big)=0
\end{equation}
with convergence in $\cD'(G, V_{\xi})$.
\item\label{Thm constant term - item 2}
For every compact subset $\fC$ of  $i(\fa/\fa_{\fh})^{*}\setminus i\cS$, every compact subset $B$ of $G$ and every closed cone $\Upsilon\subset \cF^{\circ}\cup\{0\}$, there exists a $\gamma\in\fa^{*}$ with $\gamma|_{\overline{\cC}}\leq 0$ and $\gamma|_{\Upsilon\setminus\{0\}}<0$,  an $N\in\N_{0}$, and a continuous seminorm $p$ on $\cD(G, V_{\xi})$,  so that
\begin{align*}
&e^{-\rho_{Q}(Y+X)}\Big| m_{\phi,\mu}\big(g \exp(Y+X)x\big)-m_{\phi,\mu_{z,\cF}}\big(g \exp(Y+X)\big)\Big|\\
&\qquad\leq e^{\gamma(X)}(1+\|Y\|)^N\|\mu\| p(\phi)
\end{align*}
for all $\lambda\in \fC$, $\mu\in \cD'(\overline{Q}:\xi:\lambda)^{H}$, $\phi\in\cD(G,V_{\xi})$,  $Y\in\overline{\cC}$, $X\in\Upsilon$ and $g\in B$.
\item\label{Thm constant term - item 4}
Let $z\in Z$ be an admissible point so that $M\cap H_{z}=M\cap H$.
For the face $\cF=\overline{\cC}$ the image of $\const_{z,\overline{\cC}}(\xi:\lambda)$ lies for all $\lambda\in \fU$ in $\cD'(\overline{Q}:\xi:\lambda)^{H_{\emptyset}}$,
where
$$
H_{\emptyset}
=(L_{Q}\cap H)\overline{N}_{Q}.
$$
\end{enumerate}
\end{Thm}

The distribution   $\mu_{z,\cF}$ is called the constant term of $\mu$ with respect to the adapted point $z$ and the face $\cF$.

\begin{proof}
Without loss of generality we may assume that $z=eH$.
We fix $\lambda_{0}\in i(\fa/\fa_{\fh})^{*}\setminus i\cS$ and set
\begin{align*}
\cQ_{\cF}^{+}
&:=\big\{\nu\in \cQ_{\cF,\lambda_{0}}:\Re\nu(X)>\rho_{Q}(X) \text{ for some } X\in\cF^{\circ}\big\},\\
\cQ_{\cF}^{-}
&:=\big\{\nu\in \cQ_{\cF,\lambda_{0}}:\big(\Re\nu-\rho_{Q}\big)\big|_{\cF^{\circ}}<0\big\},\\
\cQ_{\cF}^{0}
&:=\cQ_{\cF,\lambda_{0}}\setminus\big(\cQ_{\cF}^{+}\cup\cQ_{\cF}^{-}\big)
=\big\{\nu\in \cQ_{\cF,\lambda_{0}}: \Re\nu|_{\cF}=\rho_{Q}|_{\cF}\big\}.
\end{align*}
Let $\fA\subseteq \fa_{\cF}^{*}$ be an open polydisc centered at $\rho_{Q}|_{\fa_{\cF}}$ so that
$$
\big(\cQ_{\cF}^{+}\cup\cQ_{\cF}^{-}\big)\cap\big(\overline{\fA}+i\fa_{\cF}\big)
=\emptyset.
$$
There exits an open neighborhood $\fU_{0}$ of $\lambda_{0}$ in $\{\lambda\in (\fa/\fa_{\fh})_{\C}^{*}:\Im\lambda\notin\cS\}$ so that $\cQ_{\cF,\lambda}$ does not intersect with the boundary of $\fA$. We may choose $\fU_{0}$ so small that there exists a subspace $U_{\cF}$ of $V_{\cF}$ for which (\ref{eq Assumption on B}) holds for all $\lambda\in\fU_{0}$.
We define
$$
E:\fU_{0}\to\End(U_{\cF})
$$
for $\lambda\in \fU_{0}$ to be the projection onto the generalized eigenspaces of  $\Gamma_{\cF,\lambda}$ with eigenvalues in $\fA$, i.e.,
$$
E(\lambda)
:=\sum_{\nu\in \fA\cap \cQ_{\cF,\lambda}}E_{\nu}.
$$
It follows from the Cauchy integral formula for spectral projections from functional calculus that $E$ is holomorphic.

We fix an element $X\in\cF^{\circ}$. After shrinking $\fU_{0}$ we may assume that
$$
\Re\nu(X)
>\rho_{Q}(X)+\langle\lambda\rangle(X)+\beta_{\cF,\lambda}(X)
\qquad \big(\lambda\in \fU_{0}, \nu\in \cQ_{\cF,\lambda}^{0}\big).
$$
From the estimate (\ref{eq Psi estimate}) with $\zeta=\rho_{Q}$, we obtain that for $\mu\in \cD'(\overline{Q}:\xi:\lambda)^{H}$, with $\lambda\in \fU_{0}$, and $\phi\in\cD(G,V_{\xi})$, the $U_{\cF}^{*}$-valued integral
$$
\int_{0}^{\infty}E(\lambda)e^{-s\Gamma_{\cF,\lambda}(X)}\Psi_{\mu,\phi,X}\big(\exp(sX)\big)\,ds
$$
converges uniformly on any compact subset of $\fU_{0}$. We may thus define $\Const(\xi:\lambda)\mu\in \cD'(G, V_{\xi})$  for $\phi\in\cD(G,V_{\xi})$ by
$$
\big(\Const_{z,\cF}(\xi:\lambda)\mu\big)(\phi)
=\Big(E(\lambda)\circ\Phi_{\mu,\phi}(e)
    +\int_{0}^{\infty}E(\lambda)e^{-s\Gamma_{\cF,\lambda}(X)}\Psi_{\mu,\phi,X}\big(\exp(sX)\big)\,ds\Big)(1).
$$
For every $\eta\in V^{*}(\xi)$ the family of distributions
$$
\fU_{0}\ni\lambda\mapsto \Const_{z,\cF}(\xi:\lambda)\circ\mu^{\circ}(\xi:\lambda)\eta
$$
is holomorphic.

In view of Theorem \ref{Thm temperedness} all distributions in $\cD'(\overline{Q}:\xi:\lambda)^{H}$ with $\lambda\in i(\fa/\fa_{\fh})^{*}\setminus i\cS$ are tempered. We may therefore apply \cite[Theorem 6.9]{DelormeKrotzSouaifi_ConstantTerm} to these distributions. It follows from \cite[(5.36) \& (6.1)]{DelormeKrotzSouaifi_ConstantTerm} that the constant term-map coincides with $\Const_{z,\cF}(\xi:\lambda)$ for $\lambda\in\fU_{0}\cap i(\fa/\fa_{\fh})^{*}$. It follows that $\Const_{z,\cF}(\xi:\lambda)$ maps $\cD'(\overline{Q}:\xi:\lambda)^{H}$ to $\cD'(\overline{Q}:\xi:\lambda)^{H_{z,\cF}}$ for these $\lambda$. By analytic continuation the same holds for all $\lambda\in \fU_{0}$.

With the above construction we find for every $\lambda_{0}\in i(\fa/\fa_{\fh})^{*}\setminus i\cS$ an open neighborhood $\fU_{0}$ so that the constant term map $\cD'(\overline{Q}:\xi:\lambda)^{H}\to\cD'(\overline{Q}:\xi:\lambda)^{H_{z,\cF}}$ from \cite{DelormeKrotzSouaifi_ConstantTerm} extends holomorphically to $\lambda\in \fU_{0}$. It follows that there exists an open neighborhood $\fU$ of $i(\fa/\fa_{\fh})^{*}\setminus i\cS$ so that the constant term map extends holomorphically to $\fU$.

The remaining assertions in (\ref{Thm constant term - item 1}) and (\ref{Thm constant term - item 2}) are a reformulation of \cite[Theorem 6.9]{DelormeKrotzSouaifi_ConstantTerm} with uniformity in the estimate in $\lambda\in\fC$.
The uniform estimates are obtained by using the estimates (\ref{eq Phi estimate - 1}), (\ref{eq Psi estimate}) and (\ref{eq E_nu estimate}), which are uniform in $\lambda\in\fC$, instead of the estimates in \cite[Lemmas 5.8 \& 5.9]{DelormeKrotzSouaifi_ConstantTerm}.

Finally, we turn to (\ref{Thm constant term - item 4}).
By holomorphicity it suffices to prove the assertion for $\lambda\in i(\fa/\fa_{\fh})^{*}\setminus i\cS$.
For these $\lambda $ it follows from  \cite[Lemma 6.5]{DelormeKrotzSouaifi_ConstantTerm} and Corollary \ref{Cor Description D'(Z,overline Q:xi:lambda) horospherical case} that $\Const_{z, \overline{\cC}}(\xi:\lambda)\mu$ for a $\mu\in\cD'(\overline{Q}:\xi:\lambda)^{H}$ is uniquely determined by (\ref{eq Constant term limit formula}). Since $M$ centralizes $\fa$, it is easily seen that
$$
R^{\vee}(m)\circ\Const_{z,\overline{\cC}}(\xi:\lambda)
=\Const_{z,\overline{\cC}}(\xi:\lambda)\circ R^{\vee}(m)
\qquad(m\in M).
$$
In particular, for every $m\in M\cap H$ and $\mu\in \cD'(\overline{Q}:\xi:\lambda)^{H}$
$$
R^{\vee}(m)\big(\Const_{z,\overline{\cC}}(\xi:\lambda)\mu\big)
=\Const_{z,\overline{\cC}}(\xi:\lambda)\big(R^{\vee}(m)\mu\big)
=\Const_{z,\overline{\cC}}(\xi:\lambda)\mu.
$$
As
$$
H_{\emptyset}
=(M\cap H)(H_{\emptyset})_{e}
=(M\cap H)H_{z,\overline{\cC}},
$$
this proves (\ref{Thm constant term - item 4}).
\end{proof}

\subsection{Construction of wave packets}
\label{Subsection Wave packets - Construction of wave packets}
We use the notation from Section \ref{Subsection Construction - Description}, and recall the finite union $\cS$ of hyperplanes in $(\fa/\fa_{\fh})^{*}$ from the end of Section \ref{Subsection Construction - Description}.

For a finite dimensional unitary representation $\xi$ of $M_{Q}$ we define the wave packet transform
$$
\WP_{\xi}:\cD\big(i(\fa/\fa_{\fh})^{*}\setminus i\cS\big)\otimes V^{*}(\xi)\otimes \cD(G,V_{\xi}) \to \cE(Z),
$$
to be given by
$$
\WP_{\xi}(\psi\otimes \eta\otimes\phi)(gH)
=\int_{i(\fa/\fa_{\fh})^{*}}\psi(\lambda) \Big(R^{\vee}(g)\circ\mu^{\circ}(\xi:\lambda)(\eta)\Big)(\phi)\,d\lambda
$$
for $\psi\in \cD\big(i(\fa/\fa_{\fh})^{*}\setminus i\cS\big)$, $\eta\in V^{*}(\xi)$, $\phi\in\cD(G,V_{\xi})$ and $g\in G$.

\begin{Thm}\label{Thm Wave packets are L^2}
Let $\xi$ be a finite dimensional unitary representation of $M_{Q}$.
The image of $\WP_{\xi}$ is consists of square integrable functions on $Z$.
\end{Thm}

Before we prove the theorem, we first describe an important consequence. In view of (\ref{eq Def B}) the decomposition of $L^{2}_{\mc}(Z)$ into a direct integral of irreducible representations is of the form
\begin{equation}\label{eq abstract Plancherel decomp L^2_mc}
L^{2}_{\mc}(Z)
\simeq \widehat{\bigoplus_{[\xi]\in \widehat{M}_{Q,\mathrm{fu}}}}\int^{\oplus}_{i(\fa/\fa_{\fh})^{*}}
 \cM_{\xi,\lambda}\otimes \Ind_{\overline{Q}}^{G}\big(\xi\otimes\lambda\otimes \1\big)\,d\lambda.
\end{equation}
Here each multiplicity space $\cM_{\xi,\lambda}$ is a subspace of the space of $H$-fixed functionals on $C^{\infty}(\overline{Q}:\xi:\lambda)$. In view of the topological isomorphism (\ref{eq def theta}) we may view $\cM_{\xi,\lambda}$  as a subspace of $\cD'(\overline{Q}:\xi:\lambda)^{H}$.
The Theorems \ref{Thm Description D'(overline Q:xi:lambda)^H} and \ref{Thm Wave packets are L^2} now have the following immediate corollary.

\begin{Cor}\label{Cor Multiplicity space is V^*(xi)}
Let $\xi$ be a finite dimensional unitary representation of $M_{Q}$.
For almost every $\lambda\in i(\fa/\fa_{\fh})^{*}$ the multiplicity space $\cM_{\xi,\lambda}$ is equal to $\cD'(\overline{Q}:\xi:\lambda)^{H}$ and the map
$$
\mu^{\circ}(\xi:\lambda):V^{*}(\xi)\to\cM_{\xi,\lambda}
$$
is a linear isomorphism.
\end{Cor}

\begin{proof}[Proof of Theorem \ref{Thm Wave packets are L^2}]
In view of (\ref{eq Polar decomp}) we have every integrable function $\chi$ on $Z$
$$
\int_{Z}\chi(z)\,dz
=\sum_{j=1}^{r}\sum_{\cO\in (P\bs Z)_{\open}}\int_{K_{j}}\int_{\overline{\cC}}\chi\big(f_{j}k\exp(X) x_{\cO}H\big)J_{j,\cO}(k,X)\,dX\,dk.
$$
Here $dX$ is the Lebesgue-measure on $\fa$, $dk$ denotes for each $j$  the Haar measure on $K_{j}$.
The Jacobian
$$
J_{j,\cO}:K_{j}\times\overline{\cC}\to\R_{\geq0}
$$
are readily seen to be constant in the first variable. We therefore consider these functions as functions on $\overline{\cC}$ only.
Important for our consideration is the estimate
$$
J_{j,\cO}(X)\leq C e^{-2\rho_{Q}(X)}
\qquad\big(f\in F, \cO\in(P\bs Z)_{\open}, X\in\overline{\cC}\big)
$$
for some constant $C>0$.
This estimate follows from \cite[Proposition 4.3]{KnopKrotzSayagSchlichtkrull_VolumeGrowthTemperednessAndIntegrabilityOfMatrixCoefficients}.

We will decompose the integral over $\overline{\cC}$ as a sum of integrals over suitable subsets which allow to apply Theorem \ref{Thm constant term}.
We recall from Section \ref{Subsection Orbits of max rank - Admissible points and little Weyl group} that the little Weyl group is the Weyl group of the spherical root system $\Sigma_{Z}$ in $(\fa/\fa_{E})^{*}$. The faces of $\overline{\cC}$ are in bijection with the power set of $\Sigma_{Z}$. To be more precise, to a face $\cF$ of $\overline{\cC}$ a subset $S_{\cF}$ of $\Sigma_{Z}$ is attached with the property
$$
\cF
=\big\{X\in \fa:\sigma(X)=0 \text{ for all } \sigma\in S_{\cF} \text{ and }\sigma(X)<0 \text{ for all } \sigma\in \Sigma_{Z}\setminus S_{\cF}\big\}.
$$
The assignment $\cF\mapsto S_{\cF}$ is a bijection between the faces of $\overline{\cC}$ and the power set of $\Sigma_{Z}$.
Let $\cF$ be a face of $\overline{\cC}$. If $\cF'$ is the unique face of $\overline{\cC}$ with $S_{\cF'}=\Sigma_{Z}\setminus S_{\cF}$, then $\cF\cap \cF'=\overline{\cC}\cap (-\overline{\cC})=\fa_{E}$ is the edge of $\overline{\cC}$.
We then define the cone
$$
\cF_{\perp}
:=\cF'\cap \fa_{E}^{\perp}
$$
and set
$$
\fa_{\cF,\perp}
:=\spn(\cF_{\perp}).
$$
Now
\begin{equation}\label{eq fa=spn(F)+spn(F^0)}
\fa
=\fa_{\cF}\oplus\fa_{\cF_{\perp}}
\end{equation}
and
$$
\overline{\cC}
=\cF+\cF_{\perp}.
$$
We write $p_{\cF}$ to be the  projection $\fa\to\fa_{\cF}$ along the decomposition (\ref{eq fa=spn(F)+spn(F^0)}). We fix a $\delta>0$ and define cones $C_{\cF}$ in $\overline{\cC}$ by setting
$$
C_{\fa_{E}}
:=\{X\in\overline{\cC}: \|X\|\leq (1+\delta)\big\|p_{\cF}(X)\big\|\}
$$
if $\cF=\fa_{E}$, and then recursively by
$$
C_{\cF}
:=\Big\{X\in\overline{\cC}\setminus \bigcup_{\substack{\cF' \text{ face of }\overline{\cC}\\\cF'\subsetneq\cF}}C_{\cF'}:
        \|X\|\leq (1+\delta)\big\|p_{\cF}(X)\big\|\Big\}
$$
for the remaining faces $\cF$. Then for every face $\cF$ of $\overline{\cC}$ there exists a closed cone $\Upsilon_{\cF}$ in $\cF^{\circ}\cup \{0\}$  so that
$$
\overline{C_{\cF}}
\subseteq\Big\{X\in\overline{\cC}:p_{\cF}(X)\in\Upsilon_{\cF}, \|X\|\leq (1+\delta)\big\|p_{\cF}(X)\big\|\Big\}.
$$
Moreover,
$$
\overline{\cC}
=\bigsqcup_{\cF \text{ face of }\overline{\cC}}C_{\cF}.
$$
Now
\begin{align*}
\int_{Z}\chi(z)\,dz
&=\sum_{j=1}^{r}\sum_{\cO\in (P\bs Z)_{\open}}\sum_{\cF \text{ face of }\overline{\cC}}
    \int_{K_{j}}\int_{C_{\cF}}\chi\big(f_{j}k\exp(X) x_{\cO}H\big)J_{j,\cO}(X)\,dX\,dk\\
&\leq C\sum_{j=1}^{r}\sum_{\cO\in (P\bs Z)_{\open}}\sum_{\cF \text{ face of }\overline{\cC}}
    \int_{K_{j}}\int_{C_{\cF}}\chi\big(f_{j}k\exp(X) x_{\cO}H\big)e^{-2\rho_{Q}(X)}\,dX\,dk
\end{align*}
for every non-negative measurable function $\chi$ on $Z$.

Let $K'$ be a maximal compact subgroup, $\cF$ a face of $\overline{\cC}$ and $z\in Z$ an adapted point. To prove the theorem, it suffices to prove that for every $\psi\in \cD\big(i(\fa/\fa_{\fh})^{*}\setminus i\cS\big)$, $\eta\in V^{*}(\xi)$ and $\phi\in\cD(G,V_{\xi})$ the function
$$
K'\times C_{\cF}\to\C;
\quad (k,X)\mapsto e^{-\rho_{Q}(X)}\WP_{\xi}(\psi\otimes\eta\otimes\phi)\big(k\exp(X)\cdot z\big)
$$
is square integrable.

It follows from Proposition \ref{Prop Action N_A(H)} that for each $w\in W$ there exists a linear map $\mu_{\cF,w}(\xi:\lambda):V^{*}(\xi)\to\cD'(G,V_{\xi})^{H_{z,\cF}}$  so that
$$
\Const_{z,\cF}(\xi:\lambda)\circ\mu^{\circ}(\xi:\lambda)
=\sum_{w\in W}\mu_{\cF,w}(\xi:\lambda)
$$
and
$$
R^{\vee}\big(\exp(X)\big)\circ\mu_{\cF,w}(\xi:\lambda)
=e^{\big(-\Ad^{*}(w^{-1})\lambda+\rho_{Q}\big)(X)}\mu_{\cF,w}(\xi:\lambda)
\qquad\big(w\in W, X\in\fa_{\cF}\big).
$$
The family $\lambda\mapsto \mu_{\cF,w}(\xi:\lambda)$ is meromorphic and holomorphic on $\{\lambda\in (\fa/\fa_{\fh})^{*}_{\C}:\Im\lambda\notin\cS\}$.

We now fix a face $\cF$ of $\overline{\cC}$ and a $w\in W$.
Let $\gamma_{\cF}\in\fa^{*}$ and $p=p_{z,\cF}$ satisfy the properties of Theorem \ref{Thm constant term} (\ref{Thm constant term - item 2}) with the closed cone $\Upsilon$ taken to be $\Upsilon_{\cF}$ and the compact subset $B$ equal to $K'$.
As $\gamma_{\cF}|_{\Upsilon_{\cF}}<0$, we have
$$
\int_{C_{\cF}}\Big(e^{\gamma_{\cF}\big(p_{\cF}(X)\big)}(1+\|X-p_{\cF}(X)\|)^{N}\Big)^{2}\,dX<\infty
$$
for every $N\in\N$. Therefore,  it suffices to prove that for every $\psi\in\cD\big(i(\fa/\fa_{\fh})^{*}\setminus i\cS\big)$, $\eta\in V^{*}(\xi)$ and $\phi\in\cD(G,V_{\xi})$ the function
\begin{align*}
\Omega_{\psi, \eta, \phi}:G\times (\fa/\fa_{\fh})&\to\C;\\
(g,X)&\mapsto e^{-\rho_{Q}(X)}\int_{i(\fa/\fa_{\fh})^{*}}\psi(\lambda)\Big(R^{\vee}\big(g\exp(X)\big)\mu_{\cF,w}(\xi:\lambda)\eta\Big)(\phi)\,d\lambda
\end{align*}
is square integrable on $K'\times C_{\cF}$.

Let $\Ft_{\eucl}$ be the Euclidean Fourier transform on $\fa/\fa_{\fh}$, i.e., the transform
$$
\Ft_{\eucl}:\cS\big(\fa/\fa_{\fh}\big)\to\cS\big(i(\fa/\fa_{\fh})^{*}\big)
$$
given by
$$
\Ft_{\eucl}\psi(\xi)
=\int_{\fa/\fa_{\fh}}\psi(X)e^{\xi(X)}\,dX
\qquad\big(\psi\in \cS(\fa/\fa_{\fh}), \xi\in i(\fa/\fa_{\fh})^{*}\big).
$$
For every $X\in \fa_{\cF}$ and $\eta\in V^{*}(\xi)$
$$
e^{-\rho_{Q}(X)}R^{\vee}\big(\exp(X)\big)\mu_{\cF,w}(\xi:\lambda)\eta
=e^{-\lambda\big(\Ad(w)X\big)}\mu_{\cF,w}(\xi:\lambda)\eta,
$$
and hence
\begin{align*}
\Omega_{\psi, \eta, \phi}(g,X)
&=e^{-\rho_{Q}(X)}\int_{i(\fa/\fa_{\fh})^{*}}\psi(\lambda)\Big(R^{\vee}\big(g\exp(X)\big)\mu_{\cF,w}(\xi:\lambda)\eta\Big)(\phi)\,d\lambda\\
&=\Ft_{\eucl}^{-1}\Big(\lambda\mapsto\psi(\lambda)\big(R^{\vee}(g)\mu_{\cF,w}(\xi:\lambda)\eta\big)(\phi)\Big)\big(\Ad(w)X\big).
\end{align*}
Since $\lambda\mapsto\psi(\lambda)\big(R^{\vee}(g)\mu_{\cF,w}(\xi:\lambda)\eta\big)(\phi)$ is compactly supported and smooth, the function $X\mapsto \Omega_{\psi, \eta, \phi}(g,X)$ is contained in $\cS\big(\fa_{\cF}\big)$. Moreover, the continuity of $\Ft_{\eucl}$ implies that for every continuous seminorm $p$ on $\cS(\fa_{\cF})$ there exists a continuous seminorm $q$ on $\cD\big(i(\fa/\fa_{\fh})^{*}\big)$, independent of $g\in G$ and $\phi\in \cD(G,V_{\xi})$,  so that
\begin{equation}\label{eq estimate Omega^ct}
p\big(\Omega_{\psi, \eta, \phi}(g,\dotvar)\big)
\leq q\Big(\lambda\mapsto\psi(\lambda)\big(R^{\vee}(g)\mu_{\cF,w}(\xi:\lambda)\eta\big)(\phi)\Big).
\end{equation}

Let $\fC\subseteq i(\fa/\fa_{\fh})^{*}\setminus i\cS$ be a compact subset.
We claim that for every differential operator $D$ on $i(\fa/\fa_{\fh})^{*}$ with constant coefficients there exists an $N\in \N_{0}$ and a continuous seminorm $r$ on $\cD(G,V_{\xi})$ so that for all $\lambda\in\fC$, $\phi\in \cD(G, V_{\xi})$, $k\in K'$ and $Y\in \fa$ the estimate
\begin{equation}\label{eq Estimate Derivatives in lambda of matrix coefficient}
\big|D\big(R^{\vee}\big(k\exp(Y)\big)\mu_{\cF,w}(\xi:\lambda)\eta\big)(\phi)\big|
\leq e^{\rho_{Q}(Y)}(1+\|Y\|)^{N} r(\phi)
\end{equation}
holds.
It suffices to prove the claim for $D= \partial_{\lambda}^{\alpha}$, with $\alpha$ a multi-index.
We first note that it follows from Theorem \ref{Thm constant term} that $\lambda\mapsto \mu_{\cF,w}(\xi:\lambda)\eta$ extends to a holomorphic family of distributions with family parameter $\lambda$ in an open neighborhood $\fU$ of $i(\fa/\fa_{\fh})^{*}\setminus i\cS$.
Let $\epsilon>0$ be so small that the polydisc $\Delta$ with radius $\epsilon$ and center $\lambda$ is contained in $\fU$. Let now $\Delta_{\delta}$ be the polydisc centered at $\lambda$ of radius $\delta>0$.  For every $\delta\leq \epsilon$ we obtain from Cauchy's integral formula the estimate
$$
\big|\partial_{\lambda}^{\alpha}\big(R^{\vee}\big(k\exp(Y)\big)\mu_{\cF,w}(\xi:\lambda)\eta\big)(\phi)\big|
\leq\frac{\alpha!}{\delta^{|\alpha|}}\sup_{\lambda'\in \Delta_{\delta}}\big|\big(R^{\vee}\big(k\exp(Y)\big)\mu_{\cF,w}(\xi:\lambda')\eta\big)(\phi)\big|.
$$
We now invoke Theorem \ref{Thm temperedness}. This yields the existence of an $N'\in\N_{0}$ and a continuous seminorm $r'$ on $\cD(G,V_{\xi})$ so that  \begin{align*}
&\sup_{\lambda'\in \Delta_{\delta}}\big|\big(R^{\vee}\big(k\exp(Y)\big)\mu_{\cF,w}(\xi:\lambda')\eta\big)(\phi)\big|\\
&\quad\leq\sup_{\lambda'\in \Delta_{\delta}} \max_{w\in W} e^{\rho_{Q}(Y)+\Re\Ad^{*}(w)\lambda'(Y)}\big(1+\|Y\|\big)^{N'}r'(\phi)
\leq e^{\rho_{Q}(Y)+\delta\|Y\|}\big(1+\|Y\|\big)^{N'}r'(\phi).
 \end{align*}
The claim now follows with $N=N'+|\alpha|$ and $r=e\alpha!\, r'$ by taking $\delta$ equal to the minimum of $\epsilon$ and $(1+\|Y\|)^{-1}$.

We now consider the space $\cD_{\fC}\big(i(\fa/\fa_{\fh})^{*}\big)$ of functions $\psi\in\cD\big(i(\fa/\fa_{\fh})^{*}\big)$ with $\supp(\psi)\subseteq \fC$.
Every continuous seminorm on $\cD_{\fC}\big(i(\fa/\fa_{\fh})^{*}\big)$ can be dominated by a sum of seminorms of the sort $\psi\mapsto \sup|D\phi|$, where $D$ is a differential operator with constant coefficients. It follows from  (\ref{eq estimate Omega^ct}), the Leibnitz rule and (\ref{eq Estimate Derivatives in lambda of matrix coefficient}) that for every continuous seminorm $p$ on $\cS(\fa_{\cF})$ there exist continuous seminorms $r$ and $s$ on $\cD_{\fC}\big(i(\fa/\fa_{\fh})^{*}\big)$ and $\cD(G,V_{\xi})$, respectively, and an $N\in \N_{0}$,  so that for all $\psi\in\cD_{\fC}\big(i(\fa/\fa_{\fh})^{*}\big)$, $\eta\in V^{*}(\xi)$, $\phi\in\cD(G,V_{\xi})$, $\lambda\in\fC$, $k\in K'$ and $Y\in\fa$ the estimate
$$
p\Big(\fa_{\cF}\ni X\mapsto \Omega_{\psi, \eta, \phi}\big(k,X+Y\big)\Big)
\leq (1+\|Y\|)^{N} r(\psi)\|\eta\|s(\phi).
$$
holds.
In particular, for every $n\in\N_{0}$ there exists continuous seminorms $r_{n}$ on $\cD_{\fC}\big(i(\fa/\fa_{\fh})^{*}\big)$ and $s_{n}$ on $\cD(G,V_{\xi})$ so that
$$
\sup_{X\in\fa_{\cF}}(1+\|X\|)^{n}\big|\Omega_{\psi, \eta, \phi}(k,X+Y)\big|
\leq (1+\|Y\|)^{N} r_{n}(\psi)\|\eta\|s_{n}(\phi).
$$
for every $\psi\in\cD_{\fC}\big(i(\fa/\fa_{\fh})^{*}\big)$, $\eta\in V^{*}(\xi)$, $\phi\in\cD(G,V_{\xi})$, $k\in K'$ and $Y\in \fa$.

From the definition of $C_{\cF}$ it follows that there exists a constant $c>0$, so that if  $X\in\fa_{\cF}$, $Y\in\fa_{\cF_{\perp}}$ and $X+Y\in C_{\cF}$, then $\|Y\|\leq c\|X\|$.
For $n> N+\dim(\fa/\fa_{\fh})/2$ the integral
$$
\int_{K'}\int_{C_{\cF}}\big|\Omega_{\psi, \eta, \phi}(k,X)\big|^{2}\,dX\,dk
$$
is therefore absolutely convergent and bounded by
$$
\mathrm{vol}(K')r_{n}(\psi)\|\eta\|s_{n}(\phi)\int_{\Upsilon_{\cF}}(1+\|X\|)^{2N+\dim(\fa_{\cF_{\perp}})-2n}\,dX.
$$
This proves the theorem.
\end{proof}

\section{The most continuous part of $L^{2}(Z)$}
\label{Section Most continuous part}
\subsection{Abstract Plancherel decomposition}
\label{Subsection Most continuous part - Abstract Plancherel decomposition}
We denote by $\widehat{G}$ the unitary dual of $G$. For each equivalence class $[\pi]\in\widehat{G}$ we choose a representative $(\pi,\cH_{\pi})$, i.e., $\cH_{\pi}$ is a Hilbert space and $\pi$ is a unitary representation of $G$ on $\cH_{\pi}$ in the equivalence class $[\pi]$.
We denote the space of smooth vectors of $\pi$ by $\cH_{\pi}^{\infty}$.

Let $[\pi]\in \widehat{G}$.
Since $Z$ is real spherical, the space $({\cH_{\pi^{\vee}}^{\infty}}')^{H}$ is finite dimensional. See \cite[Theorem C]{KobayashiOshima_FiniteMultiplicitiyTheorems} and \cite{KrotzSchlichtkrull_MultiplicityBoundsAndSubrepresentationTheorem}.
For every $\mu\in ({\cH_{\pi^{\vee}}^{\infty}}')^{H}$ and $f\in \cD(Z)$ the functional
$$
\cH_{\pi^{\vee}}^{\infty}\ni v\mapsto \int_{Z}f(g H) \Big( \pi(g)\mu\Big)(v)\,dgH
$$
actually defines a smooth vector for $\pi$.
We define the Fourier transform
$$
\Ft f(\pi)\in \Hom_{\C}\big(({\cH_{\pi^{\vee}}^{\infty}}')^{H},\cH_{\pi}^{\infty}\big)
$$
of a function $f\in\cD(Z)$ and $\mu\in({\cH_{\pi^{\vee}}^{\infty}}')^{H}$ by
\begin{align*}
\Ft f(\pi)\mu
=\int_{Z}f(g H) \pi(g)\mu\,dgH.
\end{align*}

By the abstract Plancherel Theorem there exists a Radon measure $d_{\Pl}[\pi]$ on $\widehat{G}$ and for every $[\pi] \in\widehat{G}$ a Hilbert space
$$
\cM_{\pi}
\subseteq ({\cH_{\pi}^{\infty}}')^{H},
$$
depending measurably on $[\pi]$, so that the Fourier transform
$$
\Ft : \cD(Z)\to\int_{\widehat{G}}^{\oplus}\Hom_{\C}\big(\cM_{\pi^{\vee}},\cH_{\pi}\big)\,d_{\Pl}[\pi]
$$
with the induced Hilbert space structure on $\Hom_{\C}\big(\cM_{\pi^{\vee}},\cH_{\pi}\big)$ extends to a unitary $G$-isomorphism
\begin{equation}\label{eq abstract Plancherel decomposition}
\Ft:L^{2}(Z)\to\int_{\widehat{G}}^{\oplus}\Hom_{\C}\big(\cM_{\pi^{\vee}},\cH_{\pi}\big)\,d_{\Pl}[\pi].
\end{equation}
The measure class of  the Plancherel measure $d_{\Pl}[\pi]$ is uniquely determined by $Z$. Once $d_{\Pl}[\pi]$ has been fixed, the multiplicity spaces $\cM_{\pi}$, including their inner products, are uniquely determined for almost all $[\pi]\in \widehat{G}$.
By dualizing (\ref{eq abstract Plancherel decomposition}) we obtain that the dual space of $\cM_{\pi^{\vee}}$ is equal to $\cM_{\pi}$. Therefore,  the abstract Plancherel decomposition may also be written in its more common form
$$
L^{2}(Z)\simeq
\int_{\widehat{G}}^{\oplus}\cM_{\pi}\otimes\cH_{\pi}\,d_{\Pl}[\pi].
$$

We recall the Bernstein morphisms $B_{I}$ with $I\subseteq\Sigma_{Z}$ from (\ref{eq Def B}) and (\ref{eq Def B_I}).
In the remainder of Section \ref{Section Most continuous part} we will derive the decomposition of the most continuous part of $L^{2}(Z)$
$$
L^{2}_{\mc}(Z)
:=\Im(B_{\emptyset})\cap L^{2}(Z)
$$
into a direct integral of irreducible unitary representations of $G$.

\subsection{Plancherel decomposition for $Z_{\emptyset}$}
\label{Subsection Most continuous part - Plancherel decomposition for Z_empty}
We first determine the Plancherel decomposition for $Z_{\emptyset}$.

We choose a set of representatives $\fN$ of $\cN/\cZ$ in $\cN\cap K$ as in Section \ref{Subsection Construction - Horospherical case}
and define
$$
V_{\emptyset}^{*}(\xi)
:=\bigoplus_{v\in\fN}(V_{\xi}^{*})^{M_{Q}\cap vHv^{-1}}.
$$
We write
$$
\mu_{\emptyset}^{\circ}(\xi:\lambda):V^{*}_{\emptyset}(\xi)\to\cD'(\overline{Q}:\xi:\lambda)^{H_{\emptyset}}
$$
for the map from Corollary \ref{Cor Description D'(Z,overline Q:xi:lambda) horospherical case} for the space $Z_{\emptyset}$.
Now we define the Fourier transform
$$
\Ft_{\emptyset}f(\xi:\lambda)\in \Hom_{\C}\big(V_{\emptyset}^{*}(\xi^{\vee}),C^{\infty}(\overline{Q}:\xi:\lambda)\big)
=V_{\emptyset}^{*}(\xi)\otimes C^{\infty}(\overline{Q}:\xi:\lambda).
$$
of a function $f\in\cD(Z_{\emptyset})$ by
\begin{align*}
\Ft_{\emptyset}f(\xi:\lambda)\eta
=\int_{Z_{\emptyset}}f(g H_{\emptyset})
    R^{\vee}(g)\Big(\mu^{\circ}_{\emptyset}(\xi^{\vee}:-\lambda)\eta\Big)\,dgH_{\emptyset}.
\end{align*}
Let $\langle\cdot,\cdot\rangle_{\emptyset,\xi^{\vee}}$ be the inner product on $V_{\emptyset}^{*}(\xi^{\vee})$ induced by the inner product on $V_{\xi^{\vee}}$, and let $\langle\cdot,\cdot\rangle_{\emptyset,\xi,\lambda}$ be the inner product on
$V_{\emptyset}^{*}(\xi)\otimes \Ind_{\overline{Q}}^{G}(\xi\otimes\lambda\otimes\1)$ induced by the inner products $\langle\cdot,\cdot\rangle_{\emptyset,\xi}$ and $\langle \cdot, \cdot\rangle_{\overline{Q},\xi,\lambda}$ on $V_{\emptyset}^{*}(\xi)$ and $\Ind_{\overline{Q}}^{G}(\xi\otimes\lambda\otimes\1)$, respectively.

Recall that $\widehat{M}_{Q,\mathrm{fu}}$ denotes the set of equivalence classes of finite dimensional unitary representations of $M_{Q}$.
Let $(\fa/\fa_{\fh})_{+}^{*}$ be a fundamental domain for the action of $\cN$ on $(\fa/\fa_{\fh})^{*}$.
We recall that we normalize Lebesgue measure on $i(\fa/\fa_{\fh})^{*}$ by requiring that
$$
\phi(e)
=\int_{i(\fa/\fa_{\fh})^{*}}\int_{A/(A\cap H)}\phi(a)a^{\lambda}\,d\lambda\
\qquad\Big(\phi\in\cD\big(A/(A\cap H)\big)\Big).
$$
We then have the following Plancherel decomposition.

\begin{Thm}\label{Thm Plancherel Theorem horospherical case}
The Fourier transform $f\mapsto\Ft_{\emptyset}f$ extends to a continuous linear operator
\begin{equation}\label{eq Plancherel Thm horospherical case}
L^{2}(Z_{\emptyset})\to \widehat{\bigoplus_{\xi\in\widehat{M}_{Q,\mathrm{fu}}}}\int_{i(\fa/\fa_{\fh})^{*}}^{\oplus} V_{\emptyset}^{*}(\xi)\otimes C^{\infty}(\overline{Q}:\xi:\lambda)\,d\lambda.
\end{equation}
Moreover, for every $f_{1},f_{2}\in L^{2}(Z_{\emptyset})$
\begin{equation}\label{eq Plancherel Thm horospherical case - Pl identity}
\int_{Z_{\emptyset}}f_{1}(z)\overline{f_{2}(z)}\,dz
=\sum_{[\xi]\in \widehat{M}_{Q,\mathrm{fu}}}\dim(V_{\xi})
    \int_{(\fa/\fa_{\fh})_{+}^{*}}\Big\langle\Ft_{\emptyset}f_{1}(\xi:\lambda),\Ft_{\emptyset}f_{2}(\xi:\lambda)\Big\rangle_{\emptyset,\xi,\lambda}\,d\lambda.
\end{equation}
\end{Thm}

\begin{Rem}
In view of the following assertions the decomposition (\ref{eq Plancherel Thm horospherical case}) is in fact the Plancherel decomposition for $Z_{\emptyset}$.
\begin{enumerate}[(a)]
\item
Let $\xi\in\widehat{M}_{Q}$. For every $\lambda\in i(\fa/\fa_{\fh})^{*}$ the representation $\Ind_{\overline{Q}}^{G}(\xi\otimes\lambda\otimes\1)$ is irreducible.
See \cite[p. 203, Th{\'e}or{\`e}me 4]{Bruhat_SurLesRepresentationsInduitesDesGroupesDeLie} and \cite[Theorem 4.11]{KolkVaradarajan_TransverseSymbolOfVectorialDistributions}.
\item
Let $\xi,\xi'\in\widehat{M}_{Q}$. For almost all $\lambda,\lambda'\in  i(\fa/\fa_{\fh})^{*}$ the representations $\Ind_{\overline{Q}}^{G}(\xi\otimes\lambda\otimes\1)$ and $\Ind_{\overline{Q}}^{G}(\xi'\otimes\lambda'\otimes\1)$ are equivalent if and only if there exists a $w\in \cN$ so that $\xi=w\cdot \xi'$ and $\lambda=\Ad^{*}(w)\lambda'$.
This assertion is proven by comparing the infinitesimal characters of the principal series representations.
\end{enumerate}
\end{Rem}

We first prove a lemma. Recall the inclusions $\iota_{v}$ for $v\in\cN$ from (\ref{eq Def iota_v}).

\begin{Lemma}\label{Lemma Identities cF_empty}
Let $\xi$ be a finite dimensional unitary representation of $M_{Q}$, $\lambda\in i(\fa/\fa_{\fh})^{*}\setminus i\cS$ and $f\in \cD(Z_{\emptyset})$. Then for every $\eta\in V_{\xi}^{M_{Q}\cap H}$
\begin{equation}\label{eq Identity cF_empty - 1}
\Ft_{\emptyset}f(\xi:\lambda)\big(\iota_{e}\eta\big):g\mapsto
\int_{M_{Q}/(M_{Q}\cap H)}\int_{A/(A\cap H)}f(g^{-1}am H_{\emptyset})a^{\lambda-\rho_{Q}}\xi(m)\eta\,da\,dm.
\end{equation}
Furthermore, for every $v\in \fN$
\begin{align}\label{eq Identity cF_empty - 2}
&\Ft_{\emptyset}f\big(v\cdot\xi:\Ad^{*}(v)\lambda\big)\circ\iota_{v}\\
\nonumber&\qquad=\frac{1}{\gamma(v^{-1}\overline{Q}v:\overline{Q}:\xi:\lambda\big)}
    L(v)\circ A(v^{-1}\overline{Q}v:\overline{Q}:\xi:\lambda)\circ\Ft_{\emptyset}f(\xi:\lambda)\circ\iota_{e}.
\end{align}
Finally, for every $v\in \fN$ and $\eta\in V_{v\cdot\xi}^{M_{Q}\cap vHv^{-1}}=V_{\xi}^{M_{Q}\cap H}$ we have
\begin{equation}\label{eq Identity cF_empty - 3}
\big\|\Ft_{\emptyset}f\big(v\cdot\xi:\Ad^{*}(v)\lambda\big)\big(\iota_{v}\eta\big)\big\|_{\overline{Q},\xi,\lambda}
=\big\|\Ft_{\emptyset}f\big(\xi:\lambda\big)\big(\iota_{e}\eta\big)\big\|_{\overline{Q},\xi,\lambda}.
\end{equation}
\end{Lemma}

\begin{proof}
Let $\eta\in V_{\xi}^{M_{Q}\cap H}$.
By (\ref{eq formula mu horospherical case}) the distribution $\mu_{\emptyset}^{\circ}(\xi^{\vee}:-\lambda)(\iota_{e}\eta)$ is for $\phi\in \cD(G,V_{\xi}^{*})$ and $\lambda\in (\fa/\fa_{\fh})^{*}_{\C}$ given by
$$
\big(\mu_{\emptyset}^{\circ}(\xi^{\vee}:-\lambda)(\iota_{e}\eta)\big)(\phi)
=\int_{M_{Q}}\int_{A}\int_{\overline{N}_{Q}}a^{\lambda-\rho_{Q}}
        \Big(\xi(m)\eta,\phi(ma\overline{n})\Big)\,d\overline{n}\,da\,dm\,dn.
$$
Let $f\in \cD(Z_{\emptyset})$. Then
\begin{align*}
&\Big(\Ft_{\emptyset}f(\xi:\lambda)\big(\iota_{e}\eta\big)\Big)(\phi)\\
&\quad=\int_{Z_{\emptyset}}f(g H_{\emptyset})
    R^{\vee}(g)\Big(\mu^{\circ}_{\emptyset}(\xi^{\vee}:-\lambda)(\iota_{e}\eta)\Big)(\phi)\,dgH_{\emptyset}\\
&\quad=\int_{Z_{\emptyset}}
    \int_{M_{Q}}\int_{A}\int_{\overline{N}_{Q}}f(g H_{\emptyset})a^{\lambda-\rho_{Q}}
        \Big(\xi(m)\eta,\phi(ma\overline{n}g^{-1})\Big)\,d\overline{n}\,da\,dm\,dgH_{\emptyset}\\
&\quad=\int_{Z_{\emptyset}}
    \int_{M_{Q}/(M_{Q}\cap H)}\int_{A/(A\cap H)}\int_{H_{\emptyset}}f(g H_{\emptyset})a^{\lambda-\rho_{Q}}
        \Big(\xi(m)\eta,\phi(mah^{-1}g^{-1})\Big)\,dh\,da\,dm\,dgH_{\emptyset}.
\end{align*}
Let $M_{0}$ be a submanifold of $M_{Q}$ so that
$$
M_{0}\to M_{Q}/(M_{Q}\cap H);
\quad m_{0}\mapsto m_{0}(M_{Q}\cap H)
$$
is a diffeomorphism onto an open and dense subset of $M_{Q}/(M_{Q}\cap H)$ and let $d\mu$ be the pull back of the invariant measure on $M_{Q}/(M_{Q}\cap H)$ along this map. Let further $A_{0}$ be a closed subgroup of $A$ so that
$$
A_{0}\to A/(A\cap H);
\quad a_{0}\mapsto a_{0}(A\cap H)
$$
is a diffeomorphism. Then
\begin{align*}
&\Big(\Ft_{\emptyset}f(\xi:\lambda)\big(\iota_{e}\eta\big)\Big)(\phi)\\
&\qquad=\int_{Z_{\emptyset}}\int_{H_{\emptyset}}
    \int_{M_{0}}\int_{A_{0}}f(gh H_{\emptyset})a^{\lambda-\rho_{Q}}
        \Big(\xi(m)\eta,\phi(mah^{-1}g^{-1})\Big)\,da\,d\mu(m)\,dh\,dgH_{\emptyset}\\
&\qquad=\int_{G}\int_{M_{0}}\int_{A_{0}}f(g H_{\emptyset})a^{\lambda-\rho_{Q}}
        \Big(\xi(m)\eta,\phi(mag^{-1})\Big)\,da\,d\mu(m)\,dg\\
&\qquad=\int_{G}\bigg(\int_{M_{0}}\int_{A_{0}}f(g^{-1}am H_{\emptyset})a^{\lambda-\rho_{Q}}
        \xi(m)\eta\,da\,d\mu(m),\phi(g)\bigg)\,dg\\
&\qquad=\int_{G}\bigg( \int_{M_{Q}/(M_{Q}\cap H)}\int_{A/(A\cap H)}f(g^{-1}am H_{\emptyset})a^{\lambda-\rho_{Q}}
        \xi(m)\eta\,da\,dm,\phi(g)\bigg)\,dg.
\end{align*}
This proves (\ref{eq Identity cF_empty - 1}).

The identity (\ref{eq Identity cF_empty - 2}) follows from (\ref{eq transformation rule mu horospherical case}) as $\cI^{\circ}_{v}(\xi^{\vee}:-\lambda)$ acts on the subspace $C^{\infty}(\overline{Q}:\xi:\lambda)$ of $\cD'(\overline{Q}:\xi^{\vee}:-\lambda)$ by
\begin{equation}\label{eq Intertwiner 1/gamma L(v) circ A(v^-1oQv:oQ)}
\frac{1}{\gamma(v^{-1}\overline{Q}v:\overline{Q}:\xi:\lambda\big)}
    L(v)\circ A(v^{-1}\overline{Q}v:\overline{Q}:\xi:\lambda).
\end{equation}

Finally, (\ref{eq Identity cF_empty - 3}) follows from (\ref{eq Identity cF_empty - 2}) as (\ref{eq Intertwiner 1/gamma L(v) circ A(v^-1oQv:oQ)}) is a unitary map.
\end{proof}

\begin{proof}[Proof of Theorem \ref{Thm Plancherel Theorem horospherical case}]
Let $f\in\cD(Z_{\emptyset})$. In view of the decomposition polar $G=KAH_{\emptyset}$, we have
$$
\int_{Z_{\emptyset}}|f(z)|^{2}\,dz
=\int_{K}\int_{A/(A\cap H)}a^{-2\rho_{Q}} |f(kaH_{\emptyset})|^{2}\,da\,dk.
$$
By Fubini's theorem the function
$$
A/A\cap H\ni a\mapsto a^{-\rho_{Q}}f(kaH_{\emptyset})
$$
is square integrable for almost every $k\in K$. We now apply the Plancherel theorem for the euclidean Fourier transform on $A/A\cap H$ to the inner integral and obtain
$$
\int_{Z_{\emptyset}}|f(z)|^{2}\,dz
=\int_{K}\int_{i(\fa/\fa_{\fh})^{*}}\left|\int_{A/(A\cap H)}a^{\lambda-\rho_{Q}}f(kaH_{\emptyset})\,da\right|^{2}\,d\lambda\,dk.
$$
Since $M\subseteq K$, we have for every $\theta\in\cD(G/H_{\emptyset})$ and $a\in A$
$$
\int_{K}\theta(kaH_{\emptyset})\,dk
=\int_{K}\int_{M/(M\cap H)}\theta(kmaH_{\emptyset})\,dm\,dk.
$$
It follows that
\begin{align*}
\int_{Z_{\emptyset}}|f(z)|^{2}\,dz
&=\int_{K}\int_{i(\fa/\fa_{\fh})^{*}}
\int_{M/(M\cap H)}\left|\int_{A/(A\cap H)}a^{\lambda-\rho_{Q}}f(kmaH_{\emptyset})\,da\right|^{2}\,dm\,d\lambda\,dk.
\end{align*}
By Fubini's theorem the function
$$
M/(M\cap H)\ni m\mapsto \int_{A/(A\cap H)}a^{\lambda-\rho_{Q}}f(kmaH_{\emptyset})\,da
$$
is square integrable for almost every $k\in K$ and $\lambda\in i(\fa/\fa_{\fh})^{*}$.
For every finite dimensional representation  $\sigma$ of $M$ we choose an orthonormal basis $E_{\sigma}$ of $V_{\sigma}^{M\cap H}$. The set of equivalence classes of irreducible unitary representations of $M$ we denote by $\widehat{M}$. We now apply the Peter-Weyl theorem for $M/M\cap H$ and obtain
\begin{align*}
&\int_{M/(M\cap H)}\left|\int_{A/(A\cap H)}a^{\lambda-\rho_{Q}}f(kmaH_{\emptyset})\,da\right|^{2}\,dm\\
&\qquad=\sum_{[\sigma]\in \widehat{M}}\dim(V_{\sigma})\sum_{\eta\in E_{\sigma}}
    \left\|\int_{M/(M\cap H)}\int_{A/(A\cap H)}a^{\lambda-\rho_{Q}}f(kmaH_{\emptyset})\sigma(m)\eta\,da\,dm\right\|^{2}_{\sigma}.
\end{align*}
In view of Lemma \ref{Lemma decomposition of M_Q} and Corollary \ref{Cor M_Q,fu simeq M reps with trivial restriction to M cap L_Q,nc} we may replace $M$ by $M_{Q}$, hence the right-hand side equals
$$
\sum_{[\xi]\in \widehat{M}_{Q,\mathrm{fu}}}\dim(V_{\xi})\sum_{\eta\in E_{\xi,e}}
    \left\|\int_{M_{Q}/(M_{Q}\cap H)}\int_{A/(A\cap H)}a^{\lambda-\rho_{Q}}f(kmaH_{\emptyset})\xi(m)\eta\,da\,dm\right\|^{2}_{\xi}.
$$
Here $E_{\xi,e}$ denotes a choice of an orthonormal basis of $V_{\xi}^{M_{Q}\cap H}$.
By (\ref{eq Identity cF_empty - 1}) in Lemma \ref{Lemma Identities cF_empty}
\begin{align*}
\int_{Z_{\emptyset}}|f(z)|^{2}\,dz
&=\sum_{[\xi]\in \widehat{M}_{Q,\mathrm{fu}}}\dim(V_{\xi})\sum_{\eta\in E_{\xi,e}}\int_{i(\fa/\fa_{\fh})^{*}}
\int_{K}\big\|\Ft_{\emptyset}f(\xi:\lambda)\big(\iota_{e}\eta\big)(k)\big\|_{\xi}^{2}\,dk\,d\lambda\\
&=\sum_{[\xi]\in \widehat{M}_{Q,\mathrm{fu}}}\dim(V_{\xi})\sum_{\eta\in E_{\xi,e}}\int_{i(\fa/\fa_{\fh})^{*}}
\big\|\Ft_{\emptyset}f(\xi:\lambda)\big(\iota_{e}\eta\big)\big\|_{\overline{Q},\xi,\lambda}^{2}\,d\lambda.
\end{align*}
Since $\fN$ is a set of representatives of $\cN/\cZ$ in $K\cap\cN$ and $(\fa/\fa_{\fh})^{*}_{+}$ a fundamental domain for the action of $\cN/\cZ$ on $(\fa/\fa_{\fh})^{*}$, the right-hand side equals the sum over $v\in\fN$ of
$$
\sum_{[\xi]\in \widehat{M}_{Q,\mathrm{fu}}}\dim(V_{v^{-1}\cdot\xi})\sum_{\eta\in E_{v^{-1}\cdot\xi,e}}\int_{i(\fa/\fa_{\fh})^{*}_{+}}
\big\|\Ft_{\emptyset}f(v^{-1}\cdot\xi:\Ad^{*}(v^{-1})\lambda)\big(\iota_{e}\eta\big)\big\|_{\overline{Q},v^{-1}\cdot\xi,\Ad^{*}(v^{-1})\lambda}^{2}\,d\lambda.
$$
Since $V_{v^{-1}\cdot\xi}^{M_{Q}\cap H}=V_{\xi}^{M_{Q}\cap vHv^{-1}}$, the set $E_{\xi,v}:=E_{v^{-1}\cdot\xi,e}$ is an orthonormal basis of $V_{\xi}^{M_{Q}\cap vHv^{-1}}$.
Therefore, $\bigcup_{v\in\fN}E_{\xi,v}$ is an orthonormal basis of $V_{\emptyset}^{*}(\xi)$.
We now apply (\ref{eq Identity cF_empty - 3}). This yields
\begin{align*}
\int_{Z_{\emptyset}}|f(z)|^{2}\,dz
&=\sum_{v\in\fN}\sum_{[\xi]\in \widehat{M}_{Q,\mathrm{fu}}}\dim(V_{\xi})\sum_{\eta\in E_{\xi,v}}\int_{i(\fa/\fa_{\fh})^{*}_{+}}
\big\|\Ft_{\emptyset}f(\xi:\lambda)\big(\iota_{v}\eta\big)\big\|_{\overline{Q},\xi,\lambda}^{2}\,d\lambda\\
&=\sum_{[\xi]\in \widehat{M}_{Q,\mathrm{fu}}}\dim(V_{\xi})\sum_{\eta\in E_{\xi}}\int_{i(\fa/\fa_{\fh})^{*}_{+}}
\big\|\Ft_{\emptyset}f(\xi:\lambda)\eta\big\|_{\overline{Q},\xi,\lambda}^{2}\,d\lambda\\
&=\sum_{[\xi]\in \widehat{M}_{Q,\mathrm{fu}}}\dim(V_{\xi})\int_{i(\fa/\fa_{\fh})^{*}_{+}}
\big\|\Ft_{\emptyset}f(\xi:\lambda)\big\|_{\emptyset,\xi,\lambda}^{2}\,d\lambda.
\end{align*}
This proves (\ref{eq Plancherel Thm horospherical case - Pl identity}) for $f\in \cD(Z_{\emptyset})$.

From (\ref{eq Plancherel Thm horospherical case - Pl identity}) and the density of $\cD(Z_{\emptyset})$ in $L^{2}(Z_{\emptyset})$ it follows that $f\mapsto\Ft_{\emptyset}f$ extends uniquely to a continuous linear operator (\ref{eq Plancherel Thm horospherical case}) and the identity (\ref{eq Plancherel Thm horospherical case - Pl identity}) holds for $f\in L^{2}(Z_{\emptyset})$ as well.
\end{proof}

\subsection{Constant Term}
\label{Subsection Most continuous part - Constant Term}
In Section \ref{Subsection Most continuous part - Plancherel decomposition} we will relate the Plancherel decomposition for $Z_{\emptyset}$ to the decomposition of $L^{2}_{\mc}(Z)$ using the Maa\ss-Selberg relations from \cite[\S 9.4]{DelormeKnopKrotzSchlichtkrull_PlancherelTheoryForRealSphericalSpacesConstructionOfTheBernsteinMorphisms}. For this we need a description of the constant term map from Section \ref{Subsection Wave packets - Constant term} for $\cF=\overline{\cC}$.

We recall the finite union of proper subspaces $\cS\subseteq (\fa/\fa_{\fh})^{*}$ and the set of elements $\{x_{\cO}:\cO\in (P\bs Z)_{\fa_{\fh}}\}$ from Section \ref{Subsection Construction - Description}.
For a finite dimensional unitary representation $\xi$ of $M_{Q}$, $\lambda\in i(\fa/\fa_{\fh})^{*}\setminus i\cS$ and $\cO\in (P\bs Z)_{\open}$ we write
$$
\Const_{\cO}(\xi:\lambda):\cD'(\overline{Q}:\xi:\lambda)^{H}\to\cD'(\overline{Q}:\xi:\lambda)^{H_{\emptyset}};
\quad \mu\mapsto \mu_{x_{\cO}H,\overline{\cC}}
$$
for the constant term map for the adapted point $z=x_{\cO}H$ and face $\cF=\overline{\cC}$.
Since we only consider $\cF=\overline{\cC}$ we have dropped the subscript $\cF$.

Our description of the constant term map will be given in terms of the intertwining operators from Section \ref{Subsection Distribution vectors - Intertwining operators}. We recall the choice of a set of representatives $\fN$ of $\cN/\cZ$ in $\cN\cap K$  and the space $V_{\emptyset}^{*}(\xi)$ from Section \ref{Subsection Most continuous part - Plancherel decomposition for Z_empty}.

\begin{Prop}\label{Prop Const mu formula}
Let $\xi$ be a finite dimensional unitary representation of $M_{Q}$, $\lambda\in i(\fa/\fa_{\fh})^{*}\setminus i\cS$ and $\cO\in (P\bs Z)_{\open}$. For every $\mu\in \cD'(\overline{Q}:\xi:\lambda)^{H}$ the distribution
$$
\cA\big(Q:\overline{Q}:v^{-1}\cdot \xi:\Ad^{*}(v^{-1})\lambda\big)\circ \cI^{\circ}_{v^{-1}}(\xi:\lambda)(\mu)
$$
is smooth in the point $x_{\cO}$. We have
$$
\Const_{\cO}(\xi:\lambda)\mu
=\mu^{\circ}_{\emptyset}(\xi:\lambda)\eta,
$$
where $\eta\in V^{*}_{\emptyset}(\xi)$ is given by
$$
\eta_{v}
=\ev_{x_{\cO}}\circ \cA\big(Q:\overline{Q}:v^{-1}\cdot \xi:\Ad^{*}(v^{-1})\lambda\big)\circ \cI^{\circ}_{v^{-1}}(\xi:\lambda)(\mu)
\qquad(v\in \fN).
$$
\end{Prop}

\begin{proof}
Let $\mu\in\cD'(\overline{Q}:\xi:\lambda)^{H}$.
Since $\Const_{\cO}(\xi:\lambda)\mu\in \cD'(\overline{Q}:\xi:\lambda)^{H_{\emptyset}}$, it follows from Corollary \ref{Cor Description D'(Z,overline Q:xi:lambda) horospherical case} that there exists an $\eta\in V_{\emptyset}^{*}(\xi)$ so that $\Const_{\cO}(\xi:\lambda)\mu=\mu^{\circ}_{\emptyset}(\xi:\lambda)\eta$. Moreover, $\eta$ is given by
$$
\eta_{v}
=\ev_{e}\circ\cA\big(Q:\overline{Q}:v^{-1}\cdot \xi:\Ad^{*}(v^{-1})\lambda\big)\circ \cI^{\circ}_{v^{-1}}(\xi:\lambda)\circ\Const_{\cO}(\xi:\lambda)\mu
\qquad(v\in\fN).
$$
We set
$$
\mu_{v}
:=\cA\big(Q:\overline{Q}:v^{-1}\cdot \xi:\Ad^{*}(v^{-1})\lambda\big)\circ \cI^{\circ}_{v^{-1}}(\xi:\lambda)(\mu)
$$
and
$$
\mu_{v,\emptyset}
:=\cA\big(Q:\overline{Q}:v^{-1}\cdot \xi:\Ad^{*}(v^{-1})\lambda\big)\circ \cI^{\circ}_{v^{-1}}(\xi:\lambda)(\mu)\circ\Const_{\cO}(\xi:\lambda)(\mu).
$$
It then suffices to prove that for every $v\in\fN$
\begin{equation}\label{eq Const computation}
\ev_{x_{\cO}}(\mu_{v})
=\ev_{e}(\mu_{v,\emptyset}).
\end{equation}

It follows from  Theorem \ref{Thm constant term} that for every $X\in\cC$ the limit
$$
\lim_{t\to \infty} e^{-t\rho_{Q}(X)}\Big(R^{\vee}\big(\exp(tX)x_{\cO}\big)\mu_{v}-R^{\vee}\big(\exp(tX)\big)\mu_{v,\emptyset}\Big)
$$
exists and equals $0$.
Since $\mu_{v}$ is contained in $\cD'\big(Q:v^{-1}\cdot\xi:\Ad^{*}(v^{-1})\lambda\big)^{H}$, it is given by a smooth function on the open subset $\cO$. Let $\chi=\ev_{x_{\cO}}(\mu_{v})$. Then
$$
\mu_{v}(manx_{\cO}h)=a^{-\lambda+\rho_{Q}}\xi^{\vee}(m)\chi
\qquad\big(m\in M, a\in A, n\in N_{P}, h\in H).
$$
By Lemma \ref{Lemma Occurence of M-reps in regular functions} there exists a $\nu\in(\fa/\fa_{\fh})^{*}$ and a regular function $f_{\chi}:G\to V_{\xi}^{*}$ so that
$$
f_{\chi}(manx_{\cO}h)=a^{\nu}\xi^{\vee}(m)\chi
\qquad(m\in M, a\in A, n\in N_{P}, h\in H).
$$
Let $\nu_{1},\dots,\nu_{r}\in \Lambda$ be a basis of $(\fa/\fa_{\fh})^{*}$. Then there exist regular functions $f_{1},\dots,f_{r}:G\to\R$ so that
$$
f_{j}(manx_{\cO}h)=a^{\nu_{j}}
\qquad(1\leq j\leq r, m\in M, a\in A, n\in N_{P}, h\in H).
$$
Note that each $f_{j}$ is real valued and thus $f_{j}^{2}$ is non-negative.
Now
$$
\mu_{v}\big|_{\cO}
=\Big(\prod_{j=1}^{r}\big(f_{j}^{2}\big)^{u_{j}}f_{\chi}\Big)\big|_{\cO}
$$
where $u_{j}\in \C$ is determined by
$$
-\lambda+\rho_{Q}-\nu
=2\sum_{j=1}^{r}u_{j}\nu_{j}.
$$
Let $V$ be the span of $R(G)f_{\chi}$. Then $V$ is finite dimensional and the restriction of $R$ to $V$ has lowest weight $\nu$. Note that $f_{\chi}$ is an $H$-fixed vector in $V$.
The limit of
$$
e^{-t\nu(X)}R\big(\exp(tX)x_{\cO}\big)f_{\chi}
$$
for $t\to\infty$ exist and is a non-zero lowest weight vector in $V$. In fact,
$$
\ev_{e}\Big(\lim_{t\to\infty}e^{-t\nu(X)}R\big(\exp(tX)x_{\cO}\big)f_{\chi}\Big)
=f_{\chi}(x_{\cO})
=\chi.
$$
Likewise, for every $1\leq j\leq r$ the span $V_{j}$ of $R(G)f_{j}$ is finite dimensional and the restriction of $R$ to $V_{j}$ has lowest weight $\nu_{j}$. The limits of
$$
e^{-t\nu_{j}(X)}R\big(\exp(tX)x_{\cO}\big)f_{j}
\qquad(1\leq j\leq r)
$$
for $t\to\infty$ exist, and
$$
\ev_{e}\Big(\lim_{t\to\infty}e^{-t\nu_{j}(X)}R\big(\exp(tX)x_{\cO}\big)f_{j}\Big)
=f_{j}(x_{\cO})
=1
\qquad(1\leq j\leq r).
$$
It follows that
$$
e^{-t\nu(X)}R^{\vee}\big(\exp(tX)x_{\cO}\big)\mu_{v}
$$
converges for $t\to\infty$ uniformly on a neighborhood of $e$ in $G$, and the limit $\mu_{v,\emptyset}$ satisfies
$$
\ev_{e}(\mu_{v,\emptyset})
=\ev_{e}\Big(\lim_{t\to\infty}e^{-t\nu(X)}R^{\vee}\big(\exp(tX)x_{\cO}\big)\mu_{v}\Big)
=\chi
=\ev_{x_{\cO}}(\mu_{v}).
$$
This establishes (\ref{eq Const computation}).
\end{proof}

In view of Theorem \ref{Thm Description D'(overline Q:xi:lambda)^H} and Corollary \ref{Cor Description D'(Z,overline Q:xi:lambda) horospherical case} there exists for every $\cO\in(P\bs Z)_{\open}$ a unique linear map
$$
\const_{\cO}(\xi:\lambda):V^{*}(\xi)\to V^{*}_{\emptyset}(\xi)
$$
so that the diagram
$$
\xymatrix{
   \cD'(\overline{Q}:\xi:\lambda)^{H} \ar[rr]^{\Const_{\cO}(\xi:\lambda)}
        && \cD'(\overline{Q}:\xi:\lambda)^{H_{\emptyset}}  \\
\\
   V^{*}(\xi)\ar@<-2pt>[uu]^{\mu^{\circ}(\xi:\lambda)} \ar[rr]^{\const_{\cO}(\xi:\lambda)}
       && V_{\emptyset}^{*}(\xi)\ar@<-2pt>[uu]_{\mu_{\emptyset}^{\circ}(\xi:\lambda)}
}
$$
commutes. We end this section with a description of this map $\const_{\cO}(\xi:\lambda)$ in terms of the $B$-matrices and the map $\beta(\xi:\lambda)$  from Section \ref{Subsection Construction - B-matrices}.

We recall the maps $s_{w}$ for $w\in\cN$ from (\ref{eq def s_w}).

\begin{Prop}\label{Prop const formula}
Let $\xi$ be a finite dimensional unitary representation of $M_{Q}$, $\lambda\in i(\fa/\fa_{\fh})^{*}\setminus i\cS$ and $\cO\in (P\bs Z)_{\open}$.
Then for every $\eta\in V^{*}(\xi)$ and $v\in \fN$
\begin{equation}\label{eq Formula const}
\Big(\const_{\cO}(\xi:\lambda)\eta\Big)_{v}
=\frac{1}{\gamma(v\overline{Q}v^{-1}:\overline{Q}:\xi:\lambda)}\Big(\cB_{v^{-1}}(Q:\xi:\lambda)\circ\beta(\xi:\lambda)^{-1}\eta\Big)_{\cO}.
\end{equation}
In particular, if $\eta\in V^{*}(\xi)$, $\cO\in (P\bs Z)_{\open}$ and $v_{w}$ is the representative in $K\cap \cN$ of an element $w\in \cN/\cW$ from Section \ref{Subsection Construction - Description}, then
\begin{equation}\label{eq Formula const for v_w components}
\Big(\const_{\cO}(\xi:\lambda)\eta\Big)_{v_{w}}
=\eta_{s_{v_{w}}(\cO)}.
\end{equation}
\end{Prop}

\begin{proof}
Let $\eta\in V^{*}(\xi)$, $\cO\in (P\bs Z)_{\open}$ and $v\in\fN$. Then by Proposition \ref{Prop Const mu formula}
$$
\Big(\const_{\cO}(\xi:\lambda)\eta\Big)_{v}
=\ev_{x_{\cO}}\circ \cA\big(Q:\overline{Q}:v^{-1}\cdot \xi:\Ad^{*}(v^{-1})\lambda\big)\circ \cI^{\circ}_{v^{-1}}(\xi:\lambda)\circ\mu^{\circ}(\xi:\lambda)(\eta).
$$
Using (\ref{eq Def mu^circ}) and  the identity
\begin{align*}
&\cA\big(Q:\overline{Q}:v^{-1}\cdot\xi:\Ad^{*}(v^{-1})\lambda\big)\circ\cI^{\circ}_{v^{-1}}(\xi:\lambda)
    \circ\cA\big(Q:\overline{Q}:\xi:\lambda\big)^{-1}\\
&\qquad=\frac{1}{\gamma\big(v\overline{Q}v^{-1}:\overline{Q}:\xi:\lambda\big)} L^{\vee}(v^{-1})\circ\cA\big(vQv^{-1}:Q:\xi:\lambda\big)\\
&\qquad=\frac{1}{\gamma\big(v\overline{Q}v^{-1}:\overline{Q}:\xi:\lambda\big)} I_{v^{-1}}(Q:\xi:\lambda),
\end{align*}
we find
\begin{align*}
&\Big(\const_{\cO}(\xi:\lambda)\eta\Big)_{v}\\
&\qquad=\frac{1}{\gamma(v\overline{Q}v^{-1}:\overline{Q}:\xi:\lambda)}\ev_{x_{\cO}}\circ I_{v^{-1}}(Q:\xi:\lambda)\circ\mu(Q:\xi:\lambda)
\circ\beta(\xi:\lambda)^{-1}(\eta).
\end{align*}
By (\ref{eq B-matrix diagram}) we thus have
\begin{align*}
&\Big(\const_{\cO}(\xi:\lambda)\eta\Big)_{v}\\
&\qquad=\frac{1}{\gamma(v\overline{Q}v^{-1}:\overline{Q}:\xi:\lambda)}\ev_{x_{\cO}}\circ \mu(Q:\xi:\lambda)\circ \cB_{v^{-1}}(Q:\xi:\lambda)
\circ\beta(\xi:\lambda)^{-1}(\eta)\\
&\qquad=\frac{1}{\gamma(v\overline{Q}v^{-1}:\overline{Q}:\xi:\lambda)}\Big(\cB_{v^{-1}}(Q:\xi:\lambda)\circ\beta(\xi:\lambda)^{-1}(\eta)\Big)_{\cO}.
\end{align*}
This proves (\ref{eq Formula const}). The identity (\ref{eq Formula const for v_w components}) follows from (\ref{eq Formula const}) and the definition (\ref{eq Def beta}) of the function $\beta(\xi:\lambda)$.
\end{proof}

For a finite dimensional unitary representation $\xi$ of $M_{Q}$ and $v\in\fN$ we define the space
\begin{equation}\label{eq Def V_empty,w(xi)}
V_{\emptyset,v}^{*}(\xi)
:=\bigoplus_{\cO\in (P\bs Z)_{\open}}(V_{\xi}^{*})^{M_{Q}\cap vHv^{-1}}.
\end{equation}
We view $V_{\emptyset,v}^{*}(\xi)$ as a subspace of $\bigoplus_{\cO\in(P\bs Z)_{\open}}V_{\emptyset}^{*}(\xi)$.
We now reorder the components of the constant term maps and thus define
$$
\const_{v}(\xi:\lambda)
:V^{*}(\xi)\to V_{\emptyset,v}^{*}(\xi)\subseteq\bigoplus_{\cO\in (P\bs Z)_{\open}}V_{\emptyset}^{*}(\xi)
$$
by setting
$$
\Big(\const_{v}(\xi:\lambda)\eta\Big)_{\cO}
:=\pr_{v}\circ\const_{\cO}(\xi:\lambda)
\qquad\big(\eta\in V^{*}(\xi),\cO\in (P\bs Z)_{\open}\big),
$$
where
\begin{equation}\label{eq Def pr_w}
\pr_{v}:V_{\emptyset}^{*}(\xi)\to (V_{\xi}^{*})^{M_{Q}\cap vHv^{-1}};
\quad\eta\mapsto \eta_{v}.
\end{equation}
Now Proposition \ref{Prop const formula} has the following corollary.

\begin{Cor}\label{Cor Formula const_v_w}
Let $\xi$ be a finite dimensional unitary representation of $M_{Q}$ and $\lambda\in i(\fa/\fa_{\fh})^{*}\setminus i\cS$.
Then for every $w\in \cN/\cW$
$$
\const_{v_{w}}(\xi:\lambda)\eta
=\big(\eta_{s_{v_{w}}(\cO)}\big)_{\cO\in (P\bs Z)_{\open}}
\in V_{\emptyset,v_{w}}^{*}(\xi).
$$

\end{Cor}

\begin{proof}
The identity is a reformulation of (\ref{eq Formula const for v_w components}).
\end{proof}

\subsection{Invariant differential operators}
\label{Subsection Most continuous part - Invariant differential operators}

Let $\xi$ be a finite dimensional unitary representation of $M_{Q}$. For $w\in \cN/\cW$ define the subspace of $V^{*}(\xi)$
\begin{equation}\label{eq Decom V^*(xi)}
V_{w}^{*}(\xi)
:=\bigoplus_{\cO\in w\cdot(P\bs Z)_{\open}}\big(V_{\xi}^{*}\big)^{M_{Q}\cap v_{w}Hv_{w}^{-1}}.
\end{equation}
We view $V^{*}_{w}(\xi)$ as a subspace of $V^{*}(\xi)$.
In this section we show that the subspaces $\mu^{\circ}(\xi:\lambda)\big(V^{*}_{w}(\xi)\big)$ of $\cD'(\overline{Q}:\xi:\lambda)^{H}$ for $w\in\cN/\cW$ are spectrally separated by the invariant differential operators on $Z$.
Recall the maps $\iota_{\cO}$ for $\cO\in (P\bs Z)_{\fa_{\fh}}$ from (\ref{eq Def iota_O}).

\begin{Prop}\label{Prop distributions separated by inv diff operators}
Let $\xi$ be a finite dimensional unitary representation of $M_{Q}$. For  every $w\in \cN/\cW$ and for $\lambda\in i(\fa/\fa_{\fh})^{*}$ outside of a finite union of proper subspaces  there exists a differential operator $D_{w}$ in the center of $\Diff(Z)$ so that
$$
D_{w}\circ\mu^{\circ}(\xi:\lambda)\circ\iota_{\cO}
=\left\{
   \begin{array}{ll}
     \mu^{\circ}(\xi:\lambda)\circ\iota_{\cO} & (\cO\in w\cdot(P\bs Z)_{\open}), \\
     0 & (\cO\notin w\cdot(P\bs Z)_{\open}).
   \end{array}
 \right.
$$
\end{Prop}

Before we prove the proposition we first give a corollary, which we will use in Section \ref{Subsection Most continuous part - Plancherel decomposition}.
By Corollary \ref{Cor Multiplicity space is V^*(xi)} the multiplicity spaces $\cM_{\xi,\lambda}$ in (\ref{eq abstract Plancherel decomp L^2_mc}) can for almost every $\lambda\in i(\fa/\fa_{\fh})^{*}$  be identified with $V^{*}(\xi)$ via the map $\mu^{\circ}(\xi:\lambda)$. The multiplicity spaces are Hilbert spaces and thus equipped with an inner product. We write $\langle\cdot,\cdot\rangle_{\Pl,\xi,\lambda}$ for the inner product on $V^{*}(\xi)$ that is induced by the inner product on the multiplicity space $\cM_{\xi,\lambda}$.

\begin{Cor}\label{Cor Orthogonal decomposition V^*(xi)}
The decomposition
$$
V^{*}(\xi)
=\bigoplus_{w\in \cN/\cW}V^{*}_{w}(\xi).
$$
is for almost every $\lambda\in i(\fa/\fa_{\fh})^{*}$ orthogonal with respect to the inner product $\langle\cdot,\cdot\rangle_{\Pl,\xi,\lambda}$.
\end{Cor}

\begin{proof}
Every differential operator $D\in \Diff(Z)$ defines a operator on $L^{2}(Z)$ with domain $\cD(Z)$. To every $D\in \Diff(Z)$ we can associate a formal adjoint $D^{*}$ which is defined by
$$
\int_{Z}D^{*}\phi(z)\overline{\psi(z)}\,dz
=\int_{Z}\phi(z)\overline{D\psi(z)}\,dz
\qquad\big(\phi,\psi\in\cD(Z)\big).
$$
It is easy to see that $D^{*}$ is a $G$-invariant differential operators and thus is contained in $\Diff(Z)$. Furthermore, if $D$ is contained in the center of $\Diff(Z)$, then for all $\phi,\psi\in\cD(Z)$ and $D'\in \Diff(Z)$
\begin{align*}
\int_{Z}D'D^{*}\phi(z)\overline{\psi(z)}\,dz
&=\int_{Z}\phi(z)\overline{DD'^{*}\psi(z)}\,dz
=\int_{Z}\phi(z)\overline{D'^{*}D\psi(z)}\,dz\\
&=\int_{Z}D^{*}D'\phi(z)\overline{\psi(z)}\,dz,
\end{align*}
and hence $D^{*}$ is contained in the center of $\Diff(Z)$ as well.

Since the differential operators in $\Diff(Z)$ commute with the regular representation of $G$,  each $D\in \Diff(Z)$ induces for $\xi\in\widehat{M}_{Q,\mathrm{fu}}$ and $\lambda\in i(\fa/\fa_{\fh})^{*}$ an operator $r(\xi:\lambda)(D)$ on the multiplicity space $\cM_{\xi,\lambda}=V^{*}(\xi)$. The operator is given by
\begin{equation}\label{eq Def r(xi,lambda)(D)}
D\circ\mu^{\circ}(\xi:\lambda)
=\mu^{\circ}(\xi:\lambda)\circ r(\xi:\lambda)(D)
\qquad\big(D\in \Diff(Z)\big).
\end{equation}
Furthermore, if ${}^{\dagger}$ denotes the hermitian conjugate with respect to the inner product $\langle\cdot,\cdot,\rangle_{\Pl,\xi,\lambda}$, then
$$
r(\xi:\lambda)(D^{*})
=\big(r(\xi:\lambda)(D)\big)^{\dagger}
\qquad\big(D\in \Diff(Z)\big).
$$
If $D$ is contained in the center of $\Diff(Z)$, then $r(\xi,\lambda)(D)$ commutes with $r(\xi,\lambda)(D)^{\dagger}$, and hence $r(\xi,\lambda)(D)$ is normal. In particular, eigenspaces corresponding to different eigenvalues are orthogonal to each other. The assertion now follows from Proposition \ref{Prop distributions separated by inv diff operators}.

\end{proof}

In the remainder of this section we give the proof of Proposition \ref{Prop distributions separated by inv diff operators}. We begin with a description and comparison of $\Diff(Z)$ and $\Diff(Z_{\emptyset})$.

The right-action induces a natural isomorphism
$$
\Diff(Z)
\simeq \cU(\fg)_{H}/\cU(\fg)\fh,
$$
where
$$
\cU(\fg)_{H}
:=\big\{u\in \cU(\fg):\Ad(h)u-u\in \cU(\fg)\fh\text{ for all }h\in H\big\}.
$$
Likewise, we have
$$
\Diff(Z_{\C})
\simeq \cU(\fg)_{H_{\C}}/\cU(\fg)\fh,
$$
where
$$
\cU(\fg)_{H_{\C}}
:=\big\{u\in \cU(\fg):\Ad(h)u-u\in \cU(\fg)\fh\text{ for all }h\in H_{\C}\big\}.
$$
Clearly, $\Diff(Z_{\C})\subseteq \Diff(Z)$.
Let $\fU(Z_{\C})$ be the algebra of completely regular differential operators on $Z_{\C}$ from \cite[Section 3]{Knop_HarishChandraHomomorphism} and let $\Cen(Z_{\C})$ be its center. We consider $\Cen(Z_{\C})$ as a subalgebra of $\Diff(Z)$. By \cite[Corollary 9.2]{Knop_HarishChandraHomomorphism} $\Cen(Z_{\C})$ is contained in the center of $\Diff(Z_{\C})$. Knop's Harish-Chandra homomorphism identifies $\Cen(Z_{\C})$ as a polynomial ring.

In order to describe Knop's Harish-Chandra homomorphism we first apply the local structure theorem, \cite[Theorem 4.2]{KnopKrotz_ReductiveGroupActions}, to $Z_{\C}$. Let $B$ be a Borel subgroup of $G_{\C}$ that is contained in $P_{\C}$. The local structure theorem then yields a parabolic subgroup $R$ of $G_{\C}$ and a Levi-decomposition $R=L_{R}N_{R}$ so that $R\cap H_{\C}=L_{R}\cap H_{\C}$ is a normal subgroup of $L_{R}$ and $L_{R}/(L_{R}\cap H_{\C})$ is a torus. By \cite[Lemma  9.3]{KnopKrotz_ReductiveGroupActions} $L_{R}$ may be chosen so that $A_{\C}\subseteq L_{R}$. Let now $\ft$ be a maximal abelian subalgebra of $\fm\cap\fl_{R}$. Then $\fj:=\fa\oplus\ft$ is a Cartan subalgebra of $\fg$. Without loss of generality we may assume that $\fj_{\C}=\fa_{\C}\oplus \ft_{\C}$ is contained in the Lie algebra of $B$. Let $\rho_{B}$ be the half-sum of the roots of $\Lie(B)$ in $\fj$. Further, let $W_{Z,\C}$ be the little Weyl group of $Z_{\C}$, see \cite[(9.13)]{KnopKrotz_ReductiveGroupActions}, and $W_{\C}$ the Weyl group of the root system of $\fg_{\C}$ in $\fj$. Finally, let $\gamma:\Cen(\fg)\to\C[\fj_{\C}^{*}]^{W_{\C}}$ be the Harish-Chandra isomorphism.
By \cite[Theorem 6.5]{Knop_HarishChandraHomomorphism} there exists a canonical isomorphism $\gamma_{Z_{\C}}:\Cen(Z_{\C})\to\C[-\rho_{B}+(\fj/\fj\cap \fh)_{\C}^{*}]^{W_{Z,\C}}$ so that the diagram
\begin{equation}\label{eq Knop HC-iso}
\xymatrixcolsep{3pc}\xymatrix{
   \Cen(\fg)\ar@<-2pt>[d]\ar[r]^{\gamma}    &  \C[\fj_{\C}]^{W_{\C}}\ar@<-2pt>[d]\\
  \Cen(Z_{\C})\ar[r]^-{\gamma_{Z_{\C}}} & \C\big[-\rho_{B}+(\fj/\fj\cap \fh)_{\C}^{*}\big]^{W_{Z,\C}}
}
\end{equation}
commutes. Here the right vertical arrow is the restriction map. The map $\gamma_{Z_{\C}}$ is Knop's Harish-Chandra homomorphism.

Let $\fb=\fm\oplus\fa\oplus\fn_{Q}$ and $\fb_{H}=(\fm\cap\fh)\oplus\fa_{\fh}$. Then $\cU(\fb)\fb_{H}$ is a two-sided ideal of $\cU(\fb)$. From the Poincaré-Birkhoff-Witt theorem we obtain the isomorphism
$$
\cU(\fb)_{H}/\cU(\fb)\fb_{H}
\simeq \cU(\fg)_{H}/\cU(\fg)\fh
\simeq\Diff(Z),
$$
where $\cU(\fb)_{H}=\cU(\fb)\cap \cU(\fg)_{H}$.
We thus may view $\Diff(Z)$ as a subring of $\cU(\fb)/\cU(\fb)\fb_{H}$. The same holds for $\Diff(Z_{\emptyset})$.
We recall from \cite[Lemma 5.2]{DelormeKrotzSouaifi_ConstantTerm} that the limit
$$
\lim_{t\to\infty}\Ad\big(\exp(tX)\big)D
$$
exists (in $\cU(\fb)/\cU(\fb)\fb_{H}$) for every $D\in \Diff(Z)$ and $X\in \cC$ and defines a $G$-invariant differential operator on $Z_{\emptyset}$. The limit does not depend on the choice of $X$. Moreover, the map
$$
\delta_{\emptyset}:\Diff(Z)\to\Diff(Z_{\emptyset});\quad D\mapsto \lim_{t\to\infty}\Ad\big(\exp(tX)\big)D
$$
is an injective algebra morphism.

\begin{Lemma}\label{Lemma Knop HC-hom diagram commutes}
The diagram
$$
\xymatrixcolsep{3pc}\xymatrix{
\Cen(\fg)\ar@<-2pt>[d]\ar[r]^{\gamma}    &  \C[\fj_{\C}]^{W_{\C}}\ar@<-2pt>[d]\\
  \Cen(Z_{\C}) \ar@{^{(}->}[d]^{\delta_{\emptyset}}\ar[r]^-{\gamma_{Z_{\C}}}&  \C\big[-\rho_{B}+(\fj/\fj\cap \fh)_{\C}^{*}\big]^{W_{Z,\C}} \ar@{^{(}->}[d]\\
  \Cen(Z_{\emptyset,\C})\ar[r]^-{\gamma_{Z_{\emptyset,\C}}} & \C\big[-\rho_{B}+(\fj/\fj\cap \fh)_{\C}^{*}\big]^{W_{Z_{\emptyset},\C}}
}
$$
commutes.
\end{Lemma}

\begin{proof}
We first follow \cite[Section 9]{Delorme_ScatteringOperators} by considering the partial toroidal $G_{\C} \times \C^{*}$-compactification $Y_{\emptyset}$ of $Z_{\C}\times \C^{*}$ attached to the fan with only one non trivial cone $\R^{+}(X, 1)$, where $X\in\cC/\fa_{\fh}$ is an element contained in the lattice of cocharacters of $\fa/\fa_{\fh}$.
By functoriality of compactifications we then find a regular map $\Delta_{\emptyset}:Y_{\emptyset} \to\C$ so that the diagram
$$
\xymatrixcolsep{3pc}\xymatrix{
Z_{\C}\times\C^{*}\ar@<-2pt>[d]\ar[r]    &  Y_{\emptyset}\ar@<-2pt>[d]^{\Delta_{\emptyset}}\\
  \C^{*}\ar@{^{(}->}[r]&  \C
}
$$
commutes. Now $\Delta_{\emptyset}^{-1}(\{t\})$ is isomorphic to $Z_{\C}$ for $t\neq 0$, and to $Z_{\emptyset,\C}$ for $t=0$. In the same way as in the proof for \cite[(9.1)]{Delorme_ScatteringOperators} it follows that $Y_{\emptyset}$ is a degeneration in the sense of Knop, \cite[Theorem 1.1]{Knop_HarishChandraHomomorphism}. To be more precise, embed $Z_{\C}$ in $\P(V)$ for some finite dimensional representation $V$ of $G_{\C}$ and let $\tilde{Z}\subseteq V$ be the affine cone of the closure of $Z_{\C}$ in $\P(V)$. We now provide $\C[\tilde{Z}]$ with a filtration induced by the cocharacter $X$ by setting
$$
\C[\tilde{Z}]^{(n)}
=\bigoplus_{k=0}^{n}\C[\tilde{Z}]_{(k)}
\qquad(n\in\N),
$$
where $\C[\tilde{Z}]_{(k)}$ is the eigenspace of $\C[\tilde{Z}]$ for the action of $X$ with eigenvalue $k$. We now define the ring
$$
R
=\bigoplus_{n=0}^{\infty}\C[\tilde{Z}]^{(n)}t^{n}
\subseteq \C[\tilde{Z}][t]
$$
and set $\tilde{Y}_{\emptyset}=\spec(R)$. Let $\tilde{\Delta}_{\emptyset}:\tilde{Y}_{\emptyset}\to \C$ be the map corresponding to inclusion homomorphism $\C[t]\hookrightarrow R$.
Now $Y_{\emptyset}$ is obtained as a suitable open $G_{\C}\times\C^{\times}$-invariant subset of the projectivization of $\tilde{Y}_{\emptyset}$. The map $\tilde{\Delta}_{\emptyset}$ factorizes and in this way we recover the map $\Delta_{\emptyset}$.

Let $\fU(\cX)$ be the algebra of completely regular differential operators on a $G_{\C}$-variety $\cX$, as defined on p. 262 in \cite{Knop_HarishChandraHomomorphism}, and let $\Cen(\cX):=\fU(\cX)^{G_{\C}}$. In view of \cite[Corollary 9.2]{Knop_HarishChandraHomomorphism} $\Cen(\cX)$ is contained in the center of $\Diff(\cX)$.
As in the proof of \cite[Theorem 6.5]{Knop_HarishChandraHomomorphism} we obtain a canonical map
$$
i_{\emptyset}:\fU(Z_{\C})=\fU(Z_{\C}\times\C)=\fU(Y_{\emptyset})\to\fU(Z_{\emptyset,\C}) .
$$
The first two equalities follow from \cite[Lemma 3.5]{Knop_HarishChandraHomomorphism} and the canonical map $\fU(Y_{\emptyset})\to\fU(Z_{\emptyset,\C})$ is obtained by applying \cite[Lemma 3.1]{Knop_HarishChandraHomomorphism} to the injection $Z_{\emptyset,\C}\hookrightarrow Y_{\emptyset}$. We note that $i_{\emptyset}$ maps $\Cen(Z_{\C})$ to $\Cen(Z_{\emptyset,\C})$. With the same arguments as in the proof of \cite[Proposition 9.2]{Delorme_ScatteringOperators} one shows that the restrictions of $i_{\emptyset}$ and $\delta_{\emptyset}$ to $\Cen(Z_{\C})$ coincide.

We now consider a horospherical degeneration $\Delta_{\hor}: Y_{\hor}\to\C$, in the sense of Knop, of the $G_{\C}\times \C^{\times}$-variety $Y_{\emptyset}$. For this, let $X'\in (\fa\oplus i\ft)/\big(\fa_{\fh}\oplus i(\ft\cap \fh)\big)$ be a cocharacter so that $\alpha(X')<0$ for every spherical root $\alpha$ of $Z_{\C}$.
We now provide $\C[\tilde{Z}]$ with a double filtration induced by the cocharacters $X$ and $X'$ by setting
$$
\C[\tilde{Z}]^{(m,n)}
=\bigoplus_{k=0}^{m}\bigoplus_{l=0}^{n}\C[\tilde{Z}]_{(k,l)}
\qquad(n\in\N),
$$
where $\C[\tilde{Z}]_{(k,l)}$ is the joint eigenspace of $\C[\tilde{Z}]$ for the actions of $X$ and $X'$ with eigenvalue $k$ and $l$, respectively. We now define the ring
$$
S
=\bigoplus_{m,n=0}^{\infty}\C[\tilde{Z}]^{(m,n)}s^{m}t^{n}
\subseteq \C[\tilde{Z}][s,t]
$$
and set $\tilde{Y}_{\hor}=\spec(S)$. Let $\tilde{\Delta}:\tilde{Y}_{\hor}\to \C$ be the map corresponding to the inclusion homomorphism $\C[s,t]\hookrightarrow S$.
The map $\tilde{\Delta}$ factorizes through the projectivization of $\tilde{Y}_{\hor}$. This yields a regular $G_{\C}\times \C^{\times}\times \C^{\times}$-equivariant map $\Delta$. By taking a suitable open $G_{\C}\times\C^{\times}\times\C^{\times}$-invariant subset $Y_{\hor}$ of the projectivization of $\tilde{Y}_{\hor}$ and passing to the restriction of $\Delta$ to $Y_{\hor}$, we may arrange things so that the fibers of $\Delta$ satisfy
$$
\Delta^{-1}(\{(s,t)\})\simeq
\left\{
  \begin{array}{ll}
    Z_{\C} & (s,t\neq 0) \\
    Z_{\emptyset,\C} & (s=0, t\neq 0)\\
    \cV\times G_{\C}/S & (s\in\C, t=0),\\
  \end{array}
\right.
$$
where $\cV$ is a complex algebraic variety on which $G_{\C}$ acts trivially and $S$ is a subgroup of $G_{\C}$ containing the unipotent radical of the Borel subgroup $B$.

We now consider the inclusions
\begin{align*}
&\iota_{1}:\cV\times G_{\C}/S=\Delta^{-1}(\{(0,0)\})\hookrightarrow \Delta^{-1}(\{0\}\times\C)=Y_{\emptyset},\\
&\iota_{2}:Y_{\emptyset}= \Delta^{-1}(\{0\}\times\C)\hookrightarrow Y_{\hor},\\
&\iota_{\hor}=\iota_{2}\circ\iota_{1}:\cV\times G_{\C}/S=\Delta^{-1}(\{(0,0)\})\hookrightarrow Y_{\hor}.
\end{align*}
As in the proof of \cite[Theorem 6.5]{Knop_HarishChandraHomomorphism} we may apply \cite[Lemma 3.1]{Knop_HarishChandraHomomorphism} to these inclusions and use this in combination with \cite[Lemma 3.5]{Knop_HarishChandraHomomorphism} to obtain canonical maps
$$
i_{1}:\fU(Z_{\emptyset,\C})\to\fU(G_{\C}/S),\quad
i_{2}:\fU(Z_{\C})\to \fU(Z_{\emptyset,\C}),\quad
i_{\hor}:\fU(Z_{\C})\to\fU(G_{\C}/S).
$$
The maps $\gamma_{Z_{\C}}$ and $\gamma_{Z_{\emptyset,\C}}$ are obtained from $i_{\hor}$ and $i_{1}$ by restricting them to $\Cen(Z_{\C})$ and $\Cen(Z_{\emptyset,\C})$, respectively, and applying \cite[Lemma 6.4.]{Knop_HarishChandraHomomorphism}. As $\iota_{2}\circ\iota_{1}=\iota_{\hor}$, the uniqueness of the maps obtained from\cite[Lemma 3.1]{Knop_HarishChandraHomomorphism} implies that $i_{1}\circ i_{2}=i_{\hor}$. Moreover, it follows from \cite[Lemma 3.5]{Knop_HarishChandraHomomorphism} that  $i_{2}=i_{\emptyset}$. This proves the lemma.
\end{proof}

For every $D\in\Diff(Z)$ the $\fa$-weights that occur in $D-\delta_{\emptyset}(D)$, considered as an element of $\cU(\fb)/\cU(\fb)\fb_{H}$, are strictly negative on $\cC$. When applied to $Z_{\emptyset}$ this leads to
$$
\Diff(Z_{\emptyset})\simeq S(\fa)/S(\fa)\fa_{\fh}\otimes \cU(\fm)_{H}/\cU(\fm)(\fm\cap \fh),
$$
where $\cU(\fm)_{H}:=\cU(\fm)\cap\cU(\fg)_{H}$. We note that $\cU(\fm)_{H}/\cU(\fm)(\fm\cap \fh)$ may be identified with the ring $\Diff\big(M/(M\cap H)\big)$ of $M$-invariant differential operators on $M/(M\cap H)$.
In particular,
$$
\Cen(Z_{\emptyset,\C})
\subseteq S(\fa)/S(\fa)\fa_{\fh}\otimes \Diff\big(M/(M\cap H)\big).
$$
The little Weyl group $W_{Z_{\emptyset},\C}$ of $Z_{\emptyset,\C}$ acts trivially on the subspace $(\fa/\fa_{\fh})^{*}_{\C}$ of $(\fj/\fj\cap\fh)^{*}_{\C}$ and is isomorphic to the little Weyl group $W_{Z,M}$ of $M_{\C}/(M_{\C}\cap H_{\C})$.

The set of characters of $\Cen(Z_{\C})$ is in view of (\ref{eq Knop HC-iso}) in bijection with  $\big(-\rho_{B}+(\fj/\fj\cap \fh)^{*}_{\C}\big)/W_{Z,\C}$. Likewise, the set of characters of $\Cen(Z_{\emptyset,\C})$ is in bijection with $\big(-\rho_{B}+(\fj/\fj\cap \fh)^{*}_{\C}\big)/W_{Z_{\emptyset},\C}$.

\begin{Lemma}\label{Lemma D(Z_empty)-character equals lambda on fa}
Let $\xi$ be a finite dimensional unitary representation of $M_{Q}$ and let $\lambda\in (\fa/\fa_{\fh})^{*}_{\C}$ with $\Im\lambda\notin\cS$. Let further $v\in \fN$. The image of $\mu^{\circ}_{\emptyset}(\xi:\lambda)\circ\iota_{v}$ is stable under the action of $\Diff(Z_{\emptyset})$. Moreover, if the $\Cen(Z_{\emptyset,\C})$-character $W_{Z_{\emptyset},\C}\cdot\nu\in \big(-\rho_{B}+(\fj/\fj\cap \fh)^{*}_{\C}\big)/W_{Z_{\emptyset},\C}$ occurs in the spectrum of the action of $\Cen(Z_{\emptyset,\C})$ on the image of $\mu^{\circ}_{\emptyset}(\xi:\lambda)\circ\iota_{v}$, then
\begin{align*}
&\Im \nu\big|_{\fa}=-\Ad(v^{-1})\lambda,\\
&\nu\big|_{\ft}\in i\ft^{*}/W_{Z,M}.
\end{align*}
\end{Lemma}

\begin{proof}
Let $X\in \fa$. The corresponding invariant differential operator $D_{X}=R^{\vee}(X)$ acts in view of (\ref{eq horospherical case a mu_v =a^(-v^-1 lambda+rho)}) on the image of $\mu^{\circ}_{\emptyset}(\xi:\lambda)\circ\iota_{v}$ by the scalar $-\Ad(v^{-1})\lambda(X)+\rho_{Q}(X)$.
Since the operators $D_{X}$ with $X\in \fa$ are contained in the center of $\Diff(Z_{\emptyset})$, it follows that action of $\Diff(Z_{\emptyset})$ preserves the image of $\mu^{\circ}_{\emptyset}(\xi:\lambda)\circ\iota_{v}$ if $\lambda$ is sufficiently regular. By meromorphic continuation the same then holds for all $\lambda$.

Assume that the $\Cen(Z_{\emptyset,\C})$-character $W_{Z_{\emptyset},\C}\cdot\nu\in \big(-\rho_{B}+(\fj/\fj\cap \fh)^{*}_{\C}\big)/W_{Z_{\emptyset},\C}$ occurs in the spectrum of the action of $\Cen(Z_{\emptyset,\C})$ on the image of $\mu^{\circ}_{\emptyset}(\xi:\lambda)\circ\iota_{v}$.
Since
$$
\Im \big(\gamma_{Z_{\emptyset}}(D_{X})\big)(\nu)
=\Im \nu(X)
$$
it follows that $\Im \nu|_{\fa}=-\Ad(v^{-1})\lambda$.

From the explicit formula (\ref{eq formula mu horospherical case}) for $\mu^{\circ}(\xi:\lambda)\circ\iota_{v}$ for sufficiently anti-dominant $\lambda$ and by using meromorphic continuation one easily sees that $\Cen(\fm)\subseteq\Diff\big(M/(M\cap H)\big)$ acts on the image of $\mu^{\circ}\circ\iota_{v}$ by the infinitesimal character of the restriction $(v^{-1}\cdot \xi)|_{M}$ of $v^{-1}\cdot\xi$ to $M$. Since this infinitesimal character is real, it follows that the restriction of $\nu$ to $\ft$ is contained $i\ft^{*}/W_{Z,M}$.
 \end{proof}

The final ingredient for the proof of Proposition \ref{Prop distributions separated by inv diff operators} is a relation between the constant term and the map $\delta_{\emptyset}$.

\begin{Lemma}\label{Lemma Const D=delta(D) Const}
Let $\xi$ be a finite dimensional unitary representation of $M_{Q}$ and $\lambda\in i(\fa/\fa_{\fh})^{*}\setminus i\cS$. Then
$$
\Const_{\cO}(\xi:\lambda)\circ D
=\delta_{\emptyset}(D)\circ\Const_{\cO}(\xi:\lambda)
\qquad\big(D\in \Diff(Z), \cO\in (P\bs Z)_{\open}\big).
$$
\end{Lemma}

\begin{proof}
After replacing $H$ by $x_{\cO}Hx_{\cO}^{-1}$ we may assume that $\cO=PH$.
Let $\mu\in \cD'(\overline{Q}:\xi:\lambda)^{H}$ and $D\in \Diff(Z)$.  For every $D\in\Diff(Z)$ the $\fa$-weights that occur in $D-\delta_{\emptyset}(D)$, considered as an element of $\cU(\fb)/\cU(\fb)\fb_{H}$, are strictly negative on $\cC$. Therefore, for a fixed $X\in \cC$ there exists an $\epsilon>0$ and a $u\in \cU(\fb)/\cU(\fb)\fb_{H}$ whose weights are non-positive on $\cC$ so that
$$
\Ad\big(\exp(tX)\big)\Big(D-\delta_{\emptyset}(D)\Big)
=e^{-\epsilon t}\Ad\big(\exp(tX)\big)u
\qquad(t>0).
$$
It follows from Theorem \ref{Thm temperedness} that
$$
\lim_{t\to\infty}e^{-t\rho_{Q}(X)-\epsilon t}R^{\vee}\big(\exp(tX)\big)\mu
=0.
$$
Since $\Ad\big(\exp(tX)\big)u$ converges for $t\to\infty$, it follows that
\begin{equation}\label{eq e^-t(rho+epsilon) exp(tX)mu ->0}
\lim_{t\to\infty}
e^{-t\rho_{Q}(X)-\epsilon t}R^{\vee}\big(\exp(tX)\big)R^{\vee}(u)\mu
=0.
\end{equation}
Using that $X$ centralizes $\delta_{\emptyset}(D)$, we obtain
\begin{align*}
&e^{-t\rho_{Q}(X)}\Big(R^{\vee}\big(\exp(tX)\big)D\mu-R^{\vee}\big(\exp(tX)\big)\delta_{\emptyset}(D)\mu_{\emptyset}\Big)\\
&\qquad=e^{-t\rho_{Q}(X)-\epsilon t} R^{\vee}\big(\exp(tX)\big)R^{\vee}(u)\mu\\
    &\qquad\qquad+e^{-t\rho_{Q}(X)}\delta_{\emptyset}(D)
        \Big(R^{\vee}\big(\exp(tX)\big)\mu-R^{\vee}\big(\exp(tX)\big)\mu_{\emptyset}\Big).
\end{align*}
In view of (\ref{eq e^-t(rho+epsilon) exp(tX)mu ->0}) and (\ref{eq Constant term limit formula}) in Theorem \ref{Thm constant term} the right-hand side converges to $0$ for $t\to\infty$. By \cite[Lemma 6.5]{DelormeKrotzSouaifi_ConstantTerm} this identifies $\delta_{\emptyset}(D)\Const_{\cO}(\xi:\lambda)(\mu)$ as the constant term of $D\mu$.
\end{proof}

Recall the maps $\pr_{v}$ from (\ref{eq Def pr_w}).
From Corollary \ref{Cor Description D'(Z,overline Q:xi:lambda) horospherical case} it is easily seen that for every $v\in\fN$ there exists an element $u\in S(\fa)/S(\fa)\fa_{\fh}\subseteq \Diff(Z_{\emptyset})$ so that the diagram
$$
\xymatrixcolsep{3pc}\xymatrix{
   \cD'(\overline{Q}:\xi:\lambda)^{H_{\emptyset}}\ar[r]^{R^{\vee}(u)}    &  \cD'(\overline{Q}:\xi:\lambda)^{H_{\emptyset}}\\ \\
  V_{\emptyset}^{*}(\xi)\ar[r]^-{\pr_{v}}\ar@<-2pt>[uu]^{\mu_{\emptyset}^{\circ}(\xi:\lambda)} &(V_{\xi}^{*})^{M_{Q}\cap vHv^{-1}}\ar@<-2pt>[uu]_{\mu_{\emptyset}^{\circ}(\xi:\lambda)\circ\iota_{v}}
}
$$
commutes. Note that $R^{\vee}(u)$ is contained in the center of $\Diff(Z_{\emptyset})$.
For $v\in\fN$ we define a map $r_{\emptyset,v}(\xi:\lambda):\Diff(Z_{\emptyset})\to\End\big(V_{\emptyset,v}^{*}(\xi)\big)$ similar to (\ref{eq Def r(xi,lambda)(D)}) by requiring that the identity
$$
D\mu_{\emptyset}^{\circ}(\xi:\lambda)(\eta_{\cO})
=\mu_{\emptyset}^{\circ}(\xi:\lambda)\Big( \big(r_{\emptyset,v}(\xi:\lambda)(D)\eta\big)_{\cO}\Big)
$$
holds for every $D\in \Diff(Z_{\emptyset})$, $\eta\in V^{*}_{\emptyset,v}(\xi)$ and $\cO\in(P\bs Z)_{\open}$.
As before $V^{*}_{\emptyset,v}(\xi)$ is considered here to be a subspace of $\bigoplus_{\cO\in(P\bs Z)_{\open}}V_{\emptyset}^{*}(\xi)$.
Now Lemma \ref{Lemma Const D=delta(D) Const} has the following immediate corollary.

\begin{Cor}\label{Cor r(delta D) const_w=const_w r(D)}
Let $\xi$ be a finite dimensional unitary representation of $M_{Q}$ and $\lambda\in i(\fa/\fa_{\fh})^{*}\setminus i\cS$. Then for every $D\in\Diff(Z)$ and $v\in\fN$
$$
\const_{v}(\xi:\lambda)\circ r(\xi:\lambda)(D)
=r_{\emptyset,v}(\xi:\lambda)\big(\delta_{\emptyset}(D)\big)\circ\const_{v}(\xi:\lambda).
$$
\end{Cor}

\begin{proof}[Proof of Proposition \ref{Prop distributions separated by inv diff operators}]
We only consider $\lambda\in i(\fa/\fa_{\fh})^{*}$ whose stabilizer in $\cN$ is equal to $\cZ$ and for which the implication for $v\in N_{G_{\C}}(\fj/\fj\cap \fh)$
$$
\Ad^{*}(v)\lambda\in i(\fa/\fa_{\fh})^{*}
\Rightarrow v\in N_{G_{\C}}\big((\fa/\fa_{\fh})^{*}\big)\cap N_{G_{\C}}\big((\ft/\ft\cap\fh)^{*}\big)
$$
holds.
It suffices to prove the proposition only for these $\lambda$ as all elements $i(\fa/\fa_{\fh})^{*}$ outside of a finite union of proper subspaces have these properties.

We claim that the $W_{Z,\C}$-orbit $W_{Z,\C}\cdot\lambda$ does not contain the points $\Ad^{*}(v_{w}^{-1})\lambda$ for $w\neq e\cW\in \cN/\cW$. To prove the claim we assume that there exists a $v\in W_{Z,\C}$ so that $v\cdot\lambda=\Ad^{*}(v_{w}^{-1})\lambda$ for some $w\in \cN/\cW$. We will prove the claim by showing that $w=e\cW$.

The assumption on $\lambda$ guarantees that $v$ normalizes $\fa/\fa_{\fh}$.
Let $N_{W_{Z,\C}}(\fa/\fa_{\fh})$ and $Z_{W_{Z,\C}}(\fa/\fa_{\fh})$ be the normalizer and centralizer, respectively, of $\fa/\fa_{\fh}$ in $W_{Z,\C}$.
By \cite[Theorem 9.5]{KnopKrotz_ReductiveGroupActions} the little Weyl group $W_{Z}$ of $Z$ is related to the little Weyl group $W_{Z,\C}$ of $Z_{\C}$ by
$$
W_{Z}
= N_{W_{Z,\C}}(\fa/\fa_{\fh})/Z_{W_{Z,\C}}(\fa/\fa_{\fh}).
$$
This identity is to be considered as an identity of finite reflection groups on $\fa/\fa_{\fh}$. It follows that $v\cdot \lambda\in W_{Z}\cdot\lambda=\Ad^{*}(\cW)\lambda$. Since the stabilizer of $\lambda$ in $\cN$ is by assumption equal to $\cZ$, it follows that $v_{w}\in \cW$, and hence $w=e\cW$. This proves the claim.

Let $v\in \cN$. After replacing $\lambda$ by $-\Ad^{*}(v^{-1})\lambda$ we may conclude from the claim that the $W_{Z,\C}$-orbit through $-\Ad^{*}(v^{-1})\lambda$ is disjunct from the $W_{Z,\C}$-orbits through $-\Ad^{*}({v'}^{-1})\lambda$ with $v'\in \cN\setminus v\cW$.

In view of Lemma \ref{Lemma D(Z_empty)-character equals lambda on fa} there exist $\nu_{1},\dots,\nu_{r}\in (\fa/\fa_{\fh})^{*}\oplus i(\ft/\ft\cap\fh)^{*}$ so that the $\Cen(Z_{\emptyset,\C})$-characters occurring in $\cD'(\overline{Q}:\xi:\lambda)^{H_{\emptyset}}$ are given by
$$
W_{Z,\C}\cdot \big(-\Ad(v_{w}^{-1})\lambda+\nu_{j}\big)
\qquad\big(w\in\cN/\cW, 1\leq j\leq r\big).
$$
The real subspaces $(\fa/\fa_{\fh})^{*}\oplus i(\ft/\ft\cap\fh)^{*}$ and $i(\fa/\fa_{\fh})^{*}\oplus (\ft/\ft\cap\fh)^{*}$ of $(\fj/\fj\cap \fh)^{*}_{\C}$ are stable under the action of $N_{G_{\C}}(\fj/\fj\cap \fh)$. Therefore, the $W_{Z,\C}$-orbits through $-\Ad^{*}(v^{-1})\lambda+\nu_{j}$ for $1\leq j\leq r$ are disjunct from the $W_{Z,\C}$-orbits through $-\Ad^{*}({v'}^{-1})\lambda+\nu_{j}$ for $1\leq j\leq r$ and $v'\in \cN\setminus v\cW$.

Let now $w\in\cN/\cW$. It follows that there exists a polynomial $p_{w}\in \C\big[-\rho_{B}+(\fj/\fj\cap\fh)^{*}_{\C}\big]^{W_{Z,\C}}$ so that
\begin{align*}
&p_{w}(-\Ad^{*}(v^{-1})\lambda+\nu_{j})=1,\qquad\big(1\leq j\leq r, v\in v_{w}\cW)\\
&p_{w}\big(-\Ad^{*}({v}^{-1})\lambda+\nu_{j}\big)=0\qquad\big(1\leq j\leq r, v\in\cN\setminus v_{w}\cW\big).
\end{align*}
Let $D_{w}:=\gamma_{Z}^{-1}\big(p_{w}(\lambda)\big)$. We claim that the differential operator $D$ has the desired properties.

To prove the claim, let $v\in \fN$.
Lemma \ref{Lemma Knop HC-hom diagram commutes}, Corollary \ref{Cor r(delta D) const_w=const_w r(D)} and the construction of $D$ guarantee that
\begin{align}\label{eq const_v circ r(D) formula}
\nonumber&\const_{v}(\xi:\lambda)\circ r(\xi:\lambda)(D_{w})\\
&\qquad=r_{\emptyset,v}(\xi:\lambda)\big(\delta(D_{w})\big)\circ \const_{v}(\xi:\lambda)
=\left\{
   \begin{array}{ll}
     \const_{v}(\xi:\lambda) & (v\cW=w), \\
     0 & (v\cW\neq w).
   \end{array}
 \right.
\end{align}
We now substitute $v=v_{w}$ with $w\in\cN/\cW$ and apply Corollary \ref{Cor Formula const_v_w}. We then have for $\eta\in V^{*}(\xi)$
$$
\Big(r(\xi:\lambda)(D_{w})\eta\Big)_{\cO}
=\left\{
   \begin{array}{ll}
     \eta_{\cO} & \big(\cO\in w\cdot (P\bs Z)_{\open}\big), \\
     0 & \big(\cO\notin w\cdot (P\bs Z)_{\open}\big).
   \end{array}
 \right.
$$
The proposition now follows by applying $\mu^{\circ}(\xi:\lambda)$ to both sides and using the identity $\mu^{\circ}(\xi:\lambda)\circ r(\xi:\lambda)(D_{w})=D_{w}\circ\mu^{\circ}(\xi:\lambda)$.
\end{proof}

\subsection{Plancherel decomposition of $L^{2}_{\mc}(Z)$}
\label{Subsection Most continuous part - Plancherel decomposition}
For a finite dimensional unitary representation $\xi$ of $M_{Q}$ and $\lambda\in i(\fa/\fa_{\fh})^{*}\setminus i\cS$ we define the Fourier transform
$$
\Ft f(\xi:\lambda)\in \Hom_{\C}\big(V^{*}(\xi^{\vee}),C^{\infty}(\overline{Q}:\xi:\lambda)\big)=V^{*}(\xi)\otimes C^{\infty}(\overline{Q}:\xi:\lambda)
$$
of a function $f\in\cD(Z)$ by
\begin{align*}
\Ft f(\xi:\lambda)\eta
=\int_{Z}f(g H)
    R^{\vee}(g)\Big(\mu^{\circ}(\xi^{\vee}:-\lambda)\eta\Big)\,dgH.
\end{align*}
Let $\langle\cdot,\cdot\rangle_{\xi^{\vee}}$ be the inner product on $V^{*}(\xi^{\vee})$ induced by the inner product on $V_{\xi^{\vee}}$, and let $\langle\cdot,\cdot\rangle_{\xi,\lambda}$ be the inner product on
$V^{*}(\xi)\otimes \Ind_{\overline{Q}}^{G}(\xi\otimes\lambda\otimes\1)$ induced by the inner products $\langle\cdot,\cdot\rangle_{\xi}$ and $\langle \cdot, \cdot\rangle_{\overline{Q},\xi,\lambda}$ on $V^{*}(\xi)$ and $\Ind_{\overline{Q}}^{G}(\xi\otimes\lambda\otimes\1)$, respectively.
We recall that $\widehat{M}_{Q,\mathrm{fu}}$ denotes the set of equivalence classes of finite dimensional unitary representations of $M_{Q}$ and that  $(\fa/\fa_{\fh})_{+}^{*}$ is a fundamental domain for the action of $\cN$ on $(\fa/\fa_{\fh})^{*}$.
Further, we recall that the Lebesgue measure on $i(\fa/\fa_{\fh})^{*}$ is normalized by requiring that
$$
\phi(e)
=\int_{i(\fa/\fa_{\fh})^{*}}\int_{A/(A\cap H)}\phi(a)a^{\lambda}\,d\lambda\
\qquad\Big(\phi\in\cD\big(A/(A\cap H)\big)\Big).
$$
We then have the following description of the Plancherel decomposition of $L^{2}_{\mc}(Z)$.

\begin{Thm}\label{Thm Plancherel Theorem}
The Fourier transform $f\mapsto\Ft f$ extends to a continuous linear operator
$$
L^{2}(Z)
    \to \widehat{\bigoplus_{\xi\in\widehat{M}_{Q,\mathrm{fu}}}}\int_{i(\fa/\fa_{\fh})^{*}}^{\oplus} V^{*}(\xi)\otimes\Ind_{\overline{Q}}^{G}(\xi\otimes\lambda\otimes \1)\,d\lambda.
$$
Moreover, for every $f_{1},f_{2}\in L^{2}_{\mc}(Z)$ we have
$$
\int_{Z}f_{1}(z)\overline{f_{2}(z)}\,dz
=\sum_{[\xi]\in \widehat{M}_{Q,\mathrm{fu}}}\dim(V_{\xi})
    \int_{i(\fa/\fa_{\fh})_{+}^{*}}\Big\langle\Ft f_{1}(\xi:\lambda),\Ft f_{2}(\xi:\lambda)\Big\rangle_{\xi,\lambda}\,d\lambda.
$$
\end{Thm}

\begin{proof}
By Corollary \ref{Cor Multiplicity space is V^*(xi)} the multiplicity space $\cM_{\xi,\lambda}$ is isomorphic to $V^{*}(\xi)$ for all $\xi\in \widehat{M}_{Q,\mathrm{fu}}$ and almost every $\lambda\in i(\fa/\fa_{\fh})^{*}$.
In view of (\ref{eq abstract Plancherel decomp L^2_mc}) it therefore suffices to show that for almost all $\lambda\in i(\fa/\fa_{\fh})^{*}$ we have the equality
\begin{equation}\label{eq Pl-inner product= xi inner product}
\langle\cdot,\cdot\rangle_{\Pl,\xi,\lambda}
=\dim(V_{\xi})\langle\cdot,\cdot\rangle_{\xi}
\end{equation}
of inner products on $V^{*}(\xi)$.
To prove this identity we will use Theorem \ref{Thm Plancherel Theorem horospherical case}, Corollary \ref{Cor Orthogonal decomposition V^*(xi)} and the Maa\ss-Selberg relations from \cite[Theorem 9.6]{DelormeKnopKrotzSchlichtkrull_PlancherelTheoryForRealSphericalSpacesConstructionOfTheBernsteinMorphisms}.

We fix a finite dimensional unitary representation $\xi$ of $M_{Q}$.
It follows from Theorem \ref{Thm Plancherel Theorem horospherical case} that the multiplicity space for the representation $\Ind_{\overline{Q}}^{G}\big(\xi\otimes\lambda\otimes \1\big)$ in the Plancherel decomposition of $L^{2}(Z_{\emptyset})$ is for almost all $\lambda\in i(\fa/\fa_{\fh})^{*}$ equal to $V_{\emptyset}^{*}(\xi)$. Moreover, the inner product on the multiplicity spaces induced by the Plancherel decomposition is the inner product $\dim(V_{\xi})\langle\cdot,\cdot\rangle_{\emptyset,\xi}$, where $\langle\cdot,\cdot\rangle_{\emptyset,\xi}$ is the inner product on $V_{\emptyset}^{*}(\xi)$ induced from the natural $M_{Q}$-invariant inner product on $V_{\xi}$. We recall the spaces $V_{\emptyset,v_{w}}^{*}(\xi)$ from (\ref{eq Def V_empty,w(xi)}). We view these spaces as subspaces of $V_{\emptyset}^{*}(\xi)$. The inner product  $\langle\cdot,\cdot\rangle_{v_{w},\xi}$ on $V^{*}_{\emptyset,v_{w}}(\xi)$ obtained by restriction equals the inner-product induced by the natural inner product on $V_{\xi}$.

By Corollary \ref{Cor Formula const_v_w} the kernel of $\const_{v_{w}}(\xi:\lambda)$ equals the direct sum of the subspaces $V^{*}_{w'}(\xi)$ with $w'\in \cN/\cW$, $w'\neq w$. In view of  Corollary \ref{Cor Orthogonal decomposition V^*(xi)} the latter equals the orthocomplement of $V^{*}_{w}(\xi)$ with respect to $\langle\cdot,\cdot\rangle_{\Pl,\xi,\lambda}$. Let $\const_{v_{w}}(\xi:\lambda)^{\dagger}$ be the dual map of $\const_{v_{w}}(\xi:\lambda)$ with respect to the inner products $\dim(V_{\xi})\langle\cdot,\cdot,\rangle_{v_{w},\xi}$ and $\langle\cdot,\cdot\rangle_{\Pl,\xi,\lambda}$ on $V_{\emptyset,v_{w}}^{*}(\xi)$ and $V^{*}(\xi)$, respectively. It then follows that the image of $\const_{v_{w}}(\xi:\lambda)^{\dagger}$ is equal to $V^{*}_{w}(\xi)$.
Moreover,  by the Maa\ss-Selberg relations from \cite[Theorem 9.6]{DelormeKnopKrotzSchlichtkrull_PlancherelTheoryForRealSphericalSpacesConstructionOfTheBernsteinMorphisms} the map $\const_{v_{w}}(\xi:\lambda)$ is a partial isometry for every $w\in\cN/\cW$ and  almost every $\lambda\in i(\fa/\fa_{\fh})^{*}$, i.e., $ \const_{v_{w}}(\xi:\lambda)^{\dagger}$ is for almost all $\lambda\in i(\fa/\fa_{\fh})^{*}$ a unitary map onto its image  $V^{*}_{w}(\xi)$. Since $\const_{v_{w}}(\xi:\lambda)$ is essentially the identity map,  the restrictions of $\dim(\xi)\langle\cdot,\cdot\rangle_{\xi}$ and $\langle\cdot,\cdot\rangle_{\Pl,\xi,\lambda}$ to the subspaces $V^{*}_{w}(\xi)$ with $w\in \cN/\cW$ coincide for almost every $\lambda\in i(\fa/\fa_{\fh})^{*}$. This proves the identity (\ref{eq Pl-inner product= xi inner product}) as the decomposition (\ref{eq Decom V^*(xi)}) is orthogonal with respect to both  $\dim(\xi)\langle\cdot,\cdot\rangle_{\xi}$ and $\langle\cdot,\cdot\rangle_{\Pl,\xi,\lambda}$ by Corollary \ref{Cor Orthogonal decomposition V^*(xi)}.
\end{proof}

\subsection{Corollaries I: regularity of the families of distributions}
\label{Subsection Most continuous part - Regularity}
\begin{Cor}\label{Cor mu^circ holomorphic on imaginary axis}
Let $\xi$ be a finite dimensional unitary representation of $M_{Q}$. For every $\eta\in V^{*}(\xi)$ the family of distributions $\lambda\mapsto\mu^{\circ}(\xi:\lambda)\eta$ is holomorphic on a neighborhood of $i(\fa/\fa_{\fh})^{*}$.
\end{Cor}

\begin{proof}
Let $\eta\in V^{*}(\xi)$.
By Theorem \ref{Thm Description D'(overline Q:xi:lambda)^H} the family $\lambda\mapsto\mu^{\circ}(\xi:\lambda)\eta$ is meromorphic. It therefore suffices to proof that the family does not have any singularities on $i(\fa/\fa_{\fh})^{*}$.

We aim for a contradiction an assume that $\lambda\mapsto \mu^{\circ}(\xi:\lambda)\eta$ has a singularity on $i(\fa/\fa_{\fh})^{*}$.
The poles of the family lie in view of Theorem \ref{Thm Description D'(overline Q:xi:lambda)^H} on a locally finite union of complex affine hyperplanes of the form
$$
\{\lambda\in(\fa/\fa_{\fh})^{*}_{\C}: \lambda(X)=a\}\quad\text{for some }X\in\fa\text{ and }a\in\R.
$$
Since the singular set of a meromorphic function on $(\fa/\fa_{\fh})_{\C}^{*}$ is a union of complex analytic submanifolds of  $\C$-codimension $1$, it follows that there exists a subspace $\cH$ of $i(\fa/\fa_{\fh})^{*}$ of codimension $1$ so that $\lambda\mapsto \mu^{\circ}(\xi:\lambda)$ is singular on $\cH$. For every $f\in\cD(Z)$ the assignment $\lambda\mapsto \Ft f(\xi^{\vee}:\lambda)\eta$ defines a meromorphic function on $(\fa/\fa_{\fh})^{*}_{\C}$.  Furthermore, there exist functions $f\in\cD(Z)$ so that $\Ft f(\xi^{\vee}:\lambda)\eta$ is singular on $\cH$. It follows that $\lambda\mapsto\Ft f(\xi^{\vee}:\lambda)\eta $ is not square integrable. This is however in contradiction with Theorem \ref{Thm Plancherel Theorem}.
\end{proof}

\subsection{Corollaries II: refined Maa{\ss}-Selberg relations}
\label{Subsection Most continuous part - Maass-Selberg relations}
The Maa\ss-Selberg relations from \cite[Theorem 9.6]{DelormeKnopKrotzSchlichtkrull_PlancherelTheoryForRealSphericalSpacesConstructionOfTheBernsteinMorphisms} state that each of the maps
$$
\const_{v}(\xi:\lambda)
:V^{*}(\xi)\to V_{\emptyset,v}^{*}(\xi)
\qquad\big(v\in\fN\big).
$$
is a partial isometry, i.e., its Hermitian dual of $\const_{v}(\xi:\lambda)$ with respect to the inner products on $V^{*}(\xi)$ and $V_{\emptyset,v}^{*}(\xi)$, induced by the Plancherel decompositions Theorems \ref{Thm Plancherel Theorem} and \ref{Thm Plancherel Theorem horospherical case}, is a unitary isometry. In view of the Theorems \ref{Thm Plancherel Theorem horospherical case} and \ref{Thm Plancherel Theorem} these inner products are up to factor $\dim (V_{\xi})$ equal to the inner products induced by the inner product on $V_{\xi}$.
We can now refine the Maa\ss-Selberg relations from \cite{DelormeKnopKrotzSchlichtkrull_PlancherelTheoryForRealSphericalSpacesConstructionOfTheBernsteinMorphisms} for the most continuous part of $L^{2}(Z)$ as follows.

\begin{Cor}\label{Cor Maass-Selberg relations}
Let $\xi$ be a finite dimensional unitary representation of $M_{Q}$, $\lambda\in i(\fa/\fa^{*})$ and $v\in \fN$. Then
$$
\const_{v}(\xi:\lambda)\big|_{V^{*}_{w}(\xi)}=0\qquad \big(w\in \cN/\cW, v\cW\neq w\big).
$$
Moreover, if $w=v\cW$, then
$$
\const_{v}(\xi:\lambda)\big|_{V^{*}_{w}(\xi)}:V^{*}_{w}(\xi)\to V^{*}_{\emptyset,v}(\xi)
$$
is a unitary map.
\end{Cor}

\begin{proof}
The assertions follow Proposition \ref{Prop distributions separated by inv diff operators}, (\ref{eq const_v circ r(D) formula}) and the Maa\ss-Selberg relations from \cite[Theorem 9.6]{DelormeKnopKrotzSchlichtkrull_PlancherelTheoryForRealSphericalSpacesConstructionOfTheBernsteinMorphisms}.
\end{proof}

The Maa\ss-Selberg relations from Corollary \ref{Cor Maass-Selberg relations} are reflected in the symmetries of the combined constant term  map
$$
\const(\xi:\lambda):V^{*}(\xi)\to \bigoplus_{\cO\in (P\bs Z)_{\open}}V_{\emptyset}^{*}(\xi)
$$
given by
$$
\Big(\const(\xi:\lambda)\eta\Big)_{\cO}=\const_{\cO}(\xi:\lambda)\eta
\qquad\big(\eta\in V^{*}(\xi), \cO\in (P\bs Z)_{\open}\big).
$$
Note that $\const(\xi:\lambda)$ decomposes according to the decomposition $V_{\emptyset}^{*}(\xi)=\bigoplus_{v\in \fN}V_{\emptyset,v}^{*}(\xi)$ as
$$
\Big(\const(\xi:\lambda)\eta\Big)_{v}
=\const_{v}(\xi:\lambda)\eta\in V_{\emptyset,v}^{*}(\xi)
\qquad\big(\eta\in V^{*}(\xi), v\in \fN\big).
$$

To describe the symmetries we first introduce an action of $\cN/\cZ$ on $\fN$. For $w\in \cN/\cZ$ and $v\in \fN\subseteq\cN$ we define $w\cdot v$ to be the element in $\fN$ determined by the identity in $\cN/\cZ$
$$
(w\cdot v)\cZ
=vw^{-1},
$$
i.e., $w\cdot v$ is the representative of $vw^{-1}\in\cN/\cZ$ in $\fN$.

For $w\in \cN/\cZ$ we now define the scattering operator
$$
\cS_{w}(\xi:\lambda)
:=\sum_{v\in \fN}\const_{w\cdot v}(\xi:\lambda)\circ\const_{v}(\xi:\lambda)^{\dagger}
:\bigoplus_{\cO\in (P\bs Z)_{\open}}V_{\emptyset}^{*}(\xi)\to\bigoplus_{\cO\in (P\bs Z)_{\open}}V_{\emptyset}^{*}(\xi).
$$

We then have the following immediate corollary of Corollary \ref{Cor Maass-Selberg relations}.

\begin{Cor}\label{Cor Scattering operators}
Let $\xi$ be a finite dimensional unitary representation of $M_{Q}$ and $\lambda\in i(\fa/\fa_{\fh})^{*}\setminus i\cS$. The following hold true.
\begin{enumerate}[(i)]
\item $\cS_{w}(\xi:\lambda)=0$ for every $w\in (\cN/\cZ)\setminus W_{Z}$.
\item The assignment $W_{Z}\ni w\mapsto \cS_{w}(\xi:\lambda)$ defines a unitary representation of $W_{Z}$ on
    $\bigoplus_{\cO\in (P\bs Z)_{\open}}V_{\emptyset}^{*}(\xi)$.
\item For every $w\in W_{Z}$ we have $\cS_{w}(\xi:\lambda)\circ \const(\xi:\lambda)=\const(\xi:\lambda)$.
\end{enumerate}
\end{Cor}

\begin{Rem}
In \cite{Delorme_ScatteringOperators} scattering operators were defined under the restricting assumption that $G$ is split, but for all boundary degenerations, not just for the horospherical boundary degeneration $Z_{\emptyset}$ as we do here.
\end{Rem}

We finish this section with a description of the scattering operators in terms of the action of standard intertwining operators on $\cD'(\overline{Q}:\xi:\lambda)^{H}$, or rather in terms of the induced action on the parameter spaces $V^{*}(\xi)$. We first define for $v\in\cN$ the normalized $\cB$-matrix
$$
\cB_{v}^{\circ}(\xi:\lambda):V^{*}(\xi)\to V^{*}(v\cdot\xi)
$$
by
$$
\cB_{v}^{\circ}(\xi:\lambda)
:=\frac{1}{\gamma(v^{-1}\overline{Q}v:\overline{Q}:\xi:\lambda)}
    \beta(v\cdot\xi:\Ad^{*}(v)\lambda)\circ \cB_{v}(Q:\xi:\lambda)\circ\beta(\xi:\lambda)^{-1}.
$$
The normalized $\cB$-matrices are characterized by the fact that the diagram
$$
\xymatrixcolsep{8pc}\xymatrixrowsep{3pc}\xymatrix{
   \cD'(\overline{Q}:\xi:\lambda)^{H} \ar[r]^{\cI^{\circ}_{v}(\xi:\lambda)}
        & \cD'(\overline{Q}:v\cdot\xi:\Ad^{*}(v)\lambda)^{H}  \\
   V^{*}(\xi)\ar@<-2pt>[u]^{\mu^{\circ}(\xi:\lambda)} \ar[r]^{\cB_{v}^{\circ}(\xi:\lambda)}
       & V^{*}(v\cdot\xi)\ar@<-2pt>[u]_{\mu^{\circ}\big(v\cdot \xi: \Ad^{*}(v)\lambda\big)}
}
$$
commutes.
If $v_{1},v_{2}\in \cN$ then the identity
$$
\cI^{\circ}_{v_{1}}\big(v_{2}\cdot\xi:\Ad^{*}(v_{2})\lambda\big)\circ \cI^{\circ}_{v_{2}}(\xi:\lambda)
=\cI^{\circ}_{v_{1}v_{2}}(\xi:\lambda)
$$
implies
\begin{equation}\label{eq group law normalized B-matrices}
\cB^{\circ}_{v_{1}}\big(v_{2}\cdot\xi:\Ad^{*}(v_{2})\lambda\big)\circ \cB^{\circ}_{v_{2}}(\xi:\lambda)
=\cB^{\circ}_{v_{1}v_{2}}(\xi:\lambda).
\end{equation}
Each normalized intertwining operator $\cI^{\circ}_{v}(\xi:\lambda)$ is unitary and hence is a unitary map between the multiplicity spaces $\cM_{\xi,\lambda}$ and $\cM_{v\cdot\xi,\Ad^{*}(v)\lambda}$. Furthermore,  $\mu^{\circ}(\xi:\lambda)$ is in view of Theorem \ref{Thm Plancherel Theorem} a unitary map from $V^{*}(\xi)$ (equipped with $\dim(V_{\xi})$ times the inner product induced by the one on $V_{\xi}$) to $\cM_{\xi,\lambda}$. Therefore, the normalized $\cB$-matrices $\cB_{v}^{\circ}(\xi:\lambda)$ are unitary.

For $v\in \fN$ we define the map
$$
j_{v}(\xi):V^{*}(v^{-1}\cdot\xi)\to V_{\emptyset,v}^{*}(\xi)\subseteq\bigoplus_{\cO\in (P\bs Z)_{\open}}V_{\emptyset}^{*}(\xi)
$$
by
$$
\big(j_{v}(\xi)\eta\big)_{\cO}
:=\eta_{\cO,e}
\qquad \big(\eta\in V^{*}(v^{-1}\cdot\xi), \cO\in (P\bs Z)_{\open}\big).
$$
The dual map
$$
j_{v}(\xi)^{\dagger}:\bigoplus_{\cO\in (P\bs Z)_{\open}}V_{\emptyset}^{*}(\xi)\to V_{e\cW}^{*}(v^{-1}\cdot\xi)\subseteq V^{*}(v^{-1}\cdot \xi)
$$
is given by
$$
j_{v}(\xi)^{\dagger}\big|_{V_{\emptyset,v'}^{*}(\xi)}=0\qquad(v'\in\fN, v'\neq v)
$$
and
$$
\big(j_{v}(\xi)^{\dagger}\eta\big)_{\cO}=\eta_{\cO}
\qquad\big(\eta\in V_{\emptyset,v}^{*}(\xi), \cO\in (P\bs Z)_{\open}\big).
$$
Now the scattering maps are given by the following.

\begin{Cor}\label{Cor Formula scattering operators}
Let $\xi$ be a finite dimensional unitary representation of $M_{Q}$ and let $\lambda\in i(\fa/\fa_{\fh})^{*}\setminus i\cS$.
Then
$$
\cS_{w}(\xi:\lambda)
=\sum_{v\in \fN} j_{w\cdot v}(\xi)\circ \cB^{\circ}_{(w\cdot v)^{-1}v}(\xi:\lambda)\circ j_{v}(\xi)^{\dagger}
\qquad(w\in W_{Z}).
$$
\end{Cor}

\begin{proof}
Let $w\in\cN/\cW$ and let $v\in \fN$ be so that $v\cW=w$.
In view of Proposition \ref{Prop Const mu formula} we have
$$
\const_{v}(\xi:\lambda)
=\const_{e}\big(v^{-1}\cdot\xi:\Ad^{*}(v^{-1})\lambda\big)\circ\cB^{\circ}_{v^{-1}}(\xi:\lambda)
\qquad(v\in\fN).
$$
Moreover, by Corollary \ref{Cor Formula const_v_w}
$$
\const_{e}\big(v^{-1}\cdot\xi:\Ad^{*}(v^{-1})\lambda\big)
=j_{v}(\xi),
$$
and hence
$$
\const_{v}(\xi:\lambda)
=j_{v}(\xi)\circ \cB^{\circ}_{v^{-1}}(\xi:\lambda)
\qquad(v\in\fN).
$$
The assertion now follows from the unitarity of the normalized $\cB$-matrices and (\ref{eq group law normalized B-matrices}).
\end{proof}

\section*{Appendices}
\addcontentsline{toc}{section}{Appendices}
\renewcommand{\thesubsection} {\Alph{subsection}}
\numberwithin{Thm}{subsection}
\numberwithin{equation}{subsection}

\subsection*{Appendix A: Wavefront spaces}
\setcounter{subsection}{1}
\addcontentsline{toc}{subsection}{\protect\numberline{\thesubsection} Wavefront spaces}
\setcounter{Thm}{0}
\setcounter{equation}{0}

The space $Z$ is called wavefront if the compression cone of $Z$ is given by
$$
\cC
=\fa^{-}+\fa_{\fh}.
$$
A main class of examples of wavefront spaces is the class of reductive symmetric spaces, i.e., the spaces $Z=G/H$ with $H$ an open subgroup of the fixed point subgroup of an involutive automorphism of $G$.

For the most continuous part of $L^{2}(Z)$ the $P$-orbits $\cO$ of maximal rank with $\fa_{\cO}=\fa_{\fh}$ are of relevance. The open $P$-orbits satisfy this condition.
In view of the following proposition, the open $P$-orbits are all orbits with $\fa_{\cO}=\fa_{\fh}$ in case $Z$ is wavefront.
We recall the groups $\cN$, $\cZ$ and $\cW$ from (\ref{eq def cN}), (\ref{eq def cZ}) and (\ref{eq def cW}), respectively.

\begin{Prop}\label{Prop wavefront}
Assume that $Z$ is wavefront. Then $\cW=\cN$. In particular, the little Weyl group is equal to
$$
W_{Z}
=\cN/\cZ.
$$
\end{Prop}

Theorem \ref{Thm properties of W-action} and Proposition \ref{Prop wavefront} have the following corollary.

\begin{Cor}\label{Cor wavefront}
Assume that $Z$ is wavefront and let $\cO\in (P\bs Z)_{\max}$. Then $\cO$ is open if and only if $\fa_{\cO}=\fa_{\fh}$.
\end{Cor}

\begin{proof}[Proof of Proposition \ref{Prop wavefront}]
As $\cW$ is a subgroup of $\cN$, we only have to prove the inclusion $\cN\subseteq \cW$. Let $w\in \cN$. From the fact that  $\overline{\cC}/\fa_{\fh}$ is a fundamental domain for action of  the little Weyl group on $\fa/\fa_{\fh}$ it follows that there exists a $v\in\cW$ so that $\Ad(v^{-1}w^{-1})\fa^{-}\cap\cC\neq \emptyset$. After replacing $w$ by $wv$ we may thus assume that $\Ad(w)\cC\cap \fa^{-}\neq\emptyset$. We may further adjust $w$ by multiplying it from the right by an element from $\cZ$ and assume that (\ref{eq wSigma^+ cap -Sigma^+=wSigma(Q) cap -Sigma}) holds. It now suffices to prove that $w\in MA$. As the stabilizer of $\rho_{P}$ in $N_{G}(\fa)$ is equal to $MA$, it is thus enough to show that $w$ stabilizes $\rho_{P}$.

Since
$$
\big(\Ad(w)\fa^{+}+\fa_{\fh}\big)\cap\fa^{+}
=-\big(\Ad(w)\cC\cap\fa^{-}\big)
$$
is open and nonempty, its dual cone
\begin{equation}\label{eq dual cone of Ad(w)cC cap fa^-(P) in wavefront case}
\Big(\big(\sum_{\alpha\in\Sigma(wPw^{-1})}\R_{\geq0}\alpha\big)\cap(\fa/\fa_{\fh})^{*}\Big)
    +\sum_{\alpha\in\Sigma(P)}\R_{\geq0}\alpha
\end{equation}
is proper. Note that
$$
\rho_{P}-\Ad^{*}(w)\rho_{P}
=\sum_{\alpha\in\Sigma(P)\cap-\Sigma(wPw^{-1})}\dim(\fg_{\alpha})\alpha.
$$
Since $w$ normalizes $\fa_{\fh}$, it follows from Remark \ref{Rem L_Q roots} that $w$ normalizes $\fl_{Q}$.
In view of (\ref{eq wSigma^+ cap -Sigma^+=wSigma(Q) cap -Sigma}) we have
$$
\Sigma(P)\cap-\Sigma(wPw^{-1})
=\Sigma(P)\cap-\Sigma(wQw^{-1})
=\Sigma(Q)\cap-\Sigma(wQw^{-1}),
$$
and hence
$$
\rho_{P}-\Ad^{*}(w)\rho_{P}
=\rho_{Q}-\Ad^{*}(w)\rho_{Q}.
$$
As $Z$ is unimodular, we have
$$
\rho_{Q}(X)
=\tr\ad(X)\big|_{\fn_{Q}}
=-\tr\ad(X)\big|_{\fm+\fa+\fh}
=-\tr\ad(X)\big|_{\fh}
=0\qquad(X\in\fa_{\fh}).
$$
Therefore, $\rho_{Q}\in(\fa/\fa_{\fh})^{*}$ and, as $w$ normalizes $\fa_{\fh}$, also $\Ad(w)^{*}\rho_{Q}\in(\fa/\fa_{\fh})^{*}$.
It follows that
$$
\Ad^{*}(w)\rho_{P}-\rho_{P}
\in\big(\sum_{\alpha\in\Sigma(wPw^{-1})}\R_{\geq0}\alpha\big)\cap(\fa/\fa_{\fh})^{*}
$$
and
$$
\rho_{P}-\Ad^{*}(w)\rho_{P}
\in \sum_{\alpha\in\Sigma(P)}\R_{\geq0}\alpha.
$$
Since (\ref{eq dual cone of Ad(w)cC cap fa^-(P) in wavefront case}) is proper and contains both $\Ad^{*}(w)\rho_{P}-\rho_{P}$ and $\rho_{P}-\Ad^{*}(w)\rho_{P}$, we conclude that $\rho_{P}=\Ad^{*}(w)\rho_{P}$. This proves the claim.
\end{proof}

\subsection*{Appendix B: Intertwining operators}
\setcounter{subsection}{2}
\addcontentsline{toc}{subsection}{\protect\numberline{\thesubsection} Intertwining operators}
\setcounter{Thm}{0}
\setcounter{equation}{0}
Let $S$ be a parabolic subgroup of $G$ with Langlands decomposition $S=M_{S}A_{S}N_{S}$. Further, let $\xi$ be a representation of $M_{S}$ on a Hilbert space $V_{\xi}$ and $\lambda\in\fa_{S,\C}^{*}$.
In this appendix we will be concerned with a description of the action of standard intertwining operators on $\cD'(S:\xi:\lambda)$.
In the course of this appendix we will prove all assertions in Proposition \ref{Prop int formula for a(S_2:S_1:xi:lambda)}.
We begin by introducing some spaces of functions.

Recall the map $\Iwasawa_{S}:G\to A_{S}$ which is given by
$$
x\in N_{S}\Iwasawa_{S}(x)M_{S}K\qquad(x\in G).
$$
We write $\cL_{S,\xi,\lambda}$ for the space of equivalence classes of measurable functions $\phi:G\to V_{\xi}$ such that
$$
x\mapsto \Iwasawa_{S}^{-\Re\lambda+\rho_{S}}(x)\|\phi(x)\|_{\xi}
$$
is integrable. Here two functions are equivalent if and only if only differ on a set of measure $0$.  We endow $\cL_{S,\xi,\lambda}$ with the norm
$$
\phi \mapsto \int_{G}\|\Iwasawa_{S}^{-\lambda+\rho_{S}}(x)\phi(x)\|_{\xi}\,dx<\infty.
$$
With this norm $\cL_{S,\xi,\lambda}$ is a Banach space.

\begin{Lemma}\label{Lemma equivalence of A_S^(-lambda+rho) under right translations}
For every compact subset $C$ of $G$ there exists a constant $c>0$ such that for every $g\in C$ and $x\in G$
$$
c^{-1}|\Iwasawa_{S}^{-\lambda+\rho_{S}}(xg)|
\leq |\Iwasawa_{S}^{-\lambda+\rho_{S}}(x)|
\leq c|\Iwasawa_{S}^{-\lambda+\rho_{S}}(xg)|.
$$
\end{Lemma}

\begin{proof}
Let $C$ be a compact subset of $G$, let $g\in C$ and let $x\in G$. Then
$$
\Iwasawa_{S}(xg)
\in \Iwasawa_{S}(x)\Iwasawa_{S}(Kg).
$$
Since $KC$ is a compact subset of $G$ and $\Iwasawa_{S}^{\pm\lambda\pm\rho_{S}}$ is continuous, there exists a constant $c>0$ so that
$$
c\geq|\Iwasawa_{S}^{\mp\lambda\pm\rho_{S}}(y)|
\qquad (y\in KC).
$$
With this constant $c$ the desired inequalities hold.
\end{proof}

It follows from Lemma \ref{Lemma equivalence of A_S^(-lambda+rho) under right translations} that $\cL_{S,\xi,\lambda}$ is invariant under right translations by elements of $G$.
We write $R$ for the right-regular representation of $G$ on $\cL_{S,\xi,\lambda}$.

\begin{Prop}
The representation $\big(R,\cL_{S,\xi,\lambda}\big)$ is a continuous Banach representation.
\end{Prop}

\begin{proof}
The proof is the same as the proof for \cite[Proposition 2.9]{Kuit_SupportTheorem}.
\end{proof}

Let $\cV_{S,\xi,\lambda}$ be the space of smooth $V_{\xi}$-valued functions on $G$ that represent a smooth vector for $R$ in $\cL_{S,\xi,\lambda}$. The local Sobolev lemma ensures that every smooth vector in $\cL_{S,\xi,\lambda}$ can indeed be represented by a smooth $V_{\xi}$-valued function. See also \cite[Theorem 5.1]{Poulsen_OnSmoothVectorsAndIntertwiningBilinearForms}. We endow $\cV_{S,\xi,\lambda}$ with the unique Fr\'echet topology so that the natural bijection $\cV_{S,\xi,\lambda}\to\cL_{S,\xi,\lambda}^{\infty}$ is a topological isomorphism.
Note that
$$
\cV_{S,\xi,\lambda}
=\big\{\phi\in \cE(G,V_{\xi}):\int_{G}\|\Iwasawa_{S}^{-\lambda+\rho_{S}}(x)\,R(u)\phi(x)\|_{\xi}\,dx<\infty\text{ for every }u\in\cU(\fg)\big\}.
$$

\begin{Lemma}\label{Lemma map V_(S,xi,lambda) -> C^infty(S:xi:lambda)}
For $\phi\in\cV_{S,\xi,\lambda}$ the function
$$
G\to V_{\xi};
\quad
x\mapsto\int_{M_{S}}\int_{A_{S}}\int_{N_{S}} a^{-\lambda+\rho_{S}}\xi(m^{-1}) \phi(manx)\,dn\,da\,dm
$$
is defined by absolutely convergent integrals and forms an element of $C^{\infty}(S:\xi:\lambda)$. Moreover, the map
$$
\cV_{S,\xi,\lambda}\to C^{\infty}(S:\xi:\lambda)
$$
thus obtained is $G$-equivariant and continuous.
\end{Lemma}

\begin{proof}
Let $\phi\in\cV_{S,\xi,\lambda}$. By Fubini's theorem the integral
\begin{equation}\label{eq T def}
\cT\phi(x)
:=\int_{M_{S}}\int_{A_{S}}\int_{N_{S}} a^{-\lambda+\rho_{S}}\xi(m^{-1}) \phi(manx)\,dn\,da\,dm
\end{equation}
is absolutely convergent for almost every $x\in K$ and the function $\cT\phi:K\to V_{\xi}$ thus obtained is integrable.
Since $L(man)\cT\phi=a^{-\lambda-\rho_{S}}\xi(m^{-1})\cT\phi$, it follows that the integral $\cT\phi(x)$ is absolutely convergent for almost every $x\in G$ and the function $\cT\phi:G\to V_{\xi}$ thus obtained is locally integrable.

We claim that the integral (\ref{eq T def}) is in fact absolutely convergent for every $x\in G$ and that $\cT\phi$ is a smooth function for every $\phi\in\cV_{S,\xi,\lambda}$.
From \cite[Th\'eor\`eme 3.3]{DixmierMalliavin_FactorisationsDeFunctionsEtDeVecteursIndefinimentDifferentiables} it follows that
\begin{equation}\label{eq cV=span(f*phi)}
\cV_{S,\xi,\lambda}
=\spn\big\{\pi(f)\phi:f\in\cD(G),\phi\in\cV_{S,\xi,\lambda}\big\}.
\end{equation}
Therefore, to prove the claim it suffices to show that for every $f\in\cD(G)$ and $\phi\in\cV_{S,\xi,\lambda}$ the function $\cT\big(\pi(f)\phi\big)$ is smooth. Let $f\in\cD(G)$ and $\phi\in\cV_{S,\xi,\lambda}$. Then, by Fubini's theorem
\begin{align*}
\cT(\pi(f)\phi)(x)
&=\int_{M_{S}}\int_{A_{S}}\int_{N_{S}} a^{-\lambda+\rho_{S}}\xi(m^{-1})\int_{G}f(y) \phi(manxy)\,dy\,dn\,da\,dm\\
&=\int_{G}f(x^{-1}y)\Big(\int_{M_{S}}\int_{A_{S}}\int_{N_{S}} a^{-\lambda+\rho_{S}}\xi(m^{-1}) \phi(many)\,dn\,da\,dm\Big)\,dy\\
&=\int_{G}f(x^{-1}y)\cT\phi(y)\,dy.
\end{align*}
Since $\cT\phi$ is locally integrable, the last expression defines a smooth function in $x\in G$. This proves the claim. Note that it follows from the claim that $\cT\phi\in C^{\infty}(S:\xi:\lambda)$ for every $\phi\in\cV_{S,\xi,\lambda}$. This proves the first statement in the lemma.

The equivariance of $\cT$ is clear. It thus remains to prove the continuity.
Let $L^{1}(K:\xi)$ be the space of integrable functions $\phi:K\to V_{\xi}$ that satisfy
$$
\phi(mk)=\xi(m)\phi(k)
\qquad(m\in M\cap K, k\in K).
$$
As stated above, the restriction of $\cT\phi$ to $K$ is integrable for every $\phi\in\cV_{S,\xi,\lambda}$. Moreover,
$$
\int_{K}\|\cT\phi(k)\|_{\xi}\,dk
\leq\int_{G}\|\Iwasawa_{S}^{-\lambda+\rho_{S}}(x)\phi(x)\|_{\xi}\,dx.
$$
Therefore $\cT$ defines a continuous map $\cV_{S,\xi,\lambda}\to L^{1}(K:\xi)$ which intertwines $\pi\big|_{K}$ and the right regular representation of $K$ on $L^{1}(K:\xi)$. Since $\big(\pi\big|_{K},\cV_{S,\xi,\lambda}\big)$ is a smooth representation, $\cT$ in fact defines a continuous map $\cV_{S,\xi,\lambda}\to L^{1}(K:\xi)^{\infty}$. From the local Sobolev lemma it follows that there is a natural identification between $L^{1}(K:\xi)^{\infty}$ and the space $C^{\infty}(K:\xi)$ consisting of all smooth functions $f:K\to V_{\xi}$ such that
$$
f(mk)
=\xi(m)f(k)
\qquad(m\in M\cap K, k\in K).
$$
This identification is a topological isomorphism. See also \cite[Theorem 5.1]{Poulsen_OnSmoothVectorsAndIntertwiningBilinearForms}.
Finally, the restriction map $\phi\mapsto \phi\big|_{K}$ is a $K$-equivariant topological isomorphism between $C^{\infty}(S:\xi:\lambda)$ and $C^{\infty}(K:\xi)$. This proves the second claim in the lemma.
\end{proof}

It follows from Lemma \ref{Lemma map V_(S,xi,lambda) -> C^infty(S:xi:lambda)} that for every $\eta\in C^{\infty}(S:\xi:\lambda)'$ the right-hand side of (\ref{eq def theta}) defines  a continuous linear functional on $\cV_{S,\xi,\lambda}$. We thus conclude that every $\mu\in\cD'(S:\xi:\lambda)$ extends to a continuous linear function on $\cV_{S,\xi,\lambda}$.
In fact, the injection
\begin{equation}\label{eq Bijection V'_S,xi,lambda to D'(S:xi:lambda)}
\big\{\mu\in\cV_{S,\xi,\lambda}':\mu \text{ satisfies } (\ref{eq L(man)mu=a^lambda xi(m) mu})\big\}
\hookrightarrow \cD'(S:\xi:\lambda);
\quad \mu\mapsto \mu\big|_{\cD(G,V_{\xi})}
\end{equation}
is a bijection.

Now let $S_{1}$ and $S_{2}$ be parabolic subgroups such that $A_{S_{1}}=A_{S_{2}}\subseteq A$. We identify $\fa_{S,\C}^{*}$ by the subspace of $\fa_{\C}^{*}$ of elements that vanish on $\fa\cap\fm_{S}$.

\begin{Prop}\label{Prop phi mapsto int_N phi continuous map on L^1 level}
Let $\lambda\in\fa_{S,\C}^{*}$ satisfy
\begin{equation}\label{eq dominance condition lambda}
\langle\Re\lambda,\alpha\rangle>0\qquad\big(\alpha\in\Sigma(\fa:S_{2})\cap-\Sigma(\fa:S_{1})\big).
\end{equation}
For every $\phi\in \cL_{S_{2},\xi,\lambda}$ and almost every $x\in G$ the integral
$$
\int_{N_{S_{2}}\cap \overline{N}_{S_{1}}}\phi(nx)\,dn
$$
is absolutely convergent and the function $\int_{N_{S_{2}}\cap\overline{N}_{S_{1}}}\phi(n\dotvar)\,dn$ thus obtained represents an element of $\cL_{S_{1},\xi,\lambda}$. Moreover, the map
$$
\cL_{S_{2},\xi,\lambda}\to\cL_{S_{1},\xi,\lambda};
\quad\phi\mapsto\int_{N_{S_{2}}\cap\overline{N}_{S_{1}}}\phi(n\dotvar)\,dn
$$
is continuous.
\end{Prop}

\begin{proof}
Let $\phi\in\cL_{S_{2},\xi,\lambda}$.
In view of Fubini's theorem, it suffices to show that
\begin{equation}\label{eq continuity of intertwining operator}
\int_{G}\int_{N_{S_{2}}\cap \overline{N}_{S_{1}}}\Big\|\Iwasawa_{S_{1}}^{-\lambda+\rho_{S_{1}}}(x)\phi(nx)\Big\|_{\xi}\,dn\,dx
\leq c \int_{G}\|\Iwasawa_{S_{2}}^{-\lambda+\rho_{S_{2}}}(x)\phi(x)\|_{\xi}\,dx
\end{equation}
for some $c>0$.

Using the invariance of the Haar measure on $G$, we obtain
\begin{align*}
\int_{G}\int_{N_{S_{2}}\cap \overline{N}_{S_{1}}}\Big\|\Iwasawa_{S_{1}}^{-\lambda+\rho_{S_{1}}}(x)\phi(nx)\Big\|_{\xi}\,dn\,dx
&=\int_{N_{S_{2}}\cap \overline{N}_{S_{1}}}\int_{G}\big\|\Iwasawa_{S_{1}}^{-\lambda+\rho_{S_{1}}}(x)\phi(nx)\big\|_{\xi}\,dx\,dn\\
&=\int_{N_{S_{2}}\cap \overline{N}_{S_{1}}}\int_{G}\big\|\Iwasawa_{S_{1}}^{-\lambda+\rho_{S_{1}}}(nx)\phi(x)\big\|_{\xi}\,dx\,dn\\
&=\int_{G}\int_{N_{S_{2}}\cap \overline{N}_{S_{1}}}\big|\Iwasawa_{S_{1}}^{-\lambda+\rho_{S_{1}}}(nx)\big|\,dn\, \|\phi(x)\|_{\xi}\,dx.
\end{align*}
Note that
$$
\int_{N_{S_{2}}\cap \overline{N}_{S_{1}}}\big|\Iwasawa_{S_{1}}^{-\lambda+\rho_{S_{1}}}(nx)\big|\,dn
=c(S_{2}:S_{1}:-\Re\lambda)\Iwasawa_{S_{2}}^{-\Re\lambda+\rho_{S_{2}}}(x)\qquad(x\in G)
$$
where $c(S_{2}:S_{1}:\dotvar)$ is the partial $c$-function which is given by the absolutely convergent integral
$$
c(S_{2}:S_{1}:\nu)
=\int_{N_{S_{2}}\cap \overline{N}_{S_{1}}}\Iwasawa_{S_{1}}^{\nu+\rho_{S_{1}}}(n)\,dn
$$
in case $\lambda=-\nu\in\fa_{\C}^{*}$ satisfies (\ref{eq dominance condition lambda}).
Hence (\ref{eq continuity of intertwining operator}) holds with $c=c(S_{2}:S_{1}:-\Re\lambda)$.
This proves the proposition.
\end{proof}

\begin{Cor}\label{Cor phi mapsto int_N phi continuous map on cV level}
Let $\lambda\in\fa_{\C}^{*}$ satisfy (\ref{eq dominance condition lambda}).
For every $\phi\in \cV_{S_{2},\xi,\lambda}$ and every $x\in G$ the integral
\begin{equation}\label{eq I(S_2:S_1:xi:lambda) phi}
\int_{N_{S_{2}}\cap \overline{N}_{S_{1}}}\phi(nx)\,dx
\end{equation}
is absolutely convergent and the function $\int_{N_{S_{2}}\cap \overline{N}_{S_{1}}}\phi(n\dotvar)\,dx$ thus obtained is an element of $\cV_{S_{1},\xi,\lambda}$. Moreover, the map
\begin{equation}\label{eq phi mapsto int_N phi}
\cV_{S_{2},\xi,\lambda}\to\cV_{S_{1},\xi,\lambda};\quad\phi\mapsto\int_{N_{S_{2}}\cap \overline{N}_{S_{1}}}\phi(n\dotvar)\,dn
\end{equation}
is continuous.
\end{Cor}

\begin{proof}
It follows from Proposition \ref{Prop phi mapsto int_N phi continuous map on L^1 level} that $\phi\mapsto\int_{N_{S_{2}}\cap \overline{N}_{S_{1}}}\phi(n\dotvar)\,dn$ defines  a continuous map between the spaces of smooth vectors in $\cL_{S_{2},\xi,\lambda}$ and $\cL_{S_{1},\xi,\lambda}$ respectively. It therefore suffices to show that for every $\phi\in\cV_{S_{2},\xi,\lambda}$ and $x\in G$ the integral (\ref{eq I(S_2:S_1:xi:lambda) phi}) is absolutely convergent and the function$\int_{N_{S_{2}}\cap \overline{N}_{S_{1}}}\phi(n\dotvar)\,dn$ is smooth. In view of (\ref{eq cV=span(f*phi)}) it suffices to do this for $\phi$ of the form $\phi=\pi(f)\psi$ with $f\in\cD(G)$ and $\psi\in\cV_{S_{2},\xi,\lambda}$.

Let $f\in \cD(G)$, $\phi\in\cV_{S_{2},\xi,\lambda}$ and $x\in G$.
It follows from Proposition \ref{Prop phi mapsto int_N phi continuous map on L^1 level} that the integral
$$
\int_{G}f(x^{-1}y)\int_{N_{S_{2}}\cap\overline{N}_{S_{1}}}\psi(ny)\,dn\,dy
$$
is absolutely convergent. Moreover, it depends smoothly on $x$.  By Fubini's theorem this integral is equal to
$$
\int_{N_{S_{2}}\cap\overline{N}_{S_{1}}}\big(\pi(f)\psi\big)(nx)\,dn.
$$
This proves the corollary.
\end{proof}

We define
$$
\cA(S_{2}:S_{1}:\xi:\lambda)
:=\theta^{S_{2}}_{\xi,\lambda}\circ A(S_{1}:S_{2}:\xi:\lambda)^{*}\circ\omega^{S_{1}}_{\xi,\lambda}
$$
The following diagram commutes.
$$
\xymatrix{
   \cD'(S_{1}:\xi:\lambda) \ar[rr]^{\cA(S_{2}:S_{1}:\xi:\lambda)} \ar@<-2pt>[dd]_{\omega^{S_{1}}_{\xi,\lambda}}
        && \cD'(S_{2}:\xi:\lambda) \ar@<-2pt>[dd]_{\omega^{S_{2}}_{\xi,\lambda}} \\
\\
   C^{\infty}(S_{1}:\xi:\lambda)'\ar@<-2pt>[uu]_{\theta^{S_{1}}_{\xi,\lambda}} \ar[rr]^{A(S_{1}:S_{2}:\xi:\lambda)^{*}}
        && C^{\infty}(S_{2}:\xi:\lambda)' \ar@<-2pt>[uu]_{\theta^{S_{2}}_{\xi,\lambda}}
}
$$
We recall from (\ref{eq Bijection V'_S,xi,lambda to D'(S:xi:lambda)}) that every distribution $\mu\in\cD'(S:\xi:\lambda)$ extends to a continuous linear functional on $\cV_{S,\xi,\lambda}$. Therefore, if  $\lambda\in\fa_{\C}^{*}$ satisfies (\ref{eq dominance condition lambda}), then in view of Corollary \ref{Cor phi mapsto int_N phi continuous map on cV level} the assignment
$$
\phi\mapsto\mu\Big(\int_{N_{S_{2}}\cap\overline{N}_{S_{1}}}\phi(n\dotvar)\,dn\Big)
$$
defines for every $\mu\in\cD'(S_{1}:\xi:\lambda)$  a distribution in $\cD'(G,V_{\xi})$.

\begin{Prop}\label{Prop int formula for a(S_2:S_1:xi:lambda) - appendix}
Let $\lambda\in\fa_{\C}^{*}$ satisfy (\ref{eq dominance condition lambda}).
For every $\mu\in\cD'(S_{1}:\xi:\lambda)$ the distribution $\cA(S_{2}:S_{1}:\xi:\lambda)\mu\in\cD'(S_{2}:\xi:\lambda)$ is given by
$$
\big[\cA(S_{2}:S_{1}:\xi:\lambda)\mu\big](\phi)
=\mu\Big(\int_{N_{S_{2}}\cap\overline{N}_{S_{1}}}\phi(n\dotvar)\,dn\Big)
\qquad\big(\phi\in\cD(G,V_{\xi})\big).
$$
\end{Prop}

For the proof of the proposition we need the following lemma.

\begin{Lemma}\label{Lemma expression for omega phi}
Let $\phi\in C^{\infty}(S:\xi^{\vee}:-\lambda)$ and consider $\phi$ as an element of $\cD'(S:\xi:\lambda)$. Then
$$
\big(\omega^{S}_{\xi,\lambda}\phi\big)(f)
=\int_{K}\Big(\phi(k),f(k)\Big)\,dk
\qquad\big(f\in C^{\infty}(S:\xi:\lambda)\big).
$$
\end{Lemma}

\begin{proof}
Let $f\in C^{\infty}(S:\xi:\lambda)$. Then
\begin{align*}
\big(\omega^{S}_{\xi,\lambda}\phi\big)(f)
&=\int_{G}\Big(\phi(x),\psi_{0}(x) f(x)\Big)\,dx\\
&=\int_{M_{S}}\int_{A_{S}}\int_{N_{S}}a^{2\rho_{S}}\psi_{0}(man)\,dn\,da\,dm\int_{K}\Big(\phi(k),f(k)\Big)\,dk.
\end{align*}
The claim in the lemma now follows from the observation that
\begin{align*}
\int_{M_{S}}\int_{A_{S}}\int_{N_{S}}a^{2\rho_{S}}\psi_{0}(man)\,dn\,da\,dm
&=\int_{M_{S}}\int_{A_{S}}\int_{N_{S}}\int_{K}a^{2\rho_{S}}\psi_{0}(mank)\,dn\,da\,dm\,dk\\
&=\int_{G}\Iwasawa_{S}^{2\rho_{S}}(x)\psi_{0}(x)\,dx
=1.
\end{align*}
\end{proof}

\begin{proof}[Proof of Proposition \ref{Prop int formula for a(S_2:S_1:xi:lambda) - appendix}]
Since (\ref{eq phi mapsto int_N phi}) is continuous and $C^{\infty}(S_{1}:\xi^{\vee}:-\lambda)$ is a dense subspace of $\cD'(S_{1}:\xi:\lambda)$, it suffices to prove the identity only for $\mu\in C^{\infty}(S_{1}:\xi^{\vee}:-\lambda)$. Let $\mu$ be such a function and let $\phi\in \cD(G,V_{\xi})$. Then
\begin{align*}
&\big[\cA(S_{2}:S_{1}:\xi:\lambda)\mu \big](\phi)\\
&\quad=\omega_{\xi,\lambda}^{S_{1}}(\mu)\Big(x\mapsto\int_{M_{S}}\int_{A_{S}}\int_{N_{S_{2}}}\int_{N_{S_{1}}\cap \overline{N}_{S_{2}}}
    a^{-\lambda+\rho_{S_{1}}}\xi(m^{-1})\phi(man\overline{n}x)\,d\overline{n}\,dn\,da\,dm\Big).
\end{align*}
It follows from Lemma \ref{Lemma expression for omega phi} that the right-hand side is equal to
\begin{align*}
&\int_{K}\Big(\mu(k),\int_{M_{S}}\int_{A_{S}}\int_{N_{S_{2}}}\int_{N_{S_{1}}\cap \overline{N}_{S_{2}}}
    a^{-\lambda+\rho_{S_{1}}}\xi(m^{-1})\phi(man\overline{n}k)\,d\overline{n}\,dn\,da\,dm\Big)\,dk\\
&\quad=\int_{K}\int_{M_{S}}\int_{A_{S}}\int_{N_{S_{2}}}\int_{N_{S_{1}}\cap \overline{N}_{S_{2}}}
    a^{-\lambda+\rho_{S_{1}}}\Big(\mu(k),\xi(m^{-1})\phi(man\overline{n}k)\Big)\,d\overline{n}\,dn\,da\,dm\,dk.
\end{align*}
Since the multiplication maps
$$
\big(N_{S_{1}}\cap N_{S_{2}}\big)\times \big(\overline{N}_{S_{1}}\cap N_{S_{2}}\big)\to N_{S_{2}},
\quad \big(N_{S_{1}}\cap N_{S_{2}}\big)\times \big(N_{S_{1}}\cap \overline{N}_{S_{2}}\big)\to N_{S_{1}}
$$
are diffeomorphisms with Jacobian equal to $1$, we can rewrite this repeated integral as
\begin{align*}
&\int_{K}\int_{M_{S}}\int_{A_{S}}\int_{ \overline{N}_{S_{1}}\cap N_{S_{2}}}\int_{N_{S_{1}}}a^{-\lambda+\rho_{S_{1}}}
    \Big(\mu(k),\xi(m^{-1})\phi(ma\overline{n}nk)\Big)\,dn\,d\overline{n}\,da\,dm\,dk\\
&\quad=\int_{K}\int_{M_{S}}\int_{A_{S}}\int_{N_{S_{1}}}
    \Big(\mu(mank),\int_{N_{S_{2}}\cap \overline{N}_{S_{1}}}\phi(\overline{n}mank)d\overline{n}\Big)\,dn\,\,da\,dm\,dk\\
&\quad=\int_{G}
    \Big(\mu(x),\int_{N_{S_{2}}\cap \overline{N}_{S_{1}}}\phi(\overline{n}x)d\overline{n}\Big)\,dx.
\end{align*}
This proves the proposition.
\end{proof}


\def\adritem#1{\hbox{\small #1}}
\def\distance{\hbox{\hspace{3.5cm}}}
\def\apetail{@}
\def\addSayag{\vbox{
\adritem{E. Sayag}
\adritem{Department of Mathematics}
\adritem{Ben-Gurion University of the Negev}
\adritem{P.O.B. 653}
\adritem{Be'er Sheva 8410501}
\adritem{Israel}
\adritem{E-mail: sayage{\apetail}math.bgu.ac.il}
}
}
\def\addKuit{\vbox{
\adritem{J.~J.~Kuit}
\adritem{Institut f\"ur Mathematik}
\adritem{Universit\"at Paderborn}
\adritem{Warburger Stra{\ss}e 100}
\adritem{33089 Paderborn}
\adritem{Germany}
\adritem{E-mail: j.j.kuit{\apetail}gmail.com}
}
}
\mbox{}
\vfill
\hbox{\vbox{\addKuit}\vbox{\distance}\vbox{\addSayag}}

\end{document}